\documentclass{memo-l}

\usepackage{amsmath, amsfonts, amssymb, amsthm, faktor}

\usepackage[all]{xy}
\usepackage{tikz}
\usetikzlibrary{math}
\usetikzlibrary{patterns}
\usepackage{caption}
\usepackage{subfig}
\usepackage{cite}
\usepackage{caption}
\usepackage{multirow}
\usetikzlibrary{arrows,chains,matrix,positioning,scopes,decorations.pathreplacing,
decorations.pathmorphing,decorations.markings,arrows.meta}

\usepackage{chngcntr}
\counterwithout{figure}{chapter}

\usepackage{aliascnt}
\usepackage[colorlinks, linktocpage, allcolors=black,breaklinks]{hyperref}

\usepackage{comment}
\usepackage{enumitem}

\newcommand{\G}{G}
\newcommand{\per}{\text{per}}

\newcommand\CL{{\mathcal L}}
\newcommand\CM{{\mathcal M}}

\newcommand\CR{{\mathcal R}}

\newcommand{\CG}{\mathcal{G}}

\newcommand\fA{{\mathfrak A}}
\newcommand\fB{{\mathfrak B}}

\newcommand\fm{{\mathfrak m}}

\newcommand\BBN{{\mathbb N}}

\newcommand\BBQ{{\mathbb Q}}
\newcommand\BBR{{\mathbb R}}
\newcommand\BBZ{{\mathbb Z}}
\newcommand{\N}{\BBN}
\newcommand{\Z}{\BBZ}
\newcommand{\Q}{\BBQ}
\newcommand{\R}{\BBR}

\newcommand{\sms}{\setminus}

\newcommand{\ssc}{A} 

\newcommand{\ran}{\text{\rm range}}

\newcommand{\res}{\restriction}

\newcommand{\lp}{\Pi}

\newcommand{\gnpq}{\Gamma_{n,p,q}}
\newcommand{\Gnpq}{G_{n,p,q}}
\newcommand{\gopq}{\Gamma_{1,p,q}}

\newcommand{\dom}{\text{dom}}

\newcommand{\ocl}[1]{\overline{#1}}
\newcommand{\conc}{\mathbin{\raisebox{3pt}{$\scriptstyle{\smallfrown}$}}}

\newcommand{\acts}{\curvearrowright}

\newcommand{\fzs}{F(2^{\Z^2})}
\newcommand{\fzn}{F(2^{\Z^n})}

\newcommand{\bsft}{\mathfrak{b}}

\pgfdeclarelayer{edges}
\pgfdeclarelayer{nodes}
\pgfdeclarelayer{annotations}
\pgfsetlayers{edges,nodes,annotations,main}

\newcommand{\Fof}[1]{F(2^{#1})}
\newcommand{\FofZZ}{\Fof{\Z^2}}

\newcommand{\ElemTo}[1]{\stackrel{#1}{\longrightarrow} }

\newcommand{\Tcaca}{T_{ca=ac}}
\newcommand{\Tdada}{T_{da=ad}}
\newcommand{\Tcbcb}{T_{cb=bc}}
\newcommand{\Tdbdb}{T_{db=bd}}
\newcommand{\Tdcadca}{T_{dca=acd}}
\newcommand{\Tcdacda}{T_{cda=adc}}
\newcommand{\Tcbacba}{T_{cba=abc}}
\newcommand{\Tcabcab}{T_{cab=bac}}
\newcommand{\Tcqadpa}{T_{c^qa=ad^p}}
\newcommand{\Tdpacqa}{T_{d^pa=ac^q}}
\newcommand{\Tcbpcaq}{T_{cb^p=a^qc}}
\newcommand{\Tcaqcbp}{T_{ca^q=b^pc}}

\newcommand{\Gcaca}{G_{ca=ac}}
\newcommand{\Gdada}{G_{da=ad}}
\newcommand{\Gcbcb}{G_{cb=bc}}
\newcommand{\Gdbdb}{G_{db=bd}}
\newcommand{\Gdcadca}{G_{dca=acd}}
\newcommand{\Gcdacda}{G_{cda=adc}}
\newcommand{\Gcbacba}{G_{cba=abc}}
\newcommand{\Gcabcab}{G_{cab=bac}}
\newcommand{\Gcqadpa}{G_{c^qa=ad^p}}
\newcommand{\Gdpacqa}{G_{d^pa=ac^q}}
\newcommand{\Gcbpcaq}{G_{cb^p=a^qc}}
\newcommand{\Gcaqcbp}{G_{ca^q=b^pc}}

\newcommand{\restrict}[1]{\!\!\upharpoonright_{#1}}

\newcommand{\signatureSep}{\colon}
\newcommand{\labelEdge}[1]{\stackrel{#1}{\longrightarrow}}

\newcommand{\todo}[1]{{\color{red}[[#1]]}}
\renewcommand{\todo}[1]{}

\numberwithin{section}{chapter}

\theoremstyle{plain}
\newtheorem{thmn}{Theorem}
\numberwithin{thmn}{section}

\newtheorem{lemn}[thmn]{Lemma}
\newtheorem{propn}[thmn]{Proposition}

\newtheorem{corn}[thmn]{Corollary}
\newtheorem{probn}{Problem}

\newtheorem*{ex}{Example}

\theoremstyle{definition}
\newtheorem*{fact}{Fact}
\newtheorem*{defn}{Definition}

\newtheorem*{rem}{Remark}

\newtheorem*{prob}{Problem}

\newtheorem{exn}[thmn]{Example}
\newtheorem{defnn}[thmn]{Definition}
\newtheorem{quesn}[thmn]{Question}

\newtheorem{remn}[thmn]{Remark}
\newtheorem{fctn}[thmn]{Fact}

\def\clap#1{\hbox to 0pt{\hss#1\hss}}

\def\mathclap{\mathpalette\mathclapinternal}

\def\mathclapinternal#1#2{%
\clap{$\mathsurround=0pt#1{#2}$}}

\newcommand{\floorfrac}[2]{\left\lfloor\frac{#1}{#2}\right\rfloor}
\newcommand{\binmod}{\mathbin{\mathrm{mod}}}

\tikzset{
  BdyNode/.style         ={circle, fill, inner sep=0pt, minimum size=3pt},
  nodeStyle/.style       ={circle, draw, fill=white, inner sep=0pt, minimum size=7pt},
  edgeStrongStyle/.style ={line width=3pt, line cap=rect},
  BoxLabelStyle/.style   ={scale=1.2},
  BoxStyle/.style        ={scale=1, anchor=north west, rectangle, draw},
  XBoxStyle/.style       ={BoxStyle, minimum width=1cm,   minimum height=1cm},
  ABoxStyle/.style       ={BoxStyle, minimum width=1cm,   minimum height=0.8cm},
  BBoxStyle/.style       ={BoxStyle, minimum width=1cm,   minimum height=1.2cm},
  CBoxStyle/.style       ={BoxStyle, minimum width=0.8cm, minimum height=1cm},
  DBoxStyle/.style       ={BoxStyle, minimum width=1.2cm, minimum height=1cm}
}

\tikzset{->-/.style={decoration={
  markings,
  mark=at position #1 with {\arrow[scale=2]{stealth}}},postaction={decorate}}}

\pgfmathsetmacro{\xboxsize}{0.03 cm}
\pgfmathsetmacro{\pnum}{3}
\pgfmathsetmacro{\qnum}{7}
\pgfmathsetmacro{\pqscale}{0.1cm}
\pgfmathsetmacro{\plen}{\pnum*\pqscale pt}
\pgfmathsetmacro{\qlen}{\qnum*\pqscale pt}

\tikzset{
    pqscaling/.style={x=\pqscale,y=\pqscale},
    xFrame/.pic={
        \draw[pqscaling] (0,0) rectangle (\xboxsize pt,\xboxsize pt);
    },
    aFrame/.pic={
        \draw[pqscaling] (0,\xboxsize pt) rectangle (\xboxsize pt, \pnum);
    },
    bFrame/.pic={
        \draw[pqscaling] (0,\xboxsize pt) rectangle (\xboxsize pt, \qnum);
    },
    cFrame/.pic={
        \pic[pqscaling,xscale=-1,rotate=90]{aFrame};
    },
    dFrame/.pic={
        \pic[pqscaling,xscale=-1,rotate=90]{bFrame};
    },
    Tcaca/.pic={
        \begin{scope}[pqscaling]
            \pic at (0,0) {xFrame};
            \pic at (0,0) {aFrame};
            \pic at (0,0) {cFrame};
        \end{scope}
    },
    TcacaLabel/.pic={
        \begin{scope}[pqscaling]
            \pic at (0,0) {Tcaca};
            \node[shift={(\xboxsize/2 pt,\xboxsize/2 pt)}] at (\pnum/2,\pnum/2) {$\Tcaca$};
        \end{scope}
    },
    Tdada/.pic={
        \begin{scope}[pqscaling]
            \pic at (0,0) {xFrame};
            \pic at (0,0) {aFrame};
            \pic at (0,0) {dFrame};
        \end{scope}
    },
    TdadaLabel/.pic={
        \begin{scope}[pqscaling]
            \pic at (0,0) {Tdada};
            \node[shift={(\xboxsize/2 pt,\xboxsize/2 pt)}] at (\qnum/2,\pnum/2) {$\Tdada$};
        \end{scope}
    },
    Tcbcb/.pic={
        \begin{scope}[pqscaling]
            \pic at (0,0) {xFrame};
            \pic at (0,0) {bFrame};
            \pic at (0,0) {cFrame};
        \end{scope}
    },
    Tdbdb/.pic={
        \begin{scope}[pqscaling]
            \pic at (0,0) {xFrame};
            \pic at (0,0) {bFrame};
            \pic at (0,0) {dFrame};
        \end{scope}
    },
    Tcdacda/.pic={
        \begin{scope}[pqscaling]
            \pic at (0,0) {xFrame};
            \pic at (0,0) {aFrame};
            \pic at (\qnum,0) {xFrame};
            \pic at (\qnum,0) {cFrame};
            \pic at (0,0) {dFrame};
        \end{scope}
    },
    Tdcadca/.pic={
        \begin{scope}[pqscaling]
            \pic at (0,0) {xFrame};
            \pic at (0,0) {aFrame};
            \pic at (\pnum,0) {xFrame};
            \pic at (\pnum,0) {dFrame};
            \pic at (0,0) {cFrame};
        \end{scope}
    },
    TdcadcaLabel/.pic={
        \begin{scope}[pqscaling]
            \pic at (0,0) {Tdcadca};
            \node[shift={(\xboxsize/2 pt,\xboxsize/2 pt)}] at (\pnum/2 + \qnum/2,\pnum/2) {$\Tdcadca$};
        \end{scope}
    },
    Tcqadpa/.pic={
        \begin{scope}[pqscaling]
            \pgfmathparse{\pnum-1}
            \foreach \i in {0,...,\pgfmathresult}{
                \pic at (\i*\qnum,0) {xFrame};
                \pic at (\i*\qnum,0) {dFrame};
            }
            \pic at (0,0) {aFrame};
        \end{scope}
    },
    Tdpacqa/.pic={
        \begin{scope}[pqscaling]
            \pgfmathparse{\qnum-1}
            \foreach \i in {0,...,\pgfmathresult}{
                \pic at (\i*\pnum,0) {xFrame};
                \pic at (\i*\pnum,0) {cFrame};
            }
            \pic at (0,0) {aFrame};
        \end{scope}
    }
}

\begin{document}

\title{Continuous Combinatorics of \\ Abelian Group Actions}

\author{Su Gao}
\address{School of Mathematical Sciences and LPMC, Nankai University, Tianjin 300071, P.R. China}
\email{sgao@nankai.edu.cn}

\author{Steve Jackson}
\address{Department of Mathematics, University of North Texas, 1155 Union Circle \#311430, Denton, TX 76203, U.S.A.}
\email{jackson@unt.edu}

\author{Edward Krohne}
\email{krohneew@gmail.com}

\author{Brandon Seward}
\address{Department of Mathematics, University of California San Diego, San Diego, CA 92093, U.S.A.}
\email{b.m.seward@gmail.com}

\begin{abstract}
This paper develops techniques which are used to answer a number of questions in the
theory of equivalence relations generated by continuous actions of abelian groups.
The methods center around the construction of certain specialized hyper-aperiodic elements,
which produce compact, free subflows with useful properties.
For example, we show that there is no continuous proper $3$-coloring of the Schreier graph on
$F(2^{\Z^2})$, the free part of the shift action of $\Z^2$ on $2^{\Z^2}$.
With earlier work of \cite{gao_countable_2015} this computes the continuous chromatic
number of $F(2^{\Z^2})$ to be exactly $4$. Combined with marker arguments for the positive directions,
our methods allow us to analyze
continuous homomorphisms into graphs, and more generally equivariant maps  into
subshifts of finite type. We present a general construction of a finite set of tiles
for  $2^{\Z^n}$ (there are $12$ for $n=2$) such that questions about the existence of continuous
homomorphisms into various structures reduce to finitary combinatorial questions about the tiles.
This tile analysis is used to deduce a number of results about $F(2^{\Z^n})$.
\end{abstract}


\thanks{Research of the first author was partially supported by the U.S. NSF grants DMS-1201290 and DMS-1800323, and NSFC grants  12250710128 and 12271263. Research of the second author was partially supported by the NSF grants DMS-1201290 and DMS-1800323.  Some results in this paper were included in the third author's Ph.D. dissertation at the University of North Texas, with the first and second authors serving as co-advisors, and the research was partially supported by the NSF grant DMS-0943870. The authors thank the anonymous referee for two rounds of careful reading of earlier versions of the manuscript and for useful comments which resulted in significant improvements.}

\maketitle

\tableofcontents

\chapter*{Introduction}
\label{sec:intro}

The last several decades have seen much development of {\em invariant
  descriptive set theory}, or the theory of definable equivalence
relations $E$ on standard Borel spaces $X$.  Most often the term
``definable'' is taken to mean ``Borel,'' that is, $E$ is a Borel
subset of $X\times X$.  This is a broad field of investigation,
overlapping with much of topological dynamics, functional analysis,
ergodic theory, and model theory, and interacting with computability
theory, geometric group theory, and in general algebra and
combinatorics (see \cite{gao_book} for an introduction). Understanding
of these definable equivalence relations is often accomplished by
imposing and investigating additional structures on the equivalence
classes.  In recent years the study of combinatorial problems relating
to such structures on definable equivalence relations has emerged as a
field of its own interest, which is known as the field of {\em Borel
  combinatorics}. The recent survey papers of Kechris
\cite{kechris_cbel} and Kechris, Marks \cite{kechrismarks} provide an
excellent account of this subject. The results and methods of the
present paper can be viewed as a contribution to this emerging field.

This paper is chiefly concerned with {\em countable} Borel equivalence
relations, that is, Borel equivalence relations each of whose
equivalence classes is countable. By a fundamental theorem of
Feldman--Moore \cite{feldman_ergodic_1977}, every countable Borel
equivalence relation $E$ is the {\em orbit equivalence relation} of a
Borel action by a countable group $\G$, that is, there is an action
$\cdot: \G \acts X$ such that $x E y$ if and only if $\exists g \in \G
\ (g\cdot x=y)$. In view of this theorem, it is natural to study
countable Borel equivalence relations ``group by group,'' that is, to
begin with the simplest groups and work our way up to progressively
more complicated groups.  The simplest infinite group is $\G=\Z$, and
its actions are the subject of study in the classical theory of
dynamical systems. This paper deals with actions of the next group
$\G=\Z^2$ and, to some extent, the other finitely generated abelian
groups $\G=\Z^n$ (the general finitely generated case $\G=\Z^n \times
A$, where $A$ is finite abelian, is not essentially different). Along
these lines, the next steps would concern general abelian groups,
nilpotent groups, solvable groups, amenable groups, etc.\
In \cite{gao_countable_2015} it was shown that Borel actions of countable
abelian groups are hyperfinite, in \cite{SchneiderSeward} this was
extended to countable locally nilpotent groups, and in \cite{polycyclic}
this was extended further to all countable polycyclic groups (which
includes some groups of exponential growth). In a related direction, several
authors have considered the Borel combinatorics of actions of groups of
subexponential growth \cite{ConleyTamuz, CGMPT, Thornton}.

This approach should not strike anyone as very ambitious, but we will
demonstrate that there are reasons to slow down and examine thoroughly
the properties of actions of seemingly simpler groups before we
consider the more complicated ones. Even for actions of simpler groups
such as $\Z^n$, many fundamental questions remain open, not to mention
the significant new difficulties when we attempt to generalize known
results for simpler groups to more complicated ones.  Two important
examples are the {\em hyperfiniteness problem} and the {\em union
  problem}.  A countable Borel equivalence relation $E$ is {\em
  hyperfinite} if it can be written as the increasing union $E=\bigcup
E_n$ of finite subequivalence relations $E_n$ (by ``finite'' we mean
each equivalence class of $E_n$ is finite). Kechris and Weiss asked
for which groups $\G$ must all their Borel actions be hyperfinite.
Specifically, they asked whether this is true for all amenable groups
$\G$. The union problem asks whether the increasing union of
hyperfinite equivalence relations is hyperfinite. In view of recent
progress, both of these questions are now understood to be closely
related to the question of what types of {\em marker structures} can
be put on equivalence relations generated by actions of $\G$.  By a
``marker structure'' we mean a Borel set $M\subseteq X$ which is a
{\em complete section} (i.e., $M$ meets every every $E$-class) and
often also {\em co-complete} (i.e., $X-M$ is also a complete section).
Usually we are interested in $M$ having some geometrical or
combinatorial structure which can be used to get useful results about
$E$. We thus see that several of the fundamental problems in the
theory of countable Borel equivalence relations are closely connected
with the Borel combinatorics of these equivalence relations.

The problems considered in the present paper are motivated by the observation that the most straightforward constructions of marker structures on the equivalence relations often yield {\em relatively clopen} structures instead of more complex Borel structures. However, as we impose more requirements on the marker structures, whether this phenomenon will continue becomes a question. In this paper we will concentrate on questions concerning {\em continuous combinatorics}, which is about relatively clopen structures and the existence of continuous mappings, as opposed to Borel structures and the existence of Borel mappings considered in Borel combinatorics. The decision to focus on this set of questions is largely driven by the richness of methods and results that are already present in this sub-area. Our results about the general Borel combinatorics of abelian group actions will be presented elsewhere (the forthcoming \cite{borcon} and \cite{GJKS_forcing_2015}).

We explain in some more detail the connections and differences of considerations in Borel combinatorics and continuous combinatorics of countable abelian group actions.
A good example of this concerns the currently known results on the hyperfiniteness problem.
Slaman and Steel \cite{Slaman1988DefinableFO} and independently Weiss \cite{weiss_book}
proved the basic result that $E$ is hyperfinite if $E$ can be generated by a Borel
action of $\Z$. To do this they defined a sequence $M_1 \supseteq M_2 \supseteq \cdots$ of
Borel marker sets which are decreasing and $\bigcap_n M_n=\varnothing$. We call such a sequence a
{\em Slaman--Steel} marker sequence. Some years later, Weiss showed that every Borel action
of $\Z^n$ is hyperfinite by showing that every free Borel action of $\Z^n$ admits a
marker set with a bounded geometry. More precisely,
we have the following basic marker lemma (a proof can be found in \cite{gao_countable_2015}).

In the lemma below, for a free action $\Z^n \acts X$ we define an extended metric
$\rho \colon X \times X \rightarrow [0, +\infty]$ by the rule: 
$\rho(x,y)=\infty$ if $[x]\neq [y]$, and otherwise if $y=g\cdot x$, where $g=(a_1,\dots,a_n)$,
then $\rho(x,y)=\| g\|= \max \{|a_1|, \dots, |a_n| \}$. We will also occasionally refer to the extended
metric $\rho_1$ on $X\times X$ where we use instead $\| g\|_1= \sum_{i=1}^n |a_i|$.
The following lemma holds (by the same proof) for both metrics $\rho$  and $\rho_1$. 

\begin{lemn} \label{bml}
For every integer $d>1$ there is a relatively clopen
set $M_d \subseteq F(2^{\Z^n})$ (where $F(2^{\Z^n})$ is the free part of the shift action, see \S\ref{sec:notation} for the precise definition) satisfying:
\begin{enumerate}[label={\rm (\arabic*)}]
\item
For all $x\neq y$ in $M_d$, $\rho(x,y)>d$.
\item
For all $x \in X$, there is a $y\in M_d$ with $\rho(x,y)\leq d$.
\end{enumerate}
\end{lemn}

Extending this hyperfiniteness result to the general (infinitely generated) abelian group case
runs into extra difficulties as this type of marker structure does not seem as useful for studying these group actions, and this case of the
hyperfiniteness problem remained open for some time.
In \cite{gao_countable_2015} it was shown that the Borel action of any
countable abelian group is hyperfinite.
The key ingredient in the argument is a stronger type of marker structure on the
equivalence relation, which we refer to as an {\em orthogonal marker structure}.
First, the ``marker regions'' (finite subequivalence classes associated to each
marker point) are essentially rectangles (the actual construction first puts down
rectangles and then modifies them in finitely many steps to fractal-like regions).
Second, the rectangular regions for different levels $M_{d_1}$ and $M_{d_2}$
must intersect in ``general position'', that is, no close parallel faces
(this is where the terminology ``orthogonal'' comes from). The fact that the marker
regions have a more regular geometric shape is important in the arguments.
Thus, the question of what types of regular marker structures can be put on
equivalence relations generated by actions of $\G$ has deep connections with the
theory of these equivalence relations.

Many interesting combinatorial questions about the equivalence relation $E$
also end up being connected with delicate questions about marker structures;
indeed this is one of the main themes of this paper. The orthogonal marker
method mentioned above is one recent technique which can be used to
produce intricate Borel structures on $E$. There is one important consideration,
however, which applies to constructions that rely upon a sequence of marker sets (such as the orthogonal marker method). Although the basic
marker construction for producing $M_d$ satisfying (1) and (2) above produces
{\em clopen} sets $M_d$, in passing to the ``limit'' the final construction
generally only produces Borel sets. Thus, this method often cannot be used to
produce clopen or continuous structures on $E$. In fact, it was realized some time ago that
there are limitations on what continuous structures can be attained.
A simple example of this occurs in the case of Slaman--Steel markers on actions of $\Z$.
We noted above that a Borel sequence $M_1\supseteq M_2 \supseteq \cdots$ in $F(2^\Z)$
with $\bigcap_n M_n=\varnothing$
can always be found, but in \cite{gao_countable_2015} it was shown that
there cannot be a sequence of (relatively) clopen marker sets $M_d \subseteq F(2^\Z)$
with this property. The basic method is that of a
{\em hyper-aperiodic element} (also called {\em $2$-coloring} in earlier papers),
which will be defined in \S\ref{sec:notation}.
In \cite{gao_coloring_2009} and \cite{GJS_coloring_2016} it was shown that hyper-aperiodic
elements exist for any countable group $\G$,
and a number of consequences of this are deduced. However, for many of the more delicate
questions, the mere existence of hyper-aperiodic elements is not sufficient. One of the main themes
of this paper is that by constructing hyper-aperiodic elements on $\G$ with
certain specialized properties, one can extend the method greatly to answer
many of these more delicate questions. We emphasize that these hyper-aperiodic element methods
are used to rule out continuous structures on $E$, and thus are complementary to
the marker construction methods which are used to produce clopen or Borel
structures on $E$. Alternatively, we may view the hyper-aperiodic element methods as showing
that particular orbits with specified properties must exist given any continuous
assignment of structure (e.g., a proper $3$-coloring of the Schreier graph) to the $E$ classes.
In \cite{GJKS_forcing_2015} it was shown that the hyper-aperiodic element method can be extended in
some cases to the Borel context via a forcing construction. However, not all of the
``negative'' (that is, non-structure) results at the continuous level obtained by
the hyper-aperiodic element arguments can be extended to the Borel level, as the
orthogonal marker methods can be used to obtain ``positive'' Borel constructions.
An example of this is the chromatic number problem for $F(2^{\Z^n})$
where we seek to determine both the continuous chromatic number $\chi_c(F(2^{\Z^n}))$
and the Borel chromatic number $\chi_b(F(2^{\Z^n}))$ for $F(2^{\Z^n})$
(these are the minimum number of colors needed to properly color the
Schreier graphs of $F(2^{\Z^n})$ with continuous, respectively Borel, functions).
For $n>1$ we show in this paper that there is no continuous proper $3$-coloring
of $F(2^{\Z^n})$, while recently the authors have shown using the orthogonal marker
methods that there is a Borel proper $3$-coloring of $F(2^{\Z^n})$
(the Borel result will appear in \cite{borcon}). Thus, for $n>1$
we have:
\[
3= \chi_b(F(2^{\Z^n})) < \chi_c(F(2^{\Z^n})) =4.
\]
This result shows a remarkable difference between the continuous and Borel combinatorics
of actions of $\Z^n$.

For a number of results, it is still not known whether the negative results given by
hyper-aperiodic element methods can be extended to the Borel context. There is a clear tension between
attempting to do so (say by forcing arguments) and the orthogonal marker method.
Delineating the exact boundaries of these methods remains an important problem.

As surveyed in \cite{kechris_cbel}, there is a rich theory of {\em descriptive graph combinatorics} developed by many authors which deals with the Borel combinatorics of graph structures on equivalence relations. The first systematic study of this subject was done by Kechris, Solecki, and Todorcevic \cite{KST}, in which they studied problems such as Borel and measurable chromatic numbers, edge chromatic numbers, perfect matchings, etc. For example, it is implied by results of \cite{KST} that the Borel chromatic number of $F(2^{\Z^n})$ is no more than $2n+1$ and the Borel edge chromatic number of $F(2^{\Z^n})$ is no more than $4n-1$. As mentioned above, we will show in this paper that the continuous chromatic number of $F(2^{\Z^n})$ is $4$. We will also show that the continuous edge chromatic number of $F(2^{\Z^2})$ is $5$ and that there is no continuous perfect matching of $F(2^{\Z^n})$ for any $n\geq 1$. With these results, the present paper can be viewed as a contribution to the theory of descriptive graph combinatorics.

The rest of the paper is organized as follows.
We give some background and motivations for the problems studied in this paper in the Introduction. Then in Chapter~\ref{sec:notation} we present the basic concepts, methods, and results in the continuous combinatorics of $\Z^n$ actions. Specifically, in \S\ref{section:1.1} and \S\ref{section:1.2} we recall basic notation and terminology
including the notion of hyper-aperiodic element on a general countable group $G$, and for the sake of completeness give
simple proofs that these exist for $\G=\Z^n$, the main case of interest
for this paper. In \S\ref{section:1.3} we give a quick application of the hyper-aperiodic elements in a problem about a marker structure called {\em toast}. In \S\ref{hyper} and \S\ref{section:1.5} we present the first main application of our method to the continuous chromatic number problem for $F(2^{\Z^n})$. Here we give a self-contained proof that the continuous chromatic number of $F(2^{\Z^n})$ is at least $4$. Later in the paper this result will be proved twice more, and each proof is shorter since it is based on more powerful general theorems. In \S\ref{section:1.6} and \S\ref{section:1.7} we give some more applications of the hyper-aperiodic elements in problems involving {\em clopen line sections} and {\em continuous cocyles}. In \S\ref{sec:md} we prove some results about {\em marker distortion}
for marker sets $M\subseteq F(2^\Z)$. Aside from being of independent interest,
these results serve as a warm-up and motivation for the results in the rest of the paper.
The notion of marker distortion measures the deviation from the marker set being
regularly spaced. We show that the best possible marker distortion differs
for clopen versus Borel marker sets (the Borel proof will be given in \cite{borcon}).
The clopen marker distortion result uses a specialized hyper-aperiodic element
which incorporates the notion of a ``tile.''  This notion plays an important role
in Chapter \ref{sec:twodtiles}, which analyzes the existence of
continuous equivariant maps from $F(2^{\Z^n})$ into subshifts of finite type in terms of
the existence of such maps on a certain set of finite tiles. The main result of the paper, known as the Twelve Tiles Theorem, is presented in this chapter. In \S\ref{section:2.1} we fix some notation on subshifts of finite type.
In \S\ref{sec:onedimtile} we present a much simpler case of the tile
analysis for dimension $n=1$. In \S\ref{subsec:tilethm} we state the Twelve Tiles Theorem, which is the tile analysis for dimension $n=2$. The theorem is proved in \S\ref{subsec:negtile} and \S\ref{subsec:postile}. In \S\ref{sec:continuous_structuring} we give a restatement of the Twelve Tiles Theorem in terms of {\em continuous structurings}. To some, this version of the theorem might be easier to apply as it does not involve the notion of subshifts of finite type. Our tile analysis techniques can be extended to higher dimensions. However, a general theorem for arbitrary dimensions is hard to state precisely, and we do not have any specific application for such a result at this time. So in \S\ref{subsec:hd} we state and prove a weak version of the tile theorem for higher dimensions, which we will use in Chapter 3.

Chapters \ref{chapter3} and \ref{sec:otherapp} are devoted to various applications of the Twelve Tiles Theorem.
In \S\ref{sec:tileapp} we deduce
several basic corollaries of the Twelve Tiles Theorem, such as the non-existence
of a continuous perfect matching on $F(2^{\Z^n})$ and the computation of the edge chromatic number of $F(2^{\Z^2})$. In \S\ref{section:3.2} we prove the fact that
there is no continuous proper free  subaction of $\Z^2$
on $F(2^{\Z^2})$. This can be used to answer in the negative a question raised in \cite{gao_countable_2015}.
In \S\ref{sec:undirected-graphs}, \S\ref{sec:nc} and \S\ref{sec:pc}
we delve more specifically in the question of characterizing graphs $\Gamma$ for which
there exists a continuous homomorphism from $F(2^{\Z^2})$ to $\Gamma$.
The result that $\chi_c(F(2^{\Z^2})) >3$ is the special case of this analysis
for which $\Gamma=K_3$, the complete graph on $3$ vertices. The Twelve Tiles Theorem gives,
in principle, an answer to the continuous graph homomorphism problem,
but not necessarily in a way which is easy to apply.
So, we prove  ``positive'' and ``negative'' conditions for the
existence or non-existence of such continuous homomorphisms in \S\ref{sec:pc} and \S\ref{sec:nc}, respectively. We illustrate these conditions
with a number of specific examples.  In Chapter \ref{sec:otherapp}
we consider the decidability of three basic problems concerning
the continuous combinatorics of $F(2^{\Z^n})$: the subshift, graph homomorphism,
and tiling problems. It follows from the Twelve Tiles Theorem that all of these problems
are semi-decidable, that is, they correspond to $\Sigma^0_1$ sets
(coding the various objects in question, e.g., finite graphs, by integers).
The graph homomorphism and tiling problems are special instances of the more general
subshift problem. These problem are stated in \S\ref{sec:problems}. We show in \S\ref{sec:ssonedim} 
that the general subshift
problem in one dimension, that is $n=1$, is computable (i.e., corresponds
to a $\Delta^0_1$ set of integers). Thus, there is an algorithm which when
given an integer coding a subshift of finite type $Y\subseteq 2^\Z$ will decide
if there is a continuous, equivariant map from $F(2^\Z)$ to $Y$. It follows immediately that the
graph homomorphism and tiling problems are also decidable  for $n=1$.
In contrast, we show in \S\ref{sec:sstwodim} that the subshift problem
for $n \geq 2$ is $\Sigma^0_1$-complete, and in particular not decidable. In \S\ref{sec:gh} we show that
for $n \geq 2$ even the continuous graph homomorphism problem is $\Sigma^0_1$-complete. Finally, in \S\ref{sec:tiling} we discuss the continuous
tiling problem for $n =2$. Here we do not yet know whether the problem
is decidable. In fact, we do not know the answer to the tiling question
even for some very specific simple sets of tiles. We mention some computer experiments
which may shed some light on this question.

For the convenience of the reader we include below a diagram of logical dependence and references between sections. The arrows illustrate logical dependence between sections, and the dash arrows illustrate references and analogs.
\vspace{20pt}

\begin{figure}[h!]
\begin{tikzpicture}[scale=0.045]

\draw (-100,0) rectangle (-80,-14);
\node at (-90, -7) {\S 1.1};

\draw (-50, 0) rectangle (-30, -14);
\node at (-40, -7) {\S 1.2};
\draw[->] (-50,-7) -- (-80,-7);

\draw (0, 14) rectangle (20, 0);
\node at (10, 7) {\S 1.3};
\draw[->] (0,7) -- (-30,-3);

\draw (0, -14) rectangle (20,-28);
\node at (10, -21) {\S 1.4};
\draw[->] (0,-21) -- (-30, -11);

\draw (50, -14) rectangle (70,-28);
\node at (60, -21) {\S 1.5};
\draw[->] (50, -21) -- (20,-21);

\draw (50, -35) rectangle (70, -49);
\node at (60, -42) {\S 1.6};
\draw[->] (50, -42) -- (20, -25);

\draw (100, 0) rectangle (120, -14);
\node at (110, -7) {\S 1.7};
\draw[->] (100,-7) -- (-30,-7);
\draw[dashed, ->] (100,-11) -- (70, -21);

\draw (0,-35) rectangle (20, -49);
\node at (10,-42) {\S 1.8};
\draw[->] (0,-42) -- (-33,-14);

\draw (-30, -90) rectangle (-10, -104);
\node at (-20, -97) {\S 2.3};
\draw[->] (-30,-97) -- (-80,-97);
\draw[dashed, ->] (-30,-93) -- (-45,-76);

\draw (100, -118) rectangle (120, -132);
\node at (110, -125) {\S 2.7};
\draw[->] (100,-123) -- (-10,-101);
\draw[->] (110,-118) -- (110, -14);

\draw (-100,-90) rectangle (-80,-104);
\node at (-90, -97) {\S 2.1};

\draw(-65,-69) rectangle (-45, -83);
\node at (-55, -76) {\S 2.2};
\draw[->] (-65,-76) -- (-80, -93);
\draw[->] (-55, -69) -- (-47,-14);

\draw (30,-90) rectangle (50, -104);
\node at (40, -97) {\S 2.5};
\draw[->] (30,-97) -- (-10, -97);

\draw (75,-90) rectangle (95,-104);
\node at (85,-97) {\S 2.6};
\draw[->] (75,-97) -- (50,-97);
\draw[->] (75,-94) -- (20,-76);

\draw (0,-69) rectangle (20,-83);
\node at (10, -76) {\S 2.4};
\draw[->] (0,-79) -- (-10,-94);
\draw[->] (0,-76) -- (-40, -14);

\draw (-18,-180) rectangle (2,-194);
\node at (-8, -187) {\S 3.1};
\draw[->] (-12,-180) -- (-20, -104);
\draw[->] (-8,-180) -- (37, -104);
\draw[->] (-4, -180) -- (100, -127);

\draw (17, -180) rectangle (37, -194);
\node at (27, -187) {\S 3.2};
\draw[->] (23, -180) -- (-16.5,-104);
\draw[->] (27, -180) -- (10,-83);
\draw[->] (31, -180) -- (43, -104);

\draw (52, -180) rectangle (72, -194);
\node at (62, -187) {\S 3.3};
\draw[->] (62, -180) -- (-13, -104);

\draw (100, -170) rectangle (120, -184);
\node at (110, -177) {\S 3.4};
\draw[->] (100, -177) -- (72, -184);

\draw (100,-190) rectangle (120, -204);
\node at (110, -197) {\S 3.5};
\draw[->] (100, -197) -- (72, -190);

\draw (-100, -270) rectangle (-80,-285);
\node at (-90, -277.5) {\S 4.1};
\draw[->] (-90,-270) -- (-27, -104);

\draw(-65,-250) rectangle (-45,-264);
\node at (-55,-257) {\S 4.2};
\draw[->] (-65,-259) -- (-80,-273);
\draw[->] (-55,-250) -- (-55, -83);

\draw (-30, -259) rectangle (-10, -273);
\node at (-20, -266) {\S 4.3};
\draw[->] (-30,-266) -- (-80, -276);
\draw[->] (-23.5, -259) -- (-23.5, -104);
\draw[->] (-20, -259) -- (-8,-194);

\draw(17, -272) rectangle (37,-286);
\node at (27, -279) {\S 4.4};
\draw[->] (17,-279) -- (-80, -279);
\draw[->] (27,-272) -- (27, -194);

\draw(-30, -285) rectangle (-10, -299);
\node at (-20, -292) {\S 4.5};
\draw[->] (-30,-292) -- (-80, -282);

\end{tikzpicture}
\end{figure}

\chapter{Basic Definitions and Results} \label{sec:notation}

\section{Bernoulli shift actions}\label{section:1.1}

If $\G$ is a group with identity $1_G$ and $X$ is a set, a (left) {\em action} of $\G$ on $X$
is a map $(g,x)\mapsto g\cdot x$ from $\G \times X$ to $X$ satisfying:
\begin{enumerate}
\item
$g\cdot (h \cdot x)=(gh)\cdot x$ for all $g,h \in \G$, $x \in X$.
\item
$1_G \cdot x=x$ for all $x \in X$.
\end{enumerate}

We also write $\G \acts X$ to denote that $\G$ acts on $X$.
When there is more than one action of $\G$ on $X$ being considered, we
employ notation such as $g \cdot x$, $g \bullet x$, $g*x$ to distinguish them.

When $G$ acts on $X$ and $x \in X$, we write $[x]$ for the orbit equivalence class
of $x$, that is, $[x]=\{ y\in X \colon \exists g \in G\, (g\cdot x=y)\}$. Of course, there is a danger of confusion with this notation as it does not specify the action. When there are different actions considered concurrently, we will introduce appropriate ad hoc notations to distinguish them.

If $\G$ acts on both $X$ and $Y$, we say a map $\pi \colon X\to Y$ is
{\em equivariant} if $\pi(g\cdot x)=g \cdot \pi(x)$ for all $ g\in \G$ and $x \in X$.
We say $\pi$ is an {\em equivariant embedding} if $\pi$ is also injective.

Throughout the rest of the paper, $X$ will henceforth denote a {\em Polish space}, i.e., a separable, completely metrizable topological space.
When only the Borel structure of $X$ is relevant, we refer to $X$ as a
{\em standard Borel space}. When $\G$ is also a topological group, we say the action
is {\em continuous} or {\em Borel} if the map $(g,x)\mapsto g \cdot x$ from $\G \times X$
to $X$ is continuous or Borel. When $\G$ is a countable discrete group
(which is the case for this paper), the action is continuous or Borel if and only if
for each $g \in \G$ the map $x \mapsto g\cdot x$ from $X$ to $X$ is
continuous or Borel.

For $\G$ a countable group, the set $2^\G$ is a compact zero-dimensional
Polish space
with the usual product of the discrete topology on $2=\{ 0,1\}$.
There are two  natural (left) actions of $\G$ on $2^\G$, the {\em left-shift},
and the {\em right-shift} (also called the {\em Bernoulli shift} actions).
The left-shift action is given by

\[
g \cdot x (h)= x(g^{-1}h),
\]
for $g,h \in \G$ and $x \in 2^\G$. The  right-shift action is given by
\[
g \cdot x(h)=x(hg).
\]
It is more common in the literature to use the left-shift action,
and either one could be used throughout this paper. In fact, the left-shift
and right-shift actions on $2^G$ are isomorphic $G$-actions under the map
$\pi \colon 2^G\to 2^G$ given by $\pi(x)(h)=x(h^{-1})$. However,
we prefer to use the right-shift action throughout.

An action of $\G$ on $X$ is {\em free} if $g\cdot x \neq x$ for all
$g\neq 1_\G$ in $\G$ and all $x \in X$. If $\G$ acts on $X$, then
the {\em free part} $F(X)$ is $\{ x \in X\colon  \forall g \neq 1_\G \ (g\cdot x \neq x)\}$.
$F(X)$ is the largest  invariant subset of $X$ on which $\G$ acts freely.
When $\G$ is a countable group acting continuously, $F(X)$ is
also a $G_\delta$ subset of $X$, and thus a Polish space with the
subspace topology. Except in trivial cases, however, $F(X)$ will not be
compact.

The shift action of a countable group $\G$ is of particular significance
as it is essentially universal for all actions of $\G$ according
to the following theorem of Becker--Kechris \cite{bk}
and Dougherty--Jackson--Kechris \cite{djk}.

\begin{thmn} \label{thm:shiftuniv}
The shift action of $\G$ on $(2^\omega)^\G\cong 2^{\G\times \omega}$
given by $(g \cdot x) (h,n)= x(hg,n)$ is a universal Borel $\G$-action.
That is, if $G \acts X$ is a Borel action,
then there is a Borel equivariant
embedding $\pi \colon X \to (2^\omega)^G$. Furthermore, if the action of $\G$ on $X$
is continuous and $X$ is zero-dimensional, then $\pi$ can be chosen to be continuous.
\end{thmn}

If $S$ is a generating set of $G$, the {\em Cayley graph} $\Gamma=\Gamma(G, S)$ is a labeled directed graph with the vertex set $V(\Gamma)=G$ and the edge set $E(\Gamma)$ consisting of pairs of the form $(g, sg)$ for $g\in G$ and $s\in S$, where the edge $(g, sg)$ is assigned the label $s$. In addition, if $G$ acts on a space $X$, the {\em Schreier graph} $\Gamma=\Gamma(X, G, S)$ is a labeled directed graph with the vertex set $V(\Gamma)=X$ and the edge set $E(\Gamma)$ consisting of pairs of the form $(x, s\cdot x)$ for $x\in X$ and $s\in S$, where the edge $(x, s\cdot x)$ is assigned the label $s$. In this sense the Cayley graph on $G$ is a special case of a Schreier graph, where the action of $G$ on itself is by left multiplication. When connectedness is considered, we consider the unlabeled and undirected Schreier graphs. Thus, since $S$ is always a generating set of $G$, the Cayley graph of $G$ is always connected, and each connected component of the Schreier graph is exactly an orbit equivalence class of the action. In the case that the action of $G$ on $X$ is free, there is a natural graph isomorphism between the Cayley graph on $G$ and each connected component of the Schreier graph. Specifically, if the action is free and $x\in X$, then the map $g\mapsto g\cdot x$ induces a graph isomorphism from $G$ to $[x]$.

In this paper we only deal with finitely generated groups, and the generating set in our study will always be finite and minimal (in the sense that no proper subset is a generating set). \label{con:S} In particular, a generating set does not contain the identity. In addition, for all the Schreier graphs $\Gamma$ we consider in this paper, the vertex set $V(\Gamma)=X$ will be a topological space, and the edge set $E(\Gamma)=\{(x, s\cdot x)\in X\times X\,:\ s\in S\}$ will be a closed subset of $X\times X$, and thus a topological space in its own right. In this sense, we refer to these Schreier graphs as {\em topological graphs}.

For a graph $\Gamma$ (which could be labeled, or directed, or both) and a set $K$ of colors, a {\em proper ($K$-)coloring} is a map $\kappa: V(\Gamma)\to K$ such that if $(x, y)\in E(\Gamma)$, then $\kappa(x)\neq \kappa (y)$. The {\em chromatic number} of $\Gamma$, denoted by $\chi(\Gamma)$, is the least cardinality of a set $K$ such that there exists a proper $K$-coloring for $\Gamma$. When the graph $\Gamma$ is a topological graph, we may consider {\em continuous} proper colorings and {\em Borel} proper colorings, and correspondingly its {\em continuous chromatic number}, denoted as $\chi_c(\Gamma)$, and, respectively, {\em Borel chromatic number}, denoted as $\chi_B(\Gamma)$.

The main case of interest for this paper is $\G=\Z^n$ and $X=F(2^{\Z^n})$,
the free part of the shift action of $\Z^n$ on $2^{\Z^n}$. We work with the standard generating set $\{ e_1, \dots, e_n\}$, where $e_i=(0, \cdots, 0, 1, 0, \cdots, 0)=(a_1, \dots, a_n)$ with $a_i=1$ and $a_j=0$ for all $j\neq i$. The Cayley graph on $\Z^n$ is therefore an $n$-dimensional grid, which is isomorphic to the Schreier graph on each orbit equivalence class.

Recall from the Introduction
that for $x, y \in F(2^{\Z^n})$, the ``distance'' $\rho(x,y)$
between them is the $\ell_\infty$ norm of the unique $g \in \Z^n$ such that
$g \cdot x=y$. Formally, $\rho(x,y)=\infty$ if $[x]\neq [y]$, and if $y=g\cdot x$ with $g=(a_1, \dots, a_n)\in \Z^n$, then $\rho(x,y)=\|g\|=\max\{|a_1|,\dots, |a_n|\}$. Occasionally we will use an alternative distance function $\rho_1$ on $X\times X$ where instead of the $\ell_\infty$ norm $\|g\|$ we use the $\ell_1$ norm $\| g\|_1= \sum_{i=1}^n |a_i|$.

\section{Hyper-aperiodic elements}\label{section:1.2}

We next recall the definition of a hyper-aperiodic element
from \cite{GJS_coloring_2016}.

\begin{defnn} \label{def:twocol}
Let $\G$ be a countable group. A point $x \in 2^\G$ is {\em hyper-aperiodic} if the closure of its orbit is contained in the free part, i.e. if $\overline{[x]} \subseteq F(2^\G)$.
\end{defnn}

The significance of this notion is that the closure of the orbit of a hyper-aperiodic element provides an invariant compact subset of $F(2^\G)$ which allows us to use topological arguments more fully when analyzing continuous marker structures on $F(2^\G)$. Hyper-aperiodicity can be tested via the following combinatorial characterization (see \cite{GJS_coloring_2016}).

\begin{lemn}\label{lem:hap} Let $\G$ act on $2^\G$ by right-shifts. Then $x \in 2^\G$ is hyper-aperiodic if and only if for all $s \neq 1_\G$ in $\G$ there is a finite set $T \subseteq G$ such that
\[
\forall g \in \G\ \exists t \in T\ x(t g)\neq x(t s g).
\]
\end{lemn}

The finite set $T$ in the above lemma is said to {\em witness} the hyper-aperiodicity of $x$ for $s$.

\begin{remn}
When using the left-shift action, the combinatorial
condition of the previous lemma becomes $\forall g \in \G \ \exists t \in T\  x(g t) \neq x(g s t)$.
\end{remn}

A related notion is that of $x,y \in 2^\G$ being orthogonal.

\begin{defnn} \label{def:orthogonal}
Let $x, y\in 2^\G$. We say that $x$ and $y$ are {\em orthogonal},
denoted $x\,\bot\, y$, if there exists a finite $T\subseteq \G$ such that for any
$g, h\in\G$ there is $t\in T$ with $x(tg)\neq y(th)$.
\end{defnn}

The topological significance of orthogonality is the following lemma from \cite{GJS_coloring_2016}.

\begin{lemn} Let $x, y\in 2^\G$. Then $x\bot y$ if and only if
$\overline{[x]}\cap \overline{[y]}=\varnothing$.
\end{lemn}

In \cite{GJS_coloring_2016} the existence of hyper-aperiodic elements, in fact that of large families
of pairwise orthogonal hyper-aperiodic elements, is shown for any countable
group $\G$.

\begin{thmn} [\cite{GJS_coloring_2016}] \label{thm:genhyp}
Let $\G$ be a countable group. Then there is a hyper-aperiodic element $x \in 2^\G$.
In fact there is a perfect set $P\subseteq 2^\G$ of pairwise orthogonal
hyper-aperiodic elements.
\end{thmn}

The proof of Theorem~\ref{thm:genhyp} is nontrivial, as it gives hyper-aperiodic elements $x \in 2^\G$ for an arbitrary countable
group $\G$. For specific groups like $\G=\Z^n$ these elements are actually easy to construct
directly. Since we will be constructing specialized hyper-aperiodic elements
for $\G=\Z^n$ throughout the paper, for the sake of completeness we give a new
simple argument that they exist in this case. This argument will
be similar to others used later in the paper.

\begin{lemn} \label{lem:goodtwocol}
Let $G$ be a countable group and $x \in 2^\G$ be hyper-aperiodic.
Then for any $s\neq 1_\G$ in $\G$ there is a finite
$T\subseteq \G$ such that
\[
\forall g \in \G\ [ (\exists t_1 \in T \ x(t_1g)\neq x(t_1sg)) \wedge
(\exists t_2 \in T \ x(t_2g)= x(t_2sg))].
\]
\end{lemn}

\begin{proof}
Let $x \in 2^\G$ be a hyper-aperiodic element. Let $s \neq 1_\G$. Let $T \subseteq \G$
witness the hyper-aperiodicity of $x$ for $s$ as in Lemma~\ref{lem:hap}, and we may assume $1_G \in T$.
Let $T'$ witness the hyper-aperiodicity for $s^2$.
Let $T''=T\cup T' \cup T's$. Then we claim that $T''$ witnesses the desired property for $s$.
For suppose $g \in \G$. If $x(g)=x(sg)$,
then since there is a $t \in T$ with $x(tg)\neq x(tsg)$, we can take
$t_1=t$, $t_2=1_\G$. Suppose $x(g)\neq x(sg)$. Let $t \in T'$ be such that
$x(tg)\neq x(ts^2g)$. Consider $x(tg)$, $x(tsg)$, and $x(ts^2g)$, which are elements of $\{0,1\}$. Either
$x(tg)=x(tsg)$ or $x(tsg)=x(ts^2g)$. So we can take $t_1=1_G$ and either
$t_2=t$ or $t_2=ts$.
\end{proof}

If $G$, $H$ are countable groups and $x \in 2^\G$, $y \in 2^H$,
let $x \oplus y \in 2^{\G\times H}$ be given by
$(x \oplus y)\, (g,h)=x(g)+y(h)\!\!\mod 2$.

\begin{lemn} \label{lem:twocolprod}
Let $x \in 2^{\G}$ and $y \in 2^{H}$ be hyper-aperiodic elements.
Then $x \oplus y \in 2^{\G\times H}$ is a hyper-aperiodic element.
\end{lemn}

\begin{proof}
Let $s=(s_1,s_2) \neq (1_G,1_H)$. Suppose first that $s_1 \neq 1_H$ and $s_2=1_H$.
Let $T$ witness the hyper-aperiodicity of $x$ for $s_1$.
Let $(g,h) \in \G \times H$. Then there is a $t \in T$ with $x(tg)\neq x(ts_1g)$,
and thus
$$\begin{array}{l}
 (x \oplus y)\, [(t,1_H)(s_1,s_2)(g,h)] = x(ts_1g)+y(h)\!\!\!\mod 2 \\
\neq  x(tg)+y(h)\!\!\!\mod 2= (x \oplus y)\, [(t,1_H)(g,h)].
\end{array}$$ So, $T\times \{ 1_H\}$
witnesses the hyper-aperiodicity of $x \oplus y$ for $s$ in this case.
The case $s_1=1_\G$, $s_2 \neq 1_H$ is similar.

Suppose that $s_1\neq 1_\G$ and $s_2 \neq 1_H$. Let $T_1$, $T_2$ witness the
enhanced hyper-aperiodicity in Lemma~\ref{lem:goodtwocol} respectively for $x$ for $s_1$ and for $y$ for $s_2$.
Let $t_1 \in T_1$ be such that $x(t_1g)\neq x(t_1sg)$, and let $t_2 \in T_2$ be
such that $y(t_2h)=y(t_2s_2 h)$. Then
$$\begin{array}{l} (x \oplus y)\, [(t_1,t_2)(g,h)]=
x(t_1g)+y(t_2h)\!\!\!\mod 2 \\
\neq x(t_1 s_1g)+y(t_2s_2h)\!\!\!\mod 2 = (x \oplus y)\, [(t_1,t_2)(s_1,s_2)(g,h)].\end{array}$$
Thus $T_1 \times T_2$ witnesses the hyper-aperiodicity of $x\oplus y$ for $s$.
\end{proof}

\begin{lemn} \label{hypzn}
For each $n \geq 1$, there is a hyper-aperiodic element on $\Z^n$.
\end{lemn}

\begin{proof}
From Lemma~\ref{lem:twocolprod} it is enough to show that there is a hyper-aperiodic element on $\Z$.
This can be done in many ways. For example, let
$x(a)= \beta(a) \mod 2$, where $\beta(a)$ is the number of $1$'s in the base $2$ expansion
of $|a|$. Then $x \in 2^\Z$ is hyper-aperiodic. For suppose
$s \neq 0$ and $|s|$ has highest non-zero base $2$ digit in the $\ell$th position counting from the right
(i.e., $\ell=\lfloor \log_2(|s|) \rfloor$), then we may define
$T=\{ n\in\Z\, \colon\, |n| \leq 4 \cdot 2^\ell \}$.
Now for any $a \in \Z$ we can find $n \in T$ such that the binary expansion of $a+n$ has $0$s
in positions lowers than the $\ell$th digit, and that the $\ell$th and $\ell+1$st
digits of $a+n$ can be taken to be either
$(0,0)$ or $(1,0)$. In the first case, adding $s$ changes the parity of
the number of $1$s by the parity of $\beta(s)$, and in the second case by $1+\beta(s)$.
Thus there is an $n\in T$ with $x(a+n)\neq x(a+n+s)$.
\end{proof}

We also note that once we have a hyper-aperiodic element on $\Z^n$, then we may easily
construct orthogonal such elements.

\begin{lemn} \label{getorthog}
There are $y_0,y_1 \in 2^{\Z^n}$ which are hyper-aperiodic and with $y_0 \perp y_1$.
\end{lemn}

\begin{proof}
Let $y \in 2^{\Z^n}$ be hyper-aperiodic. Let $p_1,p_2,q_1,q_2$ be distinct primes and
let $m\geq 2 \max \{p_1,p_2,q_1,q_2\}$. Let $M=2m+1$, and let $y_0$ be obtained from $y$
by replacing each each point $(a_1,\dots,a_n)\in \Z^n$  with a copy of $[0,2M]^n$, and if
$y(a_1,\dots,a_n)=0$ we make $y_0$ have period $(p_1,0,\dots,0)$ on the corresponding copy
of $[0,2M]^n$, and if $y_0(a_1,\dots, a_n)=1$ we make $y_0$ have period $(p_2,0,\dots,0)$ on this
copy. For example, we could use
\[
y_0(a_1M+r_1,\dots,a_nM+r_n)= 1 \text{ iff } p_{y(\vec a)} | r_1,
\]
where $0 \leq r_j <M$.
Likewise, $y_1$ is defined from $y$ using $q_1$, $q_2$. It is straightforward to check that $y_0,y_1$ are
hyper-aperiodic and that $y_0 \perp y_1$.

\end{proof}

\section{Toast, an application of hyper-aperiodicity}\label{section:1.3}
We give here a quick application of hyper-aperiodic elements, which only requires
Lemma~\ref{hypzn}. We show that there does not exist a continuous {\em toast structure}
on $F(2^{\Z^n})$. The concept of toast was introduced by B.\ Miller, and has proved
to be a useful concept (e.g., for the construction of Borel proper $3$-colorings
of $F(2^{\Z^n})$, a result which the authors will present in \cite{borcon}).

This easy argument shows how hyper-aperiodic elements are used to show certain
continuous structurings of shift actions cannot exist. By a {\em structuring} we mean a way to assign a
certain kind of structure on $F(2^{\Z^n})$.
A precise general definition will be given later in Definition~\ref{def:lstruc} (the
formal definition is not necessary for this section, though toast is an example of a structuring
with an infinitary language).
For example, a {\em linear order structuring} would be a way to assign a
linear order to each class of $F(2^{\Z^n})$.

Here we define directly what a toast structuring is, and what is means for it to be continuous.
In \S\ref{sec:continuous_structuring} we give the general definition of a structuring and what it
means for the structuring to be continuous. In the terminology below, we will sometimes 
for simplicity drop the word ``structuring,'' for example, we will refer to an unlayered toast structuring
as just an unlayered toast. 


\begin{rem}
The proof Theorem~\ref{thm:nct} below uses only the existence of a hyperaperiodic element
in $F(2^{\Z^d})$.
For most of the results of the current paper,
however, the mere existence of hyper-aperiodic elements is not enough. For these results, including
our main ``12 tiles'' theorem (Theorems~\ref{thm:tilethm} and  \ref{thm:tilethmb}) certain
specialized hyper-aperiodic elements must be constructed. Indeed, the construction of specialized
hyper-aperiodic elements, and their use in studying continuous structurings on shift actions,
is one of the main themes of this paper.
Nevertheless, the toast result will serve as motivation for the later results.
\end{rem}

Toast comes in two flavors, layered and unlayered (general) toast.
The following definition of toast and accompanying figure is taken from
\cite{GJKS_forcing_2015}. 

Below, for a set $C \subseteq F(2^{\Z^d})$ we write
$\partial C$ for the set of boundary points of $C$, meaning those points in $C$ that are
adjacent to a point not in $C$:
$$\partial C = \{x \in C : \exists y \in [x] \setminus C \ \rho_1(x, y) = 1\}.$$

\begin{defnn} \label{def:toast}
Let $\{T_n\}$ be a sequence of subequivalence relations of
the shift equivalence relation on $F(2^{\Z^d})$
on some subsets $\dom(T_n)\subseteq F(2^{\Z^d})$ with each
$T_n$-equivalence class finite. Assume $\bigcup_n \dom(T_n)=
F(2^{\Z^d})$. We say $\{ T_n\}$ is an {\em (unlayered) toast} if:
\begin{enumerate}
\item[(1)] \label{toast1} For each $T_n$-equivalence class $C$, and
  each $T_m$-equivalence class $C'$ where $m>n$, if $C \cap C' \neq
  \varnothing$ then $C \subseteq C'$.

\item[(2)] \label{toast2} For each $T_n$-equivalence class $C$ there
  is $m>n$ and a $T_m$-equivalence class $C'$ such that $C \subseteq
  C'\setminus \partial C'$.
\end{enumerate}

We say $\{ T_n\}$ is a {\em layered toast} if, instead of (2) above,
we have
\begin{enumerate}
\item[(2')] For each $T_n$-equivalence class $C$ there is a
  $T_{n+1}$-equivalence class $C'$ such that $C \subseteq C'\setminus
  \partial C'$.
\end{enumerate}

We call the toast {\em continuous} if for every $n$ and for every $g \in \Z^d$
the set $\{x \in F(2^{\Z^d}) \colon (x, g \cdot x) \in T_n\}$ is relatively clopen in $F(2^{\Z^d})$.
Similarly, we call the toast {\em open} if for every $g \in \Z^d$
this set is relatively open in $F(2^{\Z^d})$. 
\end{defnn}

Figure~\ref{fig:toast} illustrates the definitions of layered and
unlayered toast. The term `toast' is simply motivated by the the visual
shapes of these regions in Figure~\ref{fig:toast} together with how the
regions are stacked (specifically condition (1) above).

\begin{figure}[ht]
\centering
\begin{tikzpicture}[scale=0.035]

\draw (-150,-80) rectangle (-10,80);
\draw (10,-80) rectangle (150,80);

\draw (-80, 40) ellipse [x radius=60, y radius=30];
\draw (-80, -40) ellipse [x radius=60, y radius=30];
\draw[rotate around={45:(-110,40)}]   (-110,40) ellipse [x radius=20, y radius=10];
\draw[rotate around={45:(-80,40)}]   (-80,40) ellipse [x radius=20, y radius=10];
\draw[rotate around={45:(-50,40)}]   (-50,40) ellipse [x radius=20, y radius=10];
\draw (-100,50) circle (4);
\draw (-105,40) circle (4);
\draw (-117,30) circle (4);
\draw (-90,30) circle (4);
\draw (-75,40) circle (4);
\draw (-50,40) circle (4);
\draw[rotate around={45:(-105,-40)}]   (-105,-40) ellipse [x radius=20, y radius=10];
\draw[rotate around={45:(-55,-40)}]   (-55,-40) ellipse [x radius=20, y radius=10];
\draw (-110,-45) circle (4);
\draw (-98,-33) circle (4);
\draw (-63,-50) circle (4);
\draw (-60,-40) circle (4);
\draw (-45,-30) circle (4);

\draw (80, 40) ellipse [x radius=60, y radius=30];
\draw (80, -40) ellipse [x radius=60, y radius=30];
\draw[rotate around={45:(110,40)}]   (110,40) ellipse [x radius=20, y radius=10];
\draw[rotate around={45:(50,40)}]   (50,40) ellipse [x radius=20, y radius=10];
\draw (100,30) circle (4);
\draw (115,45) circle (4);
\draw (45,35) circle (4);
\draw (75,60) circle (4);
\draw (85,40) circle (4);
\draw (70,20) circle (4);
\draw[rotate around={45:(55,-40)}]   (55,-40) ellipse [x radius=20, y radius=10];
\draw (60,-35) circle (4);
\draw (74,-55) circle (4);
\draw (105,-45) circle (4);
\draw (100,-30) circle (4);

\end{tikzpicture}
\caption{(a) layered toast \hspace{50pt} (b) general toast} \label{fig:toast}
\end{figure}

In \cite{GJKS_forcing_2015} a forcing argument was used to show that a Borel layered
toast does not exist for $F(2^{\Z^d})$. The following result shows that
a continuous (unlayered) toast does not exist for $F(2^{\Z^d})$. Thus, the only
remaining possibility for $F(2^{\Z^d})$ is a Borel unlayered toast, which in fact
does exist for $F(2^{\Z^d})$ (a result of the authors to be presented in \cite{borcon}).

\begin{thmn} \label{thm:nct}
There does not exist an open toast structure on $F(2^{\Z^d})$
for any $d \geq 1$.
\end{thmn}

\begin{proof}
Suppose $\{ T_n\}$ were an open toast structure on $F(2^{\Z^d})$.
Let $x \in F(2^{\Z^d})$ be a hyper-aperiodic element. Let $K=\ocl{[x]}$,
so $K\subseteq F(2^{\Z^d})$ is compact. For each $m \in \N$, consider the
open set $U_m$ consisting of those points $y \in K$ having the property
that there is $n \leq m$ and a $T_n$-class $C$ with $y \in C \setminus \partial C$.
Then the $U_m$'s are increasing, open, and cover $K$. So there is a least $m$
with $U_m = K$. Since $m$ is least, $T_m \neq \varnothing$. Consider a $T_m$-class
$C'$. By definition of toast, $C'$ is finite. So we can find a point
$y \in \partial C'$. Since $y \in U_m$, there is $n \leq m$ and a $T_n$-class
$C$ with $y \in C \setminus \partial C$. Moreover, since $C \neq C'$ and $T_m$ is
an equivalence relation, we must have $n < m$. However, this implies that $C' \cap C \neq \varnothing$
but $C \not\subseteq C'$, contradicting Definition \ref{def:toast}.(1).
\end{proof}



\section{A hyper-aperiodic element with partial periodicity\label{hyper}}

We will construct a hyper-aperiodic element on $\Z^2$ that possesses some partial vertical periodicity.
We will use heavily the definitions of hyper-aperiodic and orthogonal given
in Definitions~\ref{def:twocol} and \ref{def:orthogonal}.

The following lemma is a warm-up.

\begin{lemn} \label{warmup} Let $x, y_0, y_1$ be hyper-aperiodic elements of $2^\Z$ with $y_0\,\bot\, y_1$. Then the element $z\in 2^{\Z^2}$ defined by
$$ z(u, v)=y_{x(u)}(v) $$
is hyper-aperiodic.
\end{lemn}

\begin{proof} Let $s=(a, b)$ be a non-identity element in $\Z^2$. First assume $a=0$.
Then $b\neq 0$. Let $T_0$ and $T_1$ be finite sets witnessing the hyper-aperiodicity of $y_0$ and $y_1$,
respectively, for $b$. Let $T=\{0\}\times (T_0\cup T_1)$. We claim that $T$ witnesses the hyper-aperiodicity of $z$ for $s$.
To see this, let $g=(u, v)\in \Z^2$ be arbitrary. If $x(u)=0$, then there is $t\in T_0$ such that
$$ z(u, v+t)=y_{x(u)}(v+t)=y_0(v+t)\neq
  y_0(v+b+t)=y_{x(u)}(v+b+t)=z(u, v+b+t), $$ or if we let
  $\tau=(0,t)\in T$,
$$ z(g+\tau)\neq z(g+s+\tau). $$ If $x(u)=1$, then similarly there is
  $t\in T_1$ such that $z(g+\tau)\neq z(g+s+\tau)$, where $\tau=(0,
  t)\in T$. This proves the lemma in the case $a=0$.

Now assume $a\neq 0$. Let $S$ be the finite set witnessing the hyper-aperiodicity of $x$ for $a$. Let $R$ be a symmetric (i.e. $r\in R$ implies $-r\in R$) finite set witnessing the orthogonality of $y_0$ with $y_1$. Note that $R$ also witnesses the orthogonality of $y_1$ with $y_0$. Let $T=S\times R$. We claim that $T$ witnesses the hyper-aperiodicity of $z$ for $s$. To see this, let $g=(u, v)\in \Z^2$ be arbitrary. Let $t\in S$ be such that $x(u+t)\neq x(u+a+t)$. Let $r\in R$ be such that $y_0(v+r)\neq y_1(v+b+r)$ (if $x(u+t)=0$) or $y_1(v+r)\neq y_0(v+b+r)$ (if $x(u+t)=1$). Letting $\tau=(t,r)$, we have
$$ z(u+t,v+r)=y_{x(u+t)}(v+r)\neq y_{x(u+a+t)}(v+b+r)=z(u+a+t, v+b+r) $$
or $z(g+\tau)\neq z(g+s+\tau)$. This proves the lemma in the case $a\neq 0$.
\end{proof}

Next we modify the construction of the hyper-aperiodic element $z$ by allowing it to possess some 
partial vertical periodicity. The partial vertical periodicity is governed by a 
function $f\colon \Z\to \Z^+$ in the sense that for each $u\in\Z$, $f(u)$ is a vertical period for $z(u, \cdot)$.

Let $\Lambda$ be an infinite set of prime numbers. Let $f:\Z\to \Z^+$ be a function such that
\begin{enumerate}
\item[(i)] For all $u\in\Z$ there are $p\in\Lambda$ and $n\in\Z^+$ such that $f(u)=p^n$;
\item[(ii)] For all $p\in\Lambda$ and $m\in\N$, there are $a\in \Z$ and $k\in \Z^+$ such that $f(i)=p^k$ for all $i\in[a, a+m]$;
\item[(iii)] $f(u)$ is monotone increasing for $u>0$, monotone decreasing for $u<0$, and $f(u)\to \infty$ as $|u|\to \infty$.
\end{enumerate}
Such functions $f$ can be constructed easily by diagonalization.

\begin{lemn} \label{main} Let $x, y_0, y_1$ be hyper-aperiodic elements of $2^\Z$ with $y_0\,\bot\, y_1$. Let $\Lambda$ be an infinite set of prime numbers and $f\colon \Z\to\Z^+$ be a function as above. 
Then the element $z\in 2^{\Z^2}$ defined by
$$ z(u, v)= y_{x(u)}(v\!\!\!\!\mod\! f(u)) $$
is hyper-aperiodic.
\end{lemn}

\begin{proof} Let $s=(a, b)$ be a non-identity element in $\Z^2$.
First assume $a=0$. Then $b\neq 0$. Let $T_0$ and $T_1$ be finite sets witnessing
the hyper-aperiodicity of $y_0$ and $y_1$, respectively, for $b$. Let $B\in \Z^+$
be greater than $|b|$ or $|t|$ for all $t\in T_0\cup T_1$. Let $A\in\Z^+$
be such that for all $|u|\geq A$, $f(u)>3B$. Let $T=[-A, A]\times [-3B, 3B]$.
We claim that $T$ witnesses the hyper-aperiodicity of $z$ for $s$.
To see this, let $g=(u, v)\in \Z^2$ be arbitrary.
First, there is $c\in [-A, A]$ such that $f(u+c)>3B$.
Without loss of generality assume $x(u+c)=0$.  Next, there is $d\in [-2B, 2B]$ such that
\begin{itemize}
\item the difference between $v+d$ and any multiple of $f(u+c)$ is
  greater than $B$, and
\item the difference between $v+b+d$ and any multiple of $f(u+c)$ is
  greater than $B$.
\end{itemize}
Finally, there is $t\in T_0$ such that
$$\begin{array}{rl} & z(u+c, v+d+t)=y_{x(u+c)}(v+d+t)=y_0(v+d+t)
  \\ \\ \neq & y_0(v+d+b+t)=y_{x(u+c)}(v+d+b+t)=z(u+c, v+d+b+t).
\end{array}$$
If we let $\tau=(c,d+t)\in T$, we have
$$ z(g+\tau)\neq z(g+s+\tau). $$ This proves the lemma in the case
$a=0$. The case $a\neq 0$ is similarly handled by a combination of the
above shifting technique and the proof of the corresponding case of
Lemma~\ref{warmup}.
\end{proof}

Note that in the proof of Lemma~\ref{main} only condition (iii) of the function $f$ is used to guarantee that $z$ is a hyper-aperiodic element. The other two conditions will be used in the proof of the following theorems.

\section{The continuous chromatic number of $F(2^{\Z^2})$}\label{section:1.5}

In this section we give a self-contained proof for the fact that there
are no continuous proper $3$-colorings on $F(2^{\Z^2})$. It has
been proved in \cite{gao_countable_2015} that there is a continuous proper
$4$-coloring on $F(2^{\Z^2})$. Combining these results, we conclude
that the continuous chromatic number for $F(2^{\Z^2})$ is exactly $4$.

\begin{thmn} \label{threecoloring} There is no continuous proper $3$-coloring of $F(2^{\Z^2})$.
Consequently, the continuous chromatic number of $F(2^{\Z^2})$ is $4$.
\end{thmn}

As an immediate corollary, we conclude that for any $n\geq 2$, the continuous chromatic number for $F(2^{\Z^n})$ is exactly $4$.

\begin{corn} \label{cor:threecol}
For any $n\geq 2$ the continuous chromatic number of $F(2^{\Z^n})$ is $4$.
\end{corn}

\begin{proof}
It has been proved in \cite{gao_countable_2015} that there is a
continuous proper $4$-coloring on $F(2^{\Z^n})$ for all $n\geq
2$. To see that the continuous chromatic number of $F(2^{\Z^n})$ is
not $3$, it suffices to notice that $F(2^{\Z^2})$ can be naturally identified
with a $\Z^2$-invariant subset of $F(2^{\Z^n})$, meaning there is a continuous
map $\varphi\colon F(2^{\Z^2})\to F(2^{\Z^n})$ such that whenever $(a_1, a_2)\cdot
x=y$ we have $(a_1, a_2, 0, \dots, 0)\cdot
\varphi(x)=\varphi(y)$. Granting such a $\varphi$, if $c\colon
F(2^{\Z^n})\to 3$ were a proper $3$-coloring, then $\varphi\circ c\colon
F(2^{\Z^2})\to 3$ would also be a proper $3$-coloring. The identification
$\varphi$ can be defined by sending $x\in F(2^{\Z^2})$ to the function $\varphi(x)$ satisfying
$$ \varphi(x)(a_1, a_2, \dots, a_n)=\left\{\begin{array}{ll} x(a_1,
  a_2), & \mbox{ if $a_3=\dots=a_n=0$,} \\ 0, & \mbox{
    otherwise.} \end{array}\right. $$ Then if $x\in F(2^{\Z^2})$ one
  can easily check $\varphi(x)\in F(2^{\Z^n})$. Obviously $\varphi$
  also satisfies the equivariant requirement.
\end{proof}

The rest of this subsection is devoted to a proof of Theorem~\ref{threecoloring}. 
Toward a contradiction, assume $c: F(2^{\Z^2})\to \{0,1,2\}$ is a continuous proper $3$-coloring of $F(2^{\Z^2})$.

In the following we denote $e_1=(1,0)$ and $e_2=(0,1)$. We say $x, y \in F(2^{\Z^2})$ 
are {\em adjacent} if there is $i = 1, 2$ with $x = e_i \cdot y$ or $y = e_i \cdot x$ (i.e., $x$ 
and $y$ are adjacent in the Schreier graph $\Gamma(F(2^{\Z^2}))$). For adjacent $x$ and $y$, let
$$ \theta(x,y)=\left\{\begin{array}{ll}
+1, & \mbox{ if $c(y)=c(x)+1\!\!\!\mod 3$,} \\
-1, & \mbox{ if $c(y)=c(x)-1\!\!\!\mod3$.}
\end{array}\right.
$$
We define a potential function $P(x,y)$ for $x, y\in F(2^{\Z^2})$ with $[x] = [y]$.
Let $x_0=x, x_1, \dots, x_{k+1}=y$ be a {\em walk} in the Schreier graph $\Gamma(F(2^{\Z^2}))$ 
from $x$ to $y$, i.e., $x_i$ and $x_{i+1}$ are adjacent for all $i=0, \dots, k$ (we do not 
require the $x_i$'s to be distinct). Then let (where the following sum in taken in $\Z$, not $\Z/3\Z$)
$$ P(x,y)=\sum^k_{i=0} \theta(x_i, x_{i+1}). $$
We first verify that $P(x,y)$ is well-defined.

\begin{lemn} $P(x,y)$ does not depend on the path chosen from $x$ to $y$.
\end{lemn}

\begin{proof}
For a walk $\pi = (x_0, \ldots, x_{k+1})$ in the Schreier graph
$\Gamma(F(2^{\Z^2}))$, set $P(\pi) = \sum_{i=0}^k \theta(x_i,
x_{i+1})$. We write $\pi^{-1} = (x_{k+1}, \ldots, x_0)$ for the
reversal of $\pi$. If $\pi$ ends at $y$ and $\sigma$ begins at $y$,
then we write $\pi \sigma$ for the concatenated path. Note that
$P(\pi^{-1}) = - P(\pi)$ and that $P(\pi \sigma) = P(\pi) +
P(\sigma)$. We say that $\pi$ is backtracking if there is $i$ with
$x_i = x_{i+2}$. Note that two paths $\pi$ and $\pi'$ are homotopy
equivalent if and only if they become equal after removing all
instances of back-tracking (this can be taken as a definition if
desired), and that $P(\cdot)$ is invariant under homotopy equivalence.
Let $G$ be the homotopy group of the Schreier graph at the point $x$. 

Write $K_3$ for the triangle graph having vertex set $\{0, 1,
2\}$. Note that the coloring $c$ provides a graph homomorphism from
$F(2^{\Z^2})$ to $K_3$. Write $\sigma_z$ for the path $(z, (1,
0) \cdot z, (1, 1) \cdot z, (0, 1) \cdot z, z)$. Since $K_3$ has
no $4$-cycles, the image of $\sigma_z$ under $c$ must be trivial up to
homotopy equivalence.  It follows $P(\sigma_z) = 0$. Similarly, if
$\ell_{w,z}$ is a walk from $w$ to $z$ then $P(\ell_{w,z} \sigma_z
\ell_{w,z}^{-1}) = P(\ell_{w,z}) + 0 - P(\ell_{w,z}) = 0$.

Now suppose that $\pi_1$ and $\pi_2$ are paths from $x$ to $y$. We
claim that $\ell=\pi_1 \pi_2^{-1}$ is homotopy equivalent to a
concatenation of paths of the form $\ell_{x,z} \sigma_z^{\pm 1}
\ell_{x,z}^{-1}$. Let $N$ be the normal subgroup of $G$ 
consisting of products of elements of the form $\ell_{x,z} \sigma_z^{\pm 1}
\ell_{x,z}^{-1}$. We show that every path $\ell$ in $G$ is homotopy
equivalent to an element of $N$. 

View $x$ as the origin in $\Z^2$, and identify $\ell$ 
as a sequence of directed edges in the Cayley graph of $\Z^2$. 
Let $z=(a,b)$ be a vertex of the path $\ell$ which as a point in $\Z^2$
is of maximal distance from the origin. If $z=(0,0)$, then $\ell$
is the trivial path and we are done. Otherwise, let $u$ be the vertex of $\ell$
immediately preceding $z$ on the path, and let $v$ be the vertex immediately
following $z$. One of the coordinates of $u$ is strictly less in absolute value 
than the corresponding coordinate of $z$, and the other coordinate of $u$
is equal to the corresponding coordinate of $z$, and likewise for $v$. 
If $u=v$, then we may delete the part of the path  $uzv$ from $\ell$ 
to a shorter path $\ell'$ which is homotopy equivalent to $\ell$. By induction on the length of $\ell$,
the result holds for $\ell'$ and we are done. 
If $u \neq v$, then say (the other cases being similar) $u=(a,b-1)$, $v=(a-1,b)$. 
Then we may replace $uzv$ with $u z'v$ where $z'=(a-1,b-1)$. 
The new path $\ell'$ is equal to $\ell$ modulo an element of $N$, has the same length as $\ell$, and 
either has a strictly smaller value for the maximum distance of a vertex from $(0,0)$
or else has the same maximum and a strictly smaller number of vertices attaining this maximum. 
By induction we may assume the result holds for $\ell'$ and we are done.

\end{proof}

Let $z$ be the hyper-aperiodic element on $\Z^2$ as constructed in Lemma~\ref{main}. Then $\overline{[z]}\subseteq F(2^{\Z^2})$ and $c\!\upharpoonright\!\overline{[z]}: \overline{[z]}\to \{0,1,2\}$ is a continuous function on a compact space. Thus there is $N\in\Z^+$ such that for all $x\in\overline{[z]}$, $x\!\upharpoonright\![-N, N]^2$ completely determines $c(x)$.

Fix two distinct prime numbers $p, q\in\Lambda$ where $p, q\neq 2, 3$. By condition (ii) of the function $f$, there are $a, b\in\Z$ and $n, k\in \Z^+$ such that $f(i)=p^n$ for all $i\in[a, a+2N+1]$ and $f(i)=q^k$ for all $i\in[b, b+2N+1]$. Let $x=(a+N+1, 0)\cdot z$ and $y=(b+N+1, 0)\cdot z$. By the construction of $z$, we have that for all $i\in[-N, N]$, $i\cdot x(0, \cdot)$ is vertically periodic with a period $p^n$ and $i\cdot y(0, \cdot)$ is vertically periodic with a period $q^k$. By the choice of $N$, we conclude that $c( (0,t)\cdot x)$ is periodic with respect to $t$ with a period $p^n$, i.e.,
$$ c((0, t)\cdot x)=c((0, t+p^n)\cdot x) $$
for all $t\in\Z$; similarly, $c((0,t)\cdot y)$ is periodic with respect to $t$ with a period $q^k$.

Consider the horizontal path from $x$ to $y$. Its length is $|a-b|$. Thus $|P(x,y)|\leq |a-b|$. Set $D=p^nq^k$ and let $m\in \Z^+$ be arbitrary. Let $x'=(0, mD)\cdot x$ and $y'=(0, mD)\cdot y$. Then the horizontal path from $x'$ to $y'$ is still of length $|a-b|$, and we also have $|P(x', y')|\leq |a-b|$.

Define
$$ r=\displaystyle\frac{P(x, (0,p^n)\cdot x)}{p^n} \mbox{ and } s=\displaystyle\frac{P(y, (0, q^k)\cdot y)}{q^k}. $$
Intuitively, $r$ is slope of the change of the potential from $x$ to $(0, p^n)\cdot x$, and $s$ is the slope of the change of the potential from $y$ to $(0, q^k)\cdot y$. We claim that $|r|<1$ and $r\neq 0$. Similarly for $s$. To see the claim, note first that if $r=1$ then $P(x,(0, p^n)\cdot x)=p^n$. Consider the vertical path $x_0, \dots, x_{p^n}$ from $x$ to $(0, p^n)\cdot x$. Its length is $p^n$, which implies that for all $i=0, \dots, p^n$, $\rho(x_i, x_{i+1})=+1$. It follows that $c(x_i) = c(x_j)$ if and only if $i \equiv j \mod 3$. Since $p \neq 3$ is prime, this implies $c((0,p^n) \cdot x) = c(x_{p^n}) \neq c(x_0) = c(x)$, contradicting the fact that by design $t \mapsto c((0, t) \cdot x)$ is $p^n$-periodic. A similar argument rules out the possibility that $r=-1$. To see that $r\neq 0$, just note that in order for $r=0$ we must have $P(x, (0, p^n)\cdot x)=0$, and this cannot happen unless $p^n$ is even. Since $p$ is an odd prime, we are done.

Finally, since $p\neq q$, we conclude in addition that $r\neq s$.

By periodicity, we have
$$ P(x, x')=mD r \mbox{ and } P(y, y')=mD s. $$
Since
$$ P(x, x')+P(x',y')=P(x,y)+P(y,y'), $$
we have
$$ P(x',y')-P(x,y)=P(y,y')-P(x,x')=mD(s-r). $$
Since $P(x,y), P(x',y')\in [-|a-b|, |a-b|]$ and $r\neq s$, this is a contradiction when $m$ is large enough.

We remark that in the above proof we did not use condition (i) of the function $f$. This condition will be used in the proofs of  the theorems in the next two sections.

\section{CBS line sections and single line sections}\label{section:1.6}

A {\em line section} is a symmetric Borel relation $S$, which we think
of as a graph, which is a subset of the Schreier graph
$\Gamma(F(2^{\Z^2}))$ which is acyclic and and which satisfies $\deg_S(x) \in \{0, 2\}$ for
all $x \in F(2^{\Z^2})$. We write $V(S) = \{x \in F(2^{\Z^2}) \colon
\deg_S(x) = 2\}$. Intuitively, the intersection of a line section $S$
with an orbit is a collection of lines in the graph of
$F(2^{\Z^2})$. We call these lines {\em $S$-lines}. We say that $S$ is
{\em clopen} if for $i = 1, 2$ the set $\{x \in F(2^{\Z^2}) \colon (x, e_i
\cdot x) \in S\}$ is clopen, where $e_1=(1,0)$ and $e_2=(0,1)$. We say that $S$ is a {\em complete
  section} if $[x] \cap V(S) \neq \varnothing$ for all $x$, and we
call $S$ {\em co-complete} if $[x] \cap (F(2^{\Z^2}) \setminus V(S))
\neq \varnothing$ for all $x$. If $x$ and $y$ are on the same
$S$-line, then the {\em $S$-path-distance} between $x$ and $y$,
denoted as $\rho_S(x,y)$, is the length of the shortest $S$-path
between $x$ and $y$. Thus $\rho_S(x,y)=1$ if and only if $x$ and $y$
are adjacent elements of an $S$-line.

For example, if $S$ is defined as
$$ (x, y)\in S \iff (1,0)\cdot x=y \mbox{ or } (-1,0)\cdot x=y, $$
then $S$ is a clopen complete line section, but not co-complete.
In other words, the entire space can be viewed as a collection of horizontal $S$-lines.
In general, we call a line section $S$ a {\em horizontal line section} 
if every edge in $S$ is a horizontal ($e_1$) edge.
We remark that a simple category argument gives the following result.

\begin{thmn} There does not exist any Borel horizontal line section of $F(2^{\Z^2})$ which is both complete and co-complete.
\end{thmn}

\begin{proof} Assume $S$ is a Borel complete and co-complete horizontal line section of $F(2^{\Z^2})$. 
Note that $F(2^{\Z^2})\sms S$ is also such a set. By considering vertical translations of $S$, 
we note that countably many homeomorphic images of $S$ cover the entire $F(2^{\Z^2})$. 
It follows that $S$ is non-meager, and likewise for $F(2^{\Z^2})\sms S$. On the other hand, 
by considering only the horizontal translations, which is a continuous action of $\Z$ on $F(2^{\Z^2})$, 
we have that $S$ is a Borel invariant subset of $F(2^{\Z^2})$. 
It follows from the First Topological 0-1 Law (\cite{Kechris_book} Theorem 8.46) 
that $S$ must be either meager or comeager, a contradiction.
\end{proof}

In other words, the only complete Borel horizontal line section is the above trivial one that we demonstrated.
In the following we consider line sections with bounded slopes from a large-scale perspective, which we call {\em line sections with coarsely bounded slopes} or {\em CBS line sections}.

\begin{defnn}\label{def:acute} A line section $S$ of $F(2^{\Z^2})$ has {\em coarsely bounded slopes} 
or is a {\em CBS line section}
if for any $x\in F(2^{\Z^2})$ there exist $D,M\in\Z^+$ such that for any $x_1, x_2\in [x]\cap V(S)$
lying on the same $S$-line, if $\rho_S(x_1, x_2)>D$ and $(u, v) \cdot x_1 = x_2$, then $|v|\leq M |u|$.
\end{defnn}

Thus, on each equivalence class there is a bound $M$ on the slopes of the lines in $S$, at least for
pairs of points which are at least $D$ apart with respect to $\rho_S$.
We have the following negative result about CBS line sections,
whose proof is very similar to that of Theorem~\ref{threecoloring}.

\begin{thmn} \label{ahline} There does not exist any clopen CBS
line section of $F(2^{\Z^2})$ which is both complete and co-complete.
\end{thmn}

\begin{proof}
Toward a contradiction, assume $S$ is a clopen CBS line
section of $F(2^{\Z^2})$ which is both complete and co-complete. Let
$z$ be the hyper-aperiodic element on $\Z^2$ constructed in
Lemma~\ref{main}. Let $D, M$ be given by the definition of CBS
line sections with $x=z$. By assumption the sets
$\{x \in \overline{[z]} \colon (x, e_i \cdot x) \in S\}$, $i = 1, 2$, are
clopen. It follows from compactness of $\overline{[z]}$ that there is
$N \in \Z^+$ such that for every $x \in \overline{[z]}$, $x\!
\upharpoonright\! [-N,N]^2$ determines whether $(x, e_1 \cdot x)$ and
$(x, e_2 \cdot x)$ are in $S$.

Note that by the boundedness property of $S$, every $S$-line meets every vertical
line in $[z]$, i.e., for every $a \in \Z$, the vertex set of each $S$-line
has a non-empty intersection with the set $\{(a, v)\cdot z\,|\, v\in \Z\}$. Indeed, if
$(a_n, b_n) \cdot z$ enumerates a bi-infinite path in $S$ then
$|a_{n+1}-a_n| \leq 1$ for every $n$ and the quantities $|a_{n+D+1} -
a_n|$ and $|a_{n+D+2}-a_n|$ are never $0$, so the signs of $a_{n+D+1}
- a_n$ and $a_{n+D+2}-a_n$ never change and thus $\{a_n : n \in \Z\} =
\Z$.

Fix two distinct prime numbers $p, q\in\Lambda$. By condition (ii) of
the function $f$ from \S\ref{hyper}, there are $a, b\in\Z$ with $|b-a|\geq D$ and $n,
k\in \Z^+$ such that $f(i)=p^n$ for all $i\in[a, a+2N+2D+1]$ and
$f(i)=q^k$ for all $i\in[b, b+2N+2D+1]$. We assume $a < b$, as the
other case is essentially identical. Let $x=(a+N+D+1, 0)\cdot z$ and
$y=(b+N+D+1, 0)\cdot z$. By the construction of $z$, we have that for
all $i\in[-N-D, N+D]$, $(i,0)\cdot x(0, \cdot)$ is vertically periodic
with a period $p^n$ and $(i,0)\cdot y(0, \cdot)$ is vertically periodic
with a period $q^k$. By the choice of $N$, we conclude that $S
\upharpoonright \{(i,t) \cdot x \colon |i| \leq D, \ t \in \Z\}$ is
vertically periodic with period $p^n$; similarly $S \upharpoonright
\{(i,t) \cdot y \colon |i| \leq D, \ t \in \Z\}$ is vertically periodic
with period $q^k$.  Note that for all $v,v'\in\Z$,  if $(0, v) \cdot x$ and $(0, v') \cdot
x$ lie on the same $S$-line then their $S$-path-distance must be at
most $D$, and hence the $S$-path connecting them is contained in the
vertical strip $\{(i,t) \cdot x \colon |i| \leq D, \ t \in \Z\}$. Writing
$[S]$ for the equivalence relation given by the connected components
of $S$, it follows that the restriction of $[S]$ to the vertical line
extending from $x$ is $p^n$-periodic and the restriction of $[S]$ to
the vertical line extending from $y$ is $q^k$-periodic.

Define $\ell$ to be the number of equivalence classes of the
restriction of $[S]$ to the vertical line extending from $x$
over some/any interval of size $p^n$
(that is the number of distinct $S$-classes in a set of the form
$(0,s)\cdot x, \dots, (0,(s+p^n-1))\cdot x$) divided by $p^n$.
Intuitively, $\ell$ is the density of $S$-lines along the
vertical line above $x$. Similarly define $r$ to be the
number of equivalence classes of the restriction of $[S]$ to the
vertical line extending from $y$ in some/any interval of size $q^k$
divided by $q^k$. So, $r$ is the density of the $S$-lines above $y$.
We note that $\ell, r\neq 0$ by the above
observation that every vertical line meets all the $S$-lines. Since
$S$ is co-complete, we also have that $\ell, r\neq 1$. Otherwise, if
$\ell = 1$, every point in the vertical line extending from $x$ would
belong to a distinct $S$-line, hence $S \cap [z]$ would necessarily
consist of all horizontal lines in $[z]$ (if all of the distinct points on
the vertical line through $x$ belong to distinct $S$-classes then, 
by disjointness of the distinct $S$-lines, 
for every point $(0,i)\cdot x$, the $S$-line through $(0,i)\cdot x$
must move horizontally, that is, the point $(1,i)\cdot x$ 
is on the same $S$-line). This contradicts that $S$ is
co-complete. Again, we conclude that $\ell \neq r$ since $p\neq q$.

Let $m \in \Z^+$ be an arbitrary multiple of all $f(i)$ for $i\in[a,b]$.
Let $x'=(0, m)\cdot x$ and $y'=(0, m)\cdot y$.  Let $L$ be the
segment from $x$ to $x'$ and $R$ be the segment from $y$ to $y'$.  The
restriction of $S$ to the vertical strip $H = \{(u, v)
\cdot x \colon 0 \leq u \leq b - a, \ v \in \Z\}$ is periodic with period
dividing $m$. Let $t$ be a positive integer, and consider the vertical strip 
$H_t=\{ (u,v)\cdot x\colon 0 \leq u \leq b-a, 0 \leq v \leq tm\}$. 
Let $L_t$ be the left edge of this strip and let $R_t$ be the right edge. 
By periodicity, the number of distinct $S$-classes meeting $L_i$ is 
$\ell tm$, and the number meeting the $R_t$ is $rtm$. So, on the one hand we have that 
the difference between the number of distinct $S$-lines intersecting the left-edge $L_t$
of $H_t$ and the right-edge $R_t$ of $H_t$ would be $|\ell -r| tm$.
On the other hand,  this difference is bounded by a constant
independent of $t$ using the fact that the lines in $S$
have a slope bounded by $M$. This is a contradiction for $t$ large enough. 

\end{proof}

In view of Theorem~\ref{ahline} it is natural to ask to what extent we
can remove the requirement that the line section of $F(2^{\Z^2})$
consists of CBS lines as defined in Definition~\ref{def:acute}. 
In the forthcoming \cite{borcon} we will prove that there does exist a 
Borel line section of $F(2^{\Z^2})$ such that every class is covered by a single line. 
Such a line section is called a {\em lining} of the whole space.
This result has also been obtained independently by 
N.\ Chandgotia and S.\ Unger \cite{ChandgotiaUnger} by a different argument.

In the following, we rule out a
clopen line section $S$ of $F(2^{\Z^2})$ with the property that each
class $[x]$ contains exactly one $S$-line.  We call such sections {\em
single line sections}.

\begin{thmn} \label{sline}
There does not exist any clopen single line section of $F(2^{\Z^2})$. In particular,
there does not exist any clopen lining of $F(2^{\Z^2})$.
\end{thmn}

\begin{proof}
Let $S$ be a clopen single line section. Let $x \in F(2^{\Z^2})$ be a hyper-aperiodic element (for this argument, no other
special properties of $x$ are required).

Let $K=\overline{[x]}$, so $K\subseteq F(2^{\Z^2})$
is compact. The compactness of $K$ and the clopenness of $S$ give that there is a $d_0\in \Z^+$
such that for any $y \in K$, $y\res [-d_0,d_0]^2$ determines the membership of $y$ in $V(S)$ and in $\{w : (w, e_i \cdot w) \in S\}$, $i = 1, 2$. The
compactness of $K$ and the fact that $S$ is a complete section on $F(2^{\Z^2})$ also give a
$d_1$ such that every $y \in K$ is within $\rho$ distance $d_1$ of a point in $V(S)$.
Finally, the compactness of $K$ and the fact that $S$ is a single line on each class
give that there is a $d_2$ such that for all $y \in K$ and all points $z,w$
within $\rho$ distance $5d_1$ of $y$, if $z,w \in V(S)$ then $ \rho_S(z,w)\leq d_2$.
This is because the quantity $\rho_S(z, w)$ is continuously determined from $y$
(if $\rho_S(z, w) = d$ then the $S$-path of length $d$ joining them is contained in a ball
of bounded radius around $y$, and the behavior of $S$ in that ball is determined by the
restriction of $y$ to some larger, but still finite, ball).

Fix $x_0\in [x] \cap V(S)$. Consider the horizontal line $L$ through $x_0$, that is, the points of the form
$(a,0)\cdot x_0$ for all $a\in\Z$. From the definition of $d_1$, for every $n \in \Z$ we
can fix a point $x_n \in V(S)$ with $\rho(x_n, (3n d_1, 0) \cdot x_0) \leq d_1$.
Observe that $x_{n+1}=(a,b)\cdot x_n$ with $d_1 \leq a \leq 5d_1$ and $x_n=(a',b')\cdot x_0$
with $|b'|\leq d_1$. That is, the points $x_n$ are horizontally spaced with sequential distance
between $d_1$ and $5d_1$, and are within vertical distance $d_1$ of $L$.
This is illustrated in Figure~\ref{fig:sline}. From the definition of $d_2$,
the path along $S$ between consecutive points $x_n$ and $x_{n+1}$
has length at most $d_2$. In particular, all of these paths stay within the
horizontal strip centered about $L$ of width $2(d_2+d_1)$ (see Figure~\ref{fig:sline}).

\begin{figure}[h]
\centering
\begin{tikzpicture}[scale=0.05]

\path[fill=lightgray] (-100,-10) rectangle (100,10);
\draw(-100,0) to (100,0);
\draw[blue] (-100,25) to (100,25);
\draw[blue] (-100,-25) to (100,-25);
\draw[blue] (-100,10) to (100,10);
\draw[blue] (-100,-10) to (100,-10);

\draw[red,fill=red] (0,0) circle (1);

\draw[red,fill=red] (25,2.5) circle (1);
\draw[red,fill=red] (65,2.5) circle (1);
\draw[red,fill=red] (90,-7.5) circle (1);

\node[below left]  at (5,0)  {$x_0$};

\draw[red] (0,0) to (0,5) to (-5,10) to (0, 12.5) to (5,5) to (10,2.5) to (15,7.5);
\draw[red] (15,7.5) to (10,15) to (15, 22.5) to (20,15) to (25,2.5);
\draw[red] (25,2.5) to (30, 5) to (35, -2.5) to (40, -7.5) to (45,-2.5);
\draw[red] (45,-2.5) to (50,-7.5) to (55,-15) to (60, -5) to (65, 2.5);
\draw[red] (65,2.5) to (70, 12.5) to (65, 20) to (70,22.5) to (75, 15) to (73, 5) to (90,-7.5);

\draw[red,fill=red] (-33,7.5) circle (1);
\draw[red,fill=red] (-65,-7.5) circle (1);

\draw[red] (0,0) to (-5,5) to (-10, 17.5) to (-15,10) to (-20,5);
\draw[red] (-20,5) to (-25,7.5) to (-30,15) to (-35, 20) to (-40,12.5) to (-33,7.5);
\draw[red] (-33, 7.5) to (-40, -5) to (-38, -17.5) to (-47, -20) to (-50,-12.5) to (-65,-7.5);

\node[left] at (-100,10) {$d_1$};
\node[left] at (-100,25) {$d_1+d_2$};

\end{tikzpicture}
\caption{Analyzing $S\cap [x]$} \label{fig:sline}
\end{figure}Let $S'\subseteq S\upharpoonright [x]$ be the union of these paths connecting $x_n$ and $x_{n+1}$ for all $n\in\Z$.
Arbitrarily fix an orientation of $V(S)\cap [x]$ as a line. This orientation gives rise to a linear order $<_S$ on $V(S) \cap [x]$. If $z,w \in V(S) \cap [x]$, then $z$ and $w$ lie on the same $S$-line; define $t(z,w)=\rho_S(z,w)$ if $z<_Sw$, and $t(z,w)=-\rho_S(z,w)$ if $w<_S z$.
Let $T=\{ t \in \Z \colon \exists y \in V(S')\ t(x_0,y)=t\}$. This is the set of ``times''
$t$ for which, starting at the point $x_0$, the element on the $S$-line that is $t$ steps away from $x_0$ is in the set $V(S')$. Since $S'$
contains the union of connected subsets of $S$ connecting $x_n$ to $x_{n+1}$,
it follows that $T$ is an infinite interval in $\Z$. We claim that $T=\Z$. Suppose,
for example, $T=[a,+\infty)$. Then obviously $a<0$ since $0\in T$. Consider $x_{a-1}$ and $x_{|a|+1}$. 
By our construction, we have both $\rho(x_{a-1}, x_0)\geq |a-1|=|a|+1$ and $\rho(x_{|a|+1}, x_0)\geq |a|+1$, 
and therefore $|t(x_{a-1}, x_0)|, |t(x_{|a|+1},x_0)|\geq |a|+1$. 
By construction the $S$-line segment between $x_{a-1}$ and $x_{|a|+1}$ passes 
through $x_0$, thus $t(x_{a-1},x_0)t(x_{|a|+1},x_0)<0$. 
This implies that one of $t(x_{a-1}, x_0)$ and $t(x_{|a|+1},x_0)$ must be smaller than $a$, a contradiction.

Thus, $T=\Z$ and so $S'=S$. Thus, $S$ is confined to the horizontal strip
of width $2(d_2+d_1)$ centered about $L$. This contradicts the fact that
every point of $[x]$ is within $d_1$ of a point in $V(S)$.
\end{proof}

The following questions remain unsolved.

\begin{quesn}
Does there exist a clopen line section of $F(2^{\Z^2})$ that is both complete and co-complete?
\end{quesn}

\begin{quesn}
Does there exist a clopen line section $S$ of $F(2^{\Z^2})$ such that there are finitely many $S$-lines on each class?
\end{quesn}

\section{Continuous cocycles}\label{section:1.7}

The potential function $P(x, y)$ we defined in the proof of Theorem~\ref{threecoloring} is in fact a 
cocycle from $F(2^{\Z^2})$ to $\Z$. The essential part of the proof was an analysis of how this cocycle works. 
In this subsection we give a more general analysis for any continuous cocycle from $F(2^{\Z^2})$ to $\Z$. 
The main result is that they can be approximated by homomorphisms. 
We recall the definition of a cocycle in a general context.

\begin{defnn}
For groups $\G$ and $H$ and an action $\G \acts X$, recall that a {\em cocycle} 
is a map $\sigma \colon \G \times X \rightarrow H$ satisfying the following cycle identity: 
for all $g_1, g_2 \in \G$ and $x \in X$ we have
$$\sigma(g_1 g_2, x) = \sigma(g_1, g_2 \cdot x) \cdot \sigma(g_2, x).$$
\end{defnn}

In fact, we show this cocycle analysis in the more general 
context of residually finite groups, which is stated in Theorem~\ref{rfcocycle}. 
As an immediate consequence we obtain the analysis of $\Z$-cocycles for $\Z^2$
actions, which we state as Theorem~\ref{Zcocycle}.
Recall that a countable group $\G$ is {\em residually finite} 
if it admits a decreasing sequence of finite-index normal subgroups $\G_n$ with $\bigcap_n \G_n = \{1_\G\}$. 
Finitely generated abelian groups are residually finite.

\begin{lemn} \label{proaction}
Let $\G$ be a countable residually finite group, and let $(\G_n)_{n \in \N}$ be a 
decreasing sequence of finite-index normal subgroups with $\bigcap_n \G_n = \{1_\G\}$. 
Then there is a hyper-aperiodic element $x \in F(2^\G)$ with the following property. 
If $\mathcal{P}$ is any clopen partition of $F(2^\G)$ then there is a finite set $A \subseteq \G$ 
so that for every $n \in \N$ the restriction of $\mathcal{P}$ to $(\G \setminus A \G_n) \cdot x$ is $\G_n$-invariant.
\end{lemn}

\begin{proof}
This is trivial if $\G$ is finite, so we assume $\G$ is infinite. By passing to a subsequence, we may assume $\G_{n+1}$ is a proper subgroup of $\G_n$ for every $n$. Define $x \in 2^\G$ by setting $x(1_\G) = 0$ and for $g \neq 1_\G$ setting
$$x(g) = \min \{n \in \N : g \not\in \G_n\} \mod 2.$$
We check that $x$ is hyper-aperiodic. Fix $1_\G \neq s \in \G$ and let $k$ be least with $s \not\in \G_k$. Let $T$ be a finite set satisfying $T^{-1} (\G_k \setminus \G_{k+1}) = \G$. Such a $T$ exists since $\G_k \setminus \G_{k+1}$ contains a coset of $\G_{k+1}$ and $\G_{k+1}$ has finite index in $\G$. Now fix $g \in G$. Pick $t \in T$ with $t g \in \G_k \setminus \G_{k+1}$. Then $x(t g) = k + 1 \mod 2$ while $x(t s g) = k \mod 2$ since, by normality of the $\G_n$'s, $t s g = (t s t^{-1}) t g \in (\G_{k-1} \setminus \G_k) \G_k = \G_{k-1} \setminus \G_k$. This proves $x$ is hyper-aperiodic.

We claim that $x$ is constant on $\G_n h$ when $h \not\in \G_n$. Indeed, if $g \in \G_n$ and $k < n$ satisfies $h \in \G_k \setminus G_{k+1}$ then $g h \in \G_k \setminus \G_{k+1}$ as well and thus both $x(h)$ and $x(g h)$ are equal to $k+1 \mod 2$.

Now let $\mathcal{P}$ be a clopen partition of $F(2^\G)$. Since $\overline{[x]}$ is compact there is a finite $W \subseteq G$ such that for every $y \in \overline{[x]}$, the piece of $\mathcal{P}$ to which $y$ belongs is determined by $y \upharpoonright W$. We will check the lemma is satisfied with $A = W^{-1}$. So fix $n \in \N$, and for non-triviality assume $\G\setminus W^{-1}\G_n\neq\varnothing$. Fix $y \in (\G \setminus W^{-1} \G_n) \cdot x$, say $y = s \cdot x$ with $s \not\in W^{-1} \G_n$. Note that $w s \not\in \G_n$ for every $w \in W$ and so the claim of the previous paragraph applies. Using normality of $\G_n$, we have that for every $g \in \G_n$ and $w \in W$
$$(g \cdot y)(w) = (g s \cdot x)(w) = x(w g s) = x((w g w^{-1}) w s) = x(w s) = (s \cdot x)(w) = y(w).$$
By letting $w \in W$ vary, we find that $g \cdot y$ and $y$ belong to the same piece of $\mathcal{P}$. Since $g \in \G_n$ was arbitrary we conclude that the restriction of $\mathcal{P}$ to $(\G \setminus W^{-1} \G_n) \cdot x$ is $\G_n$-invariant.
\end{proof}

Recall that a finitely generated group $\G$ is \emph{one-ended} if for every finite set 
$A \subseteq \G$ the graph on $\G \setminus A$ induced from the Cayley graph of $\G$ 
has the property that precisely one of its connected components is infinite 
(this property does not depend upon the choice of finite generating set for $\G$ 
used in constructing the Cayley graph).

\begin{thmn} \label{rfcocycle}
Let $\G$ be a finitely-generated, one-ended, residually finite group, 
let $H$ be a countable group, and let $\delta \colon \G \times F(2^\G) \rightarrow H$ be a continuous cocycle. 
Let $(\G_n)_{n \in \N}$ be a decreasing sequence of finite-index normal subgroups of $\G$ with 
$\bigcap_n \G_n = \{1_\G\}$. Then there are $x \in F(2^\G)$, $n \in \N$, a 
group homomorphism $\phi \colon \G_n \rightarrow H$, and finite sets $T \subseteq G$, $W \subseteq H$ such that
\begin{enumerate}
\item[\rm (i)] $\delta(g, x) = \phi(g)$ for all $g \in \G_n$,
\item[\rm (ii)] $T \G_n = \G$ and if $g \in \G_n$ and $t \in T$ then $\delta(t g, x) \in W \cdot \phi(g)$.
\end{enumerate}
\end{thmn}

\begin{proof}
Let $y \in F(2^\G)$ be the hyper-aperiodic element constructed in Lemma \ref{proaction}. 
Fix a finite generating set $S$ for $\G$. The map $z \in F(2^\G) 
\mapsto (\delta(s, z))_{s \in S \cup S^{-1}} \in H^{S \cup S^{-1}}$ is continuous by assumption and 
thus produces a clopen partition $\mathcal{P}$ of $F(2^\G)$. 
Thus $z_1$ and $z_2$ belong to the same set in the partition $\mathcal{P}$ 
if and only if for all $s\in S\cup S^{-1}$, $\delta(s,z_1)=\delta(s,z_2)$. 
Let $A \subseteq \G$ be the finite set given by Lemma \ref{proaction}. By one-endedness, 
removing the vertices $A$ from $\Gamma_{\G,S}$ leaves only one infinite connected component 
and possibly a finite number of finite connected components. By enlarging $A$ if necessary, 
we may assume that $\Gamma_{\G,S} \setminus A$ is connected. 
For each pair of elements $g_1, g_2 \in \G \setminus A$ which are adjacent to $A$, 
fix a path in $\Gamma_{\G,S} \setminus A$ from $g_1$ to $g_2$. Let $B$ 
be a finite set containing all such paths with $B\cap A=\emptyset$. Now pick $t \in \G$ and $n \in \N$ 
so that $\G_n$ is disjoint with $A t^{-1} \cup ((t A^{-1} A t^{-1} \cup t B^{-1} B t^{-1}) \setminus \{1_\G\})$. 
It follows from this that $\G_n\cap A t^{-1} \G_n=\varnothing$
 and that $A t^{-1} g_1 \cap A t^{-1} g_2 = B t^{-1} g_1 \cap B t^{-1} g_2 = \varnothing$ 
for all $g_1 \neq g_2 \in \G_n$.

Fix a finite symmetric generating set $V$ for $\G_n$ (this exists since $\G_n$ is finite-index in $\G$). 
Fix $v \in V$ and consider a word $w$ in $S \cup S^{-1}$ representing a path from $1_\G$ to $v$ in $\Gamma_{\G,S}$. 
If this path does not intersect $A t^{-1} \G_n$ then set $w(v) = w$. 
Otherwise we will modify the path $w(v)$, without changing its endpoints, 
so that it does not intersect $A t^{-1} \G_n$. We can do this by applying 
the definition of $B$: if $w$ meets $A t^{-1} g$ then by considering the vertex 
it first enters $A t^{-1} g$ from and the last vertex it exits $A t^{-1} g$ to, 
we can replace the path in $w$ between those two vertices by a path in $B t^{-1} g$. 
Since both families of sets $A t^{-1} g$ and $B t^{-1} g$, $g \in \G_n$, are pairwise disjoint 
as $g \in \G_n$ varies, we can make all such changes to $w$ in a single step. We let $w(v)$ 
be the newly obtained path. In summary, for each generator $v \in V$ of $\G_n$, 
we have a path $w(v)$ in $\Gamma_{\G,S}$ from $1_\G$ to $v$ which does not intersect $A t^{-1} \G_n$. 
From this it follows that, for every $g \in \G_n$ and $v \in V$, $w(v)$ provides a path 
in $\Gamma_{\G, S}$ from $g$ to $v g$ which does not intersect $A t^{-1} \G_n$.

Set $x = t \cdot y$. By Lemma \ref{proaction} the restriction of $\mathcal{P}$ 
to $(\G \setminus A t^{-1} \G_n) \cdot x = (\G \setminus A \G_n) \cdot y$ is $\G_n$-invariant. 
This means that if $h \not\in A t^{-1} \G_n$ and $g \in \G_n$ then 
$\delta(s, h \cdot x) = \delta(s, (h g h^{-1}) h \cdot x) = \delta(s, h g \cdot x)$ 
for all $s \in S \cup S^{-1}$. Now consider $g \in \G_n$ and $v \in V$. 
Write the word $w(v)$ as $s_{\ell(v)}(v) \cdots s_1(v)$ where each $s_i(v) \in S \cup S^{-1}$, 
and set $w_i(v) = s_i(v) \cdots s_1(v)$. Then
\begin{align*}
\delta(v, g \cdot x) & = \delta(s_{\ell(v)}(v), w_{\ell(v)-1}(v) g \cdot x) \cdots \delta(s_1(v), g \cdot x)\\
 & = \delta(s_{\ell(v)}(v), w_{\ell(v)-1}(v) \cdot x) \cdots \delta(s_1(v), x)\\
 & = \delta(v, x).
\end{align*}
Define $\phi : \G_n \rightarrow H$ by $\phi(g) = \delta(g, x)$. 
The above computation shows that for $g \in \G_n$ and $v \in V$
$$\phi(v g) = \delta(v g, x) = \delta(v, g \cdot x) \delta(g, x) = \delta(v, x) \delta(g, x) = \phi(v) \phi(g).$$
Thus $\phi$ is a group homomorphism.

Finally, let $T$ be a finite set satisfying $T \G_n = \G$. Since $\overline{[x]}$ is 
compact and $\delta$ is continuous, there is a finite set $W \subseteq H$ with $\delta(t, z) \in W$ 
for all $t \in T$ and all $z \in \overline{[x]}$. Then for $t \in T$ and $g \in \G_n$ we have
\begin{equation*}
\delta(t g, x) = \delta(t, g \cdot x) \delta(g, x) \in W \phi(g).\qedhere
\end{equation*}
\end{proof}

\begin{thmn} \label{Zcocycle}
Let $\delta\colon \Z^2 \times F(2^{\Z^2}) \to \Z$ be a continuous cocycle. 
Then there are $x\in F(2^{\Z^2})$, $\alpha, \beta\in \Q$, and $C\in\Z^+$ such that for any $a, b\in\Z$,
$$ \delta((a, b), x)=\alpha a+\beta b+\epsilon, $$
where $|\epsilon|\leq C$.
\end{thmn}

\begin{proof}
Apply Theorem \ref{rfcocycle} with $\G_n = (3^n \Z) \times (3^n \Z)$ to obtain $x$, a 
homomorphism $\phi : (3^n \Z) \times (3^n \Z) \rightarrow \Z$, and finite sets 
$T \subseteq \Z^2$, $W \subseteq \Z$. Define $\alpha = \phi((3^n,0)) / 3^n$ and 
$\beta = \phi((0, 3^n)) / 3^n$. Let $C$ be the maximum value of 
$|-t_1 \alpha - t_2 \beta + w|$ for $(t_1, t_2) \in T$ and $w \in W$. 
Now fix $(a, b) \in \Z^2$ and pick $(p, q) \in \Z^2$ and $(t_1, t_2) \in T$ 
with $(a, b) = (t_1, t_2) + (p 3^n, q 3^n)$. Then there is $w \in W$ with
\begin{align*}
\delta((a, b), x) & = \delta((t_1, t_2) + (p 3^n, q 3^n), x)\\
 & = w + p 3^n \alpha + q 3^n \beta \\
 & = (a - t_1) \alpha + (b - t_2) \beta + w\\
 & = a \alpha + b \beta - t_1 \alpha - t_2 \beta + w.
\end{align*}
This completes the proof since $|-t_1 \alpha - t_2 \beta + w| \leq C$.
\end{proof}

\section{Marker distortion} \label{sec:md}

In this section we present a result which serves to illustrate some of the methods we will be using, but in
a much simpler setting. Namely, we will present a result concerning  continuous actions of $\Z$.
In this one-dimensional case, the arguments reduce to simpler, but not trivial, arguments.
In particular, the proof of Theorem~\ref{thm:mdb} will use a one-dimensional version of a ``tile''
and will serve as a warm-up for the more general tile construction given later. Specifically,
in Theorem~\ref{thm:twotilesthm} we will present a general ``tiles theorem'' for $F(2^{\Z})$
and in Theorem~\ref{thm:tilethm} a corresponding result for $2^{\Z^2}$.

The result we present has to do with the notion of {\em marker distortion} which we now introduce.
We first recall the following well-known fact (see \cite{gao_countable_2015}).

\begin{fctn} \label{fact:bf}
Let $X$ be a zero-dimensional Polish space, and let $\Z\acts X$ be a continuous, free action of $\Z$ on
$X$. Then for any integer $d>1$ there is a clopen complete section  $M_d \subseteq X$
such that for any $x \in M_d$, the first point of $M_d$ to the right (or left) of
$x$ has distance $d$ or $d+1$ from $x$.
\end{fctn}

In the above, by ``to the right of $x$'' we mean that $y=n\cdot x$ where $n \in \Z^+$, and
similarly for ``left''. By ``distance'' from $x$ to $y$ we mean $\rho(x,y)=|n|$
where $n\in \Z$ is the unique integer such that $n\cdot x=y$.

A natural question, which occurs in many contexts, is how ``regular'' of a marker structure
can we put on an equivalence relation. The notion of regular can be measured in different
ways, such as size, geometric shape, etc.\ In the one-dimensional case, since we are able
to use marker distances of $d$ and $d+1$, we might ask how
close we can keep the average marker distance to $d + \frac{1}{2}$.
We make this precise in the following definition.

\begin{defnn} \label{defn:marker_distortion}
Let $d>1$ be a positive real number, and $M$  a complete section for the action of $\Z$ on the
Polish space $X$. The {\em $d$-distortion function} for $M$
is the function $f_d \colon X \times \Z^+ \to \R$ is defined by:
\[
f_d(x,n)=  \rho(x,x^n_M)- nd ,
\]
where $x^n_M$ is the $n$th point of $M$ to the right of $x$
(if this is not defined we leave $f_d(x,n)$ undefined).
When $d$ is understood we simply write $f(x,n)$.
\end{defnn}

Even more generally, we can allow $d$ to depend on the equivalence class of $x$.
If $d \colon X\to \R$ is an invariant Borel function (that is, $f(x)=f(y)$ whenever
$[x]=[y]$), then we write $f_{d(x)}(x,n)$ for $f_d(x,n)$ where $d=d(x)$.  This measures
the marker distortion with respect to the ``expected'' average distance $d(x)$
between marker points in $[x]$. One could use a prescribed function $d(x)$, or given
the  complete section $M$, use the function $\Delta(x)$ associated to $M$ by:
$\Delta(x)= \limsup_n \frac{\rho(x,x^n_M)}{n}$.


A particular case of this is when $M=M_d$ is a complete section such that
the $\rho$ distance between consecutive points of $M_d$ is in $\{ d, d+1\}$.
In this case, we can interpret the distortion function as follows.
As we move from a point to the next marker point
to the right, we record a charge of $+\frac{1}{2}$ if the distance is $d+1$, and a charge of $-\frac{1}{2}$ if
the distance is $d$. The distortion function $f_{d+\frac{1}{2}}(x,n)$ then measures the total charge accumulated
as we move from $x$ to the $n$th marker point to the right.

It is natural to ask if we can keep the distortion function bounded, or how slow-growing can we keep it?
Can we do better if we allow Borel marker sets $M_d$ instead of clopen ones?
We next state three theorems which answer some of these questions.

The first result says that bounded distortion is impossible to attain for $F(2^\Z)$.

\begin{thmn} \label{thm:mda}
For any invariant Borel function $d\colon F(2^\Z)\to\R$ with $d(x)>1$ for all $x$, 
there does not exist a Borel complete section $M \subseteq F(2^\Z)$
with bounded $d$-distortion, that is, having the property that for some $B\in \R$, $|f_{d(x)}(x,n)| \leq B$
 for all $x \in F(2^\Z)$ and all $n \in \omega$.
\end{thmn}

The second result says that using clopen marker sets in $F(2^\Z)$ we cannot do better than
linear distortion. Note that the basic Fact~\ref{fact:bf} shows that linear distortion
is possible with clopen markers, in fact we can get $f_d(x,n)\leq n$.

\begin{thmn} \label{thm:mdb}
Suppose $d \in \R$ with $d>1$ and $M$ is a clopen complete section in $F(2^\Z)$.
Then there are $C>0$ and  $x \in F(2^\Z)$ such that for all $g\in\Z$ and all large enough $n$
we have $|f_d(g\cdot x,n)| > Cn$.
\end{thmn}

The third result says that that using Borel marker sets, we can improve the distortion from
linear growth to arbitrarily slow growth.

\begin{thmn} \label{thm:mdc}
Let $X$ be a Polish space and $\Z\acts X$ be a free, Borel action. Let $d>1$ be an integer.
Let $f\colon \Z^+ \to \Z^+$ be monotonically increasing  with $\lim_{n} f(n)=+\infty$ and
$f(1)>4(2d+1)$. Then there is a Borel complete section $M_d \subseteq X$
such that for all $x \in X$ and $n \in \Z^+$ we have $|f_d(x,n)| \leq f(n)$.
\end{thmn}

The proof of Theorem~\ref{thm:mda} is a simple category argument which we give below for the sake of
completeness. The proof of Theorem~\ref{thm:mdb} given below will be a special version of the more general ``tile''
argument for $F(2^{\Z})$ to be given later in Theorem~\ref{thm:twotilesthm}.
The proof of Theorem~\ref{thm:mdc} uses different techniques, specifically 
the method of orthogonal markers, and will be presented in \cite{borcon}.

\begin{proof}[Proof of Theorem~\ref{thm:mda}]
Suppose $M \subseteq F(2^\Z)$ is a Borel complete section with
$d$-distortion bounded by $B$.  The orbit equivalence relation on
$F(2^\Z)$ is generically ergodic (see \cite{gao_book} Proposition 6.1.9), which implies (by \cite{gao_book} Proposition 10.1.2) that the function $d$ is constant on an
invariant comeager subset of $F(2^\Z)$. Denote this constant value by
$d_0$. Let $C\subseteq F(2^\Z)$ be an invariant comeager set such that $d(x)=d_0$ for all
$x \in C$ and such that $M\cap C$ is relatively clopen in $C$ (which we can obtain
as $M$ is Borel so has the Baire property). Let $\{ D_n\}_{n \in \omega}$
be dense open sets in $2^\Z$ with $\bigcap_n D_n \subseteq C$.


We construct a particular point $x\in C$.
Write $K$ for the set of
integers $k \in \Z$ satisfying $\inf_{n \in \Z} |k - n d_0| \geq 1/3$.
Note that $K$ is infinite since $d_0 > 1$. 
Let $\{(I_n,J_n)\}$ enumerate all pairs of intervals in $\Z$ with $I_n \subseteq J_n$,
such that each possible pair $(I,J)$ occurs infinitely often in the enumeration.
We construct a decreasing sequence of basic open sets $U_0\supseteq U_1\supseteq \cdots$
in $2^\Z$. In the following argument, if $U\subseteq 2^\Z$ is a basic open set,
then by $\dom(U)$ we mean the interval $[i,j]\subseteq \Z$ such that for some
$p\in 2^{[i,j]}$ we have $U=\{ y \in 2^\Z\colon y \res [i,j]\in p\}$. We say
that $U$ is the basic open set determined by $[i,j]$ and $p$.

Let $U_0=2^\Z$.
Given $U_n$, first get a basic open $U'_n \subseteq U_n$
with $U'_n \subseteq  D_n$.
If $J_n \nsubseteq \dom(U'_n)$, then set $U_{n+1}=U'_n$. If
$J_n \subseteq \dom(U'_n)$, then we take $U_{n+1} \subseteq U'_n$
such that $\dom(U_{n+1})$ contains a disjoint copy $J'_n$
(that is, $J'_n$ has the same
length as $J_n$ and $J_n \cap J'_n=\varnothing$) and such that
if $a$ is the left endpoint of $I_n$, and $a'$ the left endpoint of $I'_n$ (the
corresponding subinterval of $J'_n$),
then $a'>a$ and  $a'- a \in K$,
and finally if $U_{n+1}$ is determined by the interval $\dom(U_{n+1})=[i,j]\subseteq \Z$ and
$p\in 2^{[i,j]}$, then $p\res J_n\cong p\res J'_n$ (meaning precisely that if
$J_n=[c,d]$, $J'_n=[c',d']$, then for all $0\leq i<d-c$ we have $p(c+i)=p(c'+i)$).
Let $\dom(U_n)=[i_n,j_n]$, and we may assume that $\lim_{n\to \infty} i_n=-\infty$,
$\lim_{n\to \infty} j_n=\infty$.
Let $\{x\}=\bigcap_n U_n\subseteq C$.  See Figure~\ref{fig:xcon} for an illustration
of the construction of $x$.

\begin{figure}[h]
\begin{tikzpicture}[scale=0.05]

\draw (-100,0) to (100,0);

\draw[fill=lightgray] (-90,-3) rectangle (-20,3);
\draw[fill=gray] (-70,-3) rectangle (-40,3);
\draw[fill=lightgray] (20,-3) rectangle (90,3);
\draw[fill=gray] (40,-3) rectangle (70,3);

\draw (-90,-3) to (-90,3);
\draw (-70,-3) to (-70,3);
\draw (-40,-3) to (-40,3);
\draw (-20,-3) to (-20,3);

\draw (90,-3) to (90,3);
\draw (70,-3) to (70,3);
\draw (40,-3) to (40,3);
\draw (20,-3) to (20,3);

\draw[decorate, decoration={brace, amplitude=5pt, mirror}] (-70,-10) to (-40,-10);
\node at (-70,-7) {$a$};
\node at (-40,-7) {$b$};
\node at (-55,-19) {$I_n$};

\draw[decorate, decoration={brace, amplitude=5pt,mirror}] (40,-11) to (70,-11);
\node at (40,-7) {$a'$};
\node at (70,-7) {$b'$};
\node at (55,-20) {$I'_n$};

\draw[decorate, decoration={brace, amplitude=5pt}] (-90,5) to (-20,5);
\node at (-55,13) {$J_n$};

\draw[decorate, decoration={brace, amplitude=5pt}] (20,5) to (90,5);
\node at (55,13) {$J'_n$};
\end{tikzpicture}
\caption{The construction of $U_{n+1}$. 
We have $x \res J_n\cong x\res J'_n$ for any $x \in U_{n+1}$.}
\end{figure}

Since $x \in C$, we have for all $m,n >0$ that $|f_{d_0}(m \cdot x,n)|\leq B$.
Let $\alpha\in \R$ be the infimum of the values of $f_{d_0}(m\cdot x,n)$ as $m$ ranges
over all integers such that $m \cdot x \in M$ and $n$ ranges over all
positive integers. Note that, while $\alpha $ may be negative, it exists by the boundedness of the distortion.
Likewise, let $\beta  \in \R$ be the supremum value of $f_{d_0}(m\cdot x, n)$ as $m$ ranges within
the set of values for which $m\cdot x \in M$, $\inf_r f_{d_0}(m\cdot x, r) < \alpha  + 1/3$, and $n$ ranges over all positive integers.
Fix now $a\in \Z$ and $n_0>0$ such that $a \cdot x \in M$,
and $f_{d_0}(a \cdot x, r)$ achieves values below $\alpha  + 1/3$ and above $\beta - 1/3$ for $r \in [0, n_0]$.
From our definitions, we have that $\alpha  \leq f_{d_0}(a \cdot x, n) \leq \beta $ for all $n \in \N$.
Let $I$ be the interval $[a, b]$ where $b\cdot x$
is the $n_0$th element of $M$ to the right of $a\cdot x$.

Let $J\supseteq I$ be an interval such that if $U_{x\res J}$
denotes the basic open set determined by $J$ and $x\res J$,
then for all $r \in [a,b]$ we have that $r\cdot U_{x\res J} \subseteq M$ or $r \cdot U_{x\res J}\subseteq 2^\Z\sms M$.
This is possible as $M \cap C $ is relatively clopen in $C$.
Let $n$ be such that $(I_n,J_n)=(I,J)$ and $J\subseteq \dom(U_n)$.
By definition of $U_{n+1}$, we have that $\dom(U_{n+1})$ includes  a disjoint copy
$J'_{n}$ of $J_n$, and if $U_{n+1}$ is determined by $p \in 2^{\dom(U_{n+1})}$,
then $p\res J_n\cong p\res J'_n$.

Let $I'_n=[a',b']$ be the corresponding subinterval of $J'_n$.
Recall that $I_n = I = [a, b]$ and note $b-a=b'-a'$.
We have that $a\cdot x \in M$ and $b\cdot x \in M$.
Also, from the definition of $J$ and the fact that $p\res J\cong p\res J'$ we have that
for all $0 \leq i <b-a$ that $(a+i)\cdot x \in M$ if and only if  $(a'+i)\cdot x \in M$.
It follows that the maps $r \mapsto f_{d_0}(a\cdot x,r)$ and
$r \mapsto f_{d_0}(a'\cdot x,r)$ for $r \in [0,n_0]$ are identical. In particular,
$f_{d_0}(a' \cdot x, r)$ achieves values below $\alpha + 1/3$ and above $\beta  - 1 / 3$ for $r \in [0,n_0]$.


Suppose $a'\cdot x$ is the $q$th point of $M$
to the right of $a\cdot x$. Then $f_{d_0}(a\cdot x, q)=\rho(a\cdot x, a'\cdot x)-qd_0$. Since $\rho(a\cdot x, a'\cdot x)=a'-a\in K$, we have that $|f_{d_0}(a\cdot x, q)|\geq 1/3$. Notice that
$$f_{d_0}(a \cdot x, q + r) = f_{d_0}(a \cdot x, q) + f_{d_0}(a' \cdot x, r).$$
By construction, $f_{d_0}(a' \cdot x, r)$ achieves values below $\alpha + 1/3$ and above $\beta - 1/3$ for $r \in [0,n_0]$.
If $f_{d_0}(a\cdot x,q)\leq  -1/3$, then for some $r \leq n_0$ we have that
$f_{d_0}(a\cdot x, q+r) < -1/3 + \alpha + 1/3 =\alpha$.
If $f_{d_0}(a\cdot x,q)\geq 1/3$, then we likewise get that
$f_{d_0}(a\cdot x, q+r) > 1/3 + \beta - 1/3 =\beta$ for some $r\leq n_0$.
In either case, this contradicts $\alpha \leq f_{d_0}(a \cdot x, r) \leq \beta$ for all $r \in \N$.
\end{proof}

The proof of Theorem~\ref{thm:mdb} will involve constructing a special hyper-aperiodic element
$x \in F(2^\Z)$ which contains copies of certain ``tiles'' $T_{w,p}$ for
$w, p \in \Z^+$. The existence of a hyper-aperiodic element with these tile structures will
allow us to deduce the theorem.

\begin{proof}[Proof of Theorem~\ref{thm:mdb}] Let $d>1$ be a real number and $M\subseteq F(2^\Z)$ be a clopen
complete section.
Again let $K$ be the set of integers $k$ satisfying $\inf_{n \in \Z} |k - n d| \geq 1/3$. Note that $K$ is infinite.
For $w,p \in \Z^+$, a {\em $T_{w,p}$ tile} is a function $t \colon [a,b]\to 2$
with $b-a=2w+p$ such that $t\res [a,a+2w]\cong t\res [b-2w,b]$. Thus, a $T_{w,p}$ tile
has a block $R$ of size $2w+1$ (meaning $|\dom(R)|=2w+1$), followed by a block $S$ of size $p-2w-1$, and then
followed by the block $R$ again. This is illustrated in Figure~\ref{fig:onedtile}.
Notice that the domain of a $T_{w,p}$ tile has a central interval $[a+w, a+w+p]$, which is of length $p$ (and size $p+1$), surrounded on each
side by a ``buffer'' interval of size $w$, as shown in Figure~\ref{fig:onedtile}.

\begin{figure}[h]
\begin{centering}
\begin{tikzpicture}[scale=0.05]
\pgfmathsetmacro{\wa}{10}
\pgfmathsetmacro{\wb}{15}
\pgfmathsetmacro{\ws}{2}

\draw (-100,0) to (100,0);
\foreach \i in {0,..., 100}
{
\draw (-100+\i* \ws,-3) to (-100+\i* \ws,3);
}

\draw[fill=lightgray] (-50,-3) rectangle (-30,3);
\draw[fill=gray] (-30,-3) rectangle (30,3);
\draw[fill=lightgray] (30,-3) rectangle (50,3);

\draw[fill,radius=1,black] (-40,0) circle;
\draw[fill,radius=1,black] (40,0) circle;

\draw[decorate, decoration=brace,thick] (-40,5) to  (40,5);
\node at (0, 12) {$p$};

\draw[decorate,decoration=brace,thick] (50,-5) to (40,-5);
\node at (45,-12) {$w$};

\end{tikzpicture}
\end{centering}
\caption{A $T_{w,p}$ tile\label{fig:onedtile}}
\end{figure}

We now construct a hyper-aperiodic element $x \in F(2^\Z)$ with the property that
for arbitrarily large $w$ and $p$, with $p \in K$,
there is an interval $[a,b]$ such that $x \res [a,b]$ is a
$T_{w,p}$ tile. Fix  an arbitrary hyper-aperiodic element $y \in M$. Such $y$ exists since $M$ is a complete section. Let $\{ (w_n,p_n)\}$
be any sequence with $\lim_n w_n=\lim_n (p_n - 2 w_n - 2) =\infty$, with $p_n - 2 w_n - 2 \geq 0$, and with $p_n \in K$.
Let $I_n=[a_n,b_n]$, with $0 < a_n<b_n < a_{n+1}$,
be a pairwise disjoint sequence of intervals in $\Z$ such that
$b_n-a_n=2w_n+p_n$, and the gap $a_{n+1}-b_n$ between $I_n$ and $I_{n+1}$ also
tends to $\infty$ with $n$. Each interval $I_n$ can be viewed as the disjoint
union of three subintervals $A_n$, $B_n$, $C_n$, where $A_n=
[a_n,a_n+ 2 w_n]$, $B_n=[a_n+ 2 w_n + 1, a_n+p_n-1]$, and $C_n=[a_n+ p_n,a_n+2w_n+p_n]$.
Fix any $x \in 2^\Z$ satisfying:

\begin{enumerate}
\item
For $i \notin \bigcup_n I_n$, $x(i)=y(i)$.
\item
For all $n \in \N$, $x \res B_n \cong  y\res [-\lfloor p_n/2\rfloor+w_n+1, \lceil p_n/2\rceil-w_n-1]$.
\item
For all $n \in \N$, $x\res A_n\cong x\res C_n \cong  y \res [-w_n,w_n]$.
\end{enumerate}

Now we check that $x$ is hyper-aperiodic. Fix $s \in \Z \setminus \{0\}$, and let $T'$ be a finite set
witnessing the hyper-aperiodicity of $s$ for $y$, that is, for every $g \in \Z$ there is $t \in T'$ with $y(g+t) \neq y(g+s+t)$.
Consider the partition $Q = \{\Z \setminus \bigcup_n I_n\} \cup \bigcup_n \{A_n, B_n, C_n\}$ of $\Z$.
By construction there is a finite set $U \subseteq \Z$ such that for all $g \in \Z$
there is $u \in U$ with $(g + u + T') \cup (g + s + u + T')$ contained in a single piece of $Q$.
Set $T = U + T'$. Now fix $g \in G$. Pick $u \in U$ with $L = (g + u + T') \cup (g + s + u + T')$
contained in a single piece of $Q$. If $L$ is contained in $\Z \setminus \bigcup_n I_n$ then $x \res L = y \res L$
and thus we are done since there is $t \in T'$ with
$$x(g + u + t) = y((g + u) + t) \neq y((g + u) + s + t) = y(g + s + u + t) = x(g + s + u + t).$$
On the other hand, if $L$ is contained in some $A_n$, $B_n$, or $C_n$, then there is $h \in \Z$
with $x \res L = (h \cdot y) \res L$. Again we are done because we can find $t \in T'$ with
$$x(g + u + t) = y(h + g + u + t) \neq y(h + g + s + u + t) = x(g + s + u + t).$$

Consider now the compact set $C=\overline{[x]} \subseteq F(2^\Z)$. Since $M$
is a clopen subset of $F(2^\Z)$, for each $\alpha\in F(2^\Z)$ there is an $m\in \Z^+$ such that,
letting $U_{\alpha\res[-m,m]}$ be the basic open set of elements extending $\alpha\res [-m,m]$,
we have either $U_{\alpha\res[-m,m]}\cap F(2^\Z)\subseteq M$ or $U_{\alpha\res[-m,m]}\cap F(2^\Z)\cap M=\varnothing$.
For each $\alpha\in F(2^\Z)$ let $m(\alpha)$ be the least such $m\in\Z^+$.
Then the function $\alpha\mapsto m(\alpha)$ is continuous on $F(2^\Z)$.
In particular, its restriction on $C$ is continuous. The compactness of $C$ now gives that there is an $m_0$  such that
$m(\alpha)\leq m_0$ for all $\alpha\in C$. In particular, for all $\alpha\in C$,
either $U_{\alpha\res[-m_0, m_0]}\cap F(2^\Z)\subseteq M$ or $U_{\alpha\res [-m_0,m_0]}\cap F(2^\Z)\cap M=\varnothing$.
Applying this to elements of $[x]$, we get that for any $a \in \Z$, $x\res [a-m_0,a+m_0]$ determines
whether $\alpha \in M$ for any $\alpha \in F(2^{\Z})$ with $\alpha\in U_{x\res [a-m_0, a+m_0]}$.

Now consider a large enough $n$ in the construction of $x$ where $w_n>m_0$ and $p_n>2w_n+2m_0+2$.
Then $x\res B_n\cong y\res I$ for an interval $I$ with $[-m_0, m_0]\subseteq I$.
Since $y\in M$ we deduce that there is $i\in B_n$ such that $i\cdot x\in M$.
For notational simplicity we denote $w=w_n$, $p=p_n$, and let $\tau$ denote the $T_{w,p}$ tile $x\res I_n$.

We construct an element $\tilde{x}\in F(2^\Z)$ as follows. First let $z$ be a 
periodic element of $2^\Z$ which is defined as an overlapping concatenation of the $T_{w,p}$ tile $\tau$. 
To be precise, letting $\tau=s\conc t\conc s$ where $|\dom(s)|=2w+1$ and $ |\dom(t)|=p-2w-1$, 
then $z$ is a bi-infinite concatenation of the form $\cdots s\conc t\conc s\conc t\conc s \cdots$. 
Without loss of generality, assume $z\res[0,p-2w-2]\cong t$. Next for all $k\geq 0$ 
let $L_k=[2^{k+1}p, 2^{k+1}p+p-2w-2]$ and $N_k=[2^{k+1}p+p-2w-1, 2^{k+2}p-1]$. 
Then $z\res L_k\cong t$, $N_k$ is the interval in between $L_k$ and $L_{k+1}$, 
and $z\res N_k\cong s\conc t\conc \cdots\conc t\conc s$ with $2^{k+1}-1$ many occurrences of $t$. 
Define $\tilde{x}$ by

$$
\tilde{x}(i)= 
\begin{cases}
z(i) & \text{ if } i\notin \{ 2^{k+1}p\ \colon  k\geq 0\}, \\
z(i) & \text{ if } i=2^{k+1}p \text{ and }  z(i)\neq z(i-k-1), \\ 
1-z(i) & \text { if } i=2^{k+1}p \text{ and } z(i)=z(i-k-1).
\end{cases}
$$


\noindent
Then $\tilde{x}\in F(2^\Z)$ because the definition of $\tilde{x}$ makes it explicit that any $k+1$ 
is not a period. Note also that $\tilde{x}$ and $z$ agree on all intervals except $\bigcup_{k\geq 0} L_k$. 
In particular, $\tilde{x}\res N_k=z\res N_k$ for all $k\geq 0$. The construction of $\tilde{x}$ 
is illustrated in Figure~\ref{fig:tilecon}.

\begin{figure}[h]
\begin{centering}
\begin{tikzpicture}[scale=0.05]
\pgfmathsetmacro{\wa}{5}
\pgfmathsetmacro{\wb}{10}
\pgfmathsetmacro{\ws}{2}

\draw (-110,0) to (110,0);
\foreach \i in {0,..., 110}
{
\draw (-110+\i* \ws,-3) to (-110+\i* \ws,3);
}

\foreach \i in {0,...,2}
{
\draw[fill=lightgray] (-100+\i*\wa+\i*\wb,-3) rectangle (-100+\i*\wa+\i*\wb+\wa,3);
\draw[fill=gray] (-100+\i*\wa+\i*\wb+\wa,-3) rectangle (-100+\i*\wa+\i*\wb+\wa+\wb,3);
\draw[fill=lightgray] (-100+ 3*\wa+3*\wb,-3) rectangle (-100+3*\wa+3*\wb+\wa,3);
}

\foreach \i in {0,...,8}
{
\draw[fill=lightgray] (-40+\i*\wa+\i*\wb,-3) rectangle (-40+\i*\wa+\i*\wb+\wa,3);
\draw[fill=gray] (-40+\i*\wa+\i*\wb+\wa,-3) rectangle (-40+\i*\wa+\i*\wb+\wa+\wb,3);
\draw[fill=lightgray] (-40+9*\wa+9*\wb,-3) rectangle (-40+9*\wa+9*\wb+\wa,3);
}

\draw[decorate, decoration=brace,thick] (-40+\wa/2,5) to  (-40+\wa+\wb+\wa/2,5);
\node at (-40+\wa+\wb/2, 12) {$p$};

\draw[decorate, decoration={brace, mirror}, thick] (-49,-5) to  (-41,-5);
\node at (-45, -13) {$L_k$};

\draw[decorate, decoration={brace, amplitude=5pt, mirror}, thick] (-39,-5) to  (-41+10*\wa+9*\wb,-5);
\node at (-40+3\wa+7*\wb/2, -13) {$N_k$};

\end{tikzpicture}
\end{centering}
\caption{The element $\tilde{x}$\label{fig:tilecon}}
\end{figure}

It follows from the construction of $\tau$ that for each interval $J$ for which $\tilde{x}\res J\cong \tau$, there is $i\in J$ such that $i\cdot \tilde{x}\in M$.
Moreover, for each $k\geq 0$, $i \mapsto \chi_{M}(i\cdot \tilde{x})$ has period $p$ for $i\in N_k$. The same is true for $i\in (-\infty, 2p-1]$. Consider any fixed $k\geq 0$. Let $i_0\in N_k$ be the least such that $i_0\cdot \tilde{x} \in M$.
Let $r_0\in\Z^+$ be such that $(i_0+p)\cdot \tilde{x}$ is the $r_0$-th element of $M$
to the right of $i_0\cdot \tilde{x}$. Let $c_0=f_d(i_0\cdot \tilde{x},r_0)$ and note that $|c_0| \geq 1/3$ since $p\in K$.
It follows that $f_d(i_0\cdot \tilde{x},\ell r_0)=\ell c_0 $ for $\ell< 2^{k+1}-1$. Thus there is a constant $\gamma>0$ such that for all $q<(2^{k+1}-1)r_0$, $$\frac{q}{r_0}c_0-\gamma\leq f_d(i_0\cdot x, q)\leq \frac{q}{r_0}c_0+\gamma.$$
Now let $j_0\in N_k$ be the largest such that $j_0\cdot \tilde{x}\in M$ and $i_1\in N_{k+1}$ be the least such that $i_1\cdot \tilde{x}\in M$. Then the absolute value of the total $d$-distortion from $j_0\cdot \tilde{x}$ to $i_1\cdot\tilde{x}$ is bounded as $k$ varies through all natural numbers. Let $\beta>0$ be such a bound. By a direct computation we get that for all $g\in\Z$ and for all $\ell$,
$$\frac{c_0}{r_0}\ell-\gamma-\beta\log \ell \leq f_d(g\cdot \tilde{x},\ell)\leq\frac{c_0}{r_0}\ell+\gamma+\beta\log\ell.$$  Thus for large enough $\ell$ we have $|f_d(g\cdot \tilde{x}, \ell)|> \frac{|c_0|}{2r_0}\ell$.

\end{proof}

\chapter{The Twelve Tiles Theorem} \label{sec:twodtiles}

In this chapter we prove a theorem which completely answers the question of when
there exists a continuous equivariant map from $F(2^{\Z^2})$ into a given subshift
of finite type. This type of question can also be formulated through the notion
of a ``continuous structuring'' on $F(2^{\Z^2})$ (to be defined in \S\ref{sec:continuous_structuring}) which means
a continuous way of associating a first-order structure to the equivalence classes of $F(2^{\Z^2})$.
We prove our main theorem, Theorem~\ref{thm:tilethm}, first in the terminology of subshifts of finite type
and then in \S\ref{sec:continuous_structuring} we present the theorem in terms of
continuous structurings. Although these versions of the theorem are
equivalent, we find both points of view useful.

The main theorem generalizes to $F(2^{\Z^n})$, but the case $n=2$
illustrates the general arguments and suffices for the applications we have.
As a warm-up, we prove an analogous result in \S\ref{sec:onedimtile}
for $F(2^\Z)$. This illustrates some of the main ideas, but the argument
is significantly easier, and the answer involves a graph built from only two tiles
instead of $12$ as in the case for $F(2^{\Z^2})$.
This ``two tiles'' theorem, Theorem~\ref{thm:twotilesthm},
gives a complete answer as to when there is a continuous equivariant map from $F(2^\Z)$
to a subshift of finite type $Y\subseteq \bsft^\Z$, just as the ``twelve tiles''
theorem, Theorem~\ref{thm:tilethm}, provides a complete answer for
$F(2^{\Z^2})$ (considering subshifts of finite type $Y\subseteq \bsft^{\Z^2}$).
We will also need the simpler Theorem~\ref{thm:twotilesthm}
for the arguments of \S\ref{sec:ssonedim}.

The twelve tiles theorem for $F(2^{\Z^2})$ is given in terms of certain finite graphs which will be denoted
$\gnpq$. The graphs $\gnpq$ will have a uniform definition from the three
parameters $n,p,q$ which will be positive integers. The theorem will state
that a map from $F(2^{\Z^2})$ into a subshift of finite type (or a continuous $\CL$-structuring)
exists if and only if there is a triple $n,p,q$ with $\gcd(p,q)=1$ such that there is a corresponding map
from $\gnpq$ into the subshift (or $\CL$-structuring). These notions will be made precise in the following sections.
The theorem will also state that this existence criterion is equivalent to
the condition that for {\em all} sufficiently large $n,p,q$ with $\gcd(p,q)=1$ we have
a map from $\gnpq$ into the subshift (or $\CL$-structures), a result which is not obvious
directly. Again, \S\ref{sec:onedimtile} will present an easier version
of this result for $F(2^\Z)$, using a simpler family of graphs $\Gamma^{(1)}_{n,p,q}$.

The graphs $\gnpq$, as we said above, will all be defined in the same manner. We will start
with a collection of $12$ rectangular ``tiles.'' These will be rectangular grid graphs $G^i_{n,p, q}$ for $i=1, \dots, 12$, each of which will be divided into certain ``blocks" of vertices, with ``block labels" given to some of them. Blocks with the same labels will eventually be identified with each other to form quotient graphs $T^i_{n,p, q}$. We will refer to both collections $G^i_{n,p, q}$ and $T^i_{n,p, q}$ as ``tile graphs." To construct $\Gamma_{n, p, q}$ we will form some disjoint unions of the twelve tiles $G^i_{n,p, q}$ which will be denoted $G_{n, p, q}$, and finally
$\gnpq$ will be a quotient graph of $G_{n,p,q}$ by identifying the vertices in similarly labeled blocks. Alternatively, one can also obtain $\gnpq$ by first taking the union of all $T^i_{n,p,q}$ and then again taking the quotient by identifying the vertices in similarly labeled blocks. We refer to $G_{n,p, q}$ and $\gnpq$ as the ``$12$-tiles graphs.''  The following diagram illustrates the constructions and the terminology.
The simpler graphs $\Gamma^{(1)}_{n,p,q}$ of \S\ref{sec:onedimtile} will be
a simplified version of this construction, starting from just two tiles.

 \begin{figure}[h]
\begin{tikzpicture}[scale=0.05]

\draw[dashed] (-110,0) rectangle (-10,-90);

\node at (-60, 6) {tile graphs};

\node at (-60, -15) {grid graphs $G^i_{n,p,q}$ with blocks};

\node at (-60, -80) {tiles $T^i_{n,p, q}$};

\draw[->] (-60,-25) -- (-60, -70);

\node at (-77, -45) {quotient};

\draw[dashed] (60, 0) rectangle (120, -90);

\node at (90, 6) {12-tiles graphs};

\node at (80, -15) {$G_{n,p,q}$};

\node at (80, -80) {$\gnpq$};

\draw[->] (80,-25) -- (80,-70);

\node at (97,-45) {quotient};

\draw[->] (-8,-15) -- (58, -15);

\node at (25,-10) {disjoint union};

\draw[->] (-8,-80) -- (58,-80);

\node at (25, -86) {disjoint union};
\node at (25, -93) {and quotient};

\end{tikzpicture}
\end{figure}



With the main theorem proved in this chapter we will be able to answer many questions
about $F(2^{\Z^2})$ concerning continuous functions on or clopen sets in $F(2^{\Z^2})$.
In the rest of the paper we will see that many natural questions such as proper coloring questions, perfect matching questions,
tiling problems, and questions about continuous graph homomorphisms can be
formulated in terms of equivariant maps into subshifts of finite type or $\CL$-structures, and thus solved by an (often simple) argument
about the finite graphs $\gnpq$. In particular, we will be able to derive $\chi_c(F(2^{\Z^2}))=4$ as a quick corollary of our main theorem, thus giving another proof of Theorem~\ref{threecoloring}.

It is somewhat amusing to note that  the twelve tiles theorem makes it possible
to approach some problems in the descriptive set theory of countable Borel
equivalence relations through a computer analysis. Indeed, some of the arguments of the
following chapters were motivated by computer studies.

\section{From proper colorings to subshifts of finite type}\label{section:2.1}

In the last chapter we already studied continuous proper colorings on $F(2^{\Z^n})$. In order to study other related continuous combinatorial questions about $F(2^{\Z^n})$, we note that the proper coloring problem is a special case of the more general question on subshifts of finite type. In this section we also fix some definitions and notations that will be used throughout the rest of the paper.

First, we will be working with a notion of a $G$-graph for a group $G$, which is a notion weaker than the Cayley graph or the Schreier graph of a free action.

\begin{defnn} \label{Ggraph}
Let $G$ be a group, $S$ be a generating set of $G$, and $\Gamma$ be a directed graph. We say that $\Gamma$ is a {\em $(G, S)$-graph} (or in short a {\em $G$-graph} if there is no danger of confusion) if each element of $E(\Gamma)$ is uniquely labeled with an element of $S$.
\end{defnn}

Later in the chapter we will introduce ``block labels" for vertices of $G$-graphs. To properly distinguish different kinds of labels, we refer to the above labels for the edges of a $G$-graph as {\em $G$-labels}.

Recall our convention from \S\ref{section:1.1} on page \pageref{con:S} that $S$ is always 
finite and minimal in this paper. In particular, the identity is not an element of $S$. 
Also, since $S$ is minimal, i.e., no proper subset of $S$ is a generating set of $G$, 
we have that if $s\in S$ then $s^{-1}\not\in S$. In the case of $\Z^n$, our standard generating set 
is $\{e_1, \dots, e_n\}$, where for $1\leq i\leq n$, $e_i=(a_1,\dots, a_n)$ with $a_i=1$ 
and $a_j=0$ for $j\neq i$. Under this convention, the Cayley graphs under our consideration 
will never have both an edge $(x, y)$ and its reverse edge $(y,x)$ simultaneously contained in the edge set. 
This sometimes causes inconvenience. To remedy this, we adopt another convention that, 
whenever $(x, y)$ is an edge in a Cayley graph with label $s$, 
we consider $(y, x)$ to have label $s^{-1}$.

The $G$-graphs $\Gamma$ under our consideration will be either a part of a Cayley graph 
or a combination of some parts of a Cayley graph, and thus they will have the 
same property that if $(x, y)\in E(\Gamma)$ then $(y,x)\not\in E(\Gamma)$, 
which causes inconvenience sometimes. We will therefore also make a convention 
that in a $G$-graph $\Gamma$, whenever $(x,y)\in E(\Gamma)$ and is labeled $s$, 
then we consider $(y, x)$ to be labeled by $s^{-1}$. With this additional convention, 
each $G$-graph we consider will always be symmetric, and corresponding uniquely to an undirected graph.

We will also be working with $G$-homomorphisms between $G$-graphs, which are defined below.

\begin{defnn} Let $G$ be a group with a generating set $S$, and $\Gamma_1$, $\Gamma_2$ be $G$-graphs. 
A map $\varphi \colon \Gamma_1\to \Gamma_2$ is called a {\em $G$-homomorphism} if whenever
$(x,y)\in E(\Gamma_1)$ is labeled with $s \in S$, then $(\varphi(x),\varphi(y))\in E(\Gamma_2)$
is also labeled with $s$.
\end{defnn}

Proper colorings and chromatic numbers are defined for $G$-graphs in the same manner as for 
general directed graphs. In the case that $\Gamma$ is a topological graph, we also define 
continuous chromatic number, Borel chromatic number, etc. in a standard manner.

For $n\geq 1$, by an {\em $n$-dimensional} ({\em rectangular}) {\em grid graph} we mean the 
Cayley graph of $\Z^n$ restricted to an $n$-dimensional rectangle $[a_1, b_1]\times \cdots \times[a_n, b_n]$ 
where $a_i\leq b_i\in \Z$ for all $1\leq i\leq n$. 
An $n$-dimensional grid graph is naturally a $\Z^n$-graph. 
If we do not specify the dimension $n$, then tacitly $n=2$.

Next we turn to the shift space $\bsft^{\Z^n}$, where $\bsft \geq 1$ is an integer, and its subshifts.

We first formalize some notions related to the topology of this space. If $g\in \Z^n$ and $p:\dom(p)\to\{0,1, \ldots, \bsft-1\}$ and $q:\dom(q)\to\{0,1, \ldots, \bsft-1\}$
are finite partial functions from $\Z^n$ to $\bsft=\{0,1, \ldots, \bsft-1\}$, i.e.,
$\dom(p), \dom(q)\subseteq \Z^n$ are finite, we make the following definitions:
\begin{itemize}
\item We write $g\cdot p= q$ if $\dom(p) - g = \dom(q)$ and for all $h\in \dom(q)$, $p(h + g)=q(h)$.
\item We write $p\cong q$ if there is $h\in \Z^n$ such that $h\cdot p=q$.
\item We say that $p$ and $q$ are {\em compatible} if for all $i\in \dom(p)\cap \dom(q)$, $p(i)=q(i)$; otherwise we say that $p$ and $q$ are {\em incompatible}.
\end{itemize}
For $a \leq b \in \Z$ we write $[a, b]$, $[a, b)$, $(a, b]$, $(a, b)$ for the
set of integers between $a$ and $b$, with the obvious conventions as to whether each of
the two inequalities is strict or not. For an interval $[a,b]$ in $\Z$, we will be consistently speaking of its {\em width}, which is $b-a$, and {\em size}, which is $b-a+1$. Intuitively, the size of the interval is the number of vertices in the 1-dimensional grid graph, and the width is the length of the path from one end to the other.

If $a_1,\dots,a_n$ are
positive integers, an  {\em $(a_1,\dots,a_n)$-pattern} is a map
\[
p \colon [0,a_1)\times \cdots \times [0,a_n) \to \{0,1, \ldots, \bsft-1\}.
\]
We think of $\dom(p)$ naturally as an $n$-dimensional grid graph, which is in particular a $\Z^n$-graph. The {\em width} of an $(a_1,\dots, a_n)$-pattern is defined as $\max\{a_i-1: 1\leq i\leq n\}$. We say an $(a_1,\dots,a_n)$-pattern $p$ {\em occurs} in $x \in \bsft^{\Z^n}$ if
\[
x\res [c_1,c_1+a_1)\times \cdots \times [c_n,c_n+a_n)\cong p
\]
for some $(c_1,\dots,c_n)\in \Z^n$. Note that this is an invariant notion, that is, if $p$ occurs in $x$ then $p$ occurs in any $g\cdot x$.
Of course, patterns give the topology on $\bsft^{\Z^n}$, the basic open sets
being
\[U_p=\{ x \in \bsft^{\Z^n} \colon x \res [-k,k]^n\cong p\}
\]
for some $(2k+1,\dots,2k+1)$-pattern $p$.

Recall that a {\em $\Z^n$-subshift} $Y$ is a closed, invariant subset of $\bsft^{\Z^n}$, where $Y$ inherits a $\Z^n$-action from $\bsft^{\Z^n}$. Some $\Z^n$-subshifts are said to be of finite type, which we define below.

\begin{defnn} \label{def:sft}
A {\em $\Z^n$-subshift of finite type} is a $Y\subseteq \bsft^{\Z^n}$
for which  there is a finite set $\{p_1,\dots,p_k\}$ of patterns
such that for any $x \in \bsft^{\Z^n}$, $x \in Y$ if and only if
none of the patterns $p_1,\dots,p_k$ occur in $x$.
\end{defnn}

It is easily verified that given any patterns $p_1, \dots, p_k$, the set
$$Y=\{x\in \bsft^{\Z^n}\,:\, \mbox{none of the patterns $p_1,\dots, p_k$ occur in $x$}\} $$
(which could be empty) defines a $\Z^n$-subshift of $\bsft^{\Z^n}$. Thus any finite collection of patterns determines a $\Z^n$-subshift of finite type. Naturally we describe the above $Y$ by the sequence $(\bsft; p_1,\dots, p_k)$. We refer to the patterns $p_1,\dots, p_k$ as {\em forbidden patterns}, and the {\em width} of the $\Z^n$-subshift of finite type $Y$ is defined as the maximum of the widths of $p_1,\dots, p_k$.

Note there is no loss of
generality in requiring all the forbidden patterns $p_i$ to have the same
domain. In this case, if all the $p_i$ have domain $[0,\ell)^n$,
we describe the subshift $Y$ by the sequence $(\bsft;\ell;p_1,\dots,p_k)$. In this case, the width of $Y$ is $\ell-1$.


A special case of a width $1$ $\Z^n$-subshift of finite type is what we call
an {\em edge $\Z^n$-subshift}, which we define below.  By an {\em edge pattern} $p$ we mean
a function $p \colon \dom(p)\to \bsft$ where $\dom(p)$ is the two point
set $\{ \vec{0}, e_i\}$ for some $1\leq i\leq n$ (here $\vec{0} = (0, 0, \ldots, 0)$). An {\em edge $\Z^n$-subshift of finite type} is a $\Z^n$-subshift of finite type determined by a finite collection of edge patterns.

\begin{ex}{\em
The set of $Y \subseteq \bsft^{\Z^n}$ which are proper $\bsft$-colorings of $\Z^n$
is an edge $\Z^n$-subshift of finite type. In this case,
$Y$ is described by $\bsft n$ many forbidden edge patterns $p: \{\vec{0}, e_i\}\to \bsft$ where $p(\vec{0})=p(e_i)$ for $1\leq i\leq n$.}
\end{ex}

In the case $n=1$, the above definition of edge $\Z$-subshift of finite type is sometimes taken in the literature as the definition of
a $\Z$-subshift of finite type. We note that any $\Z^n$-subshift of finite type is equivalent to an edge $\Z^n$-subshift of finite type in the following sense.

\begin{fctn}
Let $Y\subseteq \bsft^{\Z^n}$ be a $\Z^n$-subshift of finite type. Then there is a $\bsft'$
and an edge $\Z^n$-subshift of finite type $Y' \subseteq {\bsft'}^{\Z^n}$ such that there is an
equivariant isomorphism $\varphi \colon Y \to Y'$.
\end{fctn}

\begin{proof}
Let $Y$ be determined by the forbidden patterns $p_1,\dots,p_k$, and we may assume
without loss of generality that all the $p_i$ have domain
$[0,\ell)^n$, and thus are $(\ell, \ell, \dots, \ell)$-patterns. Let $\bsft'=\bsft^{\ell^n}$. Then $\bsft'$ represents also the set of all $(\ell, \ell,\dots, \ell)$-patterns. The set of forbidden
edge patterns $p' \colon \{ \vec 0, e_i\}\to \bsft'$ determining $Y'$
is the $p'$ such that the following holds for $p'(\vec{0})$ and $p'(e_i)$ as $(\ell, \ell,\dots, \ell)$-patterns:
\begin{quote}
either $e_i\cdot p'(\vec 0)$ is incompatible with $p'(e_i)$, or
at least one of $p'(\vec 0)$, $p'(e_i)$ occurs in
the list $p_1,\dots,p_k$.
\end{quote}
For $x \in Y \subseteq \bsft^{\Z^n}$,
let $\varphi(x)\in {\bsft'}^{\Z^n}$ be defined so that for every $\vec{c}\in\Z^n$, $\varphi(x)(\vec{c})$ is the unique $(\ell, \ell, \dots, \ell)$-pattern such that $\varphi(x)(\vec{c})\cong x\!\res\! (\vec{c} +[0,\ell)^n)$.  This is easily an equivariant isomorphism.
\end{proof}

\section{The two tiles theorem in one dimension} \label{sec:onedimtile}

In this section we prove a two tiles theorem for dimension one as a warm-up for the higher dimensional cases.

For all $n, p, q\geq 1$ with $p, q\geq n$ we define the 2-tiles graph $\Gamma^{(1)}_{n,p,q}$. First consider the 1-dimensional grid graph of size $p+n$, which can be identified as the Cayley graph of $\Z$ restricted to $[0, p+n)$. Decompose the vertex set of the graph into three blocks: $[0,n)$, $[n,p)$, and $[p, p+n)$. Assign a block label $R^{(1)}_\times$ to the vertices of the blocks $[0,n)$ and $[p, p+n)$, and leave the vertices of the block $[n,p)$ unlabeled. (The notation for the label $R^{(1)}_\times$ is chosen to be consistent with the notation for higher dimensions in the next section.) We give two notations to the resulting tile graph $G^{(1),1}_{n,p,q}=T_1(n,p,q)$, the first to be consistent with our notation of higher dimensional cases, and the second for notational simplicity in the rest of this section. Similarly, we define the second tile graph $G^{(1),2}_{n, p, q}=T_2(n,p,q)$, which will be the Cayley graph of $\Z$ restricted to $[0, q+n)$, with block label $R^{(1)}_\times$ assigned to blocks $[0,n)$ and $[q,q+n)$.

\begin{figure}[h]
\centering
\subfloat
{
\begin{tikzpicture}[scale=0.5]

\pgfmathsetmacro{\n}{3}
\pgfmathsetmacro{\p}{7}
\pgfmathsetmacro{\q}{9}
\pgfmathsetmacro{\a}{-7}
\pgfmathsetmacro{\b}{5}
\pgfmathsetmacro{\d}{4.5}
\pgfmathsetmacro{\e}{5.5}

\pgfmathsetmacro{\x}{\n+\p+\n-1}
\pgfmathsetmacro{\y}{\n+\q+\n-1}
\pgfmathsetmacro{\xa}{\n+\p-2}
\pgfmathsetmacro{\ya}{\n+\q-2}
\pgfmathsetmacro{\nm}{\n-1}
\pgfmathsetmacro{\pm}{\p-\n-1}
\pgfmathsetmacro{\qm}{\q-\n-1}
\foreach \i in {0,...,\nm}
{ \draw (\a+\i,0) circle (0.15);
}
\foreach \i in {0,...,\pm}
{ \draw[fill=lightgray] (\a+\n+\i,0) circle (0.15);
}
\foreach \i in {0,...,\nm}
{ \draw  (\a+\p+\i,0) circle (0.15);
}
\foreach \i in {0,...,\xa}
{ \draw[->,>=stealth] (\a+\i+0.1,0) -- (\a+\i+0.9,0);
}

\foreach \i in {0,...,\nm}
{ \draw  (\b+\i,0) circle (0.15);
}
\foreach \i in {0,...,\qm}
{ \draw[fill=black] (\b+\n+\i,0) circle (0.15);
}
\foreach \i in {0,...,\nm}
{ \draw  (\b+\q+\i,0) circle (0.15);
}
\foreach \i in {0,...,\ya}
{ \draw[->,>=stealth] (\b+\i+0.1,0) -- (\b+\i+0.9,0);
}

\draw[decorate, decoration={brace, mirror, amplitude=5pt},yshift=10pt] (\a+\n-1,0) -- (\a,0)
node [black,midway,yshift=15pt] {$R^{(1)}_\times$};

\node at (\a+\d, -1.5) {$T_1(3,7,9)$};

\draw[decorate, decoration={brace, mirror, amplitude=5pt},yshift=10pt] (\a+\p+\n-1,0) -- (\a+\p,0)
node [black,midway,yshift=15pt] {$R^{(1)}_\times$};

\draw[decorate, decoration={brace, mirror, amplitude=5pt},yshift=10pt] (\b+\n-1,0) -- (\b,0)
node [black,midway,yshift=15pt] {$R^{(1)}_\times$};

\node at (\b+\e,-1.5) {$T_2(3,7,9)$};

\draw[decorate, decoration={brace, mirror, amplitude=5pt},yshift=10pt] (\b+\q+\n-1,0) -- (\b+\q,0)
node [black,midway,yshift=15pt] {$R^{(1)}_\times$};

\end{tikzpicture}
}
\vspace{5pt}

\subfloat
{
\begin{tikzpicture}[scale=0.5]

\pgfmathsetmacro{\n}{3}
\pgfmathsetmacro{\p}{7}
\pgfmathsetmacro{\q}{9}
\pgfmathsetmacro{\a}{-10}
\pgfmathsetmacro{\b}{5}

\pgfmathsetmacro{\c}{-0.5}
\pgfmathsetmacro{\d}{2}
\pgfmathsetmacro{\e}{\p-\n-1}
\pgfmathsetmacro{\em}{\e-1}
\pgfmathsetmacro{\h}{0.1}

\pgfmathsetmacro{\f}{\q-\n-1}
\pgfmathsetmacro{\fm}{\q-\n-2}

\pgfmathsetmacro{\x}{\n+\p+\n-1}
\pgfmathsetmacro{\y}{\n+\q+\n-1}
\pgfmathsetmacro{\nm}{\n-1}
\pgfmathsetmacro{\nmm}{\n-2}
\pgfmathsetmacro{\pm}{\p-1}
\pgfmathsetmacro{\qm}{\q-1}

\foreach \i in {0,...,\em}
{ \draw[->,>=stealth] (\a+\c+\i+0.1,\d) -- (\a+\c+\i+0.9,\d);
}

\foreach \i in {0,...,\nmm}
{ \draw[->,>=stealth] (\a+\i+0.9,0) -- (\a+\i+0.1,0);
}

\foreach \i in {0,...,\fm}
{ \draw[->,>=stealth] (\a-1.5+\i+0.1,-\d) --  (\a-1.5+\i+0.9,-\d);
}

\draw[->,>=stealth]  (\a,\h) -- (\a+\c,\d-\h);

\draw[->,>=stealth] (\a+\c+\e,\d) -- (\a+\nm,\h);

\draw[->,>=stealth] (\a,-\h) -- (\a-1.5,-\d+\h);

\draw[->,>=stealth]  (\a-1.5+\f,-\d) -- (\a+\nm,-\h);

\foreach \i in {0,...,\e}
{ \draw[fill=lightgray] (\a+\c+\i,\d) circle (0.15);
}

\foreach \i in {0,...,\nm}
{ \draw (\a+\i,0) circle (0.15);
}

\foreach \i in {0,...,\f}
{ \draw[fill=black] (\a-1.5+\i,-\d) circle (0.15);
}

\end{tikzpicture}
}

\caption{\label{fig:onedimgamma}The 2-tiles graphs $G^{(1)}_{3, 7, 9}$ and $\Gamma^{(1)}_{3,7,9}$.
}
\end{figure}

Note that each of the $R^{(1)}_\times$-labeled block is a 1-dimensional grid graph of size $n$. When taking quotient graphs the corresponding vertices in these $R^{(1)}_\times$-labeled blocks will be identified. The 2-tiles graph $G^{(1)}_{n, p, q}$ is defined as the disjoint union of $T_1(n, p, q)$ and $T_2(n, p, q)$. To obtain $\Gamma^{(1)}_{n,p,q}$, we take the quotient of $G^{(1)}_{n, p, q}$ by identifying the corresponding vertices in all $R^{(1)}_\times$-labeled blocks. This finishes the definition of the graph $\Gamma^{(1)}_{n,p,q}$. Note that $\Gamma^{(1)}_{n, p, q}$ is a $\Z$-graph. Figure~\ref{fig:onedimgamma} illustrates the 2-tiles graphs $G^{(1)}_{n,p,q}$ and $\Gamma^{(1)}_{n, p, q}$ for $n=3$, $p=7$, and $q=9$. In the figure the top graph shows $G^{(1)}_{3,7,9}$, which is a disjoint union of $T_1(3,7,9)$ and $T_2(3,7,9)$, and the bottom graph shows $\Gamma^{(1)}_{3,7,9}$, a graph with $13$ vertices.

\begin{defnn} Let $Y\subseteq \bsft^\Z$ be a subshift of finite type described by the sequence
$(\bsft;p_1,\dots,p_k)$ and let $g \colon \Gamma^{(1)}_{n,p,q}
\to \bsft$ be a map. We say that $g$  {\em respects} $Y$ if for all $1\leq i\leq k$ and for any $\Z$-homomorphism $\varphi: \dom(p_i)\to \Gamma^{(1)}_{n,p,q}$, $g\circ \varphi: \dom(p_i)\to\bsft$ is not equal to $p_i$.
\end{defnn}

\begin{thmn} [Two tiles theorem for $F(2^\Z)$] \label{thm:twotilesthm}
Let $Y\subseteq \bsft^\Z$ be a subshift of finite type. Then the following are equivalent.
\begin{enumerate}[label=\rm{(\arabic*)}, ref=\arabic*]
\item\label{oda}
There is a continuous, equivariant map $f \colon F(2^\Z) \to Y$.
\item\label{odb}
There are $n,p,q$ with $n<p,q$, $\gcd(p,q)=1$, and
$n$ greater than or equal to the width of $Y$, and there is $g \colon \Gamma^{(1)}_{n,p,q}\to \bsft$
which respects $Y$.
\item\label{odc}
For all $n$ greater than or equal to the width of $Y$ and all sufficiently large $p,q$,
 there is $g \colon  \Gamma^{(1)}_{n,p,q}\to \bsft$ which respects $Y$.
\end{enumerate}
\end{thmn}

\begin{proof} The implication (\ref{odc})$\Rightarrow$(\ref{odb}) is obvious. We show (\ref{oda})$\Rightarrow$(\ref{odc}) and (\ref{odb})$\Rightarrow$(\ref{oda}).

Suppose first that there is a continuous, equivariant map $f \colon
F(2^\Z) \to Y$. We construct a hyper-aperiodic element $x \in 2^\Z$ as follows. Let $y\in 2^\Z$ be any hyper-aperiodic element. Enumerate without repetition
all tuples $(n_i, p_i, q_i, w_i)$ of positive integers where $p_i,q_i>n_i+2w_i$.
On $\Z$ mark off a collection of disjoint intervals
$$\mathcal{I}=\{ I^{n_i,p_i,q_i,w_i}_1, I^{n_i,p_i,q_i,w_i}_2\,:\, i\geq 1\}$$
such that for each $i$, $I^{n_i,p_i,q_i,w_i}_1, I^{n_i,p_i,q_i,w_i}_2$ are of sizes $p_i+n_i+2w_i$, $q_i+n_i+2w_i$ respectively, and the distance between successive intervals in $\mathcal{I}$ goes to infinity. We
identify each $I^{n_i,p_i,q_i,w_i}_1$ with an interval of size $w_i$
followed by a copy of $T_1(n_i,p_i,q_i)$ and then followed by another interval of size $w_i$. That is, it is a copy
of $T_1(n_i,p_i,q_i)$ surrounded by two intervals of size $w_i$.
We likewise identify each $I^{n_i,p_i,q_i,w_i}_2$ with  a copy
of $T_2(n_i,p_i,q_i)$ surrounded by two intervals of size $w_i$.

For $m \notin \bigcup_i (I^{n_i,p_i,q_i,w_i}_1 \cup I^{n_i,p_i,q_i,w_i}_2)$, we set
$x(m)=y(m)$. Next, assuming $I^{n_i,p_i,q_i,w_i}_1=[a_i, a_i+p_i+n_i+2w_i)$, we set $x(k_i + m) = y(m)$ for each $m\in [0,p_i)$
and $x(k_i+p_i+m) = y(m)$ for each $m\in [0, n_i + 2 w_i)$.
We define $x$ on the intervals
$I^{n_i,p_i,q_i,w_i}_2$ similarly, using $q_i$ instead of $p_i$. The definition of $x$ on these intervals is shown
in Figure~\ref{fig:onedimhyp}.

\begin{figure}[htp]
\centering
\begin{tikzpicture}[scale=0.2]
\usetikzlibrary{math}

\pgfmathsetmacro{\n}{3}
\pgfmathsetmacro{\p}{17}
\pgfmathsetmacro{\q}{21}
\pgfmathsetmacro{\w}{3}
\pgfmathsetmacro{\a}{-70}
\pgfmathsetmacro{\b}{-40}
\pgfmathsetmacro{\h}{1.5}

\pgfmathsetmacro{\x}{\n+\p+\n-1}
\pgfmathsetmacro{\y}{\n+\q+\n-1}
\pgfmathsetmacro{\nm}{\n-1}
\pgfmathsetmacro{\pm}{\p-1}
\pgfmathsetmacro{\qm}{\q-1}


\draw[fill=yellow] (\a-\w,-0.2) rectangle (\a,0.2);
\draw[fill=red!50] (\a,-0.2) rectangle (\a+\n,0.2);
\draw[fill=green]  (\a+\n,-0.2) rectangle (\a+\n+\w,0.2);
\draw[fill=blue!50] (\a+\n+\w,-0.2) rectangle (\a+\p,0.2);
\draw[fill=yellow] (\a+\p-\w,-0.2) rectangle (\a+\p,0.2);
\draw[fill=red!50] (\a+\p,-0.2) rectangle (\a+\p+\n,0.2);
\draw[fill=green] (\a+\p+\n,-0.2) rectangle (\a+\p+\n+\w,0.2);

\draw[fill=yellow] (\b-\w,-0.2) rectangle (\b,0.2);
\draw[fill=red!50] (\b,-0.2) rectangle (\b+\n,0.2);
\draw[fill=green]  (\b+\n,-0.2) rectangle (\b+\n+\w,0.2);
\draw[fill=blue!80] (\b+\n+\w,-0.2) rectangle (\b+\q,0.2);
\draw[fill=yellow] (\b+\q-\w,-0.2) rectangle (\b+\q,0.2);
\draw[fill=red!50] (\b+\q,-0.2) rectangle (\b+\q+\n,0.2);
\draw[fill=green] (\b+\q+\n,-0.2) rectangle (\b+\q+\n+\w,0.2);

\draw[decorate, decoration={brace,amplitude=5pt},yshift=-90pt] (\a+\p+\n,-1) -- (\a,-1)
node [black,midway,yshift=-15pt] {$T_1(n_i,p_i,q_i)$};

\draw[decorate, decoration={brace,amplitude=5pt},yshift=-90pt] (\b+\q+\n,-1) -- (\b,-1)
node [black,midway,yshift=-15pt] {$T_2(n_i,p_i,q_i)$};



\node at (\a-\h,1) {$w_i$};
\node at (\a+\h,1) {$n_i$};
\node at (\a+\n+\h,1) {$w_i$};
\node at (\a+\p-\h,1){$w_i$};
\node at (\a+\p+\h, 1){$n_i$};
\node at (\a+\p+\n+\h, 1){$w_i$};

\draw (\a-\w,1) -- (\a-\w,3);
\draw (\a-\w+\p, 1) -- (\a-\w+\p, 3);
\node at (\a+8, 2) {$p_i$};
\draw[->] (\a+7,2) -- (\a-\w,2);
\draw[->] (\a+9,2) -- (\a-\w+\p,2);

\node at (\b-\h,1) {$w_i$};
\node at (\b+\h,1) {$n_i$};
\node at (\b+\n+\h,1) {$w_i$};
\node at (\b+\q-\h,1){$w_i$};
\node at (\b+\q+\h, 1){$n_i$};
\node at (\b+\q+\n+\h, 1){$w_i$};

\draw (\b-\w,1) -- (\b-\w,3);
\draw (\b-\w+\q, 1) -- (\b-\w+\q, 3);
\node at (\b+9, 2) {$q_i$};
\draw[->] (\b+8,2) -- (\b-\w,2);
\draw[->] (\b+10,2) -- (\b-\w+\q,2);

\draw[decorate, decoration={brace, mirror, amplitude=5pt},yshift=-12pt] (\a,0) -- (\a+\n,0)
node [black,midway,yshift=-14pt] {$R^{(1)}_\times$};

\draw[decorate, decoration={brace,mirror, amplitude=5pt},yshift=-12pt] (\a+\p,0) -- (\a+\p+\n,0)
node [black,midway,yshift=-14pt] {$R^{(1)}_\times$};






\draw[decorate, decoration={brace,mirror, amplitude=5pt},yshift=-12pt] (\b,0) -- (\b+\w,0)
node [black,midway,yshift=-14pt] {$R^{(1)}_\times$};

\draw[decorate, decoration={brace,mirror, amplitude=5pt},yshift=-12pt] (\b+\q,0) -- (\b+\q+\n,0)
node [black,midway,yshift=-14pt] {$R^{(1)}_\times$};

\end{tikzpicture}
\caption{The definition of $x$ on intervals $I^{n_i,p_i,q_i,w_i}_1$ and $I^{n_i,p_i,q_i,w_i}_2$.\label{fig:onedimhyp} }
\end{figure}

For each $i$ there are a total of four subintervals of $I^{n_i,p_i,q_i,w_i}_1$ and $I^{n_i,p_i,q_i,w_i}_2$ that are corresponding to the $R^{(1)}_\times$-labeled blocks of $T_1(n_i,p_i,q_i)$ and $T_2(n_i,p_i,q_i)$. Extending each of these intervals by an interval of size $w_i$ to the left and another interval of size $w_i$ to the right, we obtain four intervals of size $n_i+2w_i$ within $I^{n_i,p_i,q_i,w_i}_1$ and $I^{n_i,p_i,q_i,w_i}_2$, which happen to be also the initial segment and the final segment of each of $I^{n_i,p_i,q_i,w_i}_1$ and $I^{n_i,p_i,q_i,w_i}_2$ of size $n_i+2w_i$. The key observation is that $x$ looks the same on all these four intervals.

Since $y$ is hyper-aperiodic, the enumeration $(n_i, p_i, q_i, w_i)$ is non-repetitive, and the distance between successive intervals in the collection $\mathcal{I}$ tends to infinity, it is
straightforward to check that $x$ is also hyper-aperiodic. Let $K=\ocl{[x]}$,
so $K \subseteq F(2^\Z)$ is compact (and invariant), and define $f_0 : K \rightarrow \bsft$
by $f_0(z) = f(z)(0)$. Since $K$ is compact and $f_0$ is
continuous on $K$, there is an integer $w>0$ such that for all
$z \in K$, $f_0(z)$ is determined by $z \res [-w,w]$.

Now let $n$ be greater than or equal to the width of $Y$, and $p, q>n+2w$. Let $i$ be such that $n_i=n$, $p_i=p$, $q_i=q$, and $w_i=w$.
Consider the copy of $T_1(n_i,p_i,q_i)$ as a subinterval of $I^{n_i,p_i,q_i,w_i}_1$ and the copy of $T_2(n_i,p_i,q_i)$ as a subinterval of $I^{n_i,p_i,q_i,w_i}_2$. For any $a\in T_1(n_i,p_i,q_i)\cup T_2(n_i,p_i,q_i)$, $f_0(a\cdot x)$ is determined by $x\res [a-w, a+w]$. By the key observation, $f_0(a\cdot x)$, as a function of $a$, looks the same on all four $R^{(1)}_\times$-labeled blocks within $T_1(n_i,p_i,q_i)$ and $T_2(n_i,p_i,q_i)$. Thus, if we define
$g': G^{(1)}_{n_i,p_i,q_i}\to \bsft$ by $g'(a)=f_0(a\cdot x)$, then $g'$ gives rise to a well-defined map $g: \Gamma^{(1)}_{n_i,p_i,q_i}\to \bsft$. Since $f$ is an equivariant map into the subshift $Y$, it follows
that $g$ respects $Y$. This proves (\ref{oda})$\Rightarrow$(\ref{odc}).

Finally, we show that (\ref{odb}) implies (\ref{oda}). Fix $n,p,q$
with $n< p,q$, $\gcd(p,q)=1$, and $n$ greater than or equal to the width of $Y$,
and assume there is $g \colon \Gamma^{(1)}_{n,p,q}\to \bsft$ which respects $Y$.
Let $N$ be a positive integer large enough so that every $k \geq N$ is a
non-negative integral linear combination of $p$ and $q$. Arbitrarily fix a point $c$ in the
$R^{(1)}_\times$-labeled block in $\Gamma^{(1)}_{n,p,q}$. Then for every $k \geq N$ we can fix a positively-oriented
path $(v_0^k, v_1^k, \ldots, v_k^k)$ in $\Gamma^{(1)}_{n,p,q}$ (positively-oriented meaning that the directed path consists of only edges with $\Z$-label $e_1=+1$) that begins and ends at $v_0^k = v_k^k = c$.
From Lemma~\ref{bml}, fix a clopen set $M_N\subseteq F(2^\Z)$ which is an $N$-marker set
for $F(2^\Z)$ (i.e., the conclusions of Lemma~\ref{bml} hold).
We define $f_0 \colon F(2^\Z)\to \bsft$ as follows. For each pair $x, y \in M_N$ of consecutive points of $M_N$ with $x$
preceding $y$ (meaning there is $k>0$ with $k\cdot x=y$ but no $\ell\in (0,k)$ with $\ell\cdot x\in M_N$), let $k > 0$ satisfy $k \cdot x = y$ and
define $f_0(t \cdot x) = g(v_t^k)$ for each $0 \leq t \leq k$. The map $f_0 \colon F(2^\Z)\to \bsft$
we have just defined is easily seen to be continuous since the set $M_N$ is clopen
in $F(2^\Z)$. Now define $f \colon F(2^\Z) \to \bsft^{\Z}$ by extending $f_0$ equivariantly:
$f(x)(a) = f_0(a \cdot x)$.
The map $f$ takes values in $Y$ since $g$ respects $Y$
and since we have the following simple fact: if $T_a, T_b \in \{ T_1(n,p,q),T_2(n,p,q)\}$
and $T_a$, $T_b$ are laid down so that the final $R^{(1)}_\times$-labeled block in $T_a$
coincides with the initial $R^{(1)}_\times$-labeled block in $T_b$ to form a single interval
$I$, then any subinterval $J$ of $I$ of width $n$ lies entirely in the copy
of $T_a$ or in the copy of $T_b$.
\end{proof}

\section{Statement of the twelve tiles theorem}  \label{subsec:tilethm}
In this section we define the 12-tiles graphs $\gnpq$, and state the twelve tiles theorem,
Theorem~\ref{thm:tilethm}. This theorem is stated in terms of continuous, equivariant maps into
subshifts of finite type, but in \S\ref{sec:continuous_structuring} we give the alternate
formulation in terms of continuous structurings. The proof of Theorem~\ref{thm:tilethm} will be given in the next two sections.

We define the 12-tiles graph $\gnpq$ for any fixed $n,p,q$ with $p,q >n$ by following the outline given in the preamble of this chapter.
We begin with defining $12$ ($2$-dimensional) rectangular grid graphs with blocks. They will be denoted as $G^i_{n,p,q}$ for $1\leq i\leq 12$.
Since a grid graph is isomorphic to the Cayley graph of $\Z^2$ restricted to a rectangle $[0,k)\times [0,h)$ for some $k,h>0$, we only need to specify the ``dimensions" (in this case $k\times h$) of the rectangle in order to describe a grid graph.
In Table~\ref{table:12tiles} we give a summary of the dimensions of the 12 grid graphs $G^i_{n,p,q}$ for $1\leq i\leq 12$.

Each of the grid graphs $G^i_{n,p,q}$ will be decomposed into blocks of vertices, and some blocks (in fact all but one for each grid graph) will be assigned block labels. Each block, being a grid graph itself, also has dimensions. We will be working with five block labels $R_\times, R_a, R_b, R_c, R_d$, and their dimensions are as follows:
\begin{alignat*}{8}
R_\times: &\quad\quad &  n   &  &\;\times\;&&      &n  \\
R_a: &&             n& &\;\times\;&&      (p-&n)  \\
R_b: &&             n& &\;\times\;&&      (q-&n)  \\
R_c: &&               (p-n&)  &\;\times\;&& &n \\
R_d: &&              (q- n &)  &\;\times\;&& &n
\end{alignat*}
The specifics of the construction of $G^i_{n, p, q}$ will be given in the next few pages and mostly done by illustrations rather than formal definitions. Here we recall that the purpose of assigning the block labels is so that we can form quotient graphs by identifying all vertices that are in the same relative positions within similarly labeled blocks. We refer to these quotient graphs as obtained {\em modulo labeled blocks}.

To finish the construction of $\gnpq$, we take
$$ G_{n,p,q}=\bigcup_{i=1}^{12} G^i_{n,p,q}$$
as a disjoint union of the 12 grid graphs with blocks, and obtain $\gnpq$ as the quotient graph modulo labeled blocks.

As outlined before, an alternative construction of $\gnpq$ is as follows.  From each of $G^i_{n,p,q}$,
we obtain a tile graph $T^i_{n,p,q}$ as a quotient graph of $G^i_{n,p,q}$ modulo labeled blocks.
Then $\gnpq$ can be obtained by taking again the disjoint union of the tiles $T^i_{n,p,q}$ for $1\leq i\leq 12$,
and forming the quotient graph of the union again modulo labeled blocks.

The individual tiles $T^i_{n,p,q}$ will be convenient to use in some of the later proofs. An additional forewarning is that the grid graphs $G^i_{n,p,q}$ and the tile graphs $T^i_{n,p,q}$ will also be named in view of the homotopy they realize. These alternative names will be more convenient to use in later proofs. We summarize all these graphs and their names in Table~\ref{table:12tiles}.

We now turn to the specifics of the definition of $G^i_{n,p,q}$ for $1\leq i\leq 12$.

The construction of the first four of the 12 tile graphs are illustrated in Figure~\ref{fig:Gamma-npq-torus}. We use one example to elaborate the construction, the others being similar.
We first consider the grid graph $G^1_{n,p,q}$ of dimensions $(p+n)\times (p+n)$. Assume the vertex set of $G^1_{n,p, q}$ to be $[0,p+n)\times [0, q+n)$. We then decompose the vertex set into nine blocks by dividing the horizontal interval $[0,p+n)$ into three intervals $[0,n)$, $[n,p)$, $[p,p+n)$ and similarly dividing the vertical interval. Then we assign the block label $R_\times$ to the four ``corner" blocks $I\times J$ where $I, J\in \{[0,n), [p,p+n)\}$, assign the block label $R_a$ to the two blocks $I\times [n,p)$ where $I\in \{[0,n), [p, p+n)\}$, assign the block label $R_c$ to the two blocks $[n,p)\times J$ where $J\in \{[0,n), [p,p+n)\}$, and leave the remaining ``interior" block $[n,p)\times [n,p)$ unlabeled. This completes the definition of $G^1_{n,p,q}$. $G^1_{n,p, q}$ is also named $\Gcaca$ to signify the statement that the two (undirected) paths along the boundaries of the grid graph from the upper-left corner to the lower-right corner are homotopic to each other (in a sense that will be made precise later in the next chapter). Furthermore, the tile graph $T^1_{n,p,q}$ is obtained as a quotient graph of $G^1_{n,p,q}$ modulo labeled blocks. To be precise, for any $i, j\in [0,n)$, the four vertices $(i,j)$, $(p+i,j)$, $(i,p+j)$ and $(p+i,p+j)$ in $G^1_{n,p,q}$ are identified as one vertex in $T^1_{n,p,q}$. Similarly, for $i\in [0,n)$ and $k\in [n,p)$, the two vertices $(i,k)$ and $(p+i,k)$ are identified, and the two vertices $(k,i)$ and $(k,p+i)$ are identified. $T^1_{n,p,q}$ can be viewed as a grid graph on a 2-torus, and for this reason we refer to it as a {\em torus tile}. The other three grid graphs and torus tiles are constructed in a similar manner, the only difference being their dimensions.

\begin{figure}[htpb] 
\begin{tikzpicture}
\begin{scope}[scale=1,yscale=-1]
\begin{scope}[shift={(-5,2)}]  
  \foreach \loc in {(0,0),(0,1.8),(1.8,0),(1.8,1.8)} {
    \node[XBoxStyle] at \loc {$R_\times$};
  }
  \foreach \loc in {(0,1),(1.8,1)} {
    \node[ABoxStyle] at \loc {$R_a$};
  }
  \foreach \loc in {(1,0),(1,1.8)} {
    \node[CBoxStyle] at \loc {$R_c$};
  }

  \node[BoxLabelStyle] at (1.4,3.4) {\small $G^1_{n,p,q}$ or $\Gcaca$};
\end{scope}

\begin{scope}[shift={(-5,7)}]  
  \foreach \loc in {(0,0),(0,1.8),(2.2,0),(2.2,1.8)} {
    \node[XBoxStyle] at \loc {$R_\times$};
  }
  \foreach \loc in {(0,1),(2.2,1)} {
    \node[ABoxStyle] at \loc {$R_a$};
  }
  \foreach \loc in {(1,0),(1,1.8)} {
    \node[DBoxStyle] at \loc {$R_d$};
  }

  \node[BoxLabelStyle] at (1.6,3.4) {\small $G^3_{n,p,q}$ or $\Gdada$};
\end{scope}

\begin{scope}[shift={(0,2)}]  
  \foreach \loc in {(0,0),(0,2.2),(1.8,0),(1.8,2.2)} {
    \node[XBoxStyle] at \loc {$R_\times$};
  }
  \foreach \loc in {(0,1),(1.8,1)} {
    \node[BBoxStyle] at \loc {$R_b$};
  }
  \foreach \loc in {(1,0),(1,2.2)} {
    \node[CBoxStyle] at \loc {$R_c$};
  }

  \node[BoxLabelStyle] at (1.4,3.8) {\small $G^2_{n,p,q}$ or $\Gcbcb$};
\end{scope}

\begin{scope}[shift={(0,7)}]  
  \foreach \loc in {(0,0),(0,2.2),(2.2,0),(2.2,2.2)} {
    \node[XBoxStyle] at \loc {$R_\times$};
  }
  \foreach \loc in {(0,1),(2.2,1)} {
    \node[BBoxStyle] at \loc {$R_b$};
  }
  \foreach \loc in {(1,0),(1,2.2)} {
    \node[DBoxStyle] at \loc {$R_d$};
  }

  \node[BoxLabelStyle] at (1.6,3.8) {\small $G^4_{n,p,q}$ or $\Gdbdb$};
\end{scope}

\end{scope}
\end{tikzpicture}
\caption{\label{fig:Gamma-npq-torus}The torus tiles in $\Gamma_{n,p,q}$. }
\end{figure} 

The definition of the remaining $8$ grid graphs are illustrated in Figures~\ref{fig:Gamma-npq-commute},
\ref{fig:Gamma-npq-horiz}, and \ref{fig:Gamma-npq-vert}.
Again, tile graphs are obtained as quotient graphs modulo labeled blocks. Each of the 12 tiles contains a distinct interior block which is unlabeled.
This completes the definition of $G^i_{n,p,q}$ and $T^i_{n,p,q}$ for $1\leq i\leq 12$. The grid graphs are also named according to the homotopy between the two paths along their boundaries from the upper-left corner to the lower-right corner. The tile graphs are correspondingly named, and are classified as ``commutativity tiles", ``long horizontal tiles", and ``long vertical tiles" according to their unique features.

\begin{figure}[htpb] 
\begin{tikzpicture}
\begin{scope}[scale=1,yscale=-1]
\begin{scope}[shift={(5,2)}]  
  \foreach \loc in {(0,0),(2.2,0),(4,0),(0,1.8),(1.8,1.8),(4,1.8)} {
    \node[XBoxStyle] at \loc {$R_\times$};
  }
  \foreach \loc in {(0,1),(4,1)} {
    \node[ABoxStyle] at \loc {$R_a$};
  }
  \foreach \loc in {(1,0),(2.8,1.8)} {
    \node[DBoxStyle] at \loc {$R_d$};
  }
  \foreach \loc in {(1,1.8),(3.2,0)} {
    \node[CBoxStyle] at \loc {$R_c$};
  }

  \node[BoxLabelStyle] at (2.5,3.4) {\small $G^5_{n,p,q}$ or $\Gdcadca$};
\end{scope}

\begin{scope}[shift={(5,10)}]  
  \foreach \loc in {(0,0),(1.8,0),(4,0),(0,1.8),(2.2,1.8),(4,1.8)} {
    \node[XBoxStyle] at \loc {$R_\times$};
  }
  \foreach \loc in {(0,1),(4,1)} {
    \node[ABoxStyle] at \loc {$R_a$};
  }
  \foreach \loc in {(1,0),(3.2,1.8)} {
    \node[CBoxStyle] at \loc {$R_c$};
  }
  \foreach \loc in {(1,1.8),(2.8,0)} {
    \node[DBoxStyle] at \loc {$R_d$};
  }

  \node[BoxLabelStyle] at (2.5,3.4) {\small $G^7_{n,p,q}$ or $\Gcdacda$};
\end{scope}

\begin{scope}[shift={(12,1)}]  
  \foreach \loc in {(0,0),(0,1.8),(0,4),(1.8,0),(1.8,2.2),(1.8,4)} {
    \node[XBoxStyle] at \loc {$R_\times$};
  }
  \foreach \loc in {(1,0),(1,4)} {
    \node[CBoxStyle] at \loc {$R_c$};
  }
  \foreach \loc in {(0,1),(1.8,3.2)} {
    \node[ABoxStyle] at \loc {$R_a$};
  }
  \foreach \loc in {(1.8,1),(0,2.8)} {
    \node[BBoxStyle] at \loc {$R_b$};
  }

  \node[BoxLabelStyle] at (1.4,5.6) {\small $G^6_{n,p,q}$ or $\Gcbacba$};
\end{scope}

\begin{scope}[shift={(12,9)}]  
  \foreach \loc in {(0,0),(0,2.2),(0,4),(1.8,0),(1.8,1.8),(1.8,4)} {
    \node[XBoxStyle] at \loc {$R_\times$};
  }
  \foreach \loc in {(1,0),(1,4)} {
    \node[CBoxStyle] at \loc {$R_c$};
  }
  \foreach \loc in {(0,1),(1.8,2.8)} {
    \node[BBoxStyle] at \loc {$R_b$};
  }
  \foreach \loc in {(1.8,1),(0,3.2)} {
    \node[ABoxStyle] at \loc {$R_a$};
  }

  \node[BoxLabelStyle] at (1.4,5.6) {\small $G^8_{n,p,q}$ or $\Gcabcab$};
\end{scope}
\end{scope}
\end{tikzpicture}
\caption{\label{fig:Gamma-npq-commute}The commutativity tiles in $\Gamma_{n,p,q}$.
Tiles $\Gcabcab$ and $\Gcbacba$ commute $R_a$ with $R_b$ and
tiles $\Gcdacda$ and $\Gdcadca$ commute $R_c$ with $R_d$.}
\end{figure} 

\begin{figure}[htpb] 
\begin{tikzpicture}
\begin{scope}[scale=1,yscale=-1]
\begin{scope}[shift={(0,0)}]  
  \foreach \loc in {(0,0.0),(1.8,0.0),(3.6,0.0),(5.4,0.0),(7.2,0.0),(10,0.0),
                    (0,1.8),(2.2,1.8),(4.4,1.8),(6.6,1.8),(10,1.8)} {
    \node[XBoxStyle] at \loc {$R_\times$};
  }
  \foreach \loc in {(0,1),(10,1)} {
    \node[ABoxStyle] at \loc {$R_a$};
  }
  \foreach \loc in {(1,0),(2.8,0),(4.6,0),(6.4,0),(9.2,0)} {
    \node[CBoxStyle] at \loc {$R_c$};
  }
  \foreach \loc in {(1,1.8),(3.2,1.8),(5.4,1.8),(7.6,1.8),(8.8,1.8)} {
    \node[DBoxStyle] at \loc {$R_d$};
  }

  \fill[fill=white] (8,-.1) rectangle (9.6, 2.9);

  \node at (8.8, 1.4) {\Large$\cdots$};

  \draw [decorate,decoration={brace,amplitude=10pt}]
  (0,-0.5) -- (11,-0.5) node [scale=1, black,midway, above, yshift=20pt]
        {$q$ copies of $R_c$, $q+1$ copies of $R_\times$};

  \draw [decorate,decoration={brace,amplitude=10pt}]
  (11,3.3) -- (0,3.3) node [scale=1, black,midway, below, yshift=-20pt]
        {$p$ copies of $R_d$, $p+1$ copies of $R_\times$};

  \node[BoxLabelStyle] at (5.5,5.5) {\small $G^9_{n,p,q}$ or $\Gcqadpa$};
\end{scope}

\begin{scope}[shift={(0,10)}]  
  \foreach \loc in {(0,1.8),(1.8,1.8),(3.6,1.8),(5.4,1.8),(7.2,1.8),(10,1.8),
                    (0,0),(2.2,0),(4.4,0),(6.6,0),(10,0)} {
    \node[XBoxStyle] at \loc {$R_\times$};
  }
  \foreach \loc in {(0,1),(10,1)} {
    \node[ABoxStyle] at \loc {$R_a$};
  }
  \foreach \loc in {(1,1.8),(2.8,1.8),(4.6,1.8),(6.4,1.8),(9.2,1.8)} {
    \node[CBoxStyle] at \loc {$R_c$};
  }
  \foreach \loc in {(1,0),(3.2,0),(5.4,0),(7.6,0),(8.8,0)} {
    \node[DBoxStyle] at \loc {$R_d$};
  }

  \fill[fill=white] (8,-.1) rectangle (9.6, 2.9);

  \node at (8.8, 1.4) {\Large$\cdots$};

  \draw [decorate,decoration={brace,amplitude=10pt}]
  (0,-0.5) -- (11,-0.5) node [scale=1, black,midway, above, yshift=20pt]
        {$p$ copies of $R_d$, $p+1$ copies of $R_\times$};

  \draw [decorate,decoration={brace,amplitude=10pt}]
  (11,3.3) -- (0,3.3) node [scale=1, black,midway, below, yshift=-20pt]
        {$q$ copies of $R_c$, $q+1$ copies of $R_\times$};

  \node[BoxLabelStyle] at (5.5,5.5) {\small $G^{10}_{n,p,q}$ or $\Gdpacqa$};
\end{scope}

\end{scope}
\end{tikzpicture}
\caption{\label{fig:Gamma-npq-horiz} The long horizontal tiles in $\Gamma_{n,p,q}$.}
\end{figure} 

\begin{figure}[htpb] 
\begin{tikzpicture}
\begin{scope}[scale=1,yscale=-1]
\begin{scope}[shift={(0,0)}]  
  \foreach \loc in {(0.0,0),(0.0,1.8),(0.0,3.6),(0.0,5.4),(0.0,7.2),(0.0,10),
                    (1.8,0),(1.8,2.2),(1.8,4.4),(1.8,6.6),(1.8,10)} {
    \node[XBoxStyle] at \loc {$R_\times$};
  }
  \foreach \loc in {(1,0),(1,10)} {
    \node[CBoxStyle] at \loc {$R_c$};
  }
  \foreach \loc in {(0,1),(0,2.8),(0,4.6),(0,6.4),(0,9.2)} {
    \node[ABoxStyle] at \loc {$R_a$};
  }
  \foreach \loc in {(1.8,1),(1.8,3.2),(1.8,5.4),(1.8,7.6),(1.8,8.8)} {
    \node[BBoxStyle] at \loc {$R_b$};
  }

  \fill[fill=white] (-.1,8) rectangle (2.9,9.6);

  \node at (1.4,8.8) {\Large$\vdots$};

  \draw [decorate,decoration={brace,amplitude=10pt}]
  (-0.5,11) -- (-0.5,0) node [scale=1, black,midway, left, xshift=-30pt, rotate=-90, anchor=center]
        {\begin{tabular}{c}$q$ copies of $R_a$, \\ $q+1$ copies of $R_\times$\end{tabular}};

  \draw [decorate,decoration={brace,amplitude=10pt}]
  (3.3,0) -- (3.3,11) node [scale=1, black,midway, right, xshift=30pt, rotate=-90, anchor=center]
        {\begin{tabular}{c}$p$ copies of $R_b$, \\ $p+1$ copies of $R_\times$\end{tabular}};

  \node[BoxLabelStyle] at (1.4,12) {\small $G^{11}_{n,p,q}$ or $\Gcbpcaq$};
\end{scope}

\begin{scope}[shift={(6,0)}]  
  \foreach \loc in {(1.8,0),(1.8,1.8),(1.8,3.6),(1.8,5.4),(1.8,7.2),(1.8,10),
                    (0,0),(0,2.2),(0,4.4),(0,6.6),(0,10)} {
    \node[XBoxStyle] at \loc {$R_\times$};
  }
  \foreach \loc in {(1,0),(1,10)} {
    \node[CBoxStyle] at \loc {$R_c$};
  }
  \foreach \loc in {(1.8,1),(1.8,2.8),(1.8,4.6),(1.8,6.4),(1.8,9.2)} {
    \node[ABoxStyle] at \loc {$R_a$};
  }
  \foreach \loc in {(0,1),(0,3.2),(0,5.4),(0,7.6),(0,8.8)} {
    \node[BBoxStyle] at \loc {$R_b$};
  }

  \fill[fill=white] (-.1,8) rectangle (2.9,9.6);

  \node at (1.4,8.8) {\Large$\vdots$};

  \draw [decorate,decoration={brace,amplitude=10pt}]
  (-0.5,11) -- (-0.5,0) node [scale=1, black,midway, left, xshift=-30pt, rotate=-90, anchor=center]
        {};

  \draw [decorate,decoration={brace,amplitude=10pt}]
  (3.3,0) -- (3.3,11) node [scale=1, black,midway, right, xshift=30pt, rotate=-90, anchor=center]
        {\begin{tabular}{c}$q$ copies of $R_a$, \\ $q+1$ copies of $R_\times$\end{tabular}};

  \node[BoxLabelStyle] at (1.4,12) {\small $G^{12}_{n,p,q}$ or $\Gcaqcbp$};
\end{scope}

\end{scope}
\end{tikzpicture}

\caption{The long vertical tiles in $\Gamma_{n,p,q}$.}
\label{fig:Gamma-npq-vert}
\end{figure} 

We summarize the notation for and the basic features of the 12 tiles and their corresponding grid graphs in Table~\ref{table:12tiles}.

\begin{table}[htpb]
\centering
{\renewcommand{\arraystretch}{1.5}
\begin{tabular}{|c|c|c|c|c|l|}  \hline
\multirow{2}{*}{Category} & \multirow{2}{*}{Notation} & \multirow{2}{*}{Notation} & \multicolumn{3}{c|}{Corresponding grid graphs} \\ \cline{4-6}
 & &  & Notation & Notation & Dimensions  \\ \hline
& $\Tcaca$ & $T^1_{n,p,q}$ & $\Gcaca$ & $G^1_{n,p,q}$ & $(p+n)\times (p+n)$ \\ \cline{2-6}
Torus tiles &  $\Tcbcb$ & $T^2_{n,p,q}$ &  $\Gcbcb$ & $G^2_{n,p,q}$ & $(p+n)\times (q+n)$ \\ \cline{2-6}
 & $\Tdada$ & $T^3_{n,p,q}$ & $\Gdada$ & $G^3_{n,p,q}$ &$(q+n)\times (p+n)$ \\ \cline{2-6}
 & $\Tdbdb$ & $T^4_{n,p,q}$ & $\Gdbdb$ & $G^4_{n,p,q}$ &$(q+n)\times (q+n)$ \\ \hline
 & $\Tdcadca$ & $T^5_{n,p,q}$ & $\Gdcadca$ & $G^5_{n,p,q}$ &$(p+q+n)\times (p+n)$ \\ \cline{2-6}
Commutativity & $\Tcbacba$ & $T^6_{n,p,q}$ & $\Gcbacba$ & $G^6_{n,p,q}$ &$(p+n)\times (p+q+n)$ \\ \cline{2-6}
tiles  & $\Tcdacda$ & $T^7_{n,p,q}$ & $\Gcdacda$ & $G^7_{n,p,q}$ &$(p+q+n)\times(p+n)$ \\ \cline{2-6}
& $\Tcabcab$ & $T^8_{n,p,q}$ &$\Gcabcab$ & $G^8_{n,p,q}$ & $(p+n)\times (p+q+n)$ \\ \hline
Long horizontal & $\Tcqadpa$ & $T^9_{n,p,q}$ & $\Gcqadpa$ & $G^9_{n,p,q}$ &$(pq+n)\times (p+n)$ \\ \cline{2-6}
tiles & $\Tdpacqa$ & $T^{10}_{n,p,q}$ & $\Gdpacqa$ & $G^{10}_{n,p,q}$ &$(pq+n)\times (p+n)$ \\ \hline
Long vertical & $\Tcbpcaq$ & $T^{11}_{n,p,q}$ & $\Gcbpcaq$ & $G^{11}_{n,p,q}$ & $(p+n)\times (pq+n)$ \\ \cline{2-6}
tiles & $\Tcaqcbp$ & $T^{12}_{n,p,q}$ & $\Gcaqcbp$ & $G^{12}_{n,p,q}$ & $(p+n)\times (pq+n)$ \\
\hline
\end{tabular}
\vskip 12pt
}
\caption{The notation for the 12 tiles and their corresponding grid graphs.}
\label{table:12tiles}
\end{table}

This completes the definitions of the tile graphs $G^i_{n,p,q}$ and $T^i_{n,p,q}$ for $1\leq i\leq 12$ and the 12-tiles graphs $G_{n,p,q}$ and $\gnpq$.
Each vertex
in a quotient graph inherits the block label from the rectangular grid graphs. We make the trivial observation that the $\gnpq$ contains exactly one copy of a labeled block for each label type, which is isomorphic to the grid graph of appropriate dimensions.

We note that all of the graphs constructed are naturally $\Z^2$-graphs.
In fact, considering $\gnpq$, we note that every
directed edge in $\gnpq$ is labeled with a unique generator of $\Z^2$.
For example, suppose $(x,y)$ is an edge in $\gnpq$ and has $\Z^2$-label $e_1=(1,0)$
(the other cases being similar). If $x$, $y$ are in the same block
$R$ labeled with $R_\times, R_a, R_b, R_c$, or $R_d$, then $e_1$ is the unique $\Z^2$-label for $(x,y)$, as all
edges $(x',y')$ in one of the $12$ grid graphs in the same relative
position within a copy of $R$ have the same label. If at least one
of $x$, $y$ is in an interior block of a grid graph, then there is a unique
pair $(x',y')$ forming an edge in the union of the grid graphs which
represent $(x,y)$ in the quotient $\gnpq$, so the $\Z^2$-label is unique.
Finally, if $x$, $y$ are in different, non-interior, blocks
then in any of the $12$ grid graphs in which an edge $(x',y')$
appears which represents $(x,y)$ in the quotient, this edge has the same
$\Z^2$-label. This is because, by inspection, if $x \in R$ then two different generators cannot
both move $x$ to the same  block $R'$ where $R'\neq R$ either has a different label or is unlabeled.

The next lemma explains the significance of the parameter $n$
in $\gnpq$, relating it to a possible width of a subshift of finite type.

\begin{lemn} \label{lem:pbg}
Consider $\gnpq$ where $n <p,q$, and let $n_0\leq n+1$.
Let $A$ be the rectangular grid graph on $[0,n_0)\times [0,n_0)$.
Suppose $\varphi$ is a $\Z^2$-homomorphism from $A$ to $\gnpq$.
Then there is a $\Z^2$-homomorphism $\psi$ from $A$ to $G_{n,p,q}$
such that $\varphi= \pi\circ \psi$, where $\pi=\pi_{n,p,q} \colon G_{n,p,q}
\to \gnpq$ is the quotient map.
\end{lemn}

\begin{proof}
Recall that in the construction of $\gnpq$ only the labeled blocks of the tiles $G^i_{n,p,q}$
are identified. Thus if $\ran(\varphi)$
contains a vertex of $\gnpq$ which is an element of an interior block
of some $G_i=G^i_{n,p,q}$ ($1\leq i\leq 12$), then this $i$ is unique and since $n_0\leq n+1$ there is a unique $\Z^2$-homomorphism
$\psi$ from $A$ to $G_i$ such that $\varphi=\pi\circ \psi$. So we can assume that
$\ran(\varphi)$ contains only vertices not in interior blocks of $\gnpq$.
If $\ran(\varphi)$ consists of vertices all with a single block label, then we may take any $G_i$ where that
block appears and define $\psi$ to map into that block of $G_i$
in the obvious manner. The other possibility is that $\ran(\varphi)$
contains vertices of exactly two block labels types
(by inspection, one can verify that more than two is not possible as $n_0\leq n+1$ and $n<p,q$), one of which is $R_\times$. We can then take
any $G_i$ where these two block types appear and define $\psi$ to have range in
two such blocks of $G_i$.
\end{proof}

The following related lemma analyzes $\Z^2$-homomorphisms from the Cayley graph of $\Z^2$ to $\gnpq$.
It shows that these correspond to consistent ``tilings'' of $\Z^2$ by
grid graphs each of which is isomorphic to one of the $12$ grid graphs $G^i_{n,p,q}$ with labeled blocks.

\begin{lemn} \label{zghom} Let $n< p, q$ be positive integers.
Suppose $\varphi \colon \Z^2 \to \gnpq$ is a $\Z^2$-homomorphism.
Then there is a collection $\{ C_i\}$ of rectangles in $\Z^2$ and a map $i\mapsto \alpha(i)\in[1,12]$ such that
the following hold:
\begin{enumerate}[label={\rm (\arabic*)}]
\item \label{zghome} For each $i$, $C_i$ is of the same dimensions as $G^{\alpha(i)}_{n,p,q}$; thus there is a unique isomorphism between $G^{\alpha(i)}_{n,p,q}$ and $C_i$ which makes $C_i$ a grid graph with labeled blocks;
\item \label{zghoma}
$\bigcup_i C_i=\Z^2$;
\item \label{zghomb}
If $x\in \Z^2$ is in $C_i \cap C_j$, then there is a labeled block $R$ of $C_i$ which contains $x$ as a vertex and a labeled block $R'$ of $C_j$ which contains $x$ as a vertex so that $R=R'$ as labeled blocks (meaning they are the same rectangles with the same labels);
\item \label{zghomc}
For any $n_0\leq n+1$ and any rectangle $A\subseteq \Z^2$ of dimensions $n_0 \times n_0$, there is $i$ such that
$A \subseteq C_i$.
\end{enumerate}

\end{lemn}

\begin{proof}
Suppose $\varphi \colon \Z^2 \to \gnpq$ is a $\Z^2$-homomorphism.
For each of the $12$ rectangular grid graphs $G^i_{n,p,q}$, $1 \leq i \leq 12$,
pick a distinguished vertex in the interior block of $G^i_{n,p,q}$, say
the upper-left vertex of the interior block. Let $D\subseteq \Z^2$
be the set of all points $x\in \Z^2$ such that $\varphi(x)$ is a distinguished vertex of one of the $G^i_{n,p,q}$.
For each $x \in D$, let $\alpha(x)\in[1,12]$ such that $\varphi(x)$ is a distinguished vertex of $G^{\alpha(x)}_{n,p,q}$, and let $C_x$ be the rectangular grid subgraph of $\Z^2$
obtained by placing a copy of $G^{\alpha(x)}_{n,p,q}$ such that $x$
corresponds with the position of the distinguished vertex of $G^{\alpha(x)}_{n,p,q}$.
It is easily checked that the collection $\{C_x\}_{x \in D}$ and the map $x\mapsto \alpha(x)$ satisfy \ref{zghome}, \ref{zghoma} and \ref{zghomb},
and it follows from Lemma \ref{lem:pbg} that it satisfies \ref{zghomc}.
\end{proof}

The following is a converse of the above lemma with a weaker condition.

\begin{lemn} \label{zghomconverse} Let $n< p, q$ be positive integers.
Suppose there is a collection $\{ C_i\}$ of rectangles in $\Z^2$ and a map $i\mapsto \alpha(i)\in[1,12]$ such that
the following hold:
\begin{enumerate}[label={\rm (\arabic*)}]
\item For each $i$, $C_i$ is of the same dimensions as $G^{\alpha(i)}_{n,p,q}$; thus there is a unique isomorphism between $G^{\alpha(i)}_{n,p,q}$ and $C_i$ which makes $C_i$ a grid graph with labeled blocks;
\item
$\bigcup_i C_i=\Z^2$;
\item
If $x\in \Z^2$ is in $C_i \cap C_j$, then there is a labeled block $R$ of $C_i$ which contains $x$ as a vertex and a labeled block $R'$ of $C_j$ which contains $x$ as a vertex so that $R=R'$ as labeled blocks;
\item[\rm (4')]
For any rectangle $A\subseteq \Z^2$ of dimensions $2 \times 2$, there is $i$ such that
$A \subseteq C_i$.
\end{enumerate}
Then there is a $\Z^2$-homomorphism $\varphi: \Z^2\to \gnpq$.
\end{lemn}

\begin{proof}
Suppose $\{C_i\}$ and the map $i\mapsto \alpha(i)\in[1, 12]$ are given and satisfy (1), (2), (3) and (4'). Let $\pi: G_{n,p, q}\to \gnpq$ be the quotient map. Given any $x\in \Z^2$, fix any $i$ with $x\in C_i$, fix the unique isomorphism $\psi: C_i\to G^{\alpha(i)}_{n,p,q}$ as grid graphs with labeled blocks, and define $\varphi(x)=\pi\circ\psi(x)$.
By \ref{zghome}, \ref{zghoma} and \ref{zghomb}, $\varphi$ is well-defined. To see $\varphi$ is a $\Z^2$-homomorphism,
consider say $x$, $y=e_1\cdot x \in \Z^2$. From (4'),
there is an $i$ such that $\{ x, y\} \subseteq C_i$. If we use this same $C_i$
for the definition of $\varphi(x)$ and $\varphi(y)$, it is then clear that
the edge $(\varphi(x), \varphi(y))$ has label $e_1$ in $\gnpq$.
\end{proof}

\begin{defnn} \label{def:resp}
Let $Y$ be a subshift of finite type described by $(\bsft; p_1,\dots,p_k)$.
Let $g\colon \gnpq \to \bsft$. We say $g$ {\em respects} $Y$
if for any $1\leq i\leq k$ and for any $\Z^2$-homomorphism $\varphi$ from
$\dom(p_i)$ to $\gnpq$, $g\circ \varphi \colon \dom(p_i)\to \bsft$ is not equal to $p_i$.
\end{defnn}

From Lemma~\ref{lem:pbg} it follows that if
$n$ is greater than or equal to the width of $Y$, then $g \colon \gnpq \to \bsft$ respects $Y$
if and only if for any forbidden pattern $p_i$ defining $Y$ and for any $\Z^2$-homomorphism $\psi\colon \dom(p_i) \to G_{n,p,q}$,
$g\circ \pi\circ \psi \colon A_i\to \bsft$ is not equal to $p_i$, where $\pi=\pi_{n,p,q}:G_{n,p,q}\to \Gamma_{n,p,q}$ is the quotient map.
Thus, for $n$ greater than or equal to the width of $Y$, the condition that $g \colon \gnpq\to \bsft$
respects $Y$ is a property of $\tilde{g} \colon G_{n,p,q}\to \bsft$, where
$\tilde{g}=g \circ \pi$ is the map on $G_{n,p,q}$ that $g$ induces.

We are now ready to state our main theorem.

\begin{thmn}[Twelve tiles theorem] \label{thm:tilethm}
Let $Y\subseteq \bsft^{\Z^2}$ be a subshift of finite type. Then the following are equivalent.
\begin{enumerate}[label={\rm (\arabic*)}, ref=\arabic*]
\item \label{mta}
There is a continuous, equivariant map $f\colon F(2^{\Z^2})\to
Y$.
\item \label{mtb}
There are positive integers $n,p,q$ with $n<p,q$, $\gcd(p,q)=1$,  and $n$ greater than or equal
to the width of $Y$, and there is $g \colon \gnpq \to \bsft$
which respects $Y$.
\item \label{mtc}
For all $n$ greater than or equal to the width of $Y$ and
for all sufficiently large $p,q$, there is $g \colon \gnpq
\to \bsft$ which respects $Y$.
\end{enumerate}
\end{thmn}

The implication (\ref{mtc})$\Rightarrow$(\ref{mtb}) is obvious. We will prove (\ref{mta})$\Rightarrow$(\ref{mtc}) in \S\ref{subsec:negtile}
and (\ref{mtb})$\Rightarrow$(\ref{mta}) in \S\ref{subsec:postile}.


\section{Proof of the negative direction} \label{subsec:negtile}

In this section we prove the negative direction of the twelve tiles theorem. That is,
we prove the implication (\ref{mta})$\Rightarrow$(\ref{mtc}) for
Theorem~\ref{thm:tilethm} (in \S\ref{sec:continuous_structuring} we give the
corresponding result, Theorem~\ref{thm:tilethmb}, in terms of continuous structurings).
We refer to this implication as the
negative direction as it is often used to show that various continuous structurings
of $F(2^{\Z^2})$ do not exist.

The idea of the proof is to construct a hyper-aperiodic element $x \in F(2^{\Z^2})$
which carries the structures $\gnpq$ embedded into it for all suitable $n, p, q$ in a way that is similar to the proof of the two tiles theorem.
We use this element
for both Theorems~\ref{thm:tilethm} and \ref{thm:tilethmb}. To construct $x$, we
will use a fixed hyper-aperiodic element $y \in F(2^{\Z^2})$ which we now fix.
The element $x$ will be constructed by copying various regions in $y$ to various
places in $x$.

Let $w$ be a positive integer and fix $n, p, q$ with $p, q>2n+4w$.
The integer $w$ will be the size of a ``buffer'' region
we will use to ``pad'' the construction of $G_{n,p,q}$ to produce a
modified version $G_{w,n,p,q}$. Just as $G_{n,p,q}$
is the disjoint union of the $12$ rectangular subgraphs $G^i_{n,p,q}$
($1\leq i \leq 12$), $G_{w,n,p,q}$ will be the disjoint union of the
rectangular subgraphs $G^i_{w,n,p,q}$. The element $x$ will then be built up from $y$
and the $G^i_{w,n,p,q}$.

Consider the grid graph $G^1_{n,p,q}$ for the first torus tile. Recall this is a rectangular grid graph of dimensions
$(p+n) \times (p+n)$, and is subdivided into blocks with labels
$R_\times$, $R_a$, and $R_c$, together with an unlabeled interior block as in Figure~\ref{fig:Gamma-npq-torus}.
$G^1_{w,n,p,q}$ will be a grid graph of dimensions $(p+n+2w)\times (p+n+2w)$ which is similarly
divided into blocks with new labels $R'_{\times}$, $R'_a$, and $R'_c$.
The definition of $G^1_{w,n,p,q}$ is shown in Figure~\ref{fig:hyper-aperiodic-point-tcaca}.
The solid lines indicate the new rectangular sub-graphs $R'_{\times}$, $R'_a$ and
$R'_c$. The dashed lines indicate the original $G^1_{n,p,q}$ along with the
original $R_\times$, $R_a$, and $R_c$. Note that the $R'_\times$ block has dimensions
$(n+2w)\times (n+2w)$, the $R'_a$ block has dimensions $(n+2w)\times (p-n-2w)$, and the $R'_c$ block
has dimensions $(p-n-2w)\times (n+2w)$. The other three padded grid graphs for torus tiles,
$G^2_{w,n,p,q}, G^3_{w,n, p, q}, G^4_{w,n,p,q}$, are defined in the same manner.

\begin{figure}[htbp]
\begin{tikzpicture}
\begin{scope}[scale=2,yscale=-1,shift={(-5,2)}]  
  \foreach \loc in {(0,0),(0,1.8),(1.8,0),(1.8,1.8)} {
    \node[XBoxStyle,scale=2,dotted] at \loc {};
    \node[BoxStyle,scale=2,minimum width=1.2cm, minimum height=1.2cm, shift={(-.1,.1)}] at \loc {};
    \node[shift={(1,-1)}] at \loc {$R'_\times$};
  }
  \foreach \loc in {(0,1),(1.8,1)} {
    \node[ABoxStyle,scale=2,dotted] at \loc {};
    \node[BoxStyle,scale=2,minimum width=1.2cm, minimum height=0.6cm, shift={(-.1,-.1)}] at \loc {};
    \node[shift={(1,-.8)}] at \loc {$R'_a$};
  }
  \foreach \loc in {(1,0),(1,1.8)} {
    \node[CBoxStyle,scale=2,dotted] at \loc {};
    \node[BoxStyle,scale=2,minimum width=0.6cm, minimum height=1.2cm, shift={(.1,.1)}] at \loc {};
    \node[shift={(.8,-1)}] at \loc {$R'_c$};
  }

  \draw [|-|] (3.1,-0.1) -- (3.1, 0.0) node[midway, right] {$\,w$};
  \draw [|-|] (2.8,-0.3) -- (2.9,-0.3) node[midway, above] {$\vphantom{q}w$};
\end{scope}
\end{tikzpicture}
\caption{Construction of the padded grid graph $G^1_{w,n,p,q}$. The boundary of
$G_{n,p,q}$ has been padded by $w$.}
\label{fig:hyper-aperiodic-point-tcaca}
\end{figure}

\begin{figure}[htbp]
\begin{tikzpicture}
\begin{scope}[scale=2,yscale=-1,shift={(-5,2)}]  
  \foreach \loc in {(0,0),(2.2,0),(4.0,0),(0,1.8),(1.8,1.8),(4.0,1.8)} {
    \node[XBoxStyle,scale=2,dotted] at \loc {};
    \node[BoxStyle,scale=2,minimum width=1.2cm, minimum height=1.2cm, shift={(-.1,.1)}] at \loc {};
    \node[shift={(1,-1)}] at \loc {$R'_\times$};
  }
  \foreach \loc in {(0,1),(4.0,1)} {
    \node[ABoxStyle,scale=2,dotted] at \loc {};
    \node[BoxStyle,scale=2,minimum width=1.2cm, minimum height=0.6cm, shift={(-.1,-.1)}] at \loc {};
    \node[shift={(1,-.8)}] at \loc {$R'_a$};
  }
  \foreach \loc in {(3.2,0),(1,1.8)} {
    \node[CBoxStyle,scale=2,dotted] at \loc {};
    \node[BoxStyle,scale=2,minimum width=0.6cm, minimum height=1.2cm, shift={(.1,.1)}] at \loc {};
    \node[shift={(.8,-1)}] at \loc {$R'_c$};
  }
  \foreach \loc in {(1,0),(2.8,1.8)} {
    \node[DBoxStyle,scale=2,dotted] at \loc {};
    \node[BoxStyle,scale=2,minimum width=1.0cm, minimum height=1.2cm, shift={(.1,.1)}] at \loc {};
    \node[shift={(1.2,-1)}] at \loc {$R'_d$};
  }
\end{scope}
\end{tikzpicture}
\caption{Constructing the rectangular grid graph $G^5_{w,n,p,q}$ corresponding to $\protect\Tdcadca$.
The solid lines represent the blocks $R'_\times$ etc.\ of $G^5_{w,n,p,q}$,
while the dashed lines indicate the original blocks of $G^5_{n,p,q}$.}
\label{fig:hyper-aperiodic-point-tdcadca}
\end{figure}

The remaining $G^i_{w,n,p,q}$ are defined similarly. In each case,
the block $R'_\times$ is expanded in each dimension by $2w$,
while the other blocks ($R'_a$, $R'_b$, $R'_c$, $R'_d$) are
expanded in one direction by $2w$ and contracted in the other direction by $2w$.
Figure~\ref{fig:hyper-aperiodic-point-tdcadca} shows the construction for
the grid graph $G^5_{w,n,p,q}$, which corresponds to the tile $\protect\Tdcadca$.
This completes the description of the rectangular grid graphs $G^i_{w,n,p,q}$.

To construct the hyper-aperiodic element $x$, let $(w_k,n_k,p_k,q_k)$ enumerate, without repetition, all
$4$-tuples of positive integers such that $p_k,q_k>2n_k+4w_k$. Fix a doubly-indexed sequence $(\ell^i_k)$ of integers for $1\leq i\leq 12$ and $k\in \Z^+$ such that for any $k, i, j$,
\begin{enumerate}
\item[(i)] $\ell^i_k<\ell^j_{k+1}$,
\item[(ii)] $\ell^i_k<\ell^j_k$ if and only if $i<j$,
\item[(iii)] $\ell^1_{k+1}-\ell^{12}_k>k+p_kq_k+n_k+2w_k$, and
\item[(iv)] $\ell^{i+1}_k-\ell^i_k>k+p_kq_k+n_k+2w_k$.
\end{enumerate}
By (i) and (ii), the $\ell^i_k$ are in lexicographic order of the pairs $(k,i)$. To continue our construction, we place a copy of the rectangular
grid graph $G^i_{w_k,n_k,p_k,q_k}$ in $\Z^2$ with the lower-left corner of
$G^i_{w_k,n_n,p_k,q_k}$ at the point $(\ell^i_k,0)$. By (iii) and (iv), the distance between any point
of the copy of $G^i_{w_k,n_k,p_k,q_k}$ and any other copy is at least $k$.
Formally, when we ``place a copy'' of $G^i_{w_k,n_k,p_k,q_k}$ in $\Z^2$ we mean
we have identified a rectangular grid subgraph of $\Z^2$ and have partitioned
its elements into labeled and unlabeled blocks isomorphically to $G^i_{w_k,n_k,p_k,q_k}$.
Figure~\ref{fig:xcon} illustrates the placement of the $G^i_{w_k,n_k,p_k,q_k}$
and the construction of $x$.

\begin{figure}[h]
\begin{tikzpicture}[scale=0.03]

\pgfmathsetmacro{\s}{0.03}

\pgfmathsetmacro{\n}{20}
\pgfmathsetmacro{\r}{0.05}

\draw[->] (-200,0) to (200,0);

\draw (-150,0) to (-150,5) to (-145,5) to (-145,0);
\draw (-140,0) to (-140,10) to (-135,10) to (-135,0);
\draw (-130,0) to (-130,5) to (-120,5) to (-120,0);
\draw (-115,0) to (-115,10) to (-105,10) to (-105,0);

\node at (-80,5) {$\dots\dots$};

\draw (-50,0) to (-50,30) to (-20,30) to (-20,0);
\draw (50,0) to (50,30) to (100,30) to (100,0);
\draw (0,0) to (0,50) to (30,50) to (30,0);
\draw (120,0) to (120,50) to (170,50) to (170,0);

\node at (195,10) {$\dots\dots$};

\end{tikzpicture}
\caption{Constructing the hyper-aperiodic element $x$} \label{fig:xcon}
\end{figure}

We now define $x$. For $(a,b)\in \Z^2$ not in any copy of $G^i_{w_k,n_k,p_k,q_k}$, we define $x(a,b)=y(a,b)$.
Suppose $(a,b)$ is in a copy of $G^i_{w_k,n_k,p_k,q_k}$ (note that the $i$ and $k$ are unique).
Say $(a,b)$ is in a labeled block $R'$, and the lower-left corner point of block $R'$ has coordinates $(\alpha, \beta)$. Write
$(a,b)=(\alpha,\beta)+(a',b')$. We then set $x(a,b)=y(a',b')$. In other words, we fill-in
the portion of $x$ in the block $R'$ by a partial copy of the first quadrant of $y$. Finally, if $(a, b)$ is in an unlabeled block, then set $x(a, b)=y(a, b)$.

This completes the definition of $x \in 2^{\Z^2}$. We next show that
$x$ is hyper-aperiodic.

\begin{lemn} \label{lem:xhyp}
The element $x\in 2^{\Z^2}$ is hyper-aperiodic.
\end{lemn}

\begin{proof}
Let $s\neq (0,0)$ be a non-identity element of the group $\Z^2$.
Since $y$ is a hyper-aperiodic element, there is a finite $T_0\subseteq \Z^2$
such that for any $g \in \Z^2$ there is a $t\in T_0$ such that $y(g+t)\neq y(g+s+t)$.
Let $T_1=\{(0,0), s\}\cup T_0\cup(s+T_0)$ and $M\in \Z^+$ be such that $T_1\subseteq [-M,M]\times [-M,M]$.

Let $S$ be the set of all $g\in\Z^2$ such that either $\{g, g+s\}\cup
(g+T_0)\cup (g+s+T_0)$ does not intersect any copy of
$G^i_{w_k,n_k,p_k,q_k}$ in the construction of $x$ or there is a
labeled or unlabeled block $R'$ of some copy of
$G^i_{w_k,n_k,p_k,q_k}$ such that $\{g, g+s\}\cup (g+T_0)\cup
(g+s+T_0)\subseteq R'$. Intuitively, $S$ is the set of $g\in\Z^2$ such
that some $t\in T_0$ will witness the hyper-aperiodicity of $x$ at $g$
for $s$, i.e., there is $t\in T_0$ such that $x(g+t)\neq x(g+s+t)$.

To verify that $x$ satisfies the hyper-aperiodicity condition, it
suffices to show that $S$ is {\em syndetic}, i.e., there is a finite
set $T\subseteq \Z^2$ such that for any $g\in \Z^2$ there is $t\in T$
so that $g+t\in S$.

Let $S'$ be the set of all $g\in \Z^2$ such that $g+[-M,M]\times
[-M,M]$ does not intersect any copy of
$G^i_{w_k,n_k,p_k,q_k}$. Intuitively, $S'$ is the set of all
$g\in\Z^2$ that are of a distance $M$ away from any copy of
$G^i_{w_k,n_k,p_k,q_k}$. We have $S'\subseteq S$.

Let $T'$ be the set of all $g\in\Z^2$ such that $g$ is within distance
$M$ to a copy of $G^i_{w_k,n_k,p_k,q_k}$ for $k\leq 2M$. $T'$ is
finite. Assume that $T'\subseteq I\times \Z$ where $I$ is finite.  Let
$$ H=\max\{p_kq_k+n_k+2w_k\,:\, k\leq 2M\}. $$ Then $H$ is the maximum
height of the grid graphs $G^i_{w_k,n_k,p_k,q_k}$ for $k\leq 2M$. We
have that $T'\subseteq I\times [-M, H+M]$. Thus for any $g\in T'$,
$g+(0,H+2M)\in S'\subseteq S$.

Suppose $g\not\in S'\cup T'$, i.e., $g$ is within distance $M$ to a
copy of $G^i_{w_k,n_k,p_k,q_k}$ for some $k>2M$. Note that this $k$ is
unique, since the distance between this copy of
$G^i_{w_k,n_k,p_k,q_k}$ and the neighboring copies is greater than
$2M$. We consider several cases. Let $p'_k=p_k-n_k-2w_k$, $q'_k=q_k-n_k-2w_k$.
Note that $p'_k,q'_k \geq n_k+2w_k$.

Case 1: $n_k+2w_k>2M$. Since $n_k+2w_k$ is the smallest dimension of
any block within the copy of $G^i_{w_k,n_k,p_k,q_k}$, any interior
point of a block that is of distance $M$ away from its boundary is a
point of $S$. It follows that $g$ is within distance $2M$ of such a
point, i.e., $g$ is within distance $2M$ of a point of $S$.

Case 2: $n_k+2w_k\leq 2M<\min\{p'_k, q'_k\}$. In this case all the
labeled blocks within the copy of $G^i_{w_k,n_k,p_k,q_k}$ has one of
its dimensions $\leq 2M$, but the interior unlabeled block has
dimensions $>2M$, and any interior point of this block that is of
distance $M$ away from its boundary is a point of $S$. It follows that
$g$ is within distance $4M$ of such a point.

Case 3: $\max\{p'_k, q'_k\}\leq 2M$. In this case, if
$G^i_{w_k,n_k,p_k,q_k}$ is not correspondent to a long vertical tile,
$g$ is within a vertical distance of $4M$ to a point of $S'\subseteq
S$. If $G^i_{w_k,n_k,p_k,q_k}$ is correspondent to a long vertical
tile, then $g$ is within a horizontal distance of $4M$ to a point of
$S'$.

Case 4: $\min\{p'_k, q'_k\}\leq 2M< \max\{p'_k, q'_k\}$. This is broken up
into two subcases, depending on either $p'_k<q'_k$ or $q'_k<p'_k$.

Subcase 4.1: $p'_k\leq 2M< q'_k$. In this subcase, if
$G^i_{w_k,n_k,p_k,q_k}$ is not correspondent to the tile $\Tdbdb$, then
similar to Case 3, $g$ is within either vertical or horizontal
distance of $4M$ to a point of $S'$. If $G^i_{w_k,n_k,p_k,q_k}$ is
correspondent to the tile $\Tdbdb$, then similar to Case 2, $g$ is
within distance $4M$ of a point of $S$ which is an interior point of
the unlabeled interior block.

Subcase 4.2: $q'_k\leq 2M< p'_k$. This is opposite to Subcase 4.1. In
this subcase, if $G^i_{w_k,n_k,p_k,q_k}$ is not correspondent to one of the
tiles $\Tdbdb$, $\Tcbcb$, or $\Tdada$, then similar to Case 2, $g$ is within distance $4M$ of
a point of $S$ which is an interior point of the unlabeled interior
block. If $G^i_{w_k,n_k,p_k,q_k}$ is correspondent to one of these tiles,
then similar to Case 3, $g$ is within distance $4M$ to a
point of $S'$.

This completes the proof that $S$ is syndetic, and hence $x$ is
hyper-aperiodic.
\end{proof}

\begin{proof}[Proof of (\ref{mta})$\Rightarrow$(\ref{mtc}) of Theorem~\ref{thm:tilethm}]
Let $Y\subseteq \bsft^{\Z^2}$ be a subshift of finite type described by $(\bsft; p_1,\dots, p_m)$. Let $f\colon F(2^{\Z^2})\to Y$ be a continuous,
equivariant map. Fix $n$ greater or equal to the width of $Y$.
We need to show that for all large enough $p,q$
 there is a $g \colon \gnpq \to \bsft$ which respects $Y$.

Let $x \in F(2^{\Z^2})$ be the hyper-aperiodic element produced in
Lemma~\ref{lem:xhyp}. Let $K=\ocl{[x]}$. Since $x$ is hyper-aperiodic,
$K \subseteq F(2^{\Z^2})$, and $K$ is compact. Thus, $f$ is defined on $K$.
Define $f_0 : K \rightarrow \bsft$ by $f_0(z) = f(z)(0, 0)$.
Since $f_0$ is continuous and $K$ is compact, there is
a $w\in\Z^+$ such that for any $z \in K$, the value of $f_0(z)$ is determined by $z \res [-w,w]^2$. In particular, this holds for all $z$ in the orbit of $x$.
So, for any $(a,b)\in \Z^2$, $f_0((a,b)\cdot x)$ is determined by the values of $x$ on the $(2w+1)\times (2w+1)$ rectangle centered about $(a, b)$.

Now consider arbitrary $p, q>2n+4w$. Fix a $k$
such that $p_k=p$, $q_k=q$, $n_k=n$, and $w_k=w$. Consider the grid graphs $G^i_{w_k,n_k,p_k,q_k}=G^i_{w,n,p,q}$
along with the ``unpadded'' graph $G^i_{n,p,q}$ induced by $G^i_{w_k,n,p,q}$ (the region within the dotted lines in Figure~\ref{fig:negative}).
We first define $g' \colon G_{n,p,q}\to \bsft$ as follows.
If $(a,b) \in \Z^2$ is in a copy of $G^i_{n,p,q}$ within a copy of $G^i_{w_k, n, p, q}$, we let $g'(a,b)=f_0((a,b)\cdot x)$.
From the construction of $x$, it is easy to verify the following claims by inspection:
\begin{itemize}
\item If $(a,b)$ and $(a', b')$ are in a copy of $G^i_{n,p,q}$ within a copy of  $G^i_{w_k, n, p, q}$, and if $(a,b)$, $(a',b')$ are in the same positions
in labeled blocks with the same labels,
then $f_0((a,b)\cdot x)=f_0((a',b')\cdot x)$.
\item If $(a, b)$ is in a copy of $G^i_{n, p, q}$ within a copy of $G^i_{w_k, n, p, q}$ and $(a',b')$ is in a copy of $G^{i'}_{n, p, q}$ within a copy of $G^{i'}_{w_k, n, p, q}$, and if $(a,b)$, $(a',b')$ are in the same positions
in labeled blocks with the same labels,
then $f_0((a,b)\cdot x)=f_0((a',b')\cdot x)$.
\end{itemize}

\begin{figure}[htbp]
\begin{tikzpicture}
\begin{scope}[scale=2,yscale=-1,shift={(-5,2)}]  
  \foreach \loc in {(0,0),(2.2,0),(4.0,0),(0,1.8),(1.8,1.8),(4.0,1.8)} {
    \node[XBoxStyle,scale=2,dotted] at \loc {};
    \node[BoxStyle,scale=2,minimum width=1.2cm, minimum height=1.2cm, shift={(-.1,.1)}] at \loc {};
    \node[shift={(1,-1)}] at \loc {$R'_\times$};
\node[BoxStyle,scale=0.3,minimum width=1.2cm, minimum height=1.2cm, shift={(-.3,.5)}] at \loc {};

  }
  \foreach \loc in {(0,1),(4.0,1)} {
    \node[ABoxStyle,scale=2,dotted] at \loc {};
    \node[BoxStyle,scale=2,minimum width=1.2cm, minimum height=0.6cm, shift={(-.1,-.1)}] at \loc {};
    \node[shift={(1,-.8)}] at \loc {$R'_a$};
  }
  \foreach \loc in {(3.2,0),(1,1.8)} {
    \node[CBoxStyle,scale=2,dotted] at \loc {};
    \node[BoxStyle,scale=2,minimum width=0.6cm, minimum height=1.2cm, shift={(.1,.1)}] at \loc {};
    \node[shift={(.8,-1)}] at \loc {$R'_c$};
  }
  \foreach \loc in {(1,0),(2.8,1.8)} {
    \node[DBoxStyle,scale=2,dotted] at \loc {};
    \node[BoxStyle,scale=2,minimum width=1.0cm, minimum height=1.2cm, shift={(.1,.1)}] at \loc {};
    \node[shift={(1.2,-1)}] at \loc {$R'_d$};
  }
\end{scope}
\end{tikzpicture}
\caption{Verifying that if $(a, b)$ and $(a', b')$ are in the same positions in a labeled block of $G_{w_k,n,p,q}$, then $f((a,b)\cdot x)=f((a',b')\cdot x)$.}
\label{fig:negative}
\end{figure}

In fact, suppose, as an example, that $(a,b)$, $(a',b')$ are both in blocks with the label $R_\times$.
In this case, each of the $R_\times$ blocks is in a strictly larger
block $R'_\times$ of $G_{w_k,n,p,q}$ (see Figure~\ref{fig:negative}).
Since $w_k=w$, the $(2w+1) \times (2w+1)$ rectangle centered at $(a,b)$ lies
entirely within the block $R'_\times$, and likewise for $(a',b')$.
Since $x$ is identical on the various $R'_\times$ blocks, and the $(2w+1) \times (2w+1)$ rectangle
determines the value of $f_0$, if follows that $f_0((a,b)\cdot x)=f_0((a',b')\cdot x)$,
and so $g'(a,b)=g'(a',b')$.

Suppose next that $(a,b)$, $(a',b')$ are in blocks of $G_{n,p,q}$ with label $R_a$ (the other cases are identical). In this case,
the $R_a$ block containing $(a,b)$ is not contained in the corresponding
$R'_a$ block of $G_{w_k,n,p,q}$, as the $R'_a$ block is smaller in one direction by $2w_k$
than the $R_a$ block. However, every time an $R_a$ block appears in a $G^i_{n,p,q}$,
the block immediately above and below have the label $R_\times$.
Thus, for an $(a,b)$ in an $R_a$ block, the $(2w+1) \times (2w+1)$ rectangle about $(a,b)$
is contained in the union of two adjacent blocks with labels $R'_\times$ and $R'_a$.
Since $x$ is identical on the various $R'_\times$ blocks, and also
on the various $R'_a$ blocks, it again follows that $f_0((a,b)\cdot x)=f_0((a',b')\cdot x)$,
and so $g'(a,b)=g'(a',b')$.

In summary, if $(a, b)$ and $(a', b')$ are vertices in $G_{n,p,q}$ such that they are in the same
positions in blocks with the same label, then $g'(a, b)=g'(a', b')$. It follows that $g'$ induces a map
$g \colon \gnpq \to \bsft$. To see that $g$ respects $Y$, consider a forbidden pattern $p_h$ where $1\leq h\leq m$, and assume
$\dom(p_h)=[0,a)\times [0,b)$. By our assumption $a, b\leq n+1$.
Let $\varphi \colon \dom(p_h) \to \gnpq$ be a $\Z^2$-homomorphism. From Lemma~\ref{lem:pbg},
there is a $\Z^2$-homomorphism $\psi \colon \dom(p_h) \to G^i_{n,p,q}$
for some $1\leq i\leq 12$ such that $g\circ \varphi= g' \circ \psi$.
We need to see that $g\circ \varphi= g' \circ \psi$ is not equal to $p_h$. From the definition of $g'$, we have that
$$g'\circ \psi (a', b')=f_0(\psi(a', b')\cdot x)=f(x)(\psi(a', b'))$$ for all $a'\in [0,a)$ and $b'\in [0,b)$.
Since $f$ maps into $Y$, $f(x)\res \psi(\dom(p_h))$ is not equal to $p_h$, and hence $g'\circ \psi$ is not equal to $p_h$.
\end{proof}

\section{Proof of the positive direction} \label{subsec:postile}

In this section we show the (\ref{mtb})$\Rightarrow$(\ref{mta}) direction
of Theorem~\ref{thm:tilethm}. The proof
will follow easily from the main technical result below. In the following the finite graph $\gnpq$ is
viewed as having the discrete topology.

\begin{thmn} \label{thm:mtr}
Let $n <p,q$ with $\gcd(p,q)=1$. Then there is a continuous $\Z^2$-homomorphism
$\varphi \colon F(2^{\Z^2}) \to \gnpq$.
\end{thmn}

\begin{proof}
Fix $n<p,q$ with $\gcd(p,q)=1$. We will need the following lemma which is a simple special case of the
``orthogonal marker construction'' of \cite{gao_countable_2015} (see the proof of Theorem 3.1 in \cite{gao_countable_2015}).
For $s=(u,v)\in \Z^2$ let $\|s\|_1=|u|+|v|$ be the $\ell_1$-norm. For $x, y\in F(2^{\Z^2})$ which are equivalent, define $\rho_1(x,y)=\|s\|_1$, where $s\in \Z^2$ is the unique element such that $s\cdot x=y$.

\begin{lemn} \label{lem:soml}
Let $d>1$ be an integer.  Then there is a clopen subequivalence relation $E_d$
of $F(2^{\Z^2})$ satisfying:

\begin{enumerate}
\item \label{soml_a}
Each equivalence class of $E_d$ is rectangular, that is, isomorphic to a grid
subgraph of $\Z^2$.
\item \label{soml_b}
If $x$, $y$ are $F(2^{\Z^2})$ equivalent and both are corner points of their respective
$E_d$ classes, then either $\rho_1(x,y)\leq 1$ or $\rho_1(x,y)>d$.
\end{enumerate}

\end{lemn}

Fix $d$ large compared to $p,q$, say $d>p^{10}q^{10}$, and fix the clopen
subequivalence relation $E_d$ as in Lemma~\ref{lem:soml}. We say $x \in F(2^{\Z^2})$
is a {\em corner point} if it is a corner point of the rectangular grid graph isomorphic
to its $E_d$ class. Since the rectangular $E_d$ classes partition each $F(2^{\Z^2})$
class, it is easy to see that the corner points occur in groups of $2$ points,
such that the points in a group are within $\rho_1$ distance $1$ of each other.
Consider the set $C\subseteq F(2^{\Z^2})$ of {\em canonical corner points}, which
are the lower-left corner points of the rectangular grid-graph isomorphic
to its $E_d$ class. Each corner point occurs in a group of $4$ points, which we call
a {\em corner group}, consisting of a $2 \times 2$ square, of which two points are corner points and
the other two are boundary points of an $E_d$ class.
Figure~\ref{fig:pa} illustrates the possible arrangements involving
corner points, with the canonical corner points and corner groups  shown.

\begin{figure}[htbp]
\centering

\subfloat[][]
\centering
{
\begin{tikzpicture}[scale=0.03]
\begin{scope}
\draw (-30,0) to (30,0);
\draw (-30,5) to (0,5) to (0,30);
\draw (5,30) to (5,5) to (30,5);
\draw[fill,radius=1,black] (5,5) circle;
\draw[fill,radius=1,green] (0,0) circle;
\draw[fill,radius=1,green] (0,5) circle;
\draw[fill,radius=1,green] (5,0) circle;

\end{scope}
\end{tikzpicture}
}
\qquad
\subfloat[][]
\centering
{
\begin{tikzpicture}[scale=0.03]
\begin{scope}
\draw (0,-30) to (0,30);
\draw (5,30) to (5,5) to (30,5);
\draw (5,-30) to (5,0) to (30,0);
\draw[fill,radius=1,black] (5,5) circle;
\draw[fill,radius=1,green] (0,0) circle;
\draw[fill,radius=1,green] (0,5) circle;
\draw[fill,radius=1,green] (5,0) circle;

\end{scope}
\end{tikzpicture}
}
\qquad
\subfloat[][]
\centering
{
\begin{tikzpicture}[scale=0.03]
\begin{scope}
\draw (-30,0) to (0,0) to (0,-30);
\draw (5,-30) to (5,0) to (30,0);;
\draw (-30,5) to (30,5);;
\draw[fill,radius=1,green] (5,5) circle;
\draw[fill,radius=1,green] (0,0) circle;
\draw[fill,radius=1,green] (0,5) circle;
\draw[fill,radius=1,green] (5,0) circle;

\end{scope}
\end{tikzpicture}
}
\qquad
\subfloat[][]
\centering
{
\begin{tikzpicture}[scale=0.03]
\begin{scope}
\draw (-30,0) to (0,0) to (0,-30);
\draw (-30,5) to (0,5) to (0,30);
\draw (5,-30) to (5,30);
\draw[fill,radius=1,green] (5,5) circle;
\draw[fill,radius=1,green] (0,0) circle;
\draw[fill,radius=1,green] (0,5) circle;
\draw[fill,radius=1,green] (5,0) circle;

\end{scope}
\end{tikzpicture}
}

\caption{Arrangements involving corner points. The canonical corner points are shown in black.}
\label{fig:pa}
\end{figure}

For each $x \in F(2^{\Z^2})$, consider a rectangular $E_d$ class $R \subseteq [x]$. Say
$R$ is an $a \times b$ rectangle with lower-left corner at $y \in [x]$.
By the {\em enlargement} $R'$ of $R$ we mean the subset of $[x]$
isomorphic to the rectangular grid graph of dimensions $(a+1)\times (b+1)$
with lower-left corner point at $y$. Thus, $R'$ is obtained from $R$
by extending its right and top edges by one. The enlarged rectangles
are no longer disjoint, but have overlapping boundaries.

\begin{defnn} \label{def:scaf}
The {\em scaffolding} $S_x\subseteq [x]$, for $x \in F(2^{\Z^2})$,
is the set of points in the boundaries of the enlarged $E_d$ classes in $[x]$.
\end{defnn}

By considering the cases shown in Figure~\ref{fig:pa}, it is easy to see that the
scaffolding can also be obtained in the following manner. From each canonical
corner-point $y \in [x]$, draw a horizontal line to the right, and a vertical line
in the upwards direction, up to and including the first point which is not $E_d$
equivalent to $y$. The scaffolding is the set of points in $[x]$ which lie
on these lines. By definition, every canonical corner point is in the scaffolding.
The Cayley graph structure on $\Z^2$ induces a graph on the scaffolding. Each
point of the scaffolding has degree $2$ or $3$, with the points of degree $3$
being the upper-right points of the corner groups. The connected components
of $[x]-S_x$  are the $E_d$ classes minus the left column and bottom row of each
$E_d$ class. Figure~\ref{fig:scaffolding} gives a local view of the scaffolding in relation to the corner groups.

\begin{figure}[htbp]
\centering

\subfloat[][]
\centering
{
\begin{tikzpicture}[scale=0.03]
\begin{scope}
\draw (-30,0) to (30,0);
\draw (-30,5) to (0,5) to (0,30);
\draw (5,30) to (5,5) to (30,5);
\draw[line width=0.05cm,orange] (-30,5) to (30,5);
\draw[line width=0.05cm,orange] (5,5) to (5,30);
\draw[fill,radius=1,black] (5,5) circle;
\draw[fill,radius=1,green] (0,0) circle;
\draw[fill,radius=1,green] (0,5) circle;
\draw[fill,radius=1,green] (5,0) circle;
\end{scope}
\end{tikzpicture}
}
\qquad
\subfloat[][]
\centering
{
\begin{tikzpicture}[scale=0.03]
\begin{scope}
\draw (0,-30) to (0,30);
\draw (5,30) to (5,5) to (30,5);
\draw (5,-30) to (5,0) to (30,0);
\draw[line width=0.05cm,orange] (5, -30) to (5, 30);
\draw[line width=0.05cm,orange] (5,5) to (30,5);
\draw[fill,radius=1,black] (5,5) circle;
\draw[fill,radius=1,green] (0,0) circle;
\draw[fill,radius=1,green] (0,5) circle;
\draw[fill,radius=1,green] (5,0) circle;

\end{scope}
\end{tikzpicture}
}
\qquad
\subfloat[][]
\centering
{
\begin{tikzpicture}[scale=0.03]
\begin{scope}
\draw (-30,0) to (0,0) to (0,-30);
\draw (5,-30) to (5,0) to (30,0);;
\draw (-30,5) to (30,5);;
\draw[line width=0.05cm,orange] (-30,5) to (30,5);
\draw[line width=0.05cm,orange] (5,-30) to (5,5);
\draw[fill,radius=1,green] (5,5) circle;
\draw[fill,radius=1,green] (0,0) circle;
\draw[fill,radius=1,green] (0,5) circle;
\draw[fill,radius=1,green] (5,0) circle;

\end{scope}
\end{tikzpicture}
}
\qquad
\subfloat[][]
\centering
{
\begin{tikzpicture}[scale=0.03]
\begin{scope}
\draw (-30,0) to (0,0) to (0,-30);
\draw (-30,5) to (0,5) to (0,30);
\draw (5,-30) to (5,30);
\draw[line width=0.05cm,orange] (-30,5) to (5,5);
\draw[line width=0.05cm,orange] (5,-30) to (5,30);
\draw[fill,radius=1,green] (5,5) circle;
\draw[fill,radius=1,green] (0,0) circle;
\draw[fill,radius=1,green] (0,5) circle;
\draw[fill,radius=1,green] (5,0) circle;

\end{scope}
\end{tikzpicture}
}

\caption{A local view of the scaffolding $S_x$. Points in $S_x$ are in orange.}
\label{fig:scaffolding}
\end{figure}

According to Lemmas~\ref{zghom} and \ref{zghomconverse}, defining a continuous $\Z^2$-homomorphism $\varphi
\colon F(2^{\Z^2})\to \gnpq$ is the equivalent to describing a tiling $\{ C_k\}$
of each equivalence class
$[x]$ by the grid graphs $G^i_{n,p,q}$ (as in Lemma~\ref{zghomconverse}) which is clopen in the sense that
for each $i \in \{1,\dots,12\}$, the set of $x \in F(2^{\Z^2})$ which are lower-left corners
of a $C_k$ isomorphic to $G^i_{n,p,q}$ is clopen in $F(2^{\Z^2})$.

Let $x \in F(2^{\Z^2})$ and we define the tiling of $[x]$ by the grid graphs
$G^i_{n,p,q}$. First we use the scaffolding $S_x$ as follows. For each point $y \in [x]$
which is of degree $3$ in $S_x$, we put down a copy of the $R_\times$ block
``centered" at $y$, meaning we place the lower-left corner of the $R_\times$ block at $(-\frac{n}{2},-\frac{n}{2}) \cdot y$ if $n$ is even and at $(-\frac{n-1}{2}, -\frac{n-1}{2}) \cdot y$ if $n$ is odd.

Second, we tile the scaffolding between the $R_\times$ copies around the degree
three points. Suppose $y$ and $z$ in $[x]$ are degree three points of $S_x$
such that there is a horizontal or vertical segment of $S_x$ connecting $y$ and $z$ with no degree three points of $S_x$ on this segment other than $y$ and $z$.
Suppose that the segment of $S_x$ connecting $y$ and $z$ is vertical. We then tile the
vertical column between the two copies of $R_\times$ around $y$ and $z$ by alternating
copies of an $R_\times$ tile with a $R_a$ or $R_b$ tile. Recall that $\rho(y,z)
>p^{10}q^{10}$. Since $\gcd(p,q)=1$ and
\begin{align*}
\mathrm{height}(R_\times)+\mathrm{height}(R_a)&=p, \\
\mathrm{height}(R_\times)+\mathrm{height}(R_b)&=q,
\end{align*}
it is possible to tile the vertical column
in this manner. Furthermore, we require that the copies of $R_b$ are sparse
in the following sense: the distance between any two copies of $R_b$ block is
at least $p^5q^5$, and the distance from any $R_b$ copy to the starting and ending $R_\times$
blocks is also at least $p^5q^5$. This is possible since at most $p$ many
of the $R_b$ blocks are needed for the tiling, because a consideration of sizes gives that every $p$ many $R_b$ blocks can be replaced by $q$ many $R_a$ blocks. Similarly, if the segment of $S_x$ connecting $y$ and $z$ is horizontal, we tile the horizontal regions between
them in the same manner, using $R_\times$ alternating with
$R_c$ or $R_d$ blocks, and have the same sparsity condition on the $R_d$.
Figure~\ref{fig:scaftile} illustrates this step of the construction.

\pgfmathsetmacro{\xboxsize}{0.10cm}
\pgfmathsetmacro{\pqscale}{0.08cm}
\begin{figure}[htbp]
\begin{tikzpicture}[pqscaling]

\foreach \i/\j in {
    19/0,20/0,21/0,
    22/1,23/1,24/1,
    25/2,26/2,27/2,28/2,29/2,30/2,31/2,32/2,33/2,34/2,35/2,
    36/3,37/3,38/3,
    39/4,40/4,41/4,
    42/5,43/5,44/5,45/5,46/5,47/5,48/5,49/5,50/5,51/5,
    52/6,53/6
    } {
    \pic at (\i*\pnum + \j*\qnum, 0) {cFrame};
    \pic at (\i*\pnum + \j*\qnum, 0) {xFrame};
}
\foreach \i/\j in {22/0,25/1,36/2,39/3,42/4,52/5} {
    \pic at  (\i*\pnum + \j*\qnum, 0) {dFrame};
    \pic at  (\i*\pnum + \j*\qnum, 0) {xFrame};
    \node[scale=0.75] at (\i*\pnum + \j*\qnum+4, 3) {$R_d$};
}

\foreach \i/\j in {
    0/0,1/0,2/0,3/0,4/0,
    5/1,6/1
    } {
    \pic at (30*\pnum + 2*\qnum, \i*\pnum +         \j*\qnum) {aFrame};
    \pic at (30*\pnum + 2*\qnum, \i*\pnum + \pnum + \j*\qnum) {xFrame};
}
\pic at (30*\pnum + 2*\qnum, 5*\pnum) {bFrame};
\pic at (30*\pnum + 2*\qnum, 5*\pnum+\qnum) {xFrame};
\node[scale=0.75] at (30*\pnum + 2*\qnum+4, 5*\pnum+4) {$R_b$};

\foreach \i/\j in {
    0/0,1/0,2/0,3/0,4/0,
    5/1,6/1
    } {
    \pic at (47*\pnum + 5*\qnum, -\i*\pnum - \pnum - \j*\qnum) {aFrame};
    \pic at (47*\pnum + 5*\qnum, -\i*\pnum - \pnum - \j*\qnum) {xFrame};
}
\pic at (47*\pnum + 5*\qnum, -5*\pnum-\qnum) {bFrame};
\pic at (47*\pnum + 5*\qnum, -5*\pnum-\qnum) {xFrame};
\node[scale=0.75] at (47*\pnum + 5*\qnum+4, -5*\pnum-\qnum+4) {$R_b$};

\end{tikzpicture}
\caption[Edge boundaries for marker regions.]{Edge boundaries for marker regions.
$R_\times$ blocks alternate with $R_\ast$ blocks with $R_\ast\in\{R_a,R_b,R_c,R_d\}$.
Instances of $R_b$ and $R_d$ are labeled.  This figure differs from the actual construction
in that instances of $R_b$ and $R_d$ in the actual construction are much less
common than depicted: we require that no two copies of $R_b$ or $R_d$ blocks be within
distance $p^5q^5$ of one another.}
\label{fig:scaftile}
\end{figure} 

Let $S'_x \subseteq [x]$ denote the points in $[x]$ which have been covered by
tiles at this point in the construction. Note that $[x]-S'_x$ is a collection of
pairwise disjoint rectangular grid graphs, surrounded by alternating copies of
$R_\times$ and $R_a / R_b$ on the left and right edges, and alternating
$R_\times$ with $R_c/R_d$ on the top and bottom edges. The corners of the boundary tiling are
$R_\times$ copies, and the edges satisfy the sparsity conditions on $R_b$ and $R_d$.
The remainder of the construction is entirely finitary. We show that any such rectangular
region (including the boundary tiles) can be tiled with the grid graphs $G^i_{n,p,q}$
in a manner consistent with the boundary tiles.

Let $\CM$ denote a finite rectangular grid graph whose boundary has been tiled with
sub-blocks $R_\times$ and $R_*\in \{R_a,R_b,R_c,R_d\}$ as described above. We will describe three tiling algorithms we will use in combination to fill the interior of $\CM$.
Each algorithm is used to fill a rectangular grid graph with certain part of the boundary already specified.

The first algorithm, which we call algorithm (I), deals with a rectangular grid graph $\mathcal{N}$ with its boundary already tiled with alternating $R_\times$ blocks and $R_*$ blocks, where $R_*\in\{R_a, R_b, R_c, R_d\}$, such that only one kind of labeled blocks appear in the left and right columns and only one kind of blocks appear in the top and bottom rows. For instance, suppose $\mathcal{N}$ has a boundary devoid of copies of $R_b$ or $R_d$.
In this case we simply fill the entire grid graph with overlapping copies of tile~$\Gcaca$.
This is depicted in Figure~\ref{fig:algorithmI}. The other instances of algorithm (I) are similar, with the $\Gcaca$ tiles replaced by the appropriate tiles $\Gdada$, $\Gcbcb$, or $\Gdbdb$.

\pgfmathsetmacro{\xboxsize}{0.3cm}
\pgfmathsetmacro{\pqscale}{0.5cm}
\begin{figure}[htbp]
\begin{tikzpicture}[pqscaling]
\foreach \k in {0,1,2} {
    \foreach \i/\j in {0/0,1/0,2/1,3/1} {
        \pic at (\i*\pnum + \j*\pnum, \k*\pnum) {TcacaLabel};
    }
    \foreach \i/\j in {2/0} {
        \pic at (\i*\pnum + \j*\pnum, \k*\pnum) {TcacaLabel};
    }
    \pic at (4*\pnum + \pnum, \k*\pnum) {aFrame};
    \pic at (4*\pnum + \pnum, \k*\pnum) {xFrame};
}

\foreach \i/\j in {0/0,1/0,2/1,3/1} {
    \pic at (\i*\pnum + \j*\pnum, 3*\pnum) {cFrame};
    \pic at (\i*\pnum + \j*\pnum, 3*\pnum) {xFrame};
}
\foreach \i/\j in {2/0} {
    \pic at (\i*\pnum + \j*\pnum, 3*\pnum) {cFrame};
    \pic at (\i*\pnum + \j*\pnum, 3*\pnum) {xFrame};
}

\pic at (4*\pnum + \pnum, 3*\pnum) {xFrame};

\end{tikzpicture}
\caption{An instance of algorithm (I) with boundaries consisting of only $R_\times, R_a$ and $R_c$ blocks.}
\label{fig:algorithmI}
\end{figure} 

The next two algorithms deal with rectangular grid graphs where at least one of the $R_*$ blocks are missing. For the sake of definiteness, let us consider a rectangular grid graph $\mathcal{N}$ where $R_b$ blocks are missing, i.e., the left and right boundaries are filled with $R_\times$ blocks alternating with $R_a$ blocks. Moreover, we assume that only the top boundary is specified. In this case, algorithm (I) can still be used to achieve a propagation of the arrangement of the top boundary down through the grid graph. Figure~\ref{fig:main-positive-downward-propagation} illustrates this situation.

\pgfmathsetmacro{\xboxsize}{0.3cm}
\pgfmathsetmacro{\pqscale}{0.5cm}
\begin{figure}[htbp]
\begin{tikzpicture}[pqscaling]
\foreach \k in {0,1,2} {
    \foreach \i/\j in {0/0,1/0,2/1,3/1} {
        \pic at (\i*\pnum + \j*\qnum, \k*\pnum) {TcacaLabel};
    }
    \foreach \i/\j in {2/0} {
        \pic at (\i*\pnum + \j*\qnum, \k*\pnum) {TdadaLabel};
    }
    \pic at (4*\pnum + \qnum, \k*\pnum) {aFrame};
    \pic at (4*\pnum + \qnum, \k*\pnum) {xFrame};
}

\foreach \i/\j in {0/0,1/0,2/1,3/1} {
    \pic at (\i*\pnum + \j*\qnum, 3*\pnum) {cFrame};
    \pic at (\i*\pnum + \j*\qnum, 3*\pnum) {xFrame};
}
\foreach \i/\j in {2/0} {
    \pic at (\i*\pnum + \j*\qnum, 3*\pnum) {dFrame};
    \pic at (\i*\pnum + \j*\qnum, 3*\pnum) {xFrame};
}

\pic at (4*\pnum + \qnum, 3*\pnum) {xFrame};

\end{tikzpicture}
\caption{Using only algorithm (I), we achieve a downward propagation of an instance of $R_d$ using
tile $\Gdada$ through a region mostly tiled with $\Gcaca$.}
\label{fig:main-positive-downward-propagation}
\end{figure} 

It is obvious that a similar algorithm can achieve horizontal propagation of the left boundary tiling if one of $R_c$ or $R_d$ is missing from the top and bottom boundary.

Our next algorithm, algorithm (II), seeks to propagate the top edge tiling downward while moving the positions of the
$R_d$ tiles to the left or right. To do this we need to use the $\Gdcadca$ tiles as shown in
Figure~\ref{fig:main-positive-diagonal-propagation}.

\pgfmathsetmacro{\xboxsize}{0.1 cm}
\pgfmathsetmacro{\pqscale}{0.3cm}
\begin{figure}[htbp]
\begin{tikzpicture}[pqscaling]
\foreach \i/\j in {-1/0,0/0,1/1,2/1,3/1,4/1,5/1} {
    \pic at (\i*\pnum + \j*\qnum,   \pnum) {Tcaca};
    \pic at (\i*\pnum + \j*\qnum, 2*\pnum) {xFrame};
    \pic at (\i*\pnum + \j*\qnum, 2*\pnum) {cFrame};
}
\pic at (\pnum,   \pnum) {TdadaLabel};
\pic at (\pnum, 2*\pnum) {xFrame};
\pic at (\pnum, 2*\pnum) {dFrame};

\foreach \i/\j in {-1/0,0/0,1/0,2/0,3/0,4/1,5/1} {
    \pic at (\i*\pnum + \j*\qnum, -3*\pnum) {Tcaca};
}
\pic at (4*\pnum, -3*\pnum) {TdadaLabel};

\foreach \k in {0,1,2} {
    \foreach \i in {-1,...,\k} {
        \pic at (\i*\pnum, -\k*\pnum) {Tcaca};
    }
    \foreach \i in {\k,...,3} {
        \pic at (\i*\pnum + 2*\pnum + \qnum, -\k*\pnum) {Tcaca};
    }

    \pic at (\k*\pnum + \pnum, -\k*\pnum) {TdcadcaLabel};
}

\pic at (6*\pnum + \qnum, 2*\pnum) {xFrame};
\foreach \i in {-3,...,1} {
    \pic at (6*\pnum + \qnum, \i*\pnum) {xFrame};
    \pic at (6*\pnum + \qnum, \i*\pnum) {aFrame};
}

\end{tikzpicture}
\caption{Algorithm (II) achieves ``diagonal'' propagation of an instance of $R_d$
using tile~$\Gdcadca$, through a region mostly tiled with $\Gcaca$ (labels omitted to reduce clutter).}
\label{fig:main-positive-diagonal-propagation}
\end{figure} 

A similar version of algorithm (II) can achieve diagonal propagation
of an instance of $R_b$ though a region mostly tiled with $\Gcaca$.

We next describe the final algorithm, algorithm (III). Again the algorithm has various versions.
Here we concentrate on the version where the rectangular grid graph $\mathcal{N}$ has
exactly $q+3$ many $R_a$ blocks on the left and right boundaries and no $R_b$ blocks,
but we assume that the top boundary of $\mathcal{N}$ has been tiled with $R_c$ and $R_d$
blocks with the scarcity condition on $R_d$. The objective of this algorithm is to tile $\mathcal{N}$
so that the bottom boundary will be first tiled with less than $p$ many $R_d$ blocks and
then with $R_c$ blocks for the rest of the boundary. Note that the number of $R_d$ blocks
is uniquely determined by the top boundary. Suppose the top boundary contains $j$ many $R_c$ blocks
and $k$ many $R_d$ blocks. Then the size of the top boundary is $jp+kq+n$.
The number of $R_d$ blocks on the bottom boundary is then projected to be
$$ k-p\left\lfloor\frac{k}{p}\right\rfloor, $$
while there would be
$$ j+q\left\lfloor\frac{k}{p}\right\rfloor $$
many copies of $R_c$ on the bottom boundary. The reader should consult Figure~\ref{fig:main-positive-coprime-lemma}
in the following description of algorithm (III).

\pgfmathsetmacro{\xboxsize}{0.05cm}
\pgfmathsetmacro{\pqscale}{0.065cm}
\begin{figure}
\begin{tikzpicture}[pqscaling]
  \draw [decorate,decoration={brace,amplitude=7pt}]
  (0,9*\pnum) -- (8*\pnum*\qnum + 2*\qnum+2*\pnum,9*\pnum) node [scale=0.8, black,midway, above, yshift=10pt]
        {$j$ copies of $R_c$ and $k$ copies of $R_d$};

  \draw [decorate,decoration={brace,amplitude=7pt}]
  (8*\pnum*\qnum + 2*\qnum + 2*\pnum,-1*\pnum) -- (2*\qnum,-1*\pnum) node [scale=0.8, black, midway, below, yshift=-10pt]
        {$j+q\floorfrac{k}{p}$ copies of $R_c$};
  \draw [decorate,decoration={brace,amplitude=7pt}]
  (2*\qnum,-1*\pnum) -- (0,-1*\pnum) node [scale=0.8, black, midway, below, yshift=-10pt,xshift=30pt]
        {$k-p\floorfrac{k}{p}$ copies of $R_d$};

\foreach \i/\j in {5/0,15/1,24/2,33/3,42/4} {
    \pic at (\i*\pnum + \j*\qnum, 8*\pnum) {dFrame};
    \pic at (\i*\pnum + \j*\qnum, 8*\pnum) {xFrame};
}
\foreach \i/\j in {
     0/0, 1/0, 2/0, 3/0, 4/0,
     5/1, 6/1, 7/1, 8/1, 9/1,10/1,11/1,12/1,13/1,14/1,
    15/2,16/2,17/2,18/2,19/2,20/2,21/2,22/2,23/2,
    24/3,25/3,26/3,27/3,28/3,29/3,30/3,31/3,32/3,
    33/4,34/4,35/4,36/4,37/4,38/4,39/4,40/4,41/4,
    42/5,43/5,44/5,45/5,46/5,47/5,48/5
    } {
    \pic at (\i*\pnum + \j*\qnum, 8*\pnum) {cFrame};
    \pic at (\i*\pnum + \j*\qnum, 8*\pnum) {xFrame};
}

\foreach \i/\j in {5/0,24/2,33/3,42/4} {
    \pic at (\i*\pnum + \j*\qnum, 7*\pnum) {Tdada};
}
\foreach \i/\j in {15/1} {
    \pic at (\i*\pnum + \j*\qnum, 7*\pnum) {Tdcadca};
}
\foreach \i/\j in {
     0/0, 1/0, 2/0, 3/0, 4/0,
     5/1, 6/1, 7/1, 8/1, 9/1,10/1,11/1,12/1,13/1,14/1,
    16/2,17/2,18/2,19/2,20/2,21/2,22/2,23/2,
    24/3,25/3,26/3,27/3,28/3,29/3,30/3,31/3,32/3,
    33/4,34/4,35/4,36/4,37/4,38/4,39/4,40/4,41/4,
    42/5,43/5,44/5,45/5,46/5,47/5,48/5
    } {
    \pic at (\i*\pnum + \j*\qnum, 7*\pnum) {Tcaca};
}

\foreach \i/\j in {5/0,24/2,33/3,42/4} {
    \pic at (\i*\pnum + \j*\qnum, 6*\pnum) {Tdada};
}
\foreach \i/\j in {16/1} {
    \pic at (\i*\pnum + \j*\qnum, 6*\pnum) {Tdcadca};
}
\foreach \i/\j in {
     0/0, 1/0, 2/0, 3/0, 4/0,
     5/1, 6/1, 7/1, 8/1, 9/1,10/1,11/1,12/1,13/1,14/1,15/1,
    17/2,18/2,19/2,20/2,21/2,22/2,23/2,
    24/3,25/3,26/3,27/3,28/3,29/3,30/3,31/3,32/3,
    33/4,34/4,35/4,36/4,37/4,38/4,39/4,40/4,41/4,
    42/5,43/5,44/5,45/5,46/5,47/5,48/5
    } {
    \pic at (\i*\pnum + \j*\qnum, 6*\pnum) {Tcaca};
}

\foreach \i/\j in {5/0,33/3,42/4} {
    \pic at (\i*\pnum + \j*\qnum, 5*\pnum) {Tdada};
}
\foreach \i/\j in {17/1,24/2} {
    \pic at (\i*\pnum + \j*\qnum, 5*\pnum) {Tdcadca};
}
\foreach \i/\j in {
     0/0, 1/0, 2/0, 3/0, 4/0,
     5/1, 6/1, 7/1, 8/1, 9/1,10/1,11/1,12/1,13/1,14/1,15/1,16/1,
    18/2,19/2,20/2,21/2,22/2,23/2,
    25/3,26/3,27/3,28/3,29/3,30/3,31/3,32/3,
    33/4,34/4,35/4,36/4,37/4,38/4,39/4,40/4,41/4,
    42/5,43/5,44/5,45/5,46/5,47/5,48/5
    } {
    \pic at (\i*\pnum + \j*\qnum, 5*\pnum) {Tcaca};
}

\foreach \i/\j in {5/0,33/3,42/4} {
    \pic at (\i*\pnum + \j*\qnum, 4*\pnum) {Tdada};
}
\foreach \i/\j in {18/1,25/2} {
    \pic at (\i*\pnum + \j*\qnum, 4*\pnum) {Tdcadca};
}
\foreach \i/\j in {
     0/0, 1/0, 2/0, 3/0, 4/0,
     5/1, 6/1, 7/1, 8/1, 9/1,10/1,11/1,12/1,13/1,14/1,15/1,16/1,17/1,
    19/2,20/2,21/2,22/2,23/2,24/2,
    26/3,27/3,28/3,29/3,30/3,31/3,32/3,
    33/4,34/4,35/4,36/4,37/4,38/4,39/4,40/4,41/4,
    42/5,43/5,44/5,45/5,46/5,47/5,48/5
    } {
    \pic at (\i*\pnum + \j*\qnum, 4*\pnum) {Tcaca};
}

\foreach \i/\j in {42/4} {
    \pic at (\i*\pnum + \j*\qnum, 3*\pnum) {Tdada};
}
\foreach \i/\j in {5/0,19/1,26/2,33/3} {
    \pic at (\i*\pnum + \j*\qnum, 3*\pnum) {Tdcadca};
}
\foreach \i/\j in {
     0/0, 1/0, 2/0, 3/0, 4/0,
     6/1, 7/1, 8/1, 9/1,10/1,11/1,12/1,13/1,14/1,15/1,16/1,17/1,18/1,
    20/2,21/2,22/2,23/2,24/2,25/2,
    27/3,28/3,29/3,30/3,31/3,32/3,
    34/4,35/4,36/4,37/4,38/4,39/4,40/4,41/4,
    42/5,43/5,44/5,45/5,46/5,47/5,48/5
    } {
    \pic at (\i*\pnum + \j*\qnum, 3*\pnum) {Tcaca};
}

\foreach \i/\j in {42/4} {
    \pic at (\i*\pnum + \j*\qnum, 2*\pnum) {Tdada};
}
\foreach \i/\j in {6/0,20/1,27/2,34/3} {
    \pic at (\i*\pnum + \j*\qnum, 2*\pnum) {Tdcadca};
}
\foreach \i/\j in {
     0/0, 1/0, 2/0, 3/0, 4/0, 5/0,
     7/1, 8/1, 9/1,10/1,11/1,12/1,13/1,14/1,15/1,16/1,17/1,18/1,19/1,
    21/2,22/2,23/2,24/2,25/2,26/2,
    28/3,29/3,30/3,31/3,32/3,33/3,
    35/4,36/4,37/4,38/4,39/4,40/4,41/4,
    42/5,43/5,44/5,45/5,46/5,47/5,48/5
    } {
    \pic at (\i*\pnum + \j*\qnum, 2*\pnum) {Tcaca};
}

\foreach \i in {3,10,14,18,22} {
    \pic at (\i*\qnum,\pnum) {Tdada};
}
\foreach \i/\j in {0/0,1/1,2/1,3/2,4/3,5/4,6/5} {
    \pic at (\i*\pnum*\qnum + \j*\qnum,\pnum) {Tcqadpa};
}

\foreach \i in {0,...,7} {
    \pic at (\i*\pnum*\qnum+2*\qnum,0) {Tdpacqa};
}
\foreach \i in {0,1} {
    \pic at (\i*\qnum,0) {Tdada};
}


\foreach \i in {0,...,7} {
    \pic at (56*\pnum+2*\qnum, \i*\pnum) {Tcaca};
    \pic at (57*\pnum+2*\qnum, \i*\pnum) {Tcaca};
    \pic at (58*\pnum+2*\qnum, \i*\pnum) {aFrame};
    \pic at (58*\pnum+2*\qnum, \i*\pnum) {xFrame};
}
\pic at (56*\pnum+2*\qnum, 8*\pnum) {xFrame};
\pic at (56*\pnum+2*\qnum, 8*\pnum) {cFrame};
\pic at (57*\pnum+2*\qnum, 8*\pnum) {cFrame};
\pic at (58*\pnum+2*\qnum, 8*\pnum) {xFrame};

\pic at (57*\pnum + 2*\qnum, 8*\pnum) {xFrame};

\end{tikzpicture}

\caption{An illustration of algorithm (III).  The smallest tiles are $\Gcaca$.
The snaking paths of slightly larger tiles are $\Gdada$ and $\Gdcadca$ as in
Figure~\ref{fig:main-positive-diagonal-propagation}.  The widest tiles are
$\Gcqadpa$ and $\Gdpacqa$, with the former on the row above the latter.}
\label{fig:main-positive-coprime-lemma}
\end{figure} 

We fill $\mathcal{N}$ entirely with $\Gcaca$, $\Gdada$, $\Gdcadca$, $\Gcqadpa$, and $\Gdpacqa$.
Each such tile is of the height of $\Gcaca$, namely $p+n$, so we must specify $q+1$ many overlapping rows
of tiles in the interior of $\mathcal{N}$, since the top and bottom boundaries have been specified.

Beginning with the top row of $\mathcal{N}$, we propagate the $\Gdada$ tiles downward using algorithms (I) and (II), and at each step
either place a $\Gdada$ tile immediately below a $\Gdada$ tile from the row above, or
offset it to the right by a $\Gdcadca$ tile as shown in Figure~\ref{fig:main-positive-coprime-lemma}.
We do this so that the bottom of the $q+1$st row of tiles consists of
groups of alternating $R_\times$ and $R_c$ blocks followed by a single $R_d$
block. We require that each group of alternating $R_\times$ and $R_c$ blocks has a number of $R_c$ blocks
which is a multiple of $q$, except for the last group of alternating sequence of $R_\times$ and $R_c$ blocks
which follows the last copy of $R_d$ in this row.
We can meet this requirement since there are $q+1$ rows as we move downward, and the distance between
copies of $R_d$ in the top row is greater than $pq$.

The last two rows of $\mathcal{N}$ are tiled as follows. For the next to last row (i.e., the $q+2$nd row)
of tiles, we use copies of $\Gcqadpa$ to convert each group of $q$ many alternating
blocks of $R_\times$ and $R_c$ at the bottom of the $q+1$st row into an alternating sequence
of $R_\times$ and $R_d$ blocks (with $p$ many $R_d$ blocks). For the last
alternating sequence (whose number of $R_c$ blocks is not necessarily a multiple of $q$),
we might have an alternating sequence of $R_\times$ and $R_c$ left over, with $<q$
many copies of $R_c$. Thus, the bottom of the $q+2$nd row consists of an alternating sequence
of $R_\times$ and $R_d$, followed by a final alternating sequence of $R_\times$ and $R_c$.

For the final row in the tiling or $\mathcal{N}$, we first copy down the final alternating
sequence of $R_\times$ and $R_c$ from the row above. Then, starting from the right
and working left, we use the $\Gdpacqa$ tiles to convert each block
of $R_\times/R_d$ alternations with $p$ many copies of $R_d$, into an
alternating sequence of $R_\times/R_c$ with $q$ many copies of $R_c$.
As we approach the left edge, we will have a number, less than $p$, of
$R_d$ blocks (alternating with $R_\times$) left over, which we copy down. This completes our description of algorithm (III) in this setup.

Intuitively, what the construction accomplished is to gather copies of $R_d$ blocks on the top and pass them to the left of the bottom boundary.
This construction is readily adapted to gather copies of $R_d$ on the bottom as well,
and gather copies of $R_b$ on the left and right.

We are now ready to describe the tiling of $\CM$. Recall that the top and bottom boundary of $\CM$ are tiled with alternating copies of $R_\times$ and $R_c$ or $R_d$ tiles, with the sparsity condition for $R_d$ tiles, and the left and right boundary of $\CM$ are tiled with alternating copies of $R_\times$ and $R_a$ or $R_b$ tiles, with the sparsity condition for $R_b$ tiles. The reader should consult Figure~\ref{fig:positive} in the following description of the tiling.

\begin{figure}[h]
\begin{tikzpicture}[scale=0.05]

\pgfmathsetmacro{\s}{0.03}

\pgfmathsetmacro{\n}{20}
\pgfmathsetmacro{\r}{0.05}

\draw (-100,0) to (100,0) to (100,200) to (-100,200) to (-100,0);
\draw (-100,2) to (100,2);
\draw (-100,198) to (100,198);
\draw (-98,0) to (-98,200);
\draw (98,0) to (98,200);

\draw (-100,20) to (100,20);
\draw (-100,22) to (100, 22);
\draw (-100,180) to (100,180);
\draw (-100,178) to (100,178);
\draw (-80,0) to (-80,200);
\draw (-78,0) to (-78,200);
\draw (80, 0) to (80,200);
\draw (78,0) to (78,200);

\draw (-50,20) to (-50, 180);
\draw (-48,20) to (-48, 180);
\draw (-80,44) to (80, 44);
\draw (-80,46) to (80,46);

\node at (0,189) {$\mathcal{R}_{\text{top}}$};
\node at (0,11) {$\mathcal{R}_{\text{bottom}}$};
\node at (-89, 100) {$\mathcal{R}_{\text{left}}$};
\node at (89, 100) {$\mathcal{R}_{\text{right}}$};

\node at (-105, 11) {$R_a$};
\node at (-105, 189) {$R_a$};
\node at (105, 11) {$R_a$};
\node at (105, 189) {$R_a$};
\node at (-89, -5) {$R_c$};
\node at (89, -5) {$R_c$};
\node at (-89, 205) {$R_c$};
\node at (89, 205) {$R_c$};
\node at (-64, 25) {$R_d$};
\node at (15, 25) {$R_c$};
\node at (-73, 32) {$R_b$};
\node at (-73, 112) {$R_a$};
\node at (-64, 173) {$R_d$};
\node at (15, 173) {$R_c$};
\node at (73, 112) {$R_a$};
\node at (73, 32) {$R_b$};

\end{tikzpicture}
\caption{Tiling of any finite rectangular grid graph satisfying sparsity conditions for $R_b$ and $R_d$.} \label{fig:positive}
\end{figure}

We first introduce two full columns and two full rows in the $\CM$ region. These are placed so that they define four corner subregions whose dimensions are exactly $(q+3)p+n$. Note that the sparsity conditions guarantee that the specified boundaries of the four
corner regions consists of only $R_a$ and $R_c$ blocks and no $R_b$ or $R_d$ blocks. In Figure~\ref{fig:positive} we put these labels next to the relevant part of the boundary to indicate this fact. The full columns and rows introduced also create four regions along the boundaries of $\CM$, and we denote them by $\mathcal{R}_{\text{top}}$, $\mathcal{R}_{\text{bottom}}$, $\mathcal{R}_{\text{left}}$, and $\mathcal{R}_{\text{right}}$, respectively. We now apply algorithm (I) to each of the corner regions and the relevant version of algorithm (III) to each of the regions $\mathcal{R}_{\text{top}}$, $\mathcal{R}_{\text{bottom}}$, $\mathcal{R}_{\text{left}}$, and $\mathcal{R}_{\text{right}}$. This gives rise to an inner rectangular grid graph whose boundaries are tiled with $R_a, R_b, R_c, R_d$ tiles in a fashion specified in Figure~\ref{fig:positive}. Note that the number of $R_b$ blocks and the number of $R_d$ blocks on these boundaries are completely determined by the dimensions of $\CM$, and therefore the numbers of $R_b$ blocks on the left and right parts of the boundaries of the inner grid graph are the same, and the numbers of $R_d$ blocks on the top and bottom parts of the boundaries of the inner grid graph are also the same. At this point we can introduce a column and a row of the inner grid graph and a complete tiling of the graph with algorithm (I).

This completes the construction of the tilings of the equivalence classes
$[x]$ for $x \in F(2^{\Z^2})$ by the $G^i_{n,p,q}$,  and by Lemma~\ref{zghomconverse}
thus completes the definition of $\varphi \colon F(2^{\Z^2})\to \gnpq$. Since
the subequivalence relation $E_d$ was clopen in $F(2^{\Z^2})$, the local nature of the
constructions of the $S_x$ and the tilings of the $\CM$ show that $\varphi$ is continuous.
This completes the proof of Theorem~\ref{thm:mtr}.
\end{proof}

\begin{proof}[Proof of (\ref{mtb})$\Rightarrow$(\ref{mta}) of Theorem~\ref{thm:tilethm}]
Let $Y\subseteq {\bsft^{\Z^2}}$ be a subshift of finite type, $n$ greater than or equal to the width of $Y$, $p, q>n$, and $g: \gnpq \to\bsft$ which respects $Y$. Let $\varphi: F(2^{\Z^2})\to \gnpq$ be the continuous $\Z^2$-homomorphism in Theorem~\ref{thm:mtr}.
Define $f: F(2^{\Z^2})\to Y$ by letting $f(x)(a, b)=g\circ \varphi((a,b)\cdot x)$ for all $x\in F(2^{\Z^2})$ and $(a, b)\in \Z^2$. Then $f$ is a continuous, equivariant map.
\end{proof}

\section{Continuous structurings} \label{sec:continuous_structuring}
In  this section we present a reformulation of the main theorem, Theorem~\ref{thm:tilethm},
in terms of {\em continuous structurings} (defined below in Definition~\ref{def:lstruc}) instead of subshifts of finite type.
This makes precise the notion of a continuous assignment of a first-order structure to
an equivalence class. We show that the existence of continuous structurings can be
reformulated as the existence of continuous, equivariant maps into  certain subshifts of
finite type and conversely. We then state a version of Theorem~\ref{thm:tilethm} in terms of continuous structurings.
This alternate viewpoint is useful as some problems are naturally presented
as continuous structuring questions.
In fact, many questions of the form ``Does there exist a continuous assignment of a certain type of structure
to the classes of $F(2^{\Z^n})$?'' can be directly rephrased
in the terminology of continuous structurings.

We next make these notions precise by
introducing the notions of a continuous structuring and $\lp_1$
formulas.

\begin{defnn} \label{defn:L}
Let $\CL=(F_1,\dots,F_n,F_{n+1},\dots, F_{2n}, R_1,\dots,R_k)$ be a finite  first-order language,
where $F_1,\dots F_n, F_{n+1}, \dots, F_{2n}$ are unary function symbols
and $R_1,\dots, R_k$ are relation symbols of arity $\alpha_1, \dots, \alpha_k$ respectively.
By a {\em $\Z^n$-structure}  for $\CL$ we mean a structure of the form
$$\fA=(A; F^{\fA}_1,\dots,F^{\fA}_n, F^{\fA}_{n+1}, \dots, F^{\fA}_{2n}, R^{\fA}_1,\dots,R^{\fA}_k)$$
where $A$ is a set with an action of $\Z^n$ that is transitive and free, and for $1\leq i\leq n$,
$$ F^{\fA}_i(a)=e_i\cdot a,  \ \ \ F^{\fA}_{n+i}(a)=(-e_i)\cdot a $$
for $a\in A$, and for $1\leq j\leq k$, $R^{\fA}_j\subseteq A^{\alpha_j}$.
\end{defnn}

The interpretations of the function symbols in a $\Z^n$-structure for $\CL$ are always fixed and are given by some transitive, free action of $\Z^n$. An important case is $A=\Z^n$, where the action is understood to be the addition, and we denote the standard interpretations of the function symbols by $f_1, \dots, f_n, f_{n+1}, \dots, f_{2n}$. Similarly, if $A=[x]$ is an equivalence class in $F(2^{\Z^n})$, we also denote the standard interpretation of the function symbols by $f_1, \dots, f_n, f_{n+1}, \dots, f_{2n}$.

We will adopt the following definition of a $\lp_1$ sentence. This comes from the standard definition of $\lp_1$ formulas in logic,
except we restrict the quantifiers to a single variable.

\begin{defnn} Let $\CL=(F_1,\dots,F_n,F_{n+1},\dots, F_{2n}, R_1,\dots,R_k)$ and $\varphi$ be a sentence in $\CL$. We say that
$\varphi$ is $\lp_1$ if it is of the form
$$\varphi = \forall v \, \psi(v),$$
where $\psi$ is a boolean combination
of atomic formulas in the language $\CL$.
\end{defnn}

Note that if $\CL=(F_1,\dots,F_n,F_{n+1},\dots, F_{2n}, R_1,\dots,R_k)$ with $R_j$ of arity $\alpha_j$ for $1\leq j\leq k$, the atomic formulas in the language of $\CL$ are of the form
$$ R_j(t_1, \dots, t_{\alpha_j}) $$
for some $1\leq j\leq k$, where $t_1, \dots, t_{\alpha_j}$ are terms in the language $\CL$, or of the form $t_1=t_2$, where $t_1$ and $t_2$ are terms. Each term in the language $\CL$ is of the form
$$ F_{i_1}\dots F_{i_{\ell}}(v) $$
for a variable $v$, where $1\leq i_1,\dots, i_{\ell}\leq 2n$.

\begin{defn}
Let $\CL=(F_1,\dots,F_n,F_{n+1},\dots, F_{2n}, R_1,\dots,R_k)$, $t$ be a term in $\CL$, and $\psi$ be a formula in $\CL$. The {\em depth} of $t$ is defined as the number of function symbols appearing in $t$ (counted with repetitions).
The {\em depth } of $\psi$ is defined as the maximum depth of terms appearing in $\psi$.
\end{defn}

We will also consider {\em partial $\Z^n$-structures} as defined in the following. These are not technically $\CL$-structures in the sense of first-order logic. Nevertheless, some instantiations of $\lp_1$ formulas can be interpreted in such structures and have a truth value.

\begin{defnn}\label{def:partialLstructure} Let $\CL=(F_1,\dots,F_n,F_{n+1},\dots, F_{2n}, R_1,\dots,R_k)$ be a language as in Definition~\ref{defn:L}. Let $\ell$ be a non-negative integer. An {\em $\ell$-structure} for $\CL$ is a tuple
$$ \fA=(D_\ell; p_1, \dots, p_n, p_{n+1},\dots, p_{2n}, R^{\fA}_1, \dots, R^{\fA}_k) $$
where $D_\ell=[-\ell,\ell]^n$, $p_1, \dots, p_n, p_{n+1},\dots,p_{2n}$ are partial functions on $D_\ell$ defined as
$p_i(a)=a+e_i$ whenever $a, a+e_i\in D_\ell$, and $p_{n+i}(a)=a-e_i$ whenever $a, a-e_i\in D_\ell$
for $1\leq i\leq n$, and $R^{\fA}_j \subseteq (D_\ell)^{\alpha_j}$ for all $1\leq j\leq k$.
\end{defnn}

Since the $p_i$ are partial, the $\ell$-structure just defined is not technically an $\CL$-structure in the sense of first-order logic. Nevertheless, if $t(v)$ is a term (with only one variable $v$) in the language $\CL$ of depth $d$ and $a\in D_\ell$ is of distance at least $d$ from the boundary of $D_\ell$, then $t(a)$ is well-defined in $\fA$ with all function symbols interpreted by $p_i$. Moreover, if $\psi(v)$ is a boolean combination of atomic formulas (with only one variable $v$) in the language $\CL$ of depth $d$ and $a\in D_\ell$ is of distance at least $d$ from the boundary of $D_\ell$, then $\psi(a)$ can be interpreted in $\fA$ with all relation symbols interpreted by $R^{\fA}_j$ and has a truth value. In this case, we say $\fA$ {\em satisfies} $\psi(a)$ if $\psi(a)$ holds under the intended interpretations.

We next introduce the notion of an $\CL$-structuring of
$F(2^{\Z^n})$.

\begin{defnn} \label{def:lstruc}
Let $\CL$ be a language as in Definition~\ref{defn:L}.
An {\em $\CL$-structuring} of $F(2^{\Z^n})$ is a function $\Phi$ with domain $F(2^{\Z^n})$
which assigns to each $x\in F(2^{\Z^n})$ and each relation symbol $R_j$, $1\leq j\leq k$,
a set $\Phi(x,R_j) \subseteq [x]^{\alpha_j}$ in an invariant manner,
meaning that $\Phi(x, R_j) = \Phi(y, R_j)$ whenever $[x] = [y]$.
For each $x\in F(2^{\Z^n})$, this defines a $\Z^n$-structure for $\CL$
$$\fA_x=([x]; f_1,\dots,f_n,f_{n+1},\dots, f_{2n}, \Phi(x,R_1),\dots,\Phi(x,R_k)).$$ We say the
$\CL$-structuring $\Phi$ is {\em continuous} if for each $1\leq j\leq k$ and each
$(g_1,\dots,g_{\alpha_j})\in (\Z^n)^{\alpha_j}$, the map
$x \in F(2^{\Z^n})\mapsto \chi_{\Phi(x, R_j)}(g_1\cdot x,\dots,g_{\alpha_j}\cdot x)$
is continuous.
\end{defnn}

The next theorem says that the question of the existence of a continuous structuring of
$F(2^{\Z^n})$ can be reduced to the existence of a continuous equivariant map
to a certain $\Z^n$-subshift of finite type.

\begin{thmn} \label{thm:lsss}
For every language $\CL$ as in Definition~\ref{defn:L} and every $\lp_1$ sentence $\varphi$ in $\CL$,
there is a canonically constructed $\Z^n$-subshift of finite type $Y = Y(\varphi, \CL)$ such that the
following are equivalent:
\begin{enumerate}
\item
There is a continuous structuring $\Phi$ of
$F(2^{\Z^n})$ such that $\fA_x \models \varphi$ for all $x \in F(2^{\Z^n})$.
\item
There is a continuous, equivariant map $\pi$ from $F(2^{\Z^n})$ to $Y$.
\end{enumerate}
\end{thmn}

\begin{proof}
Let $\CL=(F_1,\dots,F_n,F_{n+1}, \dots, F_{2n}, R_1,\dots,R_k)$ and $\varphi=\forall v\, \psi(v)$ be a $\lp_1$ sentence.
Let $d$ denote the depth of $\psi$.
Let  $\bsft=\prod_{j=1}^k 2^{(2d+1)^{n \alpha_j}}$ where again $\alpha_j$ is the arity of $R_j$.
Consider the set of all $d$-structures for $\CL$ as defined in Definition~\ref{def:partialLstructure}. This set is naturally
identified with the integers in $\bsft=\{ 0,1,\dots, \bsft-1 \}$. We view each element of $\bsft$ as a {\em code} for a $d$-structure.

Note that, since $\psi$ has depth $d$, all the terms in $\psi(\vec 0)$ for $\vec 0=(0,\dots, 0)\in \Z^n$ are well-defined, and thus $\psi(\vec 0)$ has a truth value in any $d$-structure $\fA$. We need the following definition before defining $Y(\varphi, \CL)$.

If $\fA$ and $\fB$ are two $d$-structures and $1\leq i\leq n$, we say that $\fA$ is {\em $e_i$-consistent} with $\fB$
if for each relation symbol $R_j$, $1\leq j\leq k$, and all $a_1, \dots, a_{\alpha_j}\in D_d$ with also $a_1+e_i, \dots, a_{\alpha_j}+e_i \in D_d$
we have that $R_j^{\fA}(a_1+e_i, \dots, a_{\alpha_j}+e_i)$ if and only if $R_j^{\fB}(a_1,\dots, a_{\alpha_j})$.

Let $Y=Y(\varphi, \CL)$ be the set of all $y \in \bsft^{\Z^n}$ satisfying:
\begin{enumerate}

\item[(i)\label{lspa}]
For each $a \in \Z^n$ and $1 \leq i \leq n$, the $d$-structure
$\fA$ coded by $y( a)$ is $e_i$-consistent with the $d$-structure $\fB$ coded by
$y(a+e_i)$;
\item[(ii)\label{lspb}]
For each $a \in \Z^n$, the $d$-structure $\fA$ coded by $y(a)$
satisfies $\psi(\vec 0)$.
\end{enumerate}

Note $Y$ is in fact a $\Z^n$-subshift of finite type of width 1.

It is now straightforward to check that there is a continuous $\CL$-structuring $\Phi$
such that $\fA_x\models \varphi$ for all $x\in F(2^{\Z^n})$ if and only if there 
is a continuous, equivariant map into $Y$.
In fact, suppose there is such a continuous $\CL$-structuring $\Phi$.
Let $\pi \colon F(2^{\Z^n})\to Y$ be defined as follows. For $x \in F(2^{\Z^n})$
and $a \in \Z^n$, $\pi(x)(a)$ is the integer coding the $d$-structure
$\fA(x,a)$ which is given by
$$R_j^{\fA(x,a)}(a_1,\dots,a_{\alpha_j}) \leftrightarrow
\Phi(x,R_j)((a+a_1)\cdot x, \dots, (a+a_{\alpha_j})\cdot x)$$
for any $1\leq j\leq k$ and $a_1,\dots,a_{\alpha_j}
\in D_d$. The structures $\fA(x,a)$ satisfy the consistency
condition (i) for all $1\leq i\leq n$ since $\Phi(x,R_1),\dots, \Phi(x,R_k)$ defines a single
structure $\fA_x$, and (ii) holds since $\fA_x\models \varphi$ and thus $\psi(\vec a \cdot x)$.
The map $\pi$ is clearly continuous and equivariant.
Conversely, suppose $\pi$ is a continuous, equivariant map from $F(2^{\Z^n})$ to $Y$.
Then the consistency condition (i) will allow us to assemble a single structure $\fA_x$,
and condition (ii) guarantees that $\fA_x\models \varphi$.
\end{proof}

From now on we consider continuous $\CL$-structurings of $F(2^{\Z^2})$. 
To state a version of the main theorem~\ref{thm:tilethm} for continuous $\CL$-structurings
of $F(2^{\Z^2})$, we need to extend the notion of an $\CL$-structuring
to the graphs $\gnpq$. In the following definitions we will see that the interpretations of 
the relation symbols in $\gnpq$ is straightforward, but to interpret the function symbols, 
we need to work in $G_{n,p,q}$ and consider
only instantiations that are sufficiently far away from the boundary of $G_{n,p,q}$.

\begin{defnn} \label{def:lstrucgnpq}
Let $\CL=(F_1,\dots,F_4, R_1,\dots,R_k)$ be a language as in
Definition~\ref{defn:L}, where $R_j$ has arity $\alpha_j$ for $1\leq j\leq k$.
An {\em $\CL$-structuring} of $\gnpq$ is a map $\Phi$ which assigns to each relation symbol
$R_j$ a set $\Phi(R_j) \in (\gnpq)^{\alpha_j}$.
\end{defnn}

 Note that the relations $\Phi(R_j)$ on $\gnpq$  lift to relations
$\Phi'(R_j)$ on $G_{n,p,q}$ via the quotient map $\pi: G_{n,p,q}\to \gnpq$.

\begin{defnn}
Let $\Phi$ be a $\CL$-structuring of $\gnpq$, and let $\varphi=\forall v\, \psi(v)$ be a
$\Pi_1$ sentence in the language $\CL$ of depth $d$. We say $(\gnpq,\Phi)$ {\em satisfies}
$\varphi$ if for every $x \in \gnpq$ and every $a \in G_{n,p,q}$
with $\pi(a)=x$, say $a\in G_{n,p,q}^i$ for some $1\leq i\leq 12$, if $a$ is of distance at least $d$ from the boundary of $G^i_{n,p,q}$ then the formula $\psi(\vec 0)$ holds for $\vec 0=(0,0)\in \Z^2$ in the $d$-structure
$$ \fA=([-d,d]^2; p_1, \dots, p_4, R^{\fA}_1, \dots, R^{\fA}_k) $$
where for all $1\leq j\leq k$ and $a_1, \dots, a_{\alpha_j}\in \Z^2$, $R^{\fA}_j(a_1, \dots, a_{\alpha_j})$ holds if and only if $\Phi'(R_j)(a+a_1, \dots, a+a_{\alpha_j})$ holds in $G_{n,p,q}$.
\end{defnn}

The following is the $\CL$-structuring version of the main theorem.

\begin{thmn} \label{thm:tilethmb}
Let $\CL=(F_1,\dots,F_4, R_1,\dots,R_k)$ be a language as in Definition~\ref{defn:L}, and let
$\varphi=\forall v \, \psi$ be a $\Pi_1$ sentence of depth $d$. Then the following are equivalent.
\begin{enumerate}[label={\rm (\arabic*)}, ref=\arabic*]
\item
There is a continuous $\CL$-structuring of $F(2^{\Z^2})$ which satisfies $\varphi$.
\item
There are $n< p, q$ with $n \geq 2d$ and $\gcd(p,q)=1$, and there is an
$\CL$-structuring of $\gnpq$ which satisfies $\varphi$.
\item
For any $n \geq 2d$ and all large enough $p,q>n$, there
is an $\CL$-structuring of $\gnpq$ which satisfies $\varphi$.
\end{enumerate}
\end{thmn}

\begin{proof}
The proof of (\ref{mta})$\Rightarrow$(\ref{mtc}) is very similar to that
of the negative direction of Theorem~\ref{thm:tilethm} in \S\ref{subsec:negtile}. Here we sketch the changes.
Let $x$ again be the hyper-aperiodic element of Lemma~\ref{lem:xhyp}.
Given the language $\CL$ and the $\Pi_1$
formula $\varphi= \forall x\, \psi$ of depth $d$, fix $n \geq 2d$.
Let $(x,j)\mapsto \Phi(x,R_j)$ be a continuous $\CL$-structuring of
$F(2^{\Z^2})$ which satisfies $\varphi$. From the compactness of $K=\ocl{[x]}$ and the continuity of $\Phi$, there is a $w$ such that for all $y\in K$ and terms $t_1(v), \dots, t_{\alpha_j}(v)$ of depth at most $d$, whether
$\Phi(y, R_j)(t_1(y), \dots, t_{\alpha_j}(y))$ holds is determined by the value of $y$ restricted to the $(2w+1)\times (2w+1)$ rectangle centered about $(0,0)$. This implies that, for any $(a, b)\in \Z^2$, whether $\psi((a,b)\cdot x)$ holds in $\fA_{x}=\fA_{(a,b)\cdot x}$ is determined by the value of $x$ restricted to the $(2w+1)\times (2w+1)$ rectangle centered about $(a, b)$.

As in the proof in \S\ref{subsec:negtile}, let $p, q>n+2w$ and $k$ be such that
$n_k=n$, $p_k=q$, $q_k=q$, and $w_k=w$. To define the $\CL$-structuring $\Psi$
of $\gnpq$, consider the $\alpha_j$-arity relation symbol $R_j$ and points
$y_1,\dots, y_{\alpha_j}$ in $\gnpq$, and we need to define the truth value of $\Psi(R_j)(y_1, \dots, y_{\alpha_j})$.
For this we first stipulate that $\Psi(R_j)(y_1, \dots, y_{\alpha_j})$ will not hold unless there are $1\leq i\leq 12$ and $z, z_1, \dots z_{\alpha_j}\in G^i_{n, p, q}$ such that
$\pi(z_1)=y_1, \dots, \pi(z_{\alpha_j})=y_{\alpha_j}$ and $z_1, \dots, z_{\alpha_j}\in [-d, d]^2\cdot z$. In case these witnesses exist, we consider the grid graph $G^i_{w_k,n,p,q}$ in the construction of $x$, and
let $z, z_1, \dots, z_n$ continue to denote the corresponding points in $G^i_{w_k, n, p, q}$. We define that $\Psi(R_j)(y_1, \dots, y_{\alpha_j})$ holds if and only if $\Phi(x, R_j)(z_1\cdot x, \dots, z_n\cdot x)$ holds.

Since $n \geq 2d$, Lemma~\ref{lem:pbg} and the proof of Theorem~\ref{thm:tilethm}
ensure that this definition does not depend on the particular choice of $z, z_1,\dots,z_{\alpha_j}$ when such witnesses exist.
Since $\Phi$ satisfies $\varphi$ and $\varphi$ has depth $d$, it follows that the $\CL$-structuring $\Psi$ of $\gnpq$ satisfies $\varphi$.

For the proof of (\ref{mtb})$\Rightarrow$(\ref{mta}),
let $\Psi$ be an $\CL$-structuring of $\Gamma_{n,p,q}$ as in (\ref{mtb}).
Exactly as in the corresponding
proof in \S\ref{subsec:postile} of Theorem~\ref{thm:tilethm}, we first construct a continuous tiling
of $\fzs$ by partially overlapping copies of the grid graphs $G^i_{n,p,q}$, $1\leq i\leq 12$. Then $\Psi$ induces a continuous $\CL$-structuring $\Phi$ of $\fzs$ in the natural way:
$\Phi(x,R_j)(a_1\cdot x, \dots a_{\alpha_j}\cdot x)$ if and only if
there is a copy of some $G^i_{n,p,q}$ in the tiling of $[x]$ which contains the points
$a_1\cdot x, \dots, a_{\alpha_j}\cdot x$ and, letting $y_1,\dots,y_{\alpha_j}$ be the points in $G^i_{n,p,q}$
corresponding to $a_1\cdot x, \dots, a_{\alpha_j}\cdot x$, then $\Psi(R_j)(y_1,\dots,y_{\alpha_j})$ holds.
To see that $\fA_x\models\varphi$, it is enough to fix $(a,b)\in \Z^2$ and show that
$\fA_x\models\psi((a,b)\cdot x)$. Fix a copy of some $G^i_{n,p,q}$ in the tiling of $[x]$
which contains $([-d,d]^2 + (a,b))\cdot x$.
This can be done as $n \geq 2d$. Let $y$ be the point in $G^i_{n,p,q}$ corresponding to $(a,b)\cdot x$. Since $\Psi$ satisfies $\varphi$ and $y$ is of distance at least $d$ from the boundary of $G^i_{n,p,q}$, we have $\Psi$ satisfies $\psi(y)$.
It follows that $\fA_x\models\psi((a,b)\cdot x)$.
 \end{proof}

\begin{remn}
It is possible to deduce Theorem~\ref{thm:tilethmb} 
from Theorems~\ref{thm:tilethm} and \ref{thm:lsss}. For the reader willing to check a few details, we 
give an outline of this proof below. Suppose $\CL=(F_1,\dots,F_4, R_1,\dots,R_k)$ is a language and 
$\varphi=\forall v\, \psi(v)$ is a $\Pi_1$ formula. Let $Y\subseteq \bsft^{\Z^n}$
be the subshift of finite type from Theorem~\ref{thm:lsss}. Suppose first that 
there is a continuous structuring $\Phi$ with associated structures $\fA_x$ 
which satisfies $\varphi$ (that is, $\fA_x \models \varphi$ for all $x \in F(2^{\Z^n})$). 
Let $\pi \colon F(2^{\Z^n}) \to \bsft^{\Z^n}$ be the equivariant map given 
by Theorem~\ref{thm:lsss}. By Theorem~\ref{thm:tilethm} there is a map 
$g\colon \gnpq \to \bsft$ which respect $Y$. Let $g'\colon G_{n,p,q}\to \bsft$
be the lift of $g$ to the graph $G_{n,p,q}$. Define the relations 
$R'_j$ on $G_{n,p,q}^{\alpha_j}$ as follows. If $x_1,\dots, x_{\alpha_j}$ are in 
one of the $G^i_{n,p,q}$ and least one of these points is an interior point of 
this graph, then we set $R'_j(x_1,\dots,x_{\alpha_j})$ if there is a $D_d$ square 
contained in $G^i_{n,p,q}$, say centered at the point $z$, for which 
$g'(z)$ declares the points in $D_d$ corresponding to $x_1,\dots,x_{\alpha_j}$
to be in the relation $R_j$. Since interior points of the $12$ tiles are never
identified, there are no other points $y_1,\dots,y_{\alpha_j}$ in $G_{n,p,q}$
which are identified with the $x_1,\dots,x_{\alpha_j}$ in $\gnpq$ and which occur in the same relative
positions within some $D_d$ square in $G_{n,p,q}$. If none of the points 
$x_1,\dots,x_{\alpha_j}$ is an interior point of $G_{n,p,q}$, 
then we set $R'_j(x_1,\dots,x_{\alpha_j})$ if there is a $D_d$ square 
contained in $G^i_{n,p,q}$, say centered at the point $z$, with $D_d$ contained 
in the interior of some $G^i_{n,p,q}$, say centered say at $z$ 
(necessarily $z$ is not an interior point), where $g'(z)$ 
declares the corresponding points to be $R_j$ related. If 
$y_1,\dots,y_{\alpha_j}$ are equivalent to the $x_i$ points in $\gnpq$,
then it is easy to check that $R'_j(x_1,\dots,x_{\alpha_j})$ iff $R'_j(y_1,\dots,y_{\alpha_j})$. 
So, the relation $R'_j$ corresponds to a relation $R_j$ on $\gnpq$. 
It is straightforward to check that this structuring of $\gnpq $ satisfies $\varphi$.
The other direction is similar to the 
corresponding direction of the proof presented above. 
\end{remn}

\begin{remn}
The analog of Theorem~\ref{thm:tilethmb} holds for $F(2^\Z)$ as well. The proof is similar
(but easier) to that of Theorem~\ref{thm:tilethmb}, except we follow the proof of the
two tiles theorem (Theorem~\ref{thm:twotilesthm}) instead of Theorem~\ref{thm:tilethm}.
\end{remn}

\begin{remn}
In the definition of $\Pi_1$ formulas it is important that we restrict the formula $\psi(v)$
to having a single free variable. The statement of Theorem~\ref{thm:tilethmb} (or the
corresponding theorem for $F(2^\Z)$) fails for formulas of the form $\varphi=\forall u\ \forall v\ \psi(u, v)$
where $\psi$ is quantifier free. Consider, for example, the case of $F(2^\Z)$. There is a formula
$\varphi$ of the above form which asserts that the unary relations $R_1$, $R_2$, $R_3$ describe a
proper $3$-coloring of the $\Z$-graph, and that at most one vertex satisfies $R_3$.
This formula can easily be seen to hold in the two-tiles graphs $\Gamma^1_{n,p,q}$,
but there is no continuous structuring of $F(2^\Z)$ which satisfies $\varphi$.
\end{remn}

\section{Higher dimensions}  \label{subsec:hd}

The twelve tiles theorem has generalizations from the two-dimensional
case $F(2^{\Z^2})$ to the general $n$-dimensional case $F(2^{\Z^n})$. For the applications
of the current paper, however, it suffices to establish a weaker form of the negative
direction of the theorem. Namely, we show that a continuous, equivariant map from $F(2^{\Z^n})$
to a subshift of finite type implies that there is a corresponding map from
the $n$-dimensional $q_1 \times \cdots \times q_n$ torus tile to the subshift for all sufficiently
large $q_1,\dots,q_n$. This could be proved by the method of \S\ref{subsec:negtile},
but here we use instead Lemma~\ref{proaction}.

Recall Definition~\ref{def:sft} for the notion of a subshift of finite type $Y\subseteq \bsft^{\Z^n}$.
As before, a subshift $Y$ can be described by a sequence $(\bsft;p_1,\dots,p_k)$, where
$p_i\colon [0,a_{i,1})\times \cdots \times [0,a_{i,n}) \to \bsft$ are the forbidden patterns
for the subshift.

By the $q_1\times \cdots \times q_n$ {\em torus tile} $T_{q_1,\dots,q_n}$ we mean the quotient of $\Z^n$ by the
relation $(a_1,\dots,a_n)\sim (b_1,\dots,b_n)$ if and only if $a_1 \equiv b_1\!\mod q_1$, $\dots,
a_n \equiv b_n\!\mod q_n$. The tile inherits the $\Z^n$-graph structure from the $\Z^n$-graph
structure of $\Z^n$. The tile has size $q_1\cdots q_n$, and every vertex has degree
$2n$.

If $Y\subseteq \bsft^{\Z^n}$ is a subshift of finite type described by $(\bsft;p_1,\dots p_k)$, then we say that a map
$g \colon T_{q_1,\dots,q_n}\to \bsft$ {\em respects} $Y$ just as in Definition~\ref{def:resp},
using the $\Z^n$-graph $T_{q_1,\dots,q_n}$ instead of $\gnpq$. This is equivalent
to saying that the map $g \circ \pi \colon \Z^n \to \bsft$ (where $\pi \colon \Z^n \to T_{q_1,\dots,q_n}$
is the quotient map) respect $Y$ in that for any rectangle $R=[0,a_{i,1})\times \cdots
\times [0,a_{i,n}) \subseteq \Z^n$, $g \circ \pi \res R \neq p_i$.

\begin{thmn} \label{thm:weakneg}
Let $Y\subseteq \bsft^{\Z^n}$ be a subshift of finite type. If there is a continuous, equivariant map
$f \colon F(2^{\Z^n})\to Y$, then for all $q_1,\dots,q_n\geq 2$ and all sufficiently large $k$ there is
$g \colon T_{q^k_1,\dots,q^k_n}\to \bsft$ which respects $Y$.
\end{thmn}

\begin{proof}
Let $f \colon F(2^{\Z^n})\to Y$ be a continuous and equivariant map into the subshift $Y$,
and define the map $f_0 \colon F(2^{\Z^n})\to \bsft$ by $f_0(x) = f(x)(\vec 0)$, where $\vec{0}=(0,\dots, 0)\in\Z^n$. Let $x \in F(2^{\Z^n})$
be the hyper-aperiodic element of Lemma~\ref{proaction} for $G=\Z^n$, constructed with respect to
the subgroups $G_k=q_1^k\Z\times\dots\times q_n^k\Z$.
Let $\mathcal{P}=\{P_1, \dots, P_\bsft\}$ be the clopen partition of $F(2^{\Z^n})$ determined
by $f_0$, that is, for all $1\leq m\leq \bsft$,
$y \in P_m$ if and only if $f_0(y)=m$. Let $A\subseteq \Z^n$ be the finite set
given by Lemma~\ref{proaction}. Let $w \geq 1$ be larger than the width of $Y$ and let $k$ be large enough so that the diameter of
$A$ (with respect to the  $\ell_1$-norm on $\Z^n$) plus $2w$ is less than $\frac{q_i^k}{2}$
for all $i$. Set $R_1 = [w,q_1^k)\times\dots\times [w,q_n^k)$ and $R_2 = [0,q_1^k+w)\times\dots\times [0,q_n^k+w)$.

Next, our choice of $k$ allows us to pick a point $y \in [x]$ satisfying $(R_2 \setminus R_1) \cdot y \cap A G_k = \emptyset$.
From  Lemma~\ref{proaction} we have that $\mathcal{P}$ is $G_k$-invariant on $G_k (R_2 \setminus R_1) \cdot y$
since $G_k (R_2 \setminus R_1) \subseteq G \setminus A G_k$. We now define $g: T_{q_1^k,\dots, q_n^k}\to \bsft$ as follows. For any $r+G_k\in T_{q_1^k,\dots, q_n^k}=\Z^n/G_k$, let $a\in (r+G_k)\cap R_2$ be arbitrary, and define
$g(r+G_k)=f_0(a\cdot y)$. To see that $g$ is well-defined, note that $(r + G_k) \cap R_2 = \{r\}$ for every $r \in R_1$, and for $a, a'\in R_2\setminus R_1$, if $a+G_k=a'+G_k$, then $f_0(a\cdot y)=f_0(a'\cdot y)$.

Lastly,
by our choice of $R_2$, every copy of $[0, w) \times \dots \times [0, w)$ in $T_{q_1^k,\dots,q_n^k}$ is the image under
the quotient map of a translate of $[0, w) \times \dots \times [0, w)$ contained in $R_2$. Therefore $g$ respects $Y$
since the image of $f$ is contained in $Y$.
\end{proof}

\begin{remn}
The construction of the hyper-aperiodic element of \S\ref{subsec:negtile} shows a little more,
namely that if $f\colon F(2^{\Z^n}) \to Y$ is continuous and equivariant, then for all
sufficiently large $q_1,\dots,q_n$ there is $g \colon T_{q_1,\dots,q_n}\to \bsft$
which respects $Y$. Theorem~\ref{thm:weakneg} is sufficient for our purposes, however.
\end{remn}

\chapter{Applications of the Twelve Tiles Theorem}\label{chapter3}

\section{Basic applications of the Twelve Tiles Theorem} \label{sec:tileapp}

In this  section we prove several results which follow easily from the
Twelve Tiles Theorem, by which we mean either of Theorems~\ref{thm:tilethm} or  \ref{thm:tilethmb}.

For our first application, we give another proof of Theorem~\ref{threecoloring} stating that the continuous chromatic number of
$F(2^{\Z^2})$ is $4$. This includes a new proof
of Theorem~\ref{threecoloring} for the lower bound as well as a proof
different from \cite{gao_countable_2015} for the upper bound. These are included in the next two results.

\begin{thmn} \label{thm:pcn}
For any $n$ and $p, q>n$ with $\gcd(p,q)=1$, neither of the long tiles $\Tcqadpa, \Tdpacqa, \Tcbpcaq, \Tcaqcbp$ can be properly $3$-colored.
\end{thmn}

\begin{proof}
Fix $n, p, q$ with $p, q>n$ and $\gcd(p, q)=1$. We only consider the long horizontal tile $\Tcqadpa$, which is shown in
Figure~\ref{fig:Gamma-npq-horiz-3}. The proof for the other long tiles is similar.
Assume $\varphi \colon \Tcqadpa \to \{ 0,1,2\}$ is a proper
$3$-coloring. We identify $\{ 0,1,2\}$ with the vertices of $K_3$, the complete
graph on three vertices. For any closed walk $\alpha$ in $K_3$, let $w(\alpha)$
be its winding number in $K_3$, with the convention that the $3$-cycle $(0,1,2,0)$
has winding number $+1$. If $\beta$ is a cycle in $\Tcqadpa$,  then $\varphi(\beta)$ is a cycle in $K_3$.

\begin{figure}[htpb] 
\begin{tikzpicture}
\begin{scope}[scale=1,yscale=-1]
\begin{scope}[shift={(0,0)}]  
  \foreach \loc in {(0,0.0),(1.8,0.0),(3.6,0.0),(5.4,0.0),(7.2,0.0),(10,0.0),
                    (0,1.8),(2.2,1.8),(4.4,1.8),(6.6,1.8),(10,1.8)} {
    \node[XBoxStyle] at \loc {$R_\times$};
  }
  \foreach \loc in {(0,1),(10,1)} {
    \node[ABoxStyle] at \loc {$R_a$};
  }
  \foreach \loc in {(1,0),(2.8,0),(4.6,0),(6.4,0),(9.2,0)} {
    \node[CBoxStyle] at \loc {$R_c$};
  }
  \foreach \loc in {(1,1.8),(3.2,1.8),(5.4,1.8),(7.6,1.8),(8.8,1.8)} {
    \node[DBoxStyle] at \loc {$R_d$};
  }

  \node at (8.8, 1.4) {\Large$\cdots$};

\node at (0.9,-0.2) {$\gamma$};
\node at (-0.2, 0.9) {$\alpha$};
\node at (3.3, 1.5) {$\delta$};

\draw[fill,radius=0.05,black] (0,0) circle;
\draw[fill,radius=0.05,black] (1.8,0) circle;
\draw[fill,radius=0.05,black] (3.6,0) circle;
\draw[fill,radius=0.05,black] (5.4,0) circle;
\draw[fill,radius=0.05,black] (7.2,0) circle;
\draw[fill,radius=0.05,black] (10,0) circle;
\draw[fill,radius=0.05,black] (0,1.80) circle;
\draw[fill,radius=0.05,black] (2.20,1.80) circle;
\draw[fill,radius=0.05,black] (4.40,1.80) circle;
\draw[fill,radius=0.05,black] (6.60,1.80) circle;
\draw[fill,radius=0.05,black] (10,1.80) circle;

 \draw[draw=black, ultra thick] (0,0) to (10,0);
 \draw[draw=black, ultra thick] (0,1.80) to (10,1.80);
 \draw[draw=black, ultra thick] (0,0) to (0,1.80);
 \draw[draw=black, ultra thick] (10,0) to (10,1.80);

 \fill[fill=white] (8,-.1) rectangle (9.6, 2.9);



\end{scope}

\end{scope}
\end{tikzpicture}
\caption{\label{fig:Gamma-npq-horiz-3} The long horizontal tile $\Tcqadpa$ is not properly $3$-colorable. Consider the cycle indicated by the thick line.}
\end{figure} 

Consider the cycle in $\Tcqadpa$ illustrated in Figure~\ref{fig:Gamma-npq-horiz-3}. The cycle is constructed as follows. First identify the upper-left corner vertex of each $R_\times$ block. We name three cycles that start and end at this vertex. The cycle along the left edges of $R_\times$ and $R_a$ blocks, going downward in the figure, is denoted $\alpha$.
The cycle along the top edges of $R_\times$ and $R_c$ blocks, going to the right, is denoted $\gamma$.
The cycle along the top edges of  $R_\times$ and $R_d$ blocks, going to the right, is denoted $\delta$.
Note that $\alpha$ and $\gamma$ have length $p$ and $\delta$ has length $q$. For any cycle $\beta$ we also use $\beta^{-1}$ to denote the cycle obtained from $\beta$ by reversing its direction, that is, if we list out the nodes in $\beta$ as $v_1\dots v_n$, then $\beta^{-1}$ can be represented as $v_n\dots v_1$. In this notation, the cycle we are considering is $\gamma^q\alpha\delta^{-p}\alpha^{-1}$. This cycle is depicted in Figure~\ref{fig:Gamma-npq-horiz-3} by the thick line and has clockwise orientation.

The cycle $\gamma^q\alpha\delta^{-p}\alpha^{-1}$ encompasses a region in $\Tcqadpa$ which is filled by a
grid of cycles of length $4$. It follows that the winding number of $\varphi(\gamma^q\alpha\delta^{-p}\alpha^{-1})$ is the sum of the winding numbers of several $4$-cycles in $K_3$. However, it is easy to check that any $4$-cycle in $K_3$ has winding number zero. Therefore
$$q \cdot w(\varphi(\gamma)) + w(\varphi(\alpha)) - p \cdot w(\varphi(\delta)) - w(\varphi(\alpha)) = w(\varphi(\gamma^q\alpha\delta^{-p}\alpha^{-1}))=0$$
which implies that
$$q \cdot w(\varphi(\gamma))=p \cdot w(\varphi(\delta)).$$
Since $\gcd(p,q)=1$, at least one of $p,q$ must be odd. An odd cycle in $K_3$ must have non-zero winding number,
so at least one of $w(\varphi(\gamma))$, $w(\varphi(\delta))$ must be non-zero.
The above equation then shows that both are non-zero and since $\gcd(p,q)=1$,
$p | w(\varphi(\gamma))$. But, $\varphi(\gamma)$ is a $p$-cycle in $K_3$, and so can have
winding number at most $\frac{p}{3}$, a contradiction.
\end{proof}

Using either Theorem~\ref{thm:tilethm} or \ref{thm:tilethmb}, Theorem~\ref{thm:pcn} immediately implies that the continuous chromatic number of $F(2^{\Z^2})$ is greater than $3$. The proof of Theorem~\ref{threecoloring} can also be adapted to give an alternative proof of Theorem~\ref{thm:pcn}.

We next turn to the other direction, namely showing that there is a continuous proper $4$-coloring of $F(2^{\Z^2})$. By Theorem~\ref{thm:tilethm} it suffices to show that $\Gamma_{n,p,q}$ can be properly $4$-colored for some $1\leq n<p, q$ with $\gcd(p,q)=1$. We can, in fact, take $n=1, p=2, q=3$.

\begin{thmn}\label{thm:pcn2}
There is a proper $4$-coloring of $\Gamma_{1,2,3}$.
\end{thmn}

\begin{proof} We describe a proper $4$-coloring on $\Gamma_{1,2,3}$ by giving the coloring on the grid graphs $G^i_{1,2,3}$ for $1\leq i\leq 12$ so that similarly labeled blocks are colored the same way.
Because proper colorings remain proper after reflections and rotations, it
suffices to color the tiles~$\Tcaca$, $\Tdada$, $\Tdbdb$, $\Tcqadpa$ and
$\Tcdacda$.  We give this coloring in Figure~\ref{fig:four-coloring}. For better readability of the colors the  grid graphs are represented as boards, with square boxes representing vertices and adjacency of squares representing edges.
The figure also shows the coloring on the $R_\times, R_a, R_b, R_c$ blocks by themselves for reference, so that it is easy to check the consistency of the coloring across tiles.
\end{proof}

\begin{figure}[htbp]
    \newcommand{\figurescale}{.6}
    \newcommand{\subfigurespacing}{0.5cm}
   \colorlet{color1}{red!50}
   \colorlet{color2}{black!25}
   \colorlet{color3}{black!50}
   \colorlet{color0}{white}
    \begin{minipage}{1\textwidth}
        \centering
        \begin{tikzpicture}[scale=\figurescale]
            \begin{scope}[yscale=-1,local bounding box=bound]
                \foreach \row [count=\i] in {
                    {3}} {
                    \foreach \clrNum [count=\j] in \row {
                        \fill[fill=color\clrNum] (\j,\i) rectangle ++(-1,-1);
                    }
                }
                \draw[draw=black, ultra thick] (0,0) rectangle (1,1);
            \end{scope}
            \node[below=0.1cm of bound.south] {$\varphi\restrict{R_\times}$};
        \end{tikzpicture}   
        \hspace{\subfigurespacing}
        \begin{tikzpicture}[scale=\figurescale]
            \begin{scope}[yscale=-1,local bounding box=bound]
                \foreach \row [count=\i] in {
                    {0}} {
                    \foreach \clrNum [count=\j] in \row {
                        \fill[fill=color\clrNum] (\j,\i) rectangle ++(-1,-1);
                    }
                }
                \draw[draw=black, ultra thick] (0,0) rectangle (1,1);
            \end{scope}
            \node[below=0.1cm of bound.south] {$\varphi\restrict{R_a}$};
        \end{tikzpicture}   
        \hspace{\subfigurespacing}
        \begin{tikzpicture}[scale=\figurescale]
            \begin{scope}[yscale=-1,local bounding box=bound]
                \foreach \row [count=\i] in {
                    {0},
                    {2}} {
                    \foreach \clrNum [count=\j] in \row {
                        \fill[fill=color\clrNum] (\j,\i) rectangle ++(-1,-1);
                    }
                }
                \draw[draw=black, ultra thick] (0,0) rectangle (1,2);
            \end{scope}
            \node[below=0.1cm of bound.south] {$\varphi\restrict{R_b}$};
        \end{tikzpicture}   
        \hspace{\subfigurespacing}
        \begin{tikzpicture}[scale=\figurescale]
            \begin{scope}[yscale=-1,local bounding box=bound]
                \foreach \row [count=\i] in {
                    {0}} {
                    \foreach \clrNum [count=\j] in \row {
                        \fill[fill=color\clrNum] (\j,\i) rectangle ++(-1,-1);
                    }
                }
                \draw[draw=black, ultra thick] (0,0) rectangle (1,1);
            \end{scope}
            \node[below=0.1cm of bound.south] {$\varphi\restrict{R_c}$};
        \end{tikzpicture}   
        \hspace{\subfigurespacing}
        \begin{tikzpicture}[scale=\figurescale]
            \begin{scope}[yscale=-1,local bounding box=bound]
                \foreach \row [count=\i] in {
                    {0,2}} {
                    \foreach \clrNum [count=\j] in \row {
                        \fill[fill=color\clrNum] (\j,\i) rectangle ++(-1,-1);
                    }
                }
                \draw[draw=black, ultra thick] (0,0) rectangle (2,1);
            \end{scope}
            \node[below=0.1cm of bound.south] {$\varphi\restrict{R_d}$};
        \end{tikzpicture}   
    \end{minipage} \\
    \vspace{\subfigurespacing}
    \begin{minipage}{1\textwidth}
        \centering
        \begin{tikzpicture}[scale=\figurescale]
            \begin{scope}[yscale=-1,local bounding box=bound]
                \foreach \row [count=\i] in {
                    {3,0,3},
                    {0,3,0},
                    {3,0,3}} {
                    \foreach \clrNum [count=\j] in \row {
                        \fill[fill=color\clrNum] (\j,\i) rectangle ++(-1,-1);
                    }
                }
                \draw[draw=black, ultra thick] (0,0) rectangle (1,3);
                \draw[draw=black, ultra thick] (2,0) rectangle (3,3);
                \draw[draw=black, ultra thick] (0,0) rectangle (3,1);
                \draw[draw=black, ultra thick] (0,2) rectangle (3,3);
            \end{scope}
            \node[below=0.1cm of bound.south] {$\varphi\restrict{\Tcaca}$};
        \end{tikzpicture}   
        \hspace{\subfigurespacing}
        \begin{tikzpicture}[scale=\figurescale]
            \begin{scope}[yscale=-1,local bounding box=bound]
                \foreach \row [count=\i] in {
                    {3,0,2,3},
                    {0,2,3,0},
                    {3,0,2,3}} {
                    \foreach \clrNum [count=\j] in \row {
                        \fill[fill=color\clrNum] (\j,\i) rectangle ++(-1,-1);
                    }
                }
                \draw[draw=black, ultra thick] (0,0) rectangle (1,3);
                \draw[draw=black, ultra thick] (3,0) rectangle (4,3);
                \draw[draw=black, ultra thick] (0,0) rectangle (4,1);
                \draw[draw=black, ultra thick] (0,2) rectangle (4,3);
            \end{scope}
            \node[below=0.1cm of bound.south] {$\varphi\restrict{\Tdada}$};
        \end{tikzpicture}   
        \hspace{\subfigurespacing}
        \begin{tikzpicture}[scale=\figurescale]
            \begin{scope}[yscale=-1,local bounding box=bound]
                \foreach \row [count=\i] in {
                    {3,0,2,3},
                    {0,2,3,0},
                    {2,3,0,2},
                    {3,0,2,3}} {
                    \foreach \clrNum [count=\j] in \row {
                        \fill[fill=color\clrNum] (\j,\i) rectangle ++(-1,-1);
                    }
                }
                \draw[draw=black, ultra thick] (0,0) rectangle (1,4);
                \draw[draw=black, ultra thick] (3,0) rectangle (4,4);
                \draw[draw=black, ultra thick] (0,0) rectangle (4,1);
                \draw[draw=black, ultra thick] (0,3) rectangle (4,4);
            \end{scope}
            \node[below=0.1cm of bound.south] {$\varphi\restrict{\Tdbdb}$};
        \end{tikzpicture}   
    \end{minipage} \\
    \vspace{\subfigurespacing}
    \begin{minipage}{1\textwidth}
        \centering
        \begin{tikzpicture}[scale=\figurescale]
            \begin{scope}[yscale=-1,local bounding box=bound]
                \foreach \row [count=\i] in {
                    {3,0,3,0,3,0,3},
                    {0,3,0,2,0,3,0},
                    {3,0,2,3,0,2,3}} {
                    \foreach \clrNum [count=\j] in \row {
                        \fill[fill=color\clrNum] (\j,\i) rectangle ++(-1,-1);
                    }
                }
                \fill[pattern=dots, pattern color=black] (4, 1) rectangle (5,2);
                \draw[draw=black, ultra thick] (0,0) rectangle (3,1);
                \draw[draw=black, ultra thick] (2,0) rectangle (5,1);
                \draw[draw=black, ultra thick] (4,0) rectangle (7,1);
                \draw[draw=black, ultra thick] (0,0) rectangle (1,3);
                \draw[draw=black, ultra thick] (6,0) rectangle (7,3);
                \draw[draw=black, ultra thick] (0,2) rectangle (4,3);
                \draw[draw=black, ultra thick] (3,2) rectangle (7,3);

            \end{scope}
            \node[below=0.1cm of bound.south] {$\varphi\restrict{\Tcqadpa}$};
        \end{tikzpicture}   
        \hspace{\subfigurespacing}
        \begin{tikzpicture}[scale=\figurescale]
            \begin{scope}[yscale=-1,local bounding box=bound]
                \foreach \row [count=\i] in {
                    {3,0,2,3,0,3},
                    {0,3,0,2,3,0},
                    {3,0,3,0,2,3}} {
                    \foreach \clrNum [count=\j] in \row {
                        \fill[fill=color\clrNum] (\j,\i) rectangle ++(-1,-1);
                    }
                }
                \draw[draw=black, ultra thick] (0,0) rectangle (4,1);
                \draw[draw=black, ultra thick] (3,0) rectangle (6,1);
                \draw[draw=black, ultra thick] (0,2) rectangle (3,3);
                \draw[draw=black, ultra thick] (2,2) rectangle (6,3);

                \draw[draw=black, ultra thick] (0,0) rectangle (1,3);
                \draw[draw=black, ultra thick] (5,0) rectangle (6,3);
            \end{scope}
            \node[below=0.1cm of bound.south] {$\varphi\restrict{\Tcdacda}$};
        \end{tikzpicture}   
    \end{minipage}
    \caption[A proper $4$-coloring of $\Gamma_{1,2,3}$.]{\label{fig:four-coloring}A proper $4$-coloring
of $\Gamma_{1,2,3}$.  Vertices are represented by square boxes and edges are represented by adjacency. Thick lines separate
labeled blocks.}
\end{figure} 

For our next application we show that there is no continuous perfect matching
in the Schreier graph of $F(2^{\Z^n})$. Here we think of the Schreier graph of $F(2^{\Z^n})$ as an undirected graph.
Recall that a {\em  matching}
in an undirected graph is a subset $M$ of the edges such that every vertex is
incident with at most  one of the edges of $M$. A matching $M$ is
a {\em perfect matching} if every vertex is incident with exactly
one edge in $M$. For a countable group $G$ with a minimal generating set $S$, a perfect matching $M$ in
the Schreier graph $\Gamma(F(2^G), G, S)$ (see the Definition in \S\ref{section:1.1}) is given by a map $\eta\colon F(2^G) \to S\cup S^{-1}$ such that, for every $x\in F(2^G)$, $(x, \eta(x)\cdot x)\in M$. In order to give rise to a perfect matching, such $\eta$ must satisfy $\eta(\eta(x)\cdot x)=\eta(x)^{-1}$ for all $x\in F(2^G)$.
When $\eta$ is continuous, we say that $M$ is continuous.

\begin{thmn} \label{thm:npm}
There is no continuous perfect matching in $F(2^{\Z^n})$
for any $n \geq 1$.
\end{thmn}

\begin{proof} Assume a continuous perfect matching in $F(2^{\Z^n})$ exists and is given by a continuous function
$$ \eta: F(2^{\Z^n})\to \{\pm e_1, \pm e_2, \dots, \pm e_n\}. $$
The condition $\eta(\eta(x)\cdot x)=\eta(x)^{-1}=-\eta(x)$ can be described by a number of forbidden patterns of width 1, and thus gives rise to a subshift of finite type $Y\subseteq (2n)^{\Z^n}$. Let $q_1,\dots,q_n$ be odd. From
Theorem~\ref{thm:weakneg} there is $k$ such that there is $g\colon
T_{q_1^k,\dots,q_n^k}\to 2n$ which respects $Y$.
As $g$ respects $Y$, it gives a perfect matching in the torus $T_{q_1^k,\dots,q_n^k}$.
This is a contradiction as this torus has an odd number of vertices.
\end{proof}

Next we compute the continuous edge chromatic number of
$F(2^{\Z^2})$. Again we think of the Schreier graph of $F(2^{\Z^2})$ as an undirected graph. Recall that in an undirected graph two edges are {\em adjacent} if they are incident to a common vertex. An {\em edge coloring} is an assignment of colors to the edges so that adjacent edges are assigned different colors. The {\em edge chromatic number} of a graph is the smallest number of colors one can use in an edge coloring. In the Cayley graph of $\Z^2$ each edge is adjacent to 6 other edges, and it is easy to see that the edge chromatic number is $4$. In the next result we again need only consider the torus tiles.

\begin{figure}
    \noindent\centering
    \newcommand{\subfigurespacing}{0.5cm}

    \colorlet{color5}{black!65}
    \colorlet{color4}{green!50}
    \colorlet{color2}{black!25}
    \colorlet{color3}{blue!50}
    \colorlet{color1}{white}
    \tikzset{
    edgecoloringscale/.style={x=1.7em,y=1.7em,yscale=-1},
    blackboundary/.style={ultra thick},
    pics/edgecoloringnode/.style args={#1/#2/#3/#4}{code={
        \draw[color#1,fill=color#1] (0,0) -- (0.5,0.5) -- (0,1) -- (0,0); 
        \draw[color#2,fill=color#2] (0,0) -- (0.5,0.5) -- (1,0) -- (0,0); 
        \draw[color#3,fill=color#3] (1,1) -- (0.5,0.5) -- (0,1) -- (1,1); 
        \draw[color#4,fill=color#4] (1,1) -- (0.5,0.5) -- (1,0) -- (1,1); 
        \draw[black!50] (0,0) rectangle (1,1);
    }}}

    \newcommand{\edgecoloringtile}[3]{
        \begin{tikzpicture}[edgecoloringscale]
            \begin{scope}[local bounding box=bound]
                \foreach \row[count=\i] in {#2} {
                    \foreach \l/\t/\b/\r [count=\j] in \row  {
                        \draw pic at (\j,\i) {edgecoloringnode=\l/\t/\b/\r};

                    }
                }
            \end{scope}
            \node[below=0.1cm of bound.south] {#1};
            \begin{scope}[shift={(1,0)}]
                \foreach \r\s in {#3} {
                    \draw[blackboundary] \r rectangle \s;
                }
            \end{scope}
        \end{tikzpicture}
    }

    %
    %
    \begin{minipage}{1\textwidth}
        \centering\noindent
        \edgecoloringtile{$\varphi\restrict{R_\times}$}{{1/3/4/2}}{{(0,0)}/{(1,1)}}
        \hspace{\subfigurespacing}
        \edgecoloringtile{$\varphi\restrict{R_a}$}{{1/4/3/2}}{{(0,0)}/{(1,1)}}
        \hspace{\subfigurespacing}
        \edgecoloringtile{$\varphi\restrict{R_b}$}{
            {1/5/3/2},
            {1/4/5/2}}{{(0,0)}/{(1,2)}}
        \hspace{\subfigurespacing}
        \edgecoloringtile{$\varphi\restrict{R_c}$}{{2/3/4/1}}{{(0,0)}/{(1,1)}}
        \hspace{\subfigurespacing}
        \edgecoloringtile{$\varphi\restrict{R_d}$}{{2/3/4/5,5/3/4/1}}{{(0,0)}/{(2,1)}}
    \end{minipage}
    \vspace{\subfigurespacing}

    %
    %
    \begin{minipage}{1\textwidth}
        \centering\noindent
        \edgecoloringtile{$\varphi\restrict{\Tcaca}$}{
            {1/3/4/2,2/3/4/1,1/3/4/2},
            {1/4/3/2,2/4/3/1,1/4/3/2},
            {1/3/4/2,2/3/4/1,1/3/4/2}}{
            {(0,0)}/{(1,3)},
            {(0,0)}/{(3,1)},
            {(3,3)}/{(0,2)},
            {(3,3)}/{(2,0)}}
        \hspace{\subfigurespacing}
        \edgecoloringtile{$\varphi\restrict{\Tdada}$}{
            {1/3/4/2,2/3/4/5,5/3/4/1,1/3/4/2},
            {1/4/3/2,2/4/3/5,5/4/3/1,1/4/3/2},
            {1/3/4/2,2/3/4/5,5/3/4/1,1/3/4/2}}{
            {(0,0)}/{(1,3)},
            {(0,0)}/{(4,1)},
            {(4,3)}/{(0,2)},
            {(4,3)}/{(3,0)}}
        \hspace{\subfigurespacing}
        \edgecoloringtile{$\varphi\restrict{\Tcbcb}$}{
            {1/3/4/2,2/3/4/1,1/3/4/2},
            {1/5/3/2,2/5/3/1,1/5/3/2},
            {1/4/5/2,2/4/5/1,1/4/5/2},
            {1/3/4/2,2/3/4/1,1/3/4/2}}{
            {(0,0)}/{(1,4)},
            {(0,0)}/{(3,1)},
            {(3,4)}/{(0,3)},
            {(3,4)}/{(2,0)}}
        \hspace{\subfigurespacing}
        \edgecoloringtile{$\varphi\restrict{\Tdbdb}$}{
            {1/3/4/2,2/3/4/5,5/3/4/1,1/3/4/2},
            {1/5/3/2,2/1/3/4,4/2/3/1,1/5/3/2},
            {1/4/5/2,2/4/1/3,3/4/2/1,1/4/5/2},
            {1/3/4/2,2/3/4/5,5/3/4/1,1/3/4/2}}{
            {(0,0)}/{(1,4)},
            {(0,0)}/{(4,1)},
            {(4,4)}/{(0,3)},
            {(4,4)}/{(3,0)}}
    \end{minipage}
    \vspace{\subfigurespacing}

    \begin{minipage}{0.4\textwidth}
        \centering\noindent
        \edgecoloringtile{$\varphi\restrict{\Tcbacba}$}{
            {1/3/4/2,2/3/4/1,1/3/4/2},
            {1/4/3/2,2/4/3/1,1/5/3/2},
            {1/3/4/2,2/3/4/1,1/4/5/2},
            {1/5/3/2,2/5/3/1,1/3/4/2},
            {1/4/5/2,2/4/5/1,1/4/3/2},
            {1/3/4/2,2/3/4/1,1/3/4/2}}{
            {(0,0)}/{(3,1)},
            {(0,5)}/{(3,6)},
            {(0,0)}/{(1,3)},
            {(0,2)}/{(1,6)},
            {(2,0)}/{(3,4)},
            {(2,3)}/{(3,6)}}
        \hspace{\subfigurespacing}
        \edgecoloringtile{$\varphi\restrict{\Tcabcab}$}{
            {1/3/4/2,2/3/4/1,1/3/4/2},
            {1/5/3/2,2/5/3/1,1/4/3/2},
            {1/4/5/2,2/4/5/1,1/3/4/2},
            {1/3/4/2,2/3/4/1,1/5/3/2},
            {1/4/3/2,2/4/3/1,1/4/5/2},
            {1/3/4/2,2/3/4/1,1/3/4/2}}{
            {(0,0)}/{(3,1)},
            {(0,5)}/{(3,6)},
            {(0,0)}/{(1,4)},
            {(0,3)}/{(1,6)},
            {(2,0)}/{(3,3)},
            {(2,2)}/{(3,6)}}
    \end{minipage}
    \hspace{\subfigurespacing}
    \begin{minipage}{0.4\textwidth}
        \centering\noindent
        \edgecoloringtile{$\varphi\restrict{\Tcdacda}$}{
            {1/3/4/2,2/3/4/1,1/3/4/2,2/3/4/5,5/3/4/1,1/3/4/2},
            {1/4/3/2,2/4/3/1,1/4/3/2,2/4/3/5,5/4/3/1,1/4/3/2},
            {1/3/4/2,2/3/4/5,5/3/4/1,1/3/4/2,2/3/4/1,1/3/4/2}}{
            {(0,0)}/{(1,3)},
            {(5,0)}/{(6,3)},
            {(0,0)}/{(3,1)},
            {(2,0)}/{(6,1)},
            {(0,2)}/{(4,3)},
            {(3,2)}/{(6,3)}}
        \vspace{\subfigurespacing}

        \edgecoloringtile{$\varphi\restrict{\Tdcadca}$}{
            {1/3/4/2,2/3/4/5,5/3/4/1,1/3/4/2,2/3/4/1,1/3/4/2},
            {1/4/3/2,2/4/3/5,5/4/3/1,1/4/3/2,2/4/3/1,1/4/3/2},
            {1/3/4/2,2/3/4/1,1/3/4/2,2/3/4/5,5/3/4/1,1/3/4/2}}{
            {(0,0)}/{(1,3)},
            {(5,0)}/{(6,3)},
            {(0,0)}/{(4,1)},
            {(3,0)}/{(6,1)},
            {(0,2)}/{(3,3)},
            {(2,2)}/{(6,3)}}
    \end{minipage}
    \vspace{\subfigurespacing}

    \begin{minipage}{0.4\textwidth}
        \centering\noindent
        \edgecoloringtile{$\varphi\restrict{\Tcaqcbp}$}{
            {1/3/4/2,2/3/4/1,1/3/4/2},
            {1/5/3/2,2/4/3/1,1/4/3/2},
            {1/4/5/2,2/3/4/1,1/3/4/2},
            {1/3/4/2,2/4/3/1,1/4/3/2},
            {1/5/3/2,2/3/4/1,1/3/4/2},
            {1/4/5/2,2/4/3/1,1/4/3/2},
            {1/3/4/2,2/3/4/1,1/3/4/2}}{
            {(0,0)}/{(3,1)},
            {(0,6)}/{(3,7)},
            {(2,0)}/{(3,3)},
            {(2,2)}/{(3,5)},
            {(2,4)}/{(3,7)},
            {(0,0)}/{(1,4)},
            {(0,3)}/{(1,7)}}
        \hspace{\subfigurespacing}
        \edgecoloringtile{$\varphi\restrict{\Tcbpcaq}$}{
            {1/3/4/2,2/3/4/1,1/3/4/2},
            {1/4/3/2,2/4/3/1,1/5/3/2},
            {1/3/4/2,2/3/4/1,1/4/5/2},
            {1/4/3/2,2/4/3/1,1/3/4/2},
            {1/3/4/2,2/3/4/1,1/5/3/2},
            {1/4/3/2,2/4/3/1,1/4/5/2},
            {1/3/4/2,2/3/4/1,1/3/4/2}}{
            {(0,0)}/{(3,1)},
            {(0,6)}/{(3,7)},
            {(0,0)}/{(1,3)},
            {(0,2)}/{(1,5)},
            {(0,4)}/{(1,7)},
            {(2,0)}/{(3,4)},
            {(2,3)}/{(3,7)}}
    \end{minipage}
    \hspace{\subfigurespacing}
    \begin{minipage}{0.4\textwidth}
        \centering\noindent
        \edgecoloringtile{$\varphi\restrict{\Tcqadpa}$}{
            {1/3/4/2,2/3/4/5,5/3/4/1,1/3/4/2,2/3/4/5,5/3/4/1,1/3/4/2},
            {1/4/3/2,2/4/3/1,1/4/3/2,2/4/3/1,1/4/3/2,2/4/3/1,1/4/3/2},
            {1/3/4/2,2/3/4/1,1/3/4/2,2/3/4/1,1/3/4/2,2/3/4/1,1/3/4/2}}{
            {(0,0)}/{(1,3)},
            {(6,0)}/{(7,3)},
            {(0,2)}/{(3,3)},
            {(2,2)}/{(5,3)},
            {(4,2)}/{(7,3)},
            {(0,0)}/{(4,1)},
            {(3,0)}/{(7,1)}}
        \vspace{\subfigurespacing}

        \edgecoloringtile{$\varphi\restrict{\Tdpacqa}$}{
            {1/3/4/2,2/3/4/1,1/3/4/2,2/3/4/1,1/3/4/2,2/3/4/1,1/3/4/2},
            {1/4/3/2,2/4/3/1,1/4/3/2,2/4/3/1,1/4/3/2,2/4/3/1,1/4/3/2},
            {1/3/4/2,2/3/4/5,5/3/4/1,1/3/4/2,2/3/4/5,5/3/4/1,1/3/4/2}}{
            {(0,0)}/{(1,3)},
            {(6,0)}/{(7,3)},
            {(0,0)}/{(3,1)},
            {(2,0)}/{(5,1)},
            {(4,0)}/{(7,1)},
            {(0,2)}/{(4,3)},
            {(3,2)}/{(7,3)}}
    \end{minipage}
\caption[A 5-edge-coloring of $\Gamma_{1,2,3}$.]{A 5-edge-coloring of $\Gamma_{1,2,3}$.
Each node is assigned four colors representing the colors of its adjacent edges.
Also observe that the boundary nodes of each tile are colored in a consistent way,
e.g., $R_d$ is colored the same way everywhere it appears in the figure.}
\label{fig:edge-coloring}
\end{figure}  

\begin{thmn} \label{thm:ecn}
The continuous edge chromatic number of $F(2^{\Z^2})$ is $5$.
\end{thmn}

\begin{proof}

We first show that it is at least $5$. This, in fact, follows immediately from
Theorem~\ref{thm:npm}. For if there were a $4$-edge-coloring of
$F(2^{\Z^2})$, then every vertex of the Schreier graph has exactly
one edge of each color incident with it. Thus, for any fixed color, the
set of edges of that color gives a perfect matching in $F(2^{\Z^2})$.


We next show there is a continuous $5$-edge-coloring of $F(2^{\Z^2})$. Note that the condition for an edge coloring
can be described by forbidden patterns of width 1, thus we may apply Theorem~\ref{thm:tilethm} with $n=1$.
It suffices to find $p,q>1$ with $\gcd(p,q)=1$ such that $\gopq$ can be
$5$-edge-colored. We take $p=2$, $q=3$. A $5$-edge-coloring of
$\Gamma_{1,2,3}$ is shown in Figure~\ref{fig:edge-coloring}.
\end{proof}

The Borel edge chromatic number for $F(2^{\Z^2})$ is either $4$ or $5$ (as
the Borel number is less than or equal to the continuous number). We do not
know the actual answer, so we state this as a question.

\begin{quesn}
Is the Borel edge chromatic number for $F(2^{\Z^2})$ equal to $4$ or $5$?
\end{quesn}

{\it Remark.} Since the completion of an earlier version of this manuscript the above question has been answered independently by Greb\'{i}k--Rozho\v{n} \cite{GR}, Bencs--Hru\v{s}kov\'{a}--T\'oth \cite{BHT} and Weilacher \cite{Weilacher}. The Borel edge chromatic number for $F(2^{\Z^2})$ is $4$.

We also do not know the continuous chromatic edge number of $F(2^{\Z^n})$ for $n\geq 2$. Using Theorem~\ref{thm:npm}, a similar argument as above shows that it is at least $2n+1$.

In recent work Marks \cite{marks_mon} has shown, using techniques from
games and determinacy, that every Borel function
$\varphi \colon F(\omega^{F_2})\to 2$, where $F_2$ is the
free group with two generators, has an infinitely large monochromatic graph component.
In the case of $\FofZZ$, the situation is quite different. The proof
of the following result uses just the
positive direction of Theorem~\ref{thm:tilethm}.

\begin{thmn} \label{thm:moncomp}
There exists a continuous function $\varphi \colon \FofZZ \to 2$ such that
each monochromatic component is not only finite, it has size at most $3$.
\end{thmn}

\begin{proof}
We build the function $\varphi \colon F(2^{\Z^2})\to 2$ by constructing a function
$\psi\colon \gnpq \to 2$ for $n=1$, $p=5$, $q=2$. The function $\psi$ is constructed as
in Figure~\ref{fig:finite-monochromatic-components}. Theorem~\ref{thm:mtr}
produces a continuous $\Z^2$-homomorphism $\varphi' \colon F(2^{\Z^2})\to \Gamma_{1,5,2}$.
Let $\varphi= \psi \circ \varphi' \colon F(2^{\Z^2})\to 2$.

\begin{figure}[htbp]
    \colorlet{color8}{black!40}
    \colorlet{color9}{black!0}

    \newcommand{\figurescale}{0.4}
    \newcommand{\subfigurespacing}{.5cm}
    \begin{minipage}{1\textwidth}
        \centering
        \begin{tikzpicture}[scale=\figurescale]
            \begin{scope}[yscale=-1,local bounding box=bound]
                \foreach \row [count=\i] in {
                    {8}} {
                    \foreach \clrNum [count=\j] in \row {
                        \filldraw[fill=color\clrNum, draw=black] (\j,\i) rectangle ++(-1,-1);
                    }
                }
                \draw[draw=black, ultra thick] (0,0) rectangle (1,1);
            \end{scope}
            \node[below=0.1cm of bound.south] {$\psi\restrict{R_\times}$};
        \end{tikzpicture}   
        \hspace{\subfigurespacing}
        \begin{tikzpicture}[scale=\figurescale]
            \begin{scope}[yscale=-1,local bounding box=bound]
                \foreach \row [count=\i] in {
                    {9},
                    {8},
                    {9},
                    {9}} {
                    \foreach \clrNum [count=\j] in \row {
                        \filldraw[fill=color\clrNum, draw=black] (\j,\i) rectangle ++(-1,-1);
                    }
                }
                \draw[draw=black, ultra thick] (0,0) rectangle (1,4);
            \end{scope}
            \node[below=0.1cm of bound.south] {$\psi\restrict{R_a}$};
        \end{tikzpicture}   
        \hspace{\subfigurespacing}
        \begin{tikzpicture}[scale=\figurescale]
            \begin{scope}[yscale=-1,local bounding box=bound]
                \foreach \row [count=\i] in {
                    {9}} {
                    \foreach \clrNum [count=\j] in \row {
                        \filldraw[fill=color\clrNum, draw=black] (\j,\i) rectangle ++(-1,-1);
                    }
                }
                \draw[draw=black, ultra thick] (0,0) rectangle (1,1);
            \end{scope}
            \node[below=0.1cm of bound.south] {$\psi\restrict{R_b}$};
        \end{tikzpicture}   
        \hspace{\subfigurespacing}
        \begin{tikzpicture}[scale=\figurescale]
            \begin{scope}[yscale=-1,local bounding box=bound]
                \foreach \row [count=\i] in {
                    {9,8,9,9}} {
                    \foreach \clrNum [count=\j] in \row {
                        \filldraw[fill=color\clrNum, draw=black] (\j,\i) rectangle ++(-1,-1);
                    }
                }
                \draw[draw=black, ultra thick] (0,0) rectangle (4,1);
            \end{scope}
            \node[below=0.1cm of bound.south] {$\psi\restrict{R_c}$};
        \end{tikzpicture}   
        \hspace{\subfigurespacing}
        \begin{tikzpicture}[scale=\figurescale]
            \begin{scope}[yscale=-1,local bounding box=bound]
                \foreach \row [count=\i] in {
                    {9}} {
                    \foreach \clrNum [count=\j] in \row {
                        \filldraw[fill=color\clrNum, draw=black] (\j,\i) rectangle ++(-1,-1);
                    }
                }
                \draw[draw=black, ultra thick] (0,0) rectangle (1,1);
            \end{scope}
            \node[below=0.1cm of bound.south] {$\psi\restrict{R_d}$};
        \end{tikzpicture}   
    \end{minipage} \\
    \vspace{\subfigurespacing}
    \begin{minipage}{1\textwidth}
        \centering
        \begin{tikzpicture}[scale=\figurescale]

            \begin{scope}[yscale=-1,local bounding box=bound]
                \foreach \row [count=\i] in {
                    {8,9,8,9,9,8},
                    {9,8,9,8,8,9},
                    {8,9,9,8,9,8},
                    {9,8,8,9,8,9},
                    {9,8,9,8,8,9},
                    {8,9,8,9,9,8}} {
                    \foreach \clrNum [count=\j] in \row {
                        \filldraw[fill=color\clrNum, draw=black] (\j,\i) rectangle ++(-1,-1);
                    }
                }
                \draw[draw=black, ultra thick] (0,0) rectangle (1,6);
                \draw[draw=black, ultra thick] (5,0) rectangle (6,6);
                \draw[draw=black, ultra thick] (0,0) rectangle (6,1);
                \draw[draw=black, ultra thick] (0,5) rectangle (6,6);
            \end{scope}
            \node[below=0.1cm of bound.south] {$\psi\restrict{\Tcaca}$};
        \end{tikzpicture}   
        \hspace{\subfigurespacing}
        \begin{tikzpicture}[scale=\figurescale]
            \begin{scope}[yscale=-1,local bounding box=bound]
                \foreach \row [count=\i] in {
                    {8,9,8,9,9,8},
                    {9,8,9,8,8,9},
                    {8,9,8,9,9,8}} {
                    \foreach \clrNum [count=\j] in \row {
                        \filldraw[fill=color\clrNum, draw=black] (\j,\i) rectangle ++(-1,-1);
                    }
                }
                \draw[draw=black, ultra thick] (0,0) rectangle (1,3);
                \draw[draw=black, ultra thick] (5,0) rectangle (6,3);
                \draw[draw=black, ultra thick] (0,0) rectangle (6,1);
                \draw[draw=black, ultra thick] (0,2) rectangle (6,3);
            \end{scope}
            \node[below=0.1cm of bound.south] {$\psi\restrict{\Tcbcb}$};
        \end{tikzpicture}   
        \hspace{\subfigurespacing}
        \begin{tikzpicture}[scale=\figurescale]
            \begin{scope}[yscale=-1,local bounding box=bound]
                \foreach \row [count=\i] in {
                    {8,9,8},
                    {9,8,9},
                    {8,9,8}} {
                    \foreach \clrNum [count=\j] in \row {
                        \filldraw[fill=color\clrNum, draw=black] (\j,\i) rectangle ++(-1,-1);
                    }
                }
                \draw[draw=black, ultra thick] (0,0) rectangle (1,3);
                \draw[draw=black, ultra thick] (2,0) rectangle (3,3);
                \draw[draw=black, ultra thick] (0,0) rectangle (3,1);
                \draw[draw=black, ultra thick] (0,2) rectangle (3,3);
            \end{scope}
            \node[below=0.1cm of bound.south] {$\psi\restrict{\Tdbdb}$};
        \end{tikzpicture}   
    \end{minipage} \\
    \vspace{\subfigurespacing}
    \begin{minipage}{1\textwidth}
        \centering
        \begin{tikzpicture}[scale=\figurescale]
            \begin{scope}[yscale=-1,local bounding box=bound]
                \foreach \row [count=\i] in {
                    {8,9,8,9,9,8,9,8,9,9,8},
                    {9,8,9,8,8,9,8,9,8,8,9},
                    {8,9,8,9,9,8,9,8,9,9,8},
                    {9,8,9,8,8,9,8,9,8,8,9},
                    {9,8,9,8,9,8,9,8,9,8,9},
                    {8,9,8,9,8,9,8,9,8,9,8}} {
                    \foreach \clrNum [count=\j] in \row {
                        \filldraw[fill=color\clrNum, draw=black] (\j,\i) rectangle ++(-1,-1);
                    }
                }
                \draw[draw=black, ultra thick] ( 0,0) rectangle ( 1,6);
                \draw[draw=black, ultra thick] (10,0) rectangle (11,6);
                \draw[draw=black, ultra thick] ( 0,0) rectangle (11,1);
                \draw[draw=black, ultra thick] ( 0,5) rectangle (11,6);

                \draw[draw=black, ultra thick] (5,0) rectangle ++(1,1);
                \draw[draw=black, ultra thick] (2,5) rectangle ++(1,1);
                \draw[draw=black, ultra thick] (4,5) rectangle ++(1,1);
                \draw[draw=black, ultra thick] (6,5) rectangle ++(1,1);
                \draw[draw=black, ultra thick] (8,5) rectangle ++(1,1);
            \end{scope}
            \node[below=0.1cm of bound.south] {$\psi\restrict{\Tcqadpa}$};
        \end{tikzpicture}   
        \hspace{\subfigurespacing}
        \begin{tikzpicture}[scale=\figurescale]
            \begin{scope}[yscale=-1,local bounding box=bound]
                \foreach \row [count=\i] in {
                    {8,9,8,9,9,8,9,8},
                    {9,8,9,8,8,9,8,9},
                    {8,9,8,9,9,8,9,8},
                    {9,8,9,8,8,9,8,9},
                    {9,8,9,8,9,8,8,9},
                    {8,9,8,9,8,9,9,8}} {
                    \foreach \clrNum [count=\j] in \row {
                        \filldraw[fill=color\clrNum, draw=black] (\j,\i) rectangle ++(-1,-1);
                    }
                }
                \draw[draw=black, ultra thick] (0,0) rectangle (1,6);
                \draw[draw=black, ultra thick] (7,0) rectangle (8,6);
                \draw[draw=black, ultra thick] (0,0) rectangle (8,1);
                \draw[draw=black, ultra thick] (0,5) rectangle (8,6);

                \draw[draw=black, ultra thick] (5,0) rectangle ++(1,1);
                \draw[draw=black, ultra thick] (2,5) rectangle ++(1,1);
            \end{scope}
            \node[below=0.1cm of bound.south] {$\psi\restrict{\Tcdacda}$};
        \end{tikzpicture}   
    \end{minipage}
    \caption[A function of $\psi \colon\Gamma_{1,5,2}\to 2$ inducing a continuous
function $\varphi \colon \FofZZ\to 2$ which has no infinite monochromatic graph component.]
{\label{fig:finite-monochromatic-components}
A function  $\psi\colon\Gamma_{1,5,2}\to 2$
inducing a continuous function $\varphi \colon \FofZZ\to 2$ which has no infinite monochromatic
graph component.  Five of the twelve tiles are given; the other seven are
rotations and/or reflections of these five.}
\label{fig:finite-monochromatic-components}
\end{figure}  

We need to observe that each monochromatic component of $\varphi$ has size at most $3$.
From Figure~\ref{fig:finite-monochromatic-components} we see that each boundary point of
one of the $G^i_{1,5,2}$ (which are illustrated by boundary boxes in the figure)
is adjacent to interior points (i.e., points in unlabeled blocks) of a different color.
Also, the largest monochromatic component of the boundary of any $G^i_{1,5,2}$
is at most $2$ by inspection. Thus, if $\varphi'(x)$ is within some labeled block, then the monochromatic component containing $x$ has size at most $2$.
If $\varphi'(x)$ is an interior point of one of the $G^i_{1,5,2}$, then the monochromatic component
of $x$ must consist of points $y$ with $\varphi'(y)$ in the interior of the same $G^i_{1,5,2}$.
By inspection of Figure~\ref{fig:finite-monochromatic-components} we see that
this component has size at most $3$.
\end{proof}

Next we consider tiling questions about $F(2^{\Z^2})$. The general problem can be formulated
as follows. Let $T_1,T_2,\dots$ be a finite or countably infinite collection of
finite subsets of $\Z^2$ (which we call the ``tiles'').

\begin{defnn} \label{def:tiling}
We say $F(2^{\Z^2})$ has a continuous (or Borel) tiling by $\{ T_i\}$ if there is a clopen
(or Borel) subequivalence relation $E$ of $F(2^{\Z^2})$ such that every $E$ class is
isomorphic to one of the $T_i$. More precisely, for every $x \in F(2^{\Z^2})$,
$\{ g \in \Z^2 \colon g \cdot x\, E x\}$ is a translate of some $T_i$.
\end{defnn}

This leads to the following general question.

\begin{quesn} \label{ques:tiling}
For which sets of tiles $\{ T_i\}$ is there a continuous tiling
of $F(2^{\Z^2})$? For which sets of tiles is there a Borel tiling?
\end{quesn}

For finite sets of tiles, Theorems~\ref{thm:tilethm}, \ref{thm:tilethmb}  give a theoretical
answer to the continuous version of the question.
Namely, if we take $n$ larger than the maximum diameter of
a tile $T_i$, then there is a continuous tiling of $F(2^{\Z^2})$ by the $\{ T_i\}$
if and only if there are $p,q >n$ with $\gcd(p,q)=1$ such that there is a tiling of
$\gnpq$ by the $\{ T_i\}$ (as $n$ is greater than any tile diameter,
the notion of a tiling of $\gnpq$ by the $T_i$ is naturally defined).
This does not give a decision procedure for the continuous tiling question, however,
and even in some specific fairly simple cases we do not have an answer.

Let us consider one specific instance of the continuous tiling question
which is connected with the theory of marker structures.
In \cite{gao_countable_2015} it is shown that for every $d>1$ there is a tiling
of $F(2^{\Z^n})$ by rectangular regions, each of which has side lengths in
$\{ d, d+1\}$. In particular, $F(2^{\Z^2})$ has a continuous tiling
using the four rectangular tiles of dimensions $\{ d\times d, d\times (d+1),
(d+1)\times d, (d+1)\times (d+1) \}$.
We call this set of four tiles the $d,d+1$ tiles.
A natural question is to ask
for which subsets of this set can we get a continuous tiling of $F(2^{\Z^2})$.

The next result shows that if we omit the ``small''  $d \times d$ tile
or the ``large'' $(d+1)\times (d+1)$ tile, then a continuous tiling
is impossible.

\begin{thmn} \label{thm:negthreetiles}
There is no continuous tiling of $F(2^{\Z^2})$ by the $d,d+1$ tiles such that
every equivalence class omits either the $d\times d$ tile or the $(d+1)\times (d+1)$
tile. In particular, if we omit the $d\times d$ or the $(d+1)\times (d+1)$
tile from the set of $d, d+1$ tiles, then there is no continuous tiling of
$F(2^{\Z^2})$ by this set of tiles.
\end{thmn}

\begin{proof}
Suppose there is a continuous tiling of $F(2^{\Z^2})$ by the $d,d+1$ tiles,
with each class omitting either the small or large tile.
The proof of Theorem~\ref{thm:tilethm} shows that the tiling of $\gnpq$
(we assume $n>d+1$) produced comes from a single equivalence class,
that of the hyper-aperiodic element $x$ constructed. Thus, we get a tiling of
$\gnpq$ for all sufficiently large $p,q$ without using the small tile,
or without using the large tile. Say we tile $\gnpq$ without using the large tile.
Since $n>d+1$, this gives a tiling of the torus tile $\Tcaca$
without the large tile. Note that the available tiles all have a number
of vertices which is divisible by $d$. However, if we take $p$ to
be relatively prime to $d$, then $\Tcaca$ has $p^2$ vertices which is
also relatively prime to $d$, so not divisible by $d$, a contradiction. If, on the other hand, the small tile is omitted, then the number of vertices in each available tile is divisible by $d+1$. This leads to a contradiction if we similarly consider $\Tcaca$ and take $p$ to be relatively prime to $d+1$.
\end{proof}

\section{Continuous subactions and regular tilings}\label{section:3.2}
In this section we provide a negative answer to the continuous version of Problem~10.6 of \cite{gao_countable_2015}.
Specifically, we will show in Corollary~\ref{nlum} that there does not exist a clopen
rectangular marker structure on $F(2^{\Z^2})$ which is too regular. We will make this notion of regularity precise later in the section. To prove this we will first demonstrate a result
concerning continuous subactions of $\fzs$ on clopen complete sections.
We make this notion of subaction precise in the following definition.

\begin{defnn} \label{def:continuous-action}
Let $X$ be a space with a group action $(G,\cdot)$.  Let $Y\subseteq X$, and $(H,\star)$ an
action of $H$ on $Y$.  We say $(Y, H,\star)$ is a {\em subaction} of $(X, G, \cdot)$ if for all
$y_1, y_2\in Y$, there is an $h \in H$ with $h \star y_1=y_2$ if and only if there is a $g \in G$
such that $g \cdot y_1=y_2$. If $X$ is a topological space and $G, H$ are topological groups, then we say the subaction is {\em continuous} if there is a continuous $\varphi \colon
H \times Y \to G$, where the topology on $Y$ is the subspace topology inherited from $X$, such that $h\star y=\varphi(h,y) \cdot y$ for all $y\in Y$ and $h \in H$.
\end{defnn}

Note that if $(X, G, \cdot)$ is a continuous action and the subaction $(Y, H, \star)$ is continuous then as an action $(Y, H, \star)$ is continuous.
If $H=G$ in Definition~\ref{def:continuous-action}, we just write $(Y,\star)$
for the continuous subaction.
The next theorem says that there are no proper continuous subactions on clopen
complete, co-complete sections of $\fzs$.

\begin{thmn} \label{thm:no-sub-Z2-action}
There does not exist a clopen, complete, co-complete section  $M\subseteq \fzs$ and
a continuous free subaction $(M,\star)$ of $\fzs$.
\end{thmn}

\begin{proof}

Suppose towards a contradiction that $M\subseteq \fzs$ is a clopen complete, co-complete section
and $\star$ is a continuous free subaction of $\Z^2$ on $M$.
Let $x \in \fzs$ be the hyper-aperiodic element constructed in \S\ref{subsec:negtile}.
Let $K=\ocl{[x]} \subseteq \fzs$. Then $K$ is compact. As $M$ is  a complete section,
$[y] \cap M\neq \emptyset$ for any $y\in K$. As $M$ is co-complete, we also have $[y] \cap (2^{\Z^2}-M)
\neq \emptyset$ for any $y\in K$.

In the rest of the proof we will use the following terminology. 
Given $y\in F(2^{\Z^2})$ and $n>0$, the {\em $n$-ball about $y$} is the set $\{ z\in [y]\,:\, \rho(y,z)\leq n\}$, where $\rho$ is the distance function defined in \S\ref{section:1.1}. Given any $y\in F(2^{\Z^2})$, the {\em first quadrant} from $y$ is the set of all points $(a_1, a_2)\cdot y$ where $a_1, a_2> 0$. Similarly one can define the second, third, and fourth quadrants from a point $y\in F(2^{\Z^2})$.

By the compactness of $K$ and the clopen-ness of $M$, there is an $n_0\in\Z^+$ such that for any $y\in K$, whether $y\in M$ is completely determined by $y\!\upharpoonright\! [-n_0, n_0]^2$. Suppose $\varphi: \Z^2\times M\to \Z^2$ witnesses the continuity of the subaction $(M,\star)$. Since $M\cap K$ is compact and $\varphi$ is continuous, the set $E=\{\varphi(\pm e_1, y), \varphi(\pm e_2, y)\,:\, y\in M\cap K\}$ is finite. This means that  there is an $n_1\in\Z^+$ such that for any $y\in M\cap K$, $\rho(y, \pm e_1\star y), \rho(y, \pm e_2\star y)\leq n_1$; in other words, for all $y\in M\cap K$ and all $\zeta\in\{\pm e_1, \pm e_2\}$, $\zeta\star y$ is within the $n_1$-ball about $y$. In addition, we claim there is an $n_2\in\Z^+$ such that
\begin{enumerate}
\item the values $\varphi(\zeta, y)$ for all $\zeta\in\{\pm e_1, \pm e_2\}$ and $y\in M\cap K$ are completely determined by $y\!\upharpoonright\![-n_2, n_2]^2$, and
\item for any $\zeta\in\{\pm e_1, \pm e_2\}$, $y\in M\cap K$ and $\bar{y}\in M$, if $\bar{y}\!\upharpoonright\![-n_2, n_2]^2=y\!\upharpoonright\![-n_2, n_2]^2$, then $\varphi(\zeta,\bar{y})=\varphi(\zeta,y)$.
\end{enumerate}
To see this claim, we fix without loss of generality a $\zeta\in\{\pm e_1, \pm e_2\}$ and an element $g\in E$, and consider the set $W=\{ z\in M\,:\, \varphi(\zeta, z)=g\}$. $W$ is clopen in $F(2^{\Z^2})$. We write $W$ as a disjoint union of basic open neighborhoods
$$ W=\bigcup_{i\in \Z^+} N_{s_i}, $$
where each $s_i: [-m_i, m_i]^2\to\{0,1\}$ for some $m_i\in\Z^+$, and $N_{s_i}=\{u\in F(2^{\Z^2})\,:\, u\!\upharpoonright\![-m_i, m_i]^2=s_i\}$. Note that $W\cap K$ is compact, and therefore there is $k\in\Z^+$ such that
$$ W\cap K=\bigcup_{i\leq k} (N_{s_i}\cap K). $$
Without loss of generality, assume $N_{s_i}\cap K\neq\emptyset$ for all $i\leq k$. Then for all $y\in M\cap K$, $\varphi(\zeta, y)=g$ if and only if $y\in \bigcup_{i\leq k} N_{s_i}$. Moreover, for any $\bar{y}\in M$, if $\bar{y}\in \bigcup_{i\leq k} N_{s_i}$, we also have $\varphi(\zeta, \bar{y})=g$. This proves the claim.

Since $M$ is a complete section, we have that
$$ \bigcup_{n>0} [-n,n]^2\cdot M=F(2^{\Z^2}). $$
By the openness of $M$ and the compactness of $K$, there is an $n_3\in \Z^+$ such that $[-n_3, n_3]^2\cdot (M\cap K)=K$. In other words, for any $y\in K$ there is a $u\in M$ with $\rho(y, u)\leq n_3$, or the $n_3$-ball about $y$ contains a point $u\in M$. Similarly, since $M$ is also a co-complete section and $M$ is closed, there is an $n_4\in\Z^+$ such that for any $y\in K$ the $n_4$-ball about $y$ contains a point $v\notin M$. Let $$N_0=5\max\{n_0, n_1, n_2, n_3, n_4\}.$$
Since $N_0>2\max\{n_3, n_4\}+1$, it follows that the $N_0$-ball about any $y\in K$ contains points $u \in M$, $v\notin M$ in each of the
four quadrants from $y$.

Let
$$A=\{(u, g\cdot u)\,:\, u\in M, g\in [-2N_0, 2N_0]^2, g\cdot u\in M\}$$
and for each $n>0$
$$ B_n=\{(u, h\star u)\,:\, u\in M, h\in [-n, n]^2\}. $$
Note that $A\subseteq \bigcup_{n>0} B_n$. Moreover $A\cap K^2$ is compact and, since $\star$ is a continuous action, each $B_n$ is open. Thus there is an $n_5\in\Z^+$ such that $A\cap K^2\subseteq \bigcup_{n\leq n_5} B_n$. Since $\star$ is also a free action, we get that for each pair $(u, v)\in (M\cap K)^2$ with $v$ being in the $2N_0$-ball about $u$, $\varphi(g, u)\in [-n_5, n_5]^2$, where $g\in [-2N_0,2N_0]^2$ is the unique group element with $g\cdot u=v$. Using the above finite set $E$, we get  an $n_6\in\Z^+$ such that for each pair $(u, v)\in (M\cap K)^2$ with $v$ being in the $2N_0$-ball about $u$, there is a $\star$-path from $u$ to $v$ which
stays inside the $n_6$-ball about $u$. By a {\em $\star$-path} we mean a sequence $u=w_0,w_1,\dots,w_l=v$
such that for each $0\leq i<l$, $w_{i+1}= \zeta \star w_i$ for some $\zeta\in\{\pm e_1, \pm e_2\}$. Let
$$ N_1=N_0+n_6. $$
Then for any $y \in K$ and any $u,v \in M$ within the $N_0$-ball about $y$, there is a $\star$-path from $u$ to $v$ which
stays inside the $N_1$-ball about $y$.

Consider now the hyper-aperiodic element $x\in K$. Recall that for any
$n$ and sufficiently large $p, q$, there is a copy of $G^i_{n, p, q}$
for each $1\leq i\leq 12$ in the domain of $x$. We fix $n\geq
2N_1+4N_0$ and sufficiently large distinct primes $p, q>n$, and
consider a particular copy of $G^1_{n, p, q}$ in the domain of
$x$. The alternative notation for this copy of grid graph is
$\Gcaca$. We think of $\Gcaca$ both as a sub-grid-graph of $\Z^2$ and
as a subset of $[x]$ that corresponds to the points $\Gcaca\cdot
x$. Consider the $R_\times$ block in the upper-left corner of this
copy of $\Gcaca$. Let $C$ be the $N_0\times N_0$ rectangle at the
center of this $R_\times$ block. From the definition of $N_0$, there
is a $g\in C$ such that $g\cdot x\in M$. We fix this $g$ and let
$m=g\cdot x\in M$. Let $m'=(p, 0)\cdot m$ and $m''=(0,-p)\cdot
m$. Then $m'$ corresponds to the point in the same position within the
$R_\times$ block in the upper-right corner of this copy of $\Gcaca$,
and $m''$ corresponds to the point in the same position within the
$R_\times$ block in the lower-left corner of this copy of
$\Gcaca$. The top part of this copy of $\Gcaca$ along with the
elements $m$ and $m'$ are illustrated in
Figure~\ref{fig:subaction-finding-c}.

\begin{figure}[htbp]
\scalebox{0.8}{
\begin{tikzpicture}[x=1.2cm,y=1.2cm]
\tikzset{
    pics/RXZ2/.style={
        code={
            \node[anchor = north west] at (-1.5,1.5) {\tikzpictext};
            \node[rectangle, draw, minimum width=3.6cm, minimum height=3.6cm] at (0,0) {};
        }
    }
}
\pic[pic text={$R_\times$}] at (0,0) {RXZ2};
\pic[pic text={$R_\times$}] at (9,0) {RXZ2};

\node[anchor = north west] at (1.5,1.5) {$R_c$};
\node[rectangle, draw, minimum width=7.2cm, minimum height=3.6cm] at (4.5,0) {};

\node[anchor = north west] at (-.75,.75) {$R$};
\node[rectangle, draw, minimum width=12.6cm, minimum height=1.8cm] at (4.5,0) {};

\draw[dotted, every node/.style={nodeStyle, solid, thin, minimum size=3pt}]
(0.01,-0.03)   node[fill=black] (m2) {} --
(0.81,0.1)   node[fill=black]  (m3) {} --
(1.58,0.14)   node[fill=black]  (m4) {} --
(2.28,0.17)   node[fill=black]  (m5) {} --
(3.09,0.31)   node[fill=black]  (m6) {} --
(3.09,0.31)   node[fill=black] (m7) {} --
(2.75,-0.15)   node (m8) {} --
(2.47,-0.41)   node (m9) {} --
(2.22,-0.92)   node (m10) {} --
(2.5,-1.15)   node (m13) {} --
(3.02,-1.1)   node (m14) {} --
(3.42,-1.2)   node (m15) {} --
(3.9,-1.05)   node (m17) {} --
(4.1, -.85)   node (m18) {}--
(3.8, -.55) node (m19) {}--
(3.4, -.2) node (m20) {}--
(3.38,0.11)   node[fill=black]  (m21) {} --
(3.85,-0.24)   node[fill=black]  (m22) {} --
(4.4,-0.09)   node (m23)[fill=black] {} --
(4.8,0.06)   node (m24)[fill=black]  {} --
(4.6, 0.5) node (m25) {}--
(4.8, 0.9) node (m26) {}--
(5.0, 1.1) node (m27) {}--
(5.1, 0.8) node (m28) {}--
(5.2, 0.3) node (m29) {}--
(5.31,0.11)   node (m34)[fill=black]  {} --
(5.64,0.33)   node (m33)[fill=black]  {} --
(6.04,0.39)   node (m32)[fill=black]  {} --
(6.55,0.5)   node (m31)[fill=black]  {} --
(7.06,0.2)   node[fill=black]  (m45) {} --
(7.5,0.21)   node[fill=black]  (m46) {} --
(8.25,0.36)   node[fill=black]  (m47) {} --
(8.73,0.38)   node[fill=black]  (m48) {} --
(9.48,0.49)   node (m49) {} --
(10,0.46)   node (m50) {} --
(10,0.46)   node (m51) {} --
(9.56,0.28)   node (m52) {} --
(9.01,-0.03)   node (m53)[fill=black]  {};


\node[anchor = north] at (m2) {$m$};
\node[anchor = north] at (m53) {$m'$};
\end{tikzpicture}
}
\caption[The top part of $\Gcaca$. The outer rectangle has dimensions $(p+n)\times n$ with $n>2N_1+4N_0$;
the inner rectangle is positioned in the center and has dimensions $2N_0\times p$. The $\star$-path from $m$ to $m'$
stays inside the outer rectangle.]{\label{fig:subaction-finding-c}
The top part of $\Gcaca$. The outer rectangle has dimensions $n\times (p+n)$ with $n>2N_1+4N_0$;
the inner rectangle is positioned in the center and has dimensions $(p+2N_0)\times 2N_0$.
The $\star$-path from $m$ to $m'$
stays inside the outer rectangle.}
\end{figure}

In Figure~\ref{fig:subaction-finding-c}, the outer rectangle shows the
top part of $\Gcaca$ and has dimensions $(p+n)\times n $. We also
consider an inner rectangle of dimensions $(p+2N_0)\times 2N_0$ that is
positioned in the center of the outer rectangle. This region is
denoted as $R$. We see that $m, m'\in R$.

Consider the point $m$. By the definition of $N_0$, in the first or
fourth quadrant from $m$ there is another point $m_1$ in $M$ within
the $N_0$-ball of $m$, and with $x$-coordinate strictly greater than that of $m$.
From the definition of $N_1$, there is a
$\star$-path from $m$ to $m_1$ that stays within the $N_1$-ball of
$m$, and therefore stays within the outer rectangle. Continue this
construction with $m_1$ replacing $m$, and we find $m_2\in R$ to the
right of $m_1$ and a $\star$-path from $m_1$ to $m_2$ that stays
within the outer rectangle. Note that the intersection of the $N_0$-ball about $m_1$ with either the first quadrant from $m_1$ or the fourth quadrant from $m_1$ is entirely contained in the region $R$, thus we can guarantee that $m_2\in R$. This construction can be repeated until
the $\star$-path extends to $m'$. We have thus found a $\star$-path
from $m$ to $m'$ with intermediate endpoints in $R$ and the entire
path staying within the outer rectangle. This is illustrated in
Figure~\ref{fig:subaction-finding-c}.

Let $h$ be the unique element of $\Z^2$ such that $h \star m=m'$. Thus,
$h=\zeta_k\cdot \zeta_{k-1} \cdots \zeta_0$, where each $\zeta_i\in \{ \pm e_1, \pm e_2\}$
and $m, \zeta_0\star m, \zeta_1\star \zeta_0\star m, \dots, \zeta_k \star\dots \star \zeta_0 \star m=m'$ is the
$\star$-path from $m$ to $m'$ which stays inside the outer rectangle.
Likewise, there is an element $h' \in \Z^2$ such that $h' \star m=m''$ and a $\star$-path from $m$ to $m''$ that also stays inside the left part of $\Gcaca$. Let $H=\langle h,h'\rangle \leq \Z^2$, and let $G=\Z^2/H$.

Let $T=([0, p-1]\times [-p+1,0])\cdot m$ and $M_T=M\cap T$. Then $T$ corresponds to the $p\times p$ rectangular region with $m$ at the upper-left corner, and $M_T$ is the set of all points of $M$ in $T$. In the following we will define an action $\ast$ of $G$ on $M_T$.

Before defining $\ast$ we define another action $\star'$ of $\Z^2$ on $M_T$. We first define this action for the standard generators of $\Z^2$. Suppose $\zeta\in\{\pm e_1, \pm e_2\}$. Let $u\in M_T$. If $\zeta\star u\in M_T$ then let $\zeta\star' u=\zeta\star u$. Assume $\zeta\star u\notin M_T$.  Then by the definition of $N_0$, both $u, \zeta\star u$ must be within $\rho$ distance $N_0$ to the boundary of $T$, and therefore each of $u, \zeta\star u$ must be a point in one of the labeled blocks of $\Gcaca$. Consider another block of $\Gcaca$ with the same label and let $u'$ be the point within this block in the same position as $u$. Then $u'=(a_1, a_2)\cdot u$ with some $a_1, a_2\equiv 0\mod p$. By the definition of $N_0$, we have that
$\zeta\star u'=(a_1, a_2)\cdot (\zeta\star u)$. Note that exactly one of such $\zeta\star u'$ is contained in $T$. We let $\zeta\star' u$ be this unique $\zeta\star u'\in M_T$. This finishes the definition of $\zeta\star' u$ for all $u\in M_T$.

In general, suppose $g\in\Z^2$ and $u\in M_T$, we write $g=\zeta_k\zeta_{k-1}\dots \zeta_0$ where $\zeta_i\in\{\pm e_1, \pm e_2\}$, and define
$$ g\star' u=\zeta_k\star' \zeta_{k-1}\dots \star' \zeta_0\star' u. $$
We must verify that the definition does not depend on the decomposition of $g$ into a product of standard generators. For this we show that if $\zeta_1, \dots, \zeta_k$ is a sequence of standard generators of $\Z^2$ such that $\sum \zeta_i=(0,0)$, then $\zeta_k\star'\zeta_{k-1}\dots\star'\zeta_1\star' u=u$ for all $u\in M_T$. Fix such a sequence of standard generators and $u\in M_T$. We construct an element $\bar{x}\in F(2^{\Z^2})$ as follows. In the domain of $\bar{x}$ first lay down a copy of $\Gcaca$ in exactly the same position as the copy of $\Gcaca$ in the domain of $x$. Then we put down a tessellation of copies of $\Gcaca$ with overlapping boundaries as illustrated in Figure~\ref{fig:algorithmI}, with the first copy of $\Gcaca$ in the center of the tessellation, and with $2k+1$ many copies of $\Gcaca$ horizontally and $2k+1$ many copies vertically. So the total number of copies of $\Gcaca$ is $(2k+1)^2$. We then define $\bar{x}$ so that the values of $\bar{x}$ on any copy of $\Gcaca$ is the same as the values of $x$ on $\Gcaca$, and outside the tessellation of copies of $\Gcaca$ the values of $\bar{x}$ are defined in any way to make it an element of $F(2^{\Z^2})$. Let $\bar{u}\in[\bar{x}]$ be the corresponding point to $u\in[x]$. By the definition of $N_0$, we have that $\bar{u}\in M$ and that for any standard generator $\zeta$, $\varphi(\zeta, u)=\varphi(\zeta, \bar{u})$ since $u\!\upharpoonright\![-N_0, N_0]^2=\bar{u}\!\upharpoonright\![-N_0,N_0]^2$. By the same argument we can see that for all $v\in[x]\cap M$ and $\bar{v}\in [\bar{x}]$, if $v$ is within $\Gcaca$ in the definition of $x$, $\bar{v}$ is within any one of the $(2k-1)^2$ many copies of $\Gcaca$ in the definition of $\bar{x}$ that are not on the boundary of the tessellation, and $v$ and $\bar{v}$ are in corresponding positions, then $\bar{v}\in M$ and $\varphi(\zeta, v)=\varphi(\zeta, \bar{v})$. As a consequence, if we compare the two sequences
$$ u, \zeta_1\star' u, \zeta_2\star' \zeta_1\star' u, \dots, \zeta_k\star' \dots\star'\zeta_1\star' u $$
and
$$ \bar{u}, \zeta_1\star \bar{u}, \zeta_2\star\zeta_1\star \bar{u}, \dots, \zeta_k\star\dots\star\zeta_1\star \bar{u}, $$
we conclude that the corresponding terms are points in corresponding positions in some copies of $\Gcaca$. Since each sequence has $k+1$ terms, the second sequence of points is entirely contained in the tessellation of copies of $\Gcaca$ in the construction of $\bar{x}$. Since $\star$ is a free action and $\sum \zeta_i=(0,0)$, we see that $\zeta_k\star\dots\star\zeta_1\star \bar{u}=\bar{u}$. By the above comparison, we have $\zeta_k\star' \dots\star'\zeta_1\star' u=u$ as required.

We have thus defined an action $\star'$ of $\Z^2$ on $M_T$. Next we
show that this action induces an action $*$ of $G$ on $M_T$. First we
claim that $h\star' m=m$ and $h'\star' m=m$. To see this, consider the
$\star$-path from $m$ to $m'$ that stays within the top part of
$\Gcaca$. We now make a stronger observation that every point on this
$\star$-path is in fact of $\rho$-distance at least $N_0$ away from the
boundary of the top part of $\Gcaca$. This is because $n>2N_1+4N_0$
and that the intermediate endpoints used in the construction of the
$\star$-path are all within the $2N_0\times p$ region $R$. Write the
sequence of points on the $\star$-path as
$$ m, \zeta_1\star m, \dots, \zeta_k\star\dots\star\zeta_1\star m, $$
where $h=\sum_{i\leq k} \zeta_i$. Then the $N_0$-ball about each of
these points are still contained in the top part of $\Gcaca$. Compare
it with another sequence
$$ m, \zeta_1\star' m, \dots, \zeta_k\star'\dots\star'\zeta_1\star'
m, $$ which by definition is contained in $M_T$. We conclude that the
corresponding terms are in corresponding positions within some labeled
blocks of $\Gcaca$. Since $h\star m=m'$, this comparison gives that
$h\star' m=m$. Similarly $h'\star' m=m$ because $h'\star m=m''$. Next
we claim that for any $u\in M_T$, $h\star' u=u$ and $h'\star'
u=u$. Fix $u\in M_T$. We construct a $\star$-path from $m$ to $u$
which stays entirely within $\Gcaca$ and so that every point on the
$\star$-path is of $\rho$ distance $N_0$ away from the boundary of
$\Gcaca$. This $\star$-path is constructed in a similar fashion as the
$\star$-path from $m$ to $m'$, making use of the fact that from any
point $v$ of $M$ there are other points of $M$ within the $N_0$-ball
of $v$ in each of the four quadrants from $v$. Let $t\in\Z^2$ be the
unique element such that $t\star m=u$. Then the $\star$-path
corresponds to a decomposition of $t$ as $\sum_{i\leq k}\zeta_i$ for
some sequence $\zeta_1, \dots, \zeta_k$ of standard generators of
$\Z^2$. A similar comparison argument as above gives that $t\star'
m=u$. Now commutativity of $\Z^2$ gives
$$h\star' u=(tht^{-1})\star' u=t\star' h\star' t^{-1}\star' u=u. $$
Similarly $h'\star' u=u$ as well. These claims guarantee that for any
$g_1, g_2\in \Z^2$ with $g_1-g_2\in H$, and for any $u\in M_T$,
$g_1\star' u=g_2\star' u$. In other words, if we define $(g+H)*
u=g\star' u$, we obtain an action of $G$ on $M_T$.

We next claim that this action $*$ of $G$ on $M_T$ is free. For this
suppose $u\in M_T$ and $g\in \Z^2$ are such that $g\star' u=u$. We
need to show that $g\in H$. Note that we have $g\star' m=m$. Write $g$
as $\sum_{i\leq k} \zeta_i$ where $\zeta_1, \dots, \zeta_k$ are
standard generators of $\Z^2$. Consider the element $\bar{x}\in
F(2^{\Z^2})$ constructed above, with $(2k+1)^2$ many copies of
$\Gcaca$ in its domain. Let $\bar{m}\in[\bar{x}]$ correspond to
$m\in[x]$. Then the comparison argument gives that $g\star \bar{m}$
and $g\star' m$ are in corresponding positions in some copies of
$\Gcaca$. This implies that there are integers $a_1, a_2$ such that
$g\star\bar{m}=(a_1p, -a_2p)\cdot \bar{m}$. The comparison argument
also gives that $h\star \bar{m}=(p,0)\cdot \bar{m}$ and $h'\star
\bar{m}=(0,-p)\cdot \bar{m}$. This can be applied to all copies of
$\Gcaca$ to conclude that, if $\tilde{m}\in [\bar{x}]$ is in the
position corresponding to $m$ within that copy of $\Gcaca$, then
$h\star \tilde{m}=(p,0)\cdot \tilde{m}$ and $h'\star
\tilde{m}=(0,-p)\cdot \tilde{m}$. Repeatedly applying this fact, we
get that $g\star\bar{m}=(a_1h+a_2h')\star\bar{m}$. Since $\star$ is a
free action, we conclude that $g=a_1h+a_2h'\in H$ as required.

Thus we have shown that $*$ is a free action of $G$ on $M_T$. This in particular implies that $G$ is finite, and
that $|M_T|= |G|$.
Also, since $G=\Z^2/\langle h, h'\rangle$, $|G|$ is also the absolute value of the determinant
of the $2\times 2$ matrix $P$ whose rows are the components of $h$ and $h'$, as this represents
the volume of the fundamental region of the lattice in $\Z^2$ generated by the
vectors $h$ and $h'$. Finally, we show that the components of $h$ and $h'$ are all divisible by
$p$, which shows that $|\det(P)|\geq p^2$. Since $|M_T| \leq |T|=p^2$, it follows that
$M_T=T$, a contradiction to the fact that $T$ contains points not in $M$.

To see that the components of $h$ and $h'$ are divisible by $p$, we use the long tiles.
Again we concentrate on $h$ and assume $h=(h_1,h_2)$, the argument for $h'$ being similar. Consider the long tile grid graph
$\Gcqadpa$ (see Figure~\ref{fig:Gamma-npq-horiz}).  Let $m_0\in M$ be near the center (to be exact, within the $N_0$-ball of the center) of the upper-left copy of $R_\times$ block in $\Gcqadpa$. Let $m_1=(p,0)\cdot m_0$, $m_q=(pq,0)\cdot m_0$, $u_0=(0,-p)\cdot m_0$, $u_1=(q, 0)\cdot u_0$, and $u_p=(pq, 0)\cdot u_0$. Then $m_0, m_1, m_q, u_0, u_1, u_p$ are all in the same position within an $R_\times$ block. From the definition of $\Gcqadpa$ we have that $h\star m_0=m_1$, $qh\star m_0=m_q$, $h'\star m_0=u_0$, and $h'\star m_q=u_p$ because all of these relations are witnessed by points in $\Gcaca$.
Now since $n>2N_1+4N_0$, we can find a $\star$-path from $u_0$ to $u_1$ that stays within the bottom portion of $\Gcqadpa$ consisting of only the first and second $R_\times$ blocks and the $R_d$ block in between them. This gives rise to an $\ell \in \Z^2$
such that $\ell \star u_0=u_1$ and furthermore $p\ell \star u_0=u_p$.
Thus, $h'\star (q h) \star m_0=u_p= (p\ell)\star h'\star m_0$. Since $\star$ is a free action,
we must have $h'+ qh= p\ell+h'$ in $\Z^2$, so $qh=p\ell$. Since $\gcd(p,q)=1$, this shows
that $h_1,h_2$ are divisible by $p$.
\end{proof}

As a consequence of Theorem~\ref{thm:no-sub-Z2-action} we can provide a partial answer to Problem 10.6
of \cite{gao_countable_2015} concerning ``almost lined-up'' marker regions.
Recall that by a ``rectangular tiling'' of $\fzs$ we mean a
finite subequivalence relation $E$ of $\fzs$ such that
each $E$-class $R$ is a finite subset of a $\fzs$ class $[x]$, and
corresponds to a rectangle in that it is equal to $R=([a,b]\times[c,d])\cdot x$
for some $a<b$, $c<d$ in $\Z$. We have the natural notions of the ``top-edge'' of $R$,
etc.\ If $R_1$, $R_2$ are rectangular $E$-classes in the same $\fzs$ class $[x]$,
we say $R_1$, $R_2$ are {\em adjacent} if there are $y \in R_1$ and $z \in R_2$
such that $\zeta\cdot y=z$ for some $\zeta\in\{\pm e_1, \pm e_2\}$. Note that if $R_1,R_2$ are adjacent, then from the ordered pair $(R_1,R_2)$ the generator is
uniquely determined, and so it makes sense to say ``$R_2$ is above $R_1$,'' etc.\
If $R_1$ and $R_2$ are adjacent, it also makes sense to define the ``overlap'' of
$R_1$ and $R_2$. For example, if $R_2$ is above $R_1$, this is the number of points
$y\in R_1$ such that $e_2 \cdot y \in R_2$.

\begin{defnn}
Suppose $E$ is a rectangular tiling of $\fzs$ and suppose all the rectangles in the tiling
$E$ have side lengths bounded above by $M$. We say the rectangular tiling $E$ is
{\em almost lined-up} if for every rectangle $R_1$ in $E$,
there is a rectangle $R_2$ in $E$ which is adjacent to $R_1$, and to the top of $R_1$,
with the overlap of $R_1$ and $R_2$ being strictly greater than $\frac{M}{2}$. Similarly if ``top''
is replaced with ``bottom'', ``left'', or ``right.''
\end{defnn}

Problem 10.6 of \cite{gao_countable_2015} asks if there exist Borel rectangular tilings of $F(2^{\Z^2})$ that are almost lined-up.
The next result rules out the existence of clopen, almost lined-up rectangular
tilings.

\begin{corn} \label{nlum}
There does not exist a clopen  rectangular tiling of $\fzs$ which is
almost lined-up.
\end{corn}

\begin{proof}
Suppose $E$ were a clopen, almost lined-up, rectangular tiling of $\fzs$.
Since the $E$ classes are finite, and $E$ is a clopen tiling, there is a clopen set
$S\subseteq \fzs$ which is a selector for $E$. We define a free action $\star$ of $\Z^2$
on $S$ as follows. If $s \in R_1$ ($R_1$ a class of $E$), then $(0,1)\star s$ is the unique
point of $S \cap R_2$, where $R_2$ is the class of $E$ which is adjacent to $R_1$,
to the top of $R_1$, and has overlap with $R_1$  which is $>\frac{M}{2}$. By assumption,
such a class $R_2$ exists, and it is easy to see it is unique. The action of the other
generators is defined similarly. Since $E$ is clopen, it is easy to check that $\star$
is a continuous action on $S$, and it is also clearly a free action. This
contradicts Theorem~\ref{thm:no-sub-Z2-action}.
\end{proof}

\section[Undirected graph homomorphsims]{Undirected graph homomorphisms and the reduced homotopy group}
\label{sec:undirected-graphs}

In the rest of this chapter we apply the Twelve Tiles Theorem to study the existence of
continuous graph homomorphisms $\varphi : F(2^{\Z^2}) \rightarrow
\Gamma$ for various graphs $\Gamma$. For this purpose, all our graphs will be assumed to be undirected and without any assigned
$\Z^2$-labeling of their edges.

In the setting of continuous graph homomorphisms, our main theorem,
Theorem \ref{thm:tilethm}, takes the following form.

\begin{corn} \label{cor:graphthm}
Let $\Gamma$ be a graph. Then the following are equivalent.
\begin{enumerate}
\item[\rm (1)] There is a continuous graph homomorphism $\varphi \colon F(2^{\Z^2}) \rightarrow \Gamma$.
\item[\rm (2)] There is a continuous graph homomorphism $\varphi\colon F(2^{\Z^2})\rightarrow \Gamma_0$ where $\Gamma_0$ is a connected component of $\Gamma$.
\item[\rm (3)] There are positive integers $n$, $p$, $q$ with $n < p, q$ and $\gcd(p,q) = 1$ such that
$\Gamma_{n,p,q}$ admits a graph homomorphism to $\Gamma$.
\item[\rm (4)] For all $n \geq 1$ and for all sufficiently large $p, q$,
$\Gamma_{n,p,q}$ admits a graph homomorphism to $\Gamma$.
\end{enumerate}
\end{corn}

\begin{proof}
Let $S$ be the set of all functions $x \in V(\Gamma)^{\Z^2}$ such that $x
\colon \Z^2 \rightarrow V(\Gamma)$ is a graph homomorphism. It is easily
seen that $S$ is a subshift of finite type. In fact $S$ is defined by a
condition on the edges and so has width $1$ (the notion of width is defined before Definition~\ref{def:sft}).
The existence of a
continuous graph homomorphism from $F(2^{\Z^2})$ to $\Gamma$ is
equivalent to the existence of a continuous equivariant map from
$F(2^{\Z^2})$ to $S$. Similarly, the existence of a graph homomorphism
from $\Gamma_{n,p,q}$ to $\Gamma$ is equivalent to the existence of a
map $g \colon \Gamma_{n,p,q} \rightarrow V(\Gamma)$ which respects
$S$. Having noted these equivalences, the equivalence of (1), (3) and (4) follows
immediately from Theorem \ref{thm:tilethm}. 

We next argue that (1) and (2) are equivalent. It is obvious that (2) implies (1). For the other direction, assume (1) holds. By the above equivalence, we have (3) holds, and consider a graph homomorphism $\varphi$ from $\Gamma_{n,p,q}$ to $\Gamma$, where $n<p, q$ and $\gcd(p,q)=1$. Since $\Gamma_{n,p,q}$ is connected, we know that the image of $\varphi$ is contained in a connected component $\Gamma_0$ of $\Gamma$. Thus we have (3) for $\Gamma_0$ instead of $\Gamma$. Applying the equivalence between (3) and (1) for $\Gamma_0$ instead of $\Gamma$, we get that (2) holds.
\end{proof}

By the above corollary, the continuous homomorphism problem is essentially equivalent to the instances of the problem in which the graph $\Gamma$ is connected. Thus in the remainder of this chapter, we tacitly assume that $\Gamma$ is connected unless explicitly stated otherwise. 

In some sense the above corollary has provided a solution to the continuous graph homomorphism problem, but the conditions are not easy to apply in practice, and so we still seek to improve them.

We prove two main theorems on this topic. The first one, Theorem~\ref{thm:neghom},
gives a sufficient condition
on the graph $\Gamma$ to establish that there does not exist a continuous graph homomorphism
from $\fzs$ to $\Gamma$. Similarly, the second theorem, Theorem~\ref{thm:odd-length-cycle},
gives a sufficient condition on
$\Gamma$ to establish the existence of a continuous homomorphism from $\fzs$ to $\Gamma$.
We refer to these theorems as the ``negative'' and ``positive'' theorems,
and to the conditions in these two theorems as the ``negative'' and ``positive''
conditions. Both of these conditions are phrased in terms of a certain ``reduced homotopy
group'' $\pi_1^*(\Gamma)$ associated to the graph $\Gamma$. Roughly speaking,
$\pi_1^*(\Gamma)$ will be the quotient of the first homotopy group $\pi_1(\Gamma)$
of $\Gamma$ by the normal subgroup generated by the $4$-cycles in $\Gamma$
(we give the precise definitions below). The negative theorem has a number of
corollaries which give easier-to-apply, but less general, conditions sufficient
to guarantee the non-existence of a continuous graph homomorphism from
$\fzs$ to $\Gamma$.

We will illustrate each of these theorems by considering a number of specific examples
of graphs for which these theorems provide an answer. We note that even in the case of
simple graphs such as $K_3$ or $K_4$ (the complete graphs on $3$ and $4$
vertices respectively), the question is not trivial (Theorems~\ref{threecoloring}
and \ref{thm:pcn} show there is no continuous homomorphism for $\Gamma=K_3$,
while Theorem~\ref{thm:pcn2} shows there is a continuous homomorphism for $\Gamma=K_4$).
We will give some simpler versions of the negative conditions, as mentioned above,
but also give some examples of graphs $\Gamma$ where these simpler conditions do not suffice to
answer the continuous homomorphism question for $\Gamma$, but nevertheless we
can show a negative answer.

In \S\ref{sec:gh} we will return to the graph homomorphism problem. Building on
the results of this chapter, we will show in Theorem~\ref{thm:ghnr} that the continuous
graph homomorphism problem for $\fzs$ is not decidable. That is, the set of
finite graphs $\Gamma$ for which  there exists
a continuous graph homomorphism from
$\fzs$ to $\Gamma$ is not a computable set (we will observe that it is a
computably enumerable set, however).

The arguments in the next two sections will be based on homotopy theory. In the rest of this section, we review the homotopy notions we will
need within the specific setting we will use them.

Fix an undirected graph
$\Gamma=(V(\Gamma), E(\Gamma))$. For simplicity we still write edges of $\Gamma$ as pairs $(u, v)\in E(\Gamma)$ for $u, v\in V(\Gamma)$. A path is a finite sequence of vertices $\sigma = (v_0,
\ldots, v_n)$ where $(v_i, v_{i+1}) \in E(\Gamma)$ for every $0 \leq i
< n$ (we allow paths to visit the same vertex multiple times).
The length of a path is the length of the sequence minus one
(i.e. the number of edges traversed). A trivial path consists of one
vertex and has length $0$. A {\em back-tracking} path is a path
$\sigma = (v_0, \ldots, v_{2n})$ such that $v_{n-i} = v_{n+i}$ for all
$0 \leq i \leq n$. If $\sigma$ is a path from $u$ to $v$, then we
write $\sigma^{-1}$ for the reverse path from $v$ to $u$. If $\sigma$
ends at $v$ and $\tau$ begins at $v$ then we write $\sigma \tau$ for
the concatenated path which first follows $\sigma$ and then follows
$\tau$. Concatenation is done without removing instances of
back-tracking, so the length of $\sigma \tau$ is the sum of the
lengths of $\sigma$ and $\tau$. Removing instances of back-tracking is
done by passing to the homotopy equivalence class of a path. {\em
Homotopy equivalence} is an equivalence relation on paths defined as
follows: it is the smallest equivalence relation with the property
that if $\sigma$ is a path ending at $v$, $\tau$ is a path beginning
at $v$, and $\theta$ is a back-tracking path beginning and ending at
$v$, then $\sigma \theta \tau$ is homotopy equivalent to $\sigma
\tau$. Notice that $\sigma \sigma^{-1}$ is homotopy equivalent to the
trivial path.

A path which begins and ends at the same vertex will be called a {\em
cycle}, and when it has length $p$ it will be called a {\em $p$-cycle}. If
$\sigma$ is a $4$-cycle beginning at $u$ and $\tau$ is any path ending
at $u$, then $\tau \sigma \tau^{-1}$ is a cycle and we call it a {\em
square-cycle}. Notice that, for any fixed vertex $v \in
V(\Gamma)$ serving as a base point, the set of homotopy equivalence classes of cycles
beginning at $v$ forms a group under concatenation. This is called the
{\em first fundamental group of $\Gamma$} and is denoted
$\pi_1(\Gamma)$. It is a basic exercise to show that up to isomorphism
the group $\pi_1(\Gamma)$ does not depend upon the choice of the base point $v$,
assuming $\Gamma$ is connected. We
will always implicitly assume without mention that a fixed base point $v$
has been chosen, and that all cycles being discussed are cycles which
begin and end at $v$. For a cycle $\gamma$ we write $\pi_1(\gamma)$
for its homotopy equivalence class. The set of $\pi_1$-images of
square-cycles is conjugation invariant and thus generates a normal
subgroup $N$. We write $\pi_1^*(\Gamma)=\pi_1(\Gamma)/N$ for the quotient of
$\pi_1(\Gamma)$ by the subgroup $N$. We refer to $\pi_1^*(\Gamma)$
as the {\em reduced homotopy group} of $\Gamma$, and to $N$ as the {\em kernel} of $\pi^*_1$.
For a cycle $\gamma$ we write $\pi_1^*(\gamma)$ for the image of $\gamma$ in
this group. If $\pi_1^*(\gamma)=\pi_1^*(\eta)$ we say that $\gamma$ and $\eta$ are {\em reduced homotopy equivalent}.

Notice that if two paths are homotopy equivalent then
their lengths differ by an even number. Since square-cycles also have
even length, the lengths of reduced homotopy equivalent cycles also differ by an even number. It follows that $\pi_1^*(\gamma)$ is non-trivial whenever
$\gamma$ is an odd-length cycle.

Note that if $\gamma$ is a cycle in $\Gamma$ then it is well-defined,
without mentioning the base point,  to say that
$\pi_1^*(\gamma)$ satisfies a particular group-theoretic statement $\varphi$
in the group $\pi_1^*(\Gamma)$. This is because choosing a different base
point for computing the homotopy groups results in considering a conjugate
$\gamma'=\delta \gamma \delta^{-1}$ of $\gamma$, and since conjugation
is an isomorphism we have that $\pi_1^*(\Gamma)$ satisfies
$\varphi(\gamma)$ if and only if $\delta \pi_1^*(\Gamma) \delta^{-1}$ satisfies $\varphi(\gamma')$.
For example, if $\gamma$ is a cycle in $\Gamma$, it is well-defined
to say that $\pi_1^*(\gamma)$ has order $p$ in $\pi_1^*(\Gamma)$.

We mention an alternative view of $\pi_1^*(\Gamma)$ which requires a
bit of homotopy theory and fundamental groups in a slightly more
general setting. One could view $\Gamma$ as a topological space by
viewing each edge as a copy of $[0, 1]$. We can form a larger
topological space $\Gamma^*$ by attaching a copy of the unit-square
$[0, 1] \times [0, 1]$ to each $4$-cycle in $\Gamma$, identifying the
boundary of the square with the edges and vertices used by the
$4$-cycle. This is an instance of a so-called CW complex (see
\cite{Hatcher} for the definition and discussion. In \cite{Lyndon}
it is shown how a CW complex is associated to a group presentation).
In this setting, a path is any continuous function from
$[0, 1]$ to $\Gamma^*$, and two paths are homotopy equivalent if they
have the same endpoints and, while maintaining the same endpoints, one
path can be continuously morphed into the other. Again, one can fix a
point $v$ and consider the set of homotopy equivalence classes of
paths which begin and end at $v$. This forms a group denoted
$\pi_1(\Gamma^*)$. It is not difficult to check that $\pi_1^*(\Gamma)
= \pi_1(\Gamma^*)$. We will only briefly use this identity in one
example, but the reader may find it intuitively helpful.

We point out that any graph homomorphism $\Gamma \rightarrow \Gamma'$
induces a group homomorphism $\pi_1^*(\Gamma) \rightarrow
\pi_1^*(\Gamma')$ (and similarly with each $\pi_1^*$ replaced by
$\pi_1$).

As an example, we calculate the reduced homotopy group for the 12-tiles graph $\Gamma_{n,p,q}$ below. We will not need this result in the rest of the paper.

\begin{propn} \label{prop:tilehom}
Let $p, q \geq n$ with $\gcd(p,q) = 1$. Then $\pi_1^*(\Gamma_{n,p,q}) \cong \Z^2$.
\end{propn}

\begin{proof}
We simply sketch the argument as we will not need this fact,
and the argument is similar to one given later in Lemma~\ref{prestograph}.
In computing $\pi_1^*(\Gamma_{n,p,q})$ we choose our base point $v$ to
lie in the $R_\times$ block. Any cycle $\sigma$ beginning at $v$ can be viewed
as a collection of path segments on the $12$ tiles. Each of the $12$
tiles forming $\gnpq$ consists of an interior region (these points are not
identified with any others in $\gnpq$) surrounded by rectangular blocks
labeled by $R_\times$, $R_a$, $R_b$, $R_c$, or $R_d$. For the purposes of this proof we refer to these
surrounding blocks as the boundary blocks of the tile.
Since every block on the boundary either is an $R_\times$ block or has copies of
$R_\times$ blocks adjacent to it, it is possible to adjust $\sigma$ modulo the kernel of
$\pi_1^*$ so that $\sigma$ always stays in the boundary blocks and
enters and leaves tiles through a
copy of $R_\times$ (in fact through $v$). With this adjustment, each path segment
within the blocks of a given tile
represents a cycle and can be expressed as a product in $\alpha$,
$\beta$, $\gamma$, $\delta$, where $\alpha$ is the cycle at $v$
passing vertically through $R_a$ once, and similarly $\beta$,
$\gamma$, and $\delta$ pass through $R_b$, $R_c$, and $R_d$ once,
respectively. This implies $\pi_1^*(\Gamma_{n,p,q})$ is generated by
$\pi_1^*(\{\alpha, \beta, \gamma, \delta\})$. In fact, we have named each of the 12 tiles with
an equation in its subscript, and this equation describes a relation
which is satisfied in $\pi_1^*(\Gamma_{n,p,q})$. Let $R$
be the set of these $12$ relations. So,
in the group $\pi_1^*(\Gamma_{n,p,q})$ we have that each of
$\alpha$, $\beta$ commutes with each of $\gamma$, $\delta$,
$\alpha \beta=\beta\alpha$, $\gamma\delta=\delta\gamma$, and
$\alpha^q=\beta^p$, $\gamma^q=\delta^p$. Let $m,k$ be integers such that
$mp+kq=1$. Let $u=\alpha^m\beta^{k}$ and $v= \gamma^m\delta^{k}$.
Then $u^q=\alpha^{mq} \beta^{kq}=\beta^{mp+kq}=\beta$ in $\pi_1^*(\Gamma_{n,p,q})$,
and similarly $\alpha=u^p$, $\gamma=v^p$, $\delta=v^q$. So,
$\pi_1^*(\Gamma_{n,p,q})$ is generated by the commuting elements $u, v$.
It is clear then that the group $\langle \alpha,\beta,\gamma,\delta \mid R\rangle$
is isomorphic to $\Z^2$, the free abelian group generated by $u$ and $v$.

It remains to show that in the group $\pi_1^*(\Gamma_{n,p,q})$ there are no further
relations than those generated by $R$, that is,
the kernel of the map
$\varphi \colon \langle \alpha,\beta,\gamma,\delta\rangle \to
\pi_1^*(\Gamma_{n,p,q})$ is the normal subgroup generated by the words in $R$.
Suppose $\sigma=\eta_1^{\epsilon_1}\cdots \eta_\ell^{\epsilon_\ell}$
is null-homotopic in $\gnpq$, where each $\eta_i\in \{ \alpha,\beta,\gamma, \delta\}$ and $\epsilon_i\in\{\pm 1\}$.
Let $\sigma_0=\sigma, \sigma_1,\dots,\sigma_k$ be a sequence of cycles witnessing this, i.e., each $\sigma_{i+1}$
differs from $\sigma_i$ by the insertion/deletion of a four-cycle in $\gnpq$, or a
two-cycle, and $\sigma_k$ is the trivial cycle at $v$. Let $\sigma_i^*$ be the result of applying the above algorithm of adjustment
to express $\sigma_i$ as a word in $\{ \alpha,\beta,\gamma,\delta\}$. Thus,
every part $u_0,u_1,\dots,u_m$ of the path $\sigma_i$ for which
$u_0,u_m$ are in the boundary blocks of one of the $12$ tiles, but
$u_1,\dots, u_{m-1}$ are in the interior of this tile, is replaced by a path
in this tile which stays in the boundary blocks. $\sigma_i^*$
is the path in $\langle \alpha,\beta,\gamma,\delta\rangle$ which is
equivalent to this boundary-block path in the natural manner. Since
$\sigma_i^*$ differs from $\sigma_i$ by a combination of four-cycles (or two-cycles), and since
it is easy to check that all four-cycles in $\gnpq$ are part of a single tile,
it is straightforward to check that $\sigma_{i+1}^*$ and $\sigma_i^*$
differ by either the insertion/deletion of a canceling pair such as $\alpha \alpha^{-1}$,
or the insertion/deletion of one of the words in $R$. Since $\sigma_k$
is the trivial word, we must have that $\sigma_0$ is in the normal subgroup
generated by $R$.
\end{proof}

\section{Negative conditions for graph homomorphisms} \label{sec:nc}

The following result is the general ``negative condition'' sufficient
to give the non-existence of a continuous graph homomorphism from
$\fzs$ to the graph $\Gamma$.

\begin{thmn}[Homotopy-based negative condition for graph homomorphisms] \label{thm:neghom}
Let $\Gamma$ be a graph. Suppose that for every $N > 0$ there are relatively prime integers $p,
q > N$ such that for every $p$-cycle $\gamma$ of $\Gamma$,
$\pi_1^*(\gamma^q)$ is not a $p^{\text{th}}$ power in
$\pi_1^*(\Gamma)$. Then there does not exist a continuous graph homomorphism
$\varphi \colon F(2^{\Z^2}) \rightarrow \Gamma$.
\end{thmn}


\begin{proof}
We prove the contrapositive. So let $\varphi : F(2^{\Z^2}) \rightarrow
\Gamma$ be a continuous graph homomorphism. Fix any $n \geq 1$. By
Corollary \ref{cor:graphthm} there is $N$ so that for all
integers $p, q > N$ there is a graph homomorphism $\theta :
\Gamma_{n,p,q} \rightarrow \Gamma$. Fix such a pair $(p, q)$ with $p,q$ relatively prime. Consider
the rectangular grid graph given by the long tile $T_{c^q a = a d^p}$
(see Figure \ref{fig:Gamma-npq-horiz}). Recall that this tile has $q$
many copies of $R_c$ blocks on its top boundary with each $R_c$ block occurring
between two copies of $R_\times$ bocks, a single $R_a$ block on its left and right
boundaries, and has $p$ many copies of $R_d$ blocks along its bottom boundary
with each $R_d$ block occurring between two copies of $R_\times$ blocks. Let
$\sigma$ be the cycle in $T_{c^q a = a d^p}$ which begins and ends
at the vertex $v$ which is the upper-left vertex  of the upper-left
copy of $R_\times$, and
goes clock-wise through all of the
upper left corners of the labeled blocks, and which traces out a
rectangle (see Figure~\ref{fig:Gamma-npq-horiz-3}). Let $\sigma'=\theta\circ \sigma$ be the corresponding
cycle in $\Gamma$.

The arrangement of the boundary blocks tells us that $\sigma'$ has the form $\gamma^q \alpha
\delta^{-p} \alpha^{-1}$, where $\gamma$ and $\delta$ are the images
of the horizontal paths from (the upper-left vertices of)
$R_\times$ to $R_\times$ going through
$R_c$ and $R_d$, respectively, and $\alpha$ is the image of the
vertical path from $R_\times$ to $R_\times$ going through $R_a$ (see Figure~\ref{fig:Gamma-npq-horiz-3}). In
particular, $\gamma$ is a $p$-cycle. Clearly $\pi_1^*(\sigma) = 1$ so
$\pi_1^*(\sigma') = 1$ as well and thus $\pi_1^*(\gamma^q) =
\pi_1^*(\alpha \delta^p \alpha^{-1}) = \pi_1^*(\alpha \delta
\alpha^{-1})^p$.
\end{proof}

The next theorem, which is a consequence of Theorem~\ref{thm:neghom},
also gives a ``negative condition'' for the non-existence of
continuous graph homomorphisms into $\Gamma$.
This result is frequently easier to apply than
Theorem~\ref{thm:neghom}, and has the advantage of
referring directly to the graph $\Gamma$ and not the homotopy
group $\pi_1^*(\Gamma)$. A {\em weighting} of $\Gamma$, or a {\em weight function} for $\Gamma$, is
a function $w: \tilde{E}(\Gamma)\to \Z$, where
$$ \tilde{E}(\Gamma)=\{(u,v) \colon \{u,v\}\in E(\Gamma)\}, $$
such that for all $(u,v)\in \tilde{E}(\Gamma)$, $w(u,v)+w(v,u)=0$. For a weighting $w \colon \tilde{E}(\Gamma) \rightarrow \Z$ and a
path $\gamma = (v_0, \ldots, v_n)$ in $\Gamma$, we can extend the function $w$ by setting 
$$w(\gamma) = \sum_{i=0}^{n-1} w(v_i, v_{i+1}).$$ 

\begin{thmn}[Weighting-based negative condition for graph homomorphisms]
\label{thm:negweight}
Let $\Gamma$ be a graph.  Assume that for infinitely many integers $p$ and every
$p$-cycle $\gamma$ there exists a weighting $w$ such that $w(\gamma)$ is
not divisible by $p$ and $w(\sigma) = 0$ for all $4$-cycles
$\sigma$. Then there does not exist a
continuous graph homomorphism $\varphi \colon F(2^{\Z^2})
\rightarrow \Gamma$.
\end{thmn}

\begin{proof}
We again prove the contrapositive. So suppose there is a continuous
graph homomorphism. By Theorem \ref{thm:neghom}, there is $N > 0$ such
that for all relatively prime $p, q > N$ there is a $p$-cycle $\gamma$ in $\Gamma$
such that $\pi_1^*(\gamma^q)$ is a $p^{\text{th}}$ power in
$\pi_1^*(\Gamma)$. Fix such a pair $(p, q)$, let $\gamma$ be a
$p$-cycle, and let $\tau$ be a cycle satisfying $\pi_1^*(\gamma^q) =
\pi_1^*(\tau^p)$. Fix any weight $w$ such that $w(\sigma) = 0$ for all
$4$-cycles $\sigma$. Then $w(\sigma) = 0$ for all square-cycles. This
implies that $w$ induces a group homomorphism $w : \pi_1^*(\Gamma)
\rightarrow \Z$. Therefore $q \cdot w(\gamma) = w(\gamma^q) =
w(\tau^p) = p \cdot w(\tau)$. Since $\gcd(p,q) = 1$, we conclude
$w(\gamma)$ is divisible by $p$.
\end{proof}

\begin{remn}
Theorem~\ref{thm:negweight}  generalizes to weight functions
$w \colon \tilde{E}(\Gamma) \to A$ for any abelian group $A$.
In applying Theorem~\ref{thm:negweight} we frequently
use a fixed weighting $w$ on the graph $\Gamma$ (that is,
$w$ does not depend on $p$ and the $p$-cycle $\gamma$).
\end{remn}

The next example shows that we may use Theorem~\ref{thm:negweight}
to give yet another proof that the continuous chromatic number
of $\fzs$ is greater than $3$. Note that a proper $3$-coloring of $\fzs$
is equivalent to a graph homomorphism from $\fzs$ to $K_3$.

\begin{exn} \label{ex:Kthree}
There is no continuous graph homomorphism $\varphi \signatureSep F(2^{\Z^2}) \to K_3$,
where $K_3$ is the complete graph of size 3.
\end{exn}
\begin{figure}[htbp] 
\begin{center}
\begin{tikzpicture}
\begin{scope}[every node/.style={nodeStyle, minimum size=3.5ex}]
    \node (n1) at (90:1) {};
    \node (n2) at (330:1) {};
    \node (n3) at (210:1) {};
\end{scope}

\draw[-stealth] (n1) -- (n2) node[midway, shift={(.2,.2)}] {0};
\draw[-stealth] (n2) -- (n3) node[midway, shift={(0,-.3)}] {0};
\draw[-stealth] (n3) -- (n1) node[midway, shift={(-.2,.2)}] {1};
\end{tikzpicture}
\end{center}
\caption{\label{fig:K3-labeling}The complete graph of order 3, $K_3$, with a weighting function.}
\end{figure}

\todo{Define trivial cycle}

\begin{proof}
We apply Theorem \ref{thm:negweight} with the weight function $w$
given in Figure~\ref{fig:K3-labeling}. Since $K_3$ contains no
non-trivial $4$-cycles, we trivially have that $w$ assigns $0$ to
every $4$-cycle. Now consider an odd integer $p > 1$ and consider a
$p$-cycle $\gamma$. Between any two consecutive instances of the
weight $1$ edge, there must be an even number of weight $0$
edges. Since $p$ is odd, it follows $w(\gamma) \neq 0$. Clearly also
$|w(\gamma)| < p$ since $\gamma$ has length $p$ and must use at least
one edge of weight $0$. Thus $w(\gamma)$ is not divisible by $p$. We
conclude there is no continuous homomorphism to $K_3$.
\end{proof}

The next theorem, extending the above argument slightly, gives a simple (in particular, computable)
sufficient condition on the graph $\Gamma$ for there not
to exist a continuous graph homomorphism from $\fzs$ to $\Gamma$.

\begin{thmn} [Simple negative condition] \label{thm:simpleneg}
There does not exist a continuous graph homomorphism from $\fzs$ to
the graph $\Gamma$ if $\Gamma$ satisfies either of the following:
\begin{enumerate}
\item \label{sna}
$\Gamma$ is properly $3$-colorable.
\item \label{snb}
$\Gamma$ has no non-trivial $4$-cycles.
\end{enumerate}
\end{thmn}

\begin{proof}
The first statement follows immediately from Example~\ref{ex:Kthree}.
Namely, note that $\Gamma$ is properly $3$-colorable if and only if there
is a graph homomorphism from $\Gamma$ to $K_3$.
If $\Gamma$ is properly $3$-colorable, then any continuous graph homomorphism
from $\fzs$ to $\Gamma$ would immediately give a continuous graph homomorphism
from $\fzs$ to $K_3$, a contradiction to Example~\ref{ex:Kthree}.
Suppose now $\Gamma$ has no non-trivial $4$-cycles.
Fix an arbitrary direction for each edge in the graph $\Gamma$.
Let $p>1$ be odd,
and let $\gamma$ be any $p$-cycle in $\Gamma$, say beginning and ending
at $v$. There must be some edge in $\Gamma$ which occurs in $\gamma$
an odd number of times, as otherwise the total length of $\gamma$, which is $p$,
would be even. Say the edge $e_0=(u,v)$ occurs an odd number of times in $\gamma$.
Let $w \colon \tilde{E}(\Gamma)\to \Z$ be the weighting function
where $w(u,v)=1$ (so $w(v,u)=-1$) and $w(e)=0$ for all other edges $e$ of $\Gamma$.
Clearly $w(\gamma) \neq 0$, as $e_0$ is traversed an odd number of times by $\gamma$.
Also, $|w(\gamma)| <p$ as $\gamma$ contains an edge $e$ with $w(e)=0$
(since $p$ is odd, $\gamma$ must at some point use an edge other than $e_0$).
So, $w(\gamma)$ is not divisible by $p$. As there are no
non-trivial $4$-cycles in $\Gamma$, the hypotheses of Theorem~\ref{thm:negweight}
are satisfied.
\end{proof}

The simple conditions of Theorem~\ref{thm:simpleneg} suffice
to rule out the existence of continuous graph homomorphisms from
$\fzs$ to $\Gamma$ for many graphs $\Gamma$. For example,
the next example considers the {\em Petersen} graph from graph theory.

\begin{figure}[htbp] 
\begin{tikzpicture}
    \begin{scope}[every node/.style={nodeStyle, minimum size=3.5ex},scale=1]
        \node (n0) at (90 :2) {0};
        \node (n1) at (18 :2) {1};
        \node (n2) at (306:2) {2};
        \node (n3) at (234:2) {0};
        \node (n4) at (162:2) {1};
        \node (n5) at (90 :1) {1};
        \node (n6) at (18 :1) {0};
        \node (n7) at (306:1) {0};
        \node (n8) at (234:1) {2};
        \node (n9) at (162:1) {2};
    \end{scope}

    \draw (n0) -- (n1);
    \draw (n1) -- (n2);
    \draw (n2) -- (n3);
    \draw (n3) -- (n4);
    \draw (n4) -- (n0);

    \draw (n0) -- (n5);
    \draw (n1) -- (n6);
    \draw (n2) -- (n7);
    \draw (n3) -- (n8);
    \draw (n4) -- (n9);

    \draw (n5) -- (n7);
    \draw (n7) -- (n9);
    \draw (n9) -- (n6);
    \draw (n6) -- (n8);
    \draw (n8) -- (n5);

\end{tikzpicture}
\caption[The Petersen graph $\Gamma_P$.]
{The Petersen graph $\Gamma_P$, with a proper 3-coloring of its nodes.}
\label{fig:petersen-graph}
\end{figure}
\begin{exn}\label{ex:petersen-graph}
There is no continuous graph homomorphism
$\varphi \colon F(2^{\Z^2}) \rightarrow \Gamma_P$, where $\Gamma_P$ is the Petersen graph.
\end{exn}
\begin{proof}
The Petersen graph is given in Figure~\ref{fig:petersen-graph}, along
with a proper 3-coloring (i.e a homomorphism into $K_3$). There can be no
continuous homomorphism into $\Gamma_P$ by (\ref{sna}) of Theorem~\ref{thm:simpleneg}.
\end{proof}

A less immediate example is given by the following ``clamshell'' graph
$\Gamma_C$ depicted in Figure~\ref{fig:jacksons-graph}.
This graph was chosen so as to have chromatic number strictly
greater than $3$ and to have several nontrivial $4$-cycles.

\begin{exn}
There is no continuous graph homomorphism $\varphi \colon F(2^{\Z^2}) \rightarrow \Gamma_C$.
\end{exn}

\begin{figure}[hbtp] 
\begin{tikzpicture}
    \begin{scope}[every node/.style={nodeStyle, minimum size=3.5ex},scale=1.8]
        \node (x)  at (0,0) {$x$};
        \node (x') at (6,0) {$x'$};
        \node (v0) at (1,0) {$v_0$};
        \node (v1) at (2,0) {$v_1$};
        \node (v2) at (3,0) {$v_2$};
        \node (v3) at (4,0) {$v_3$};
        \node (v4) at (5,0) {$v_4$};
        \node (u0) at (1,1) {$u_0$};
        \node (u1) at (2,2) {$u_1$};
        \node (u2) at (3,3) {$u_2$};
        \node (u3) at (4,4) {$u_3$};
        \node (u4) at (5,5) {$u_4$};
    \end{scope}

    \node[shift={(-.6,.2)}] at (u0) { 1};
    \node[shift={(-.8,.2)}] at (u1) {-1};
    \node[shift={(-.8,.2)}] at (u2) {-1};
    \node[shift={(-.8,.2)}] at (u3) {-1};
    \node[shift={(-.8,.2)}] at (u4) {-1};

    \node[shift={(.8,.2)}] at (u0) { 1};
    \node[shift={(.8,.2)}] at (u1) {-1};
    \node[shift={(.8,.2)}] at (u2) {-1};
    \node[shift={(.8,.2)}] at (u3) {-1};
    \node[shift={(.8,.2)}] at (u4) {-1};

    \node[shift={(0,-2)}] at (v2) {-1};

    \draw[-stealth]              (x) to[out=-45,in=225] (x');

    \draw[-stealth] (x) to[out=70,in=180] (u0);
    \draw[-stealth] (x) to[out=75,in=180] (u1);
    \draw[-stealth] (x) to[out=80,in=180] (u2);
    \draw[-stealth] (x) to[out=85,in=180] (u3);
    \draw[-stealth] (x) to[out=90,in=180] (u4);

    \draw[stealth-]              (u0) to[out=0,in=110,distance=1.6cm] (x');
    \draw[stealth-]              (u1) to[out=0,in=105,distance=1.7cm] (x');
    \draw[stealth-]              (u2) to[out=0,in=100,distance=1.8cm] (x');
    \draw[stealth-]              (u3) to[out=0,in=95, distance=1.9cm] (x');
    \draw[stealth-] (u4) to[out=0,in=90, distance=2.0cm] (x');

    \draw[stealth-]              (u0) -- (v0) node[pos=0, shift={(.2, -.6)}] {1};
    \draw[stealth-]              (u1) -- (v1) node[pos=0, shift={(.2, -.6)}] {1};
    \draw[stealth-]              (u2) -- (v2) node[pos=0, shift={(.2, -.6)}] {1};
    \draw[stealth-]              (u3) -- (v3) node[pos=0, shift={(.2, -.6)}] {1};
    \draw[stealth-]              (u4) -- (v4) node[pos=0, shift={(.2, -.6)}] {1};

    \draw[-stealth] (x)  -- (v0) node[midway, shift={(0, .2)}] {-1};
    \draw[-stealth] (v0) -- (v1) node[midway, shift={(0, .2)}] {-1};
    \draw[-stealth] (v1) -- (v2) node[midway, shift={(0, .2)}] {-1};
    \draw[-stealth] (v2) -- (v3) node[midway, shift={(0, .2)}] {1};
    \draw[-stealth] (v3) -- (v4) node[midway, shift={(0, .2)}] {1};
    \draw[-stealth]              (v4) -- (x') node[midway, shift={(0, .2)}] {1};

\end{tikzpicture}

\caption[The ``Clamshell'' graph $\Gamma_C$.]{The ``Clamshell'' graph $\Gamma_C$,
together with a weight function.
Observe the subtlety that the edges $x\to u_0$ and $x'\to u_0$ are labeled \emph{positive} 1,
unlike the edges from $x$ and $x'$ to the other $u_i$'s.}
\label{fig:jacksons-graph}
\end{figure}

\begin{proof}
A weighting $w$ of the graph is
shown in Figure~\ref{fig:jacksons-graph}.
We first check that $w(\sigma) = 0$ for all $4$-cycles. Note that
there is only one non-trivial $4$-cycle not containing $x$, namely
$(x', u_3, v_3, v_4)$, which has weight $0$. All other non-trivial
$4$-cycles must contain $x$ and thus must contain two vertices
adjacent to $x$. We quickly see the $4$-cycles $(x, u_m, x', u_k)$ for
$0 \leq m \neq k \leq 4$, all of which have weight $0$. The only edges
leaving $x$ we have not considered are $(x, v_0)$ and $(x, x')$. The
only non-trivial $4$-cycle containing $\{x, v_0\}$ and not $\{ x,x'\}$ is $(x, v_0, v_1,
u_1)$, and the only non-trivial $4$-cycles containing $\{x, x'\}$ are
$(x, x', u_0, v_0)$ and $(x, x', v_4, u_4)$. All of these have weight
$0$, proving the claim.

Now let $p$ be odd and let $\gamma$ be a $p$-cycle. Since all edges
have weight $\pm 1$ and $p$ is odd, $w(\gamma) \neq 0$. We claim that
$|w(\gamma)| < p$. Towards a contradiction suppose not.
Since $|w(e)|=1$ for all edges $e$ in $\Gamma_C$, we have that
$|w(\gamma)|=p$. By replacing $\gamma$ with $\gamma^{-1}$ if necessary,
we can assume $w(\gamma) = p$, meaning each edge traversed by $\gamma$
contributes $+1$ to $w(\gamma)$. Notice that since $\gamma$ is a cycle,
every vertex used by $\gamma$ must have a $+1$ incoming edge and a $+1$
outgoing edge. So $\gamma$ cannot use the vertex $v_2$ as no edge entering $v_2$ has weight $+1$.
Similarly, it cannot use $v_1$ or $v_3$ since the only edges with weight $+1$
entering these vertices emanate from $v_2$. Similarly, it cannot use $v_0$
or $v_4$, as their only $+1$ entering edges come from $v_1$ and $v_3$.
It also cannot use any $u_i$ for $1 \leq i \leq 4$ since the only $+1$
edge entering $u_i$ comes from $v_i$. Similarly, $\gamma$ cannot use $x'$
because every $+1$ edge entering $x'$ comes from a $u_i$
with $1 \leq i \leq 4$ or from $v_4$. Lastly, $\gamma$ cannot use $x$
because every $+1$ edge entering $x$ comes from a $u_i$ with $1 \leq i \leq 4$ or from $v_0$ or from $x'$. This means that $u_0$ is the only vertex
used by $\gamma$, a contradiction. We conclude that $w(\gamma)$ is
not divisible by $p$, so by Theorem \ref{thm:negweight} there is no
continuous graph homomorphism into $\Gamma_C$.
\end{proof}

While simpler to use, Theorem \ref{thm:negweight} is strictly weaker
than Theorem \ref{thm:neghom}. This will follow from considering the
``Klein bottle'' graph $\Gamma_K$ which is a quotient of the graph shown in
Figure~\ref{fig:klein-bottle-graph}. In that figure we have labeled
several vertices. $\Gamma_K$ is obtained by identifying vertices of
the same label (and identifying edges whose endpoints have identical
labels). Unlabeled vertices are not identified with one
another. Notice that $\Gamma_K$ would be a $5\times 5$ torus graph,
i.e., the Cayley graph of
$\scriptstyle{\left(\faktor{\Z}{5\Z}\right)^2}$, except that the
progression of nodes $a_i$ is downward ascending on the left side of
the graph and upward ascending on the right side of the graph.

\newcommand{\vph}{\vphantom{a_4}}
\begin{figure}[htbp]
\begin{tikzpicture}
        \matrix (m) [matrix of math nodes, row sep={5ex,between origins}, column sep={5ex,between origins}, nodes={nodeStyle, minimum size=3.5ex}] {
            x\vph & c_1  & c_2  & c_3  & c_4  & x\vph\\
            a_1   & \vph & \vph & \vph & \vph & a_4   \\
            a_2   & \vph & \vph & \vph & \vph & a_3   \\
            a_3   & \vph & \vph & \vph & \vph & a_2   \\
            a_4   & \vph & \vph & \vph & \vph & a_1   \\
            x\vph & c_1  & c_2  & c_3  & c_4  & x\vph \\
        };

    \foreach \i in {1,...,6} {
        \begin{scope} [start chain, every on chain/.style={join}]
            \foreach \j in {1,...,6} {
                \chainin (m-\j-\i);
            }
        \end{scope}

        \begin{scope} [start chain, every on chain/.style={join}]
            \foreach \j in {1,...,6} {
                \chainin (m-\i-\j);
            }
        \end{scope}
    }
\end{tikzpicture}
\caption[A graph $\Gamma_K$ such that $\Gamma^*_K$ is a cell-complex
homeomorphic to the Klein
bottle.]{\label{fig:klein-bottle-graph}$\Gamma_K$ is obtained as a
quotient of the above graph by identifying labeled vertices having the
same label (and by identifying edges whose endpoints have the same
labels). Unlabeled vertices are not identified. The $2$-dimensional
cell complex $\Gamma^*_K$ is homeomorphic to the Klein bottle.}
\end{figure}

\begin{exn}\label{ex:klein-bottle-graph}
There is no continuous graph homomorphism $\varphi \colon F(2^{\Z^2}) \rightarrow \Gamma_K$.
This follows from Theorem \ref{thm:neghom} but not Theorem \ref{thm:negweight}.
\end{exn}

\begin{proof}
Define the cycles $\alpha = (x, a_1, a_2, a_3, a_4, x)$ and
$\beta = (x, c_1, c_2, c_3, c_4, x)$.
Let $\pi_1^*(\Gamma_K)=\pi_1(\Gamma_K)/N$ be the reduced homotopy
group of $\Gamma_K$ with respect to the base point $x$.
We will use the identification $\pi_1^*(\Gamma_K) =
\pi_1(\Gamma_K^*)$. The $2$-dimensional complex $\Gamma_K^*$ can be
obtained by adding a $2$-cell for every $4$-cycle in the graph in
Figure~\ref{fig:klein-bottle-graph} and then taking the quotient. It
follows from this that $\Gamma_K^*$ is homeomorphic to the Klein
bottle. So $\pi_1^*(\Gamma_K) \cong \pi_1(\Gamma_K^*) \cong \G$ where
$\G = \langle g, h | g h = h g^{-1} \rangle$ and where the isomorphism
identifies $\pi_1^*(\alpha)$ with $g$ and $\pi_1^*(\beta)$ with
$h$. Notice that $g$ and $h$ have infinite order, that $\langle g
\rangle$ is normal in $\G$, and that $\G / \langle g \rangle = \{h^i
\langle g \rangle \colon i \in \Z\} \cong \Z$. It is also possible to show directly,
without appeal to $\Gamma_K^*$, that $\pi_1^*(\Gamma_K)\cong G$, by using an argument
similar to that of Proposition~\ref{prop:tilehom} to show that there are no more
relations in $\pi_1^*(\Gamma_K)$ than those of $G$.

We will apply Theorem \ref{thm:neghom}.
Fix any pair of relatively
prime integers $(p, q)$ with $p$ odd, and let $\gamma$ be a $p$-cycle
in $\Gamma_K$, say starting and ending at $v$. Note that any node in
$\Gamma_K$ is no more than four steps away from the node $x$.
Let $\delta$ be path of length $|\delta|\leq 4$ from $v$ to $x$.
Consider $\delta^{-1} \gamma \delta \in \pi_1(\Gamma_K)$ with
$|\delta^{-1} \gamma \delta|\leq |\gamma|+8$.

Notice that in the quotient graph $\Gamma_K$, the edge set can be
unambiguously partitioned into horizontal edges and vertical edges,
and every $4$-cycle has two horizontal edges and two vertical
edges. Also, it is straightforward to check that starting
at any vertex $v \in \Gamma_K$, following a vertical edge and then a
horizontal edge is equivalent modulo a $4$-cycle to following
a horizontal edge and then a vertical edge. It follows that there
is a cycle $\gamma'$ starting and ending at $x$ with $|\gamma'| =
|\delta^{-1} \gamma \delta|$,
$\pi_1^*(\gamma') = \pi_1^*(\delta^{-1} \gamma \delta)$
(recall every square-cycle is in $N$),
and with $\gamma' = \rho_1 \rho_2$ where $\rho_1$ consists of just
horizontal edge movements and $\rho_2$ just vertical edge movements.
Each of $\rho_1,\rho_2$ must be a cycle, since if $\rho_1$ did not end at $x$,
$\rho_1\rho_2$ would not be a cycle. It follows that there are $i,j\in \Z$
and  $\gamma''=\beta^i\alpha^j$
with $|\gamma''|\leq |\gamma'|$ and $\pi_1(\gamma'')=\pi_1(\gamma')$.
Note that in going from $\gamma'$ to $\gamma''$ we eliminate backtracking
which may reduce the cycle length. Since $p > 2$ and
$5|i| + 5|j| = |\gamma''| \leq |\gamma| + 8 = p+8$,
we must have $|i|, |j| < p$.

Towards a contradiction,
suppose there is a cycle $\tau$, starting and ending with $v$,
with $\gamma^q$ equal to $\tau^p$ in the reduced homotopy group of $\Gamma_K$
with base vertex $v$. It follows that
$\pi_1^*(\delta^{-1}\gamma^q \delta)=\pi_1^*(\delta^{-1}\tau^p \delta)=
\pi_1^*((\delta^{-1}\tau \delta)^p)$. Let $i', j'$ be such that
$\pi_1^*(\delta^{-1}\tau \delta)=\pi_1^*(\beta^{i'}\alpha^{j'})$.
Thus we have $(h^ig^j)^q=(h^{i'} g^{j'})^p$ in $G$.
In $G/\langle g \rangle$ we have
$h^{qi}\langle g \rangle = h^{pi'}\langle g \rangle$. Since $h$ has infinite order
in $G/\langle g \rangle$, this gives $p | (qi)$, and so $p|i$. Since $|i|<p$, this implies $i=0$,
which then gives $i'=0$ as well. So, $g^{qj}=g^{pj'}$ in $G$. Since $g$
has infinite order in $G$, the same argument now gives that $j=j'=0$.
Thus, $\pi_1^*(\delta^{-1} \gamma \delta)$ is the identity in $\pi_1^*(\Gamma_K)$.
This is a contradiction as $\delta^{-1} \gamma \delta$ is an odd-length cycle, and
$N$ consists of only even length cycles.
We conclude, via Theorem \ref{thm:neghom}, that there is no
continuous graph homomorphism from $F(2^{\Z^2})$ into $\Gamma_K$.

Finally, we claim that Theorem \ref{thm:negweight} cannot be applied to
$\Gamma_K$. Let $w$ be any weight which assigns weight $0$ to all
$4$-cycles. Then $w$ induces a group homomorphism $w :
\pi_1^*(\Gamma_K) \rightarrow \Z$. Since $\Z$ is abelian, $w$ must
nullify the commutator subgroup. As $\pi_1^*(\beta^{-1} \alpha \beta
\alpha^{-1}) = \pi_1^*(\alpha^{-2})$ (i.e. $h^{-1} g h g^{-1} =
g^{-2}$ in $\G$), we must have $- 2 w(\alpha) = w(\alpha^{-2}) = 0$
and thus $w(\alpha) = 0$. Now consider any $p \geq 5$. If $p$ is even
then choose a $p$-cycle $\gamma$ which simply alternates between two
adjacent vertices. Then $w(\gamma) = 0$ is divisible by $p$. If $p$ is
odd then let $\gamma$ be the $p$-cycle which first follows $\alpha$
once and then alternates between two adjacent vertices. Then
$w(\gamma) = w(\alpha) = 0$ is divisible by $p$.
\end{proof}

\section{A positive condition for graph homomorphisms} \label{sec:pc}

In this section we present a ``positive condition'' on the graph $\Gamma$
which is sufficient to guarantee the existence of a continuous graph homomorphism
from $\fzs$ to $\Gamma$. The positive condition, as was the negative condition,  is stated
in terms of the reduced homotopy group $\pi_1^*(\Gamma)$.

\begin{thmn}\label{thm:odd-length-cycle}
Let $\Gamma$ be a graph. Suppose $\Gamma$ contains an odd-length cycle $\gamma$ such that
$\pi_1^*(\gamma)$ has finite order in $\pi_1^*(\Gamma)$. Then there is a continuous graph
homomorphism $\varphi \colon F(2^{\Z^2}) \rightarrow \Gamma$.
\end{thmn}

\begin{proof}
By Corollary~\ref{cor:graphthm}, it
will suffice to find a graph homomorphism
$\varphi\colon \Gamma_{1,p,q}\to \Gamma$ for some $p, q>1$ with $\gcd(p,q)=1$.
Note that $\varphi$ can be viewed as a ``coloring'' of the vertices
of $\Gamma_{1,p,q}$ by colors in $V(\Gamma)$ (with
adjacent vertices in $\Gamma_{1,p,q}$ getting colors which
are adjacent in $\Gamma$).
Let $\ell$ be the length of $\gamma$, and let $m=|\pi_1^*(\gamma)|$
be the order of $\pi_1^*(\gamma)$ in $\pi_1^*(\Gamma)$.
Let $p\gg m\ell$
be prime, and let $q=p+m$.
Because $p> m$, $p$
does not divide $m$, so $p$ does not divide $q$.  Thus, $\gcd(p,q)=1$.
This gives us our parameters $p,q$.

Before we proceed to the coloring of $\Gamma_{1,p,q}$, we make the additional observation
that $q$ must be odd.  Indeed, $\pi_1^*(\Gamma)$ is just
$\pi_1(\Gamma)$ modded out by conjugates of four-cycles in $\Gamma$,
and each such conjugate is a cycle of even length.  Thus, no odd
length cycle is null-homotopic.  Since $\gamma$ has odd length and
$\gamma^m$ (with length $m\ell$) is null-homotopic, $m$ must be even,
hence $q$ is odd.  So we have that both $p,q$ are odd.

Let $v_0, v_1 \in V(\Gamma)$ be the first two vertices in $\gamma$.
Extend $\gamma$ with alternating instances of $v_0,v_1$
until it is length $p$, and call this extension $\alpha$; we can do
this because $p$ is odd and $\gamma$ is of odd length. Similarly, extend
$\gamma$ with alternating instances of $v_0,v_1$ until it is of
length $q$, and call this extension $\beta$.  In our coloring,
$v_0$ will correspond to the $R_\times$ block (note that now the $R_\times$ block has dimensions $1\times 1$), $\alpha$ to $R_a,R_c$ blocks, and
$\beta$ to $R_b,R_d$ blocks. More precisely, we define the coloring first
on the boundaries of the $12$ tiles comprising $\Gamma_{1,p,q}$
as follows. Each occurrence of an $R_\times$ block is labeled with the vertex $v_0$.
Each occurrence of adjacent blocks of the form $R_\times R_a R_\times$ is colored
by the vertices in the cycle $\alpha$, with the order of the vertices occurring in $\alpha$ matching the upward order of the blocks $R_\times R_a R_\times$.
Similarly, the adjacent $R_\times R_b R_\times$ blocks are also
colored by vertices of $\beta$ in the same order.
Likewise, the adjacent $R_\times R_c R_\times$ and  $R_\times R_d R_\times$ blocks
are colored by the vertices in $\alpha$, $\beta$ respectively,
in the left-to-right order. This defines the coloring
$\varphi$ on the boundaries of the $12$ tiles in $\Gamma_{1,p,q}$, and
we have clearly respected the vertex identifications in $\Gamma_{1,p,q}$.

It is also instructive to note the following fact.
Consider, for the moment, the boundary of the long
tile~$\Tcqadpa$ in $\Gamma_{1,p,q}$ (see Figure~\ref{fig:Gamma-npq-horiz}).
Moving counterclockwise from the bottom-left corner and
disregarding the null-homotopic ``filler'' between copies of $\gamma$ consisting of two-cycles $(v_0,v_1,v_0)$,
we encounter $p$ copies of $\gamma$ corresponding to $R_d$, one copy
of $\gamma$ corresponding to $R_a$, $q$ \emph{backwards} copies of
$\gamma$ corresponding to $R_c$, and one final backwards copy of
$\gamma$ corresponding to $R_a$ again.  Thus, the boundary of this
tile is homotopic to $\gamma^{p-q}$.  But we have arranged that
$m=p-q$, so the boundary is actually null-homotopic.  Similarly, all
twelve tiles have null-homotopic boundaries, so our scheme has, so far,
avoided homotopy obstructions.

Having colored the boundaries of the tiles, it suffices to color each tile in
isolation.  The general strategy to get started will be to connect pairs of forwards and backwards copies of
$\gamma$ at the boundary of a tile with special paths, and then fill
the remaining space with alternating $v_0,v_1$.  This
process is illustrated in Figure~\ref{fig:parity-change-routing}. Note that in order for this strategy to work, we need the occurrences of the coloring by $\gamma$ to be of the same parity, as illustrated in Figure~\ref{fig:parity-change-routing}. More precisely, for each occurrence of the coloring by $\gamma$, whether it is a forwards or backwards copy, we concentrate on the position of the leading $v_0$ (one adjacent to $v_1$) and regard its parity to be the parity of the occurrence of $\gamma$. To apply this strategy, we assign the label of $v_0$ to the points in some chosen diagonally-moving path (the diagonal movement is why parity is important) that connects the $v_0$'s of the two instances of $\gamma$. This diagonally-moving path divides the tile into two pieces. We can fill one side by alternating $v_0$ and $v_1$. For the other side, the region consisting of points of distance at most $\ell$ to the diagonally-moving path can be filled in using $\gamma$, and the remaining portion is filled in by alternating $v_0$ and $v_1$. Thus the process in Figure~\ref{fig:parity-change-routing} works whenever the two occurrences of $\gamma$ are of the same parity but of the opposite orientation. Aside from that constraint, this process is quite flexible, as shown in Figure~\ref{fig:parity-change-routing}; a path of width $\ell+1$ and almost
arbitrary shape can connect two copies of $\gamma$ satisfying the above conditions. For future reference we refer to this connection algorithm as Algorithm {\sc (c)} and the trivial algorithm to fill in alternating $v_0$ and $v_1$ as Algorithm {\sc (t)}.

\begin{figure}[htbp]
\begin{tikzpicture}[scale=0.3]
    \colorlet{color8}{black!0}
    \colorlet{color9}{black!20}
    \colorlet{color3}{black!100}
    \colorlet{color5}{black!70}
    \colorlet{color1}{red!50}
    \colorlet{color2}{blue!30}
    \colorlet{color4}{blue!60}

\foreach \i/\j/\clrNum in {0/0/8, 0/1/9, 0/2/3, 0/3/5, 0/4/1, 0/5/2, 0/6/4, 0/7/8, 0/8/9, 0/9/8, 0/10/9, 0/11/8, 0/12/9, 0/13/8, 0/14/9, 0/15/8, 0/16/9, 0/17/8, 0/18/9, 0/19/8, 0/20/9, 0/21/8, 0/22/9, 0/23/8, 0/24/9, 0/25/8, 0/26/9, 0/27/8, 0/28/9, 0/29/8, 0/30/9, 0/31/8, 0/32/9, 0/33/8, 0/34/9, 0/35/8, 0/36/9, 0/37/8, 0/38/9, 0/39/8, 1/0/9, 1/1/3, 1/2/5, 1/3/1, 1/4/2, 1/5/4, 1/6/8, 1/7/9, 1/8/8, 1/9/9, 1/10/8, 1/11/9, 1/12/8, 1/13/9, 1/14/8, 1/15/9, 1/16/8, 1/17/9, 1/18/8, 1/19/9, 1/20/8, 1/21/9, 1/22/8, 1/23/9, 1/24/8, 1/25/9, 1/26/8, 1/27/9, 1/28/8, 1/29/9, 1/30/8, 1/31/9, 1/32/8, 1/33/9, 1/34/8, 1/35/9, 1/36/8, 1/37/9, 1/38/8, 1/39/9, 2/0/8, 2/1/9, 2/2/3, 2/3/5, 2/4/1, 2/5/2, 2/6/4, 2/7/8, 2/8/9, 2/9/8, 2/10/9, 2/11/8, 2/12/4, 2/13/8, 2/14/9, 2/15/8, 2/16/4, 2/17/8, 2/18/4, 2/19/8, 2/20/4, 2/21/8, 2/22/9, 2/23/8, 2/24/9, 2/25/8, 2/26/9, 2/27/8, 2/28/9, 2/29/8, 2/30/9, 2/31/8, 2/32/9, 2/33/8, 2/34/9, 2/35/8, 2/36/9, 2/37/8, 2/38/9, 2/39/8, 3/0/9, 3/1/3, 3/2/5, 3/3/1, 3/4/2, 3/5/4, 3/6/8, 3/7/9, 3/8/8, 3/9/9, 3/10/8, 3/11/4, 3/12/2, 3/13/4, 3/14/8, 3/15/4, 3/16/2, 3/17/4, 3/18/2, 3/19/4, 3/20/2, 3/21/4, 3/22/8, 3/23/9, 3/24/8, 3/25/9, 3/26/8, 3/27/9, 3/28/8, 3/29/9, 3/30/8, 3/31/9, 3/32/8, 3/33/9, 3/34/8, 3/35/9, 3/36/8, 3/37/9, 3/38/8, 3/39/9, 4/0/8, 4/1/9, 4/2/3, 4/3/5, 4/4/1, 4/5/2, 4/6/4, 4/7/8, 4/8/9, 4/9/8, 4/10/4, 4/11/2, 4/12/1, 4/13/2, 4/14/4, 4/15/2, 4/16/1, 4/17/2, 4/18/1, 4/19/2, 4/20/1, 4/21/2, 4/22/4, 4/23/8, 4/24/9, 4/25/8, 4/26/9, 4/27/8, 4/28/9, 4/29/8, 4/30/9, 4/31/8, 4/32/9, 4/33/8, 4/34/9, 4/35/8, 4/36/9, 4/37/8, 4/38/9, 4/39/8, 5/0/9, 5/1/8, 5/2/9, 5/3/3, 5/4/5, 5/5/1, 5/6/2, 5/7/4, 5/8/8, 5/9/4, 5/10/2, 5/11/1, 5/12/5, 5/13/1, 5/14/2, 5/15/1, 5/16/5, 5/17/1, 5/18/5, 5/19/1, 5/20/5, 5/21/1, 5/22/2, 5/23/4, 5/24/8, 5/25/9, 5/26/8, 5/27/9, 5/28/8, 5/29/9, 5/30/8, 5/31/9, 5/32/8, 5/33/9, 5/34/8, 5/35/9, 5/36/8, 5/37/9, 5/38/8, 5/39/9, 6/0/8, 6/1/9, 6/2/8, 6/3/9, 6/4/3, 6/5/5, 6/6/1, 6/7/2, 6/8/4, 6/9/2, 6/10/1, 6/11/5, 6/12/3, 6/13/5, 6/14/1, 6/15/5, 6/16/3, 6/17/5, 6/18/3, 6/19/5, 6/20/3, 6/21/5, 6/22/1, 6/23/2, 6/24/4, 6/25/8, 6/26/9, 6/27/8, 6/28/9, 6/29/8, 6/30/9, 6/31/8, 6/32/9, 6/33/8, 6/34/9, 6/35/8, 6/36/9, 6/37/8, 6/38/9, 6/39/8, 7/0/9, 7/1/8, 7/2/9, 7/3/8, 7/4/9, 7/5/3, 7/6/5, 7/7/1, 7/8/2, 7/9/1, 7/10/5, 7/11/3, 7/12/9, 7/13/3, 7/14/5, 7/15/3, 7/16/9, 7/17/3, 7/18/9, 7/19/3, 7/20/9, 7/21/3, 7/22/5, 7/23/1, 7/24/2, 7/25/4, 7/26/8, 7/27/9, 7/28/8, 7/29/9, 7/30/8, 7/31/9, 7/32/8, 7/33/9, 7/34/8, 7/35/9, 7/36/8, 7/37/9, 7/38/8, 7/39/9, 8/0/8, 8/1/9, 8/2/8, 8/3/9, 8/4/8, 8/5/9, 8/6/3, 8/7/5, 8/8/1, 8/9/5, 8/10/3, 8/11/9, 8/12/8, 8/13/9, 8/14/3, 8/15/9, 8/16/8, 8/17/9, 8/18/8, 8/19/9, 8/20/8, 8/21/9, 8/22/3, 8/23/5, 8/24/1, 8/25/2, 8/26/4, 8/27/8, 8/28/9, 8/29/8, 8/30/9, 8/31/8, 8/32/9, 8/33/8, 8/34/9, 8/35/8, 8/36/9, 8/37/8, 8/38/9, 8/39/8, 9/0/9, 9/1/8, 9/2/9, 9/3/8, 9/4/9, 9/5/8, 9/6/9, 9/7/3, 9/8/5, 9/9/3, 9/10/9, 9/11/8, 9/12/9, 9/13/8, 9/14/9, 9/15/8, 9/16/9, 9/17/8, 9/18/9, 9/19/8, 9/20/9, 9/21/8, 9/22/9, 9/23/3, 9/24/5, 9/25/1, 9/26/2, 9/27/4, 9/28/8, 9/29/9, 9/30/8, 9/31/9, 9/32/8, 9/33/9, 9/34/8, 9/35/9, 9/36/8, 9/37/9, 9/38/8, 9/39/9, 10/0/8, 10/1/9, 10/2/8, 10/3/9, 10/4/8, 10/5/9, 10/6/8, 10/7/9, 10/8/3, 10/9/9, 10/10/8, 10/11/9, 10/12/8, 10/13/9, 10/14/8, 10/15/9, 10/16/8, 10/17/9, 10/18/8, 10/19/9, 10/20/8, 10/21/9, 10/22/8, 10/23/9, 10/24/3, 10/25/5, 10/26/1, 10/27/2, 10/28/4, 10/29/8, 10/30/9, 10/31/8, 10/32/9, 10/33/8, 10/34/9, 10/35/8, 10/36/9, 10/37/8, 10/38/9, 10/39/8, 11/0/9, 11/1/8, 11/2/9, 11/3/8, 11/4/9, 11/5/8, 11/6/9, 11/7/8, 11/8/9, 11/9/8, 11/10/9, 11/11/8, 11/12/9, 11/13/8, 11/14/9, 11/15/8, 11/16/9, 11/17/8, 11/18/9, 11/19/8, 11/20/9, 11/21/8, 11/22/9, 11/23/8, 11/24/9, 11/25/3, 11/26/5, 11/27/1, 11/28/2, 11/29/4, 11/30/8, 11/31/9, 11/32/8, 11/33/9, 11/34/8, 11/35/9, 11/36/8, 11/37/9, 11/38/8, 11/39/9, 12/0/8, 12/1/9, 12/2/8, 12/3/9, 12/4/8, 12/5/9, 12/6/8, 12/7/9, 12/8/8, 12/9/9, 12/10/8, 12/11/9, 12/12/8, 12/13/9, 12/14/8, 12/15/9, 12/16/8, 12/17/9, 12/18/8, 12/19/9, 12/20/8, 12/21/9, 12/22/8, 12/23/9, 12/24/8, 12/25/9, 12/26/3, 12/27/5, 12/28/1, 12/29/2, 12/30/4, 12/31/8, 12/32/9, 12/33/8, 12/34/9, 12/35/8, 12/36/9, 12/37/8, 12/38/9, 12/39/8, 13/0/9, 13/1/8, 13/2/9, 13/3/8, 13/4/9, 13/5/8, 13/6/9, 13/7/8, 13/8/9, 13/9/8, 13/10/9, 13/11/8, 13/12/9, 13/13/8, 13/14/9, 13/15/8, 13/16/9, 13/17/8, 13/18/9, 13/19/8, 13/20/9, 13/21/8, 13/22/9, 13/23/8, 13/24/9, 13/25/8, 13/26/9, 13/27/3, 13/28/5, 13/29/1, 13/30/2, 13/31/4, 13/32/8, 13/33/9, 13/34/8, 13/35/9, 13/36/8, 13/37/9, 13/38/8, 13/39/9, 14/0/8, 14/1/9, 14/2/8, 14/3/9, 14/4/8, 14/5/9, 14/6/8, 14/7/9, 14/8/8, 14/9/9, 14/10/8, 14/11/9, 14/12/8, 14/13/9, 14/14/8, 14/15/9, 14/16/8, 14/17/9, 14/18/8, 14/19/9, 14/20/8, 14/21/9, 14/22/8, 14/23/9, 14/24/8, 14/25/9, 14/26/8, 14/27/9, 14/28/3, 14/29/5, 14/30/1, 14/31/2, 14/32/4, 14/33/8, 14/34/9, 14/35/8, 14/36/9, 14/37/8, 14/38/9, 14/39/8, 15/0/9, 15/1/8, 15/2/9, 15/3/8, 15/4/9, 15/5/8, 15/6/9, 15/7/8, 15/8/9, 15/9/8, 15/10/9, 15/11/8, 15/12/9, 15/13/8, 15/14/9, 15/15/8, 15/16/9, 15/17/8, 15/18/9, 15/19/8, 15/20/9, 15/21/8, 15/22/9, 15/23/8, 15/24/9, 15/25/8, 15/26/9, 15/27/8, 15/28/9, 15/29/3, 15/30/5, 15/31/1, 15/32/2, 15/33/4, 15/34/8, 15/35/9, 15/36/8, 15/37/9, 15/38/8, 15/39/9, 16/0/8, 16/1/9, 16/2/8, 16/3/9, 16/4/8, 16/5/9, 16/6/8, 16/7/9, 16/8/8, 16/9/9, 16/10/8, 16/11/9, 16/12/8, 16/13/9, 16/14/8, 16/15/9, 16/16/8, 16/17/9, 16/18/3, 16/19/9, 16/20/8, 16/21/9, 16/22/8, 16/23/9, 16/24/8, 16/25/9, 16/26/8, 16/27/9, 16/28/8, 16/29/9, 16/30/3, 16/31/5, 16/32/1, 16/33/2, 16/34/4, 16/35/8, 16/36/9, 16/37/8, 16/38/9, 16/39/8, 17/0/9, 17/1/8, 17/2/9, 17/3/8, 17/4/9, 17/5/8, 17/6/9, 17/7/8, 17/8/9, 17/9/8, 17/10/9, 17/11/8, 17/12/9, 17/13/8, 17/14/9, 17/15/8, 17/16/9, 17/17/3, 17/18/5, 17/19/3, 17/20/9, 17/21/8, 17/22/9, 17/23/8, 17/24/9, 17/25/8, 17/26/9, 17/27/8, 17/28/9, 17/29/8, 17/30/9, 17/31/3, 17/32/5, 17/33/1, 17/34/2, 17/35/4, 17/36/8, 17/37/9, 17/38/8, 17/39/9, 18/0/8, 18/1/9, 18/2/8, 18/3/9, 18/4/8, 18/5/9, 18/6/8, 18/7/9, 18/8/8, 18/9/9, 18/10/8, 18/11/9, 18/12/8, 18/13/9, 18/14/8, 18/15/9, 18/16/3, 18/17/5, 18/18/1, 18/19/5, 18/20/3, 18/21/9, 18/22/8, 18/23/9, 18/24/8, 18/25/9, 18/26/8, 18/27/9, 18/28/8, 18/29/9, 18/30/8, 18/31/9, 18/32/3, 18/33/5, 18/34/1, 18/35/2, 18/36/4, 18/37/8, 18/38/9, 18/39/8, 19/0/9, 19/1/8, 19/2/9, 19/3/8, 19/4/9, 19/5/8, 19/6/9, 19/7/8, 19/8/9, 19/9/8, 19/10/9, 19/11/8, 19/12/9, 19/13/8, 19/14/9, 19/15/3, 19/16/5, 19/17/1, 19/18/2, 19/19/1, 19/20/5, 19/21/3, 19/22/9, 19/23/8, 19/24/9, 19/25/8, 19/26/9, 19/27/8, 19/28/9, 19/29/8, 19/30/9, 19/31/8, 19/32/9, 19/33/3, 19/34/5, 19/35/1, 19/36/2, 19/37/4, 19/38/8, 19/39/9, 20/0/8, 20/1/9, 20/2/8, 20/3/9, 20/4/8, 20/5/9, 20/6/8, 20/7/9, 20/8/8, 20/9/9, 20/10/8, 20/11/9, 20/12/8, 20/13/9, 20/14/3, 20/15/5, 20/16/1, 20/17/2, 20/18/4, 20/19/2, 20/20/1, 20/21/5, 20/22/3, 20/23/9, 20/24/8, 20/25/9, 20/26/8, 20/27/9, 20/28/8, 20/29/9, 20/30/8, 20/31/9, 20/32/8, 20/33/9, 20/34/3, 20/35/5, 20/36/1, 20/37/2, 20/38/4, 20/39/8, 21/0/9, 21/1/8, 21/2/9, 21/3/8, 21/4/9, 21/5/8, 21/6/9, 21/7/8, 21/8/9, 21/9/8, 21/10/9, 21/11/8, 21/12/9, 21/13/3, 21/14/5, 21/15/1, 21/16/2, 21/17/4, 21/18/8, 21/19/4, 21/20/2, 21/21/1, 21/22/5, 21/23/3, 21/24/9, 21/25/8, 21/26/9, 21/27/8, 21/28/9, 21/29/8, 21/30/9, 21/31/8, 21/32/9, 21/33/3, 21/34/5, 21/35/1, 21/36/2, 21/37/4, 21/38/8, 21/39/9, 22/0/8, 22/1/9, 22/2/8, 22/3/9, 22/4/8, 22/5/9, 22/6/8, 22/7/9, 22/8/8, 22/9/9, 22/10/8, 22/11/9, 22/12/3, 22/13/5, 22/14/1, 22/15/2, 22/16/4, 22/17/8, 22/18/9, 22/19/8, 22/20/4, 22/21/2, 22/22/1, 22/23/5, 22/24/3, 22/25/9, 22/26/8, 22/27/9, 22/28/8, 22/29/9, 22/30/8, 22/31/9, 22/32/3, 22/33/5, 22/34/1, 22/35/2, 22/36/4, 22/37/8, 22/38/9, 22/39/8, 23/0/9, 23/1/8, 23/2/9, 23/3/8, 23/4/9, 23/5/8, 23/6/9, 23/7/8, 23/8/9, 23/9/8, 23/10/9, 23/11/8, 23/12/9, 23/13/3, 23/14/5, 23/15/1, 23/16/2, 23/17/4, 23/18/8, 23/19/9, 23/20/8, 23/21/4, 23/22/2, 23/23/1, 23/24/5, 23/25/3, 23/26/9, 23/27/3, 23/28/9, 23/29/3, 23/30/9, 23/31/3, 23/32/5, 23/33/1, 23/34/2, 23/35/4, 23/36/8, 23/37/9, 23/38/8, 23/39/9, 24/0/8, 24/1/9, 24/2/8, 24/3/9, 24/4/8, 24/5/9, 24/6/8, 24/7/9, 24/8/8, 24/9/9, 24/10/8, 24/11/9, 24/12/3, 24/13/5, 24/14/1, 24/15/2, 24/16/4, 24/17/8, 24/18/9, 24/19/8, 24/20/9, 24/21/8, 24/22/4, 24/23/2, 24/24/1, 24/25/5, 24/26/3, 24/27/5, 24/28/3, 24/29/5, 24/30/3, 24/31/5, 24/32/1, 24/33/2, 24/34/4, 24/35/8, 24/36/9, 24/37/8, 24/38/9, 24/39/8, 25/0/9, 25/1/8, 25/2/9, 25/3/8, 25/4/9, 25/5/8, 25/6/9, 25/7/8, 25/8/9, 25/9/8, 25/10/9, 25/11/8, 25/12/9, 25/13/3, 25/14/5, 25/15/1, 25/16/2, 25/17/4, 25/18/8, 25/19/9, 25/20/8, 25/21/9, 25/22/8, 25/23/4, 25/24/2, 25/25/1, 25/26/5, 25/27/1, 25/28/5, 25/29/1, 25/30/5, 25/31/1, 25/32/2, 25/33/4, 25/34/8, 25/35/9, 25/36/8, 25/37/9, 25/38/8, 25/39/9, 26/0/8, 26/1/9, 26/2/8, 26/3/9, 26/4/8, 26/5/9, 26/6/8, 26/7/9, 26/8/8, 26/9/9, 26/10/8, 26/11/9, 26/12/3, 26/13/5, 26/14/1, 26/15/2, 26/16/4, 26/17/8, 26/18/9, 26/19/8, 26/20/9, 26/21/8, 26/22/9, 26/23/8, 26/24/4, 26/25/2, 26/26/1, 26/27/2, 26/28/1, 26/29/2, 26/30/1, 26/31/2, 26/32/4, 26/33/8, 26/34/9, 26/35/8, 26/36/9, 26/37/8, 26/38/9, 26/39/8, 27/0/9, 27/1/8, 27/2/9, 27/3/8, 27/4/9, 27/5/8, 27/6/9, 27/7/8, 27/8/9, 27/9/8, 27/10/9, 27/11/8, 27/12/9, 27/13/3, 27/14/5, 27/15/1, 27/16/2, 27/17/4, 27/18/8, 27/19/9, 27/20/8, 27/21/9, 27/22/8, 27/23/9, 27/24/8, 27/25/4, 27/26/2, 27/27/4, 27/28/2, 27/29/4, 27/30/2, 27/31/4, 27/32/8, 27/33/9, 27/34/8, 27/35/9, 27/36/8, 27/37/9, 27/38/8, 27/39/9, 28/0/8, 28/1/9, 28/2/8, 28/3/9, 28/4/8, 28/5/9, 28/6/8, 28/7/9, 28/8/8, 28/9/9, 28/10/8, 28/11/9, 28/12/3, 28/13/5, 28/14/1, 28/15/2, 28/16/4, 28/17/8, 28/18/9, 28/19/8, 28/20/9, 28/21/8, 28/22/9, 28/23/8, 28/24/9, 28/25/8, 28/26/4, 28/27/8, 28/28/4, 28/29/8, 28/30/4, 28/31/8, 28/32/9, 28/33/8, 28/34/9, 28/35/8, 28/36/9, 28/37/8, 28/38/9, 28/39/8, 29/0/9, 29/1/8, 29/2/9, 29/3/8, 29/4/9, 29/5/8, 29/6/9, 29/7/8, 29/8/9, 29/9/8, 29/10/9, 29/11/8, 29/12/9, 29/13/3, 29/14/5, 29/15/1, 29/16/2, 29/17/4, 29/18/8, 29/19/9, 29/20/8, 29/21/9, 29/22/8, 29/23/9, 29/24/8, 29/25/9, 29/26/8, 29/27/9, 29/28/8, 29/29/9, 29/30/8, 29/31/9, 29/32/8, 29/33/9, 29/34/8, 29/35/9, 29/36/8, 29/37/9, 29/38/8, 29/39/9, -3/39/4, -3/38/2, -3/37/1, -3/36/5, -3/35/3, -3/34/9, -3/33/8} {
    \filldraw[fill=color\clrNum, draw=black] (\j,\i) rectangle ++(1,1);
}
\node[label=left:$\gamma\signatureSep\vphantom{pM}$, minimum size=0mm] at (33,-2.5) {};

\end{tikzpicture}
\caption{\label{fig:parity-change-routing}An illustration of how to ``connect''
two copies of $\gamma$ using a winding path.}
\end{figure} 

We now analyze the parity of occurrences of $\gamma$ on the boundaries
of the 12 tiles. For each tile, if we set the bottom-left corner to be
the origin, then going around the boundaries in the counter-clockwise
direction, we will see forwards copies of $\gamma$ along the bottom
and right boundaries and then backwards copies of $\gamma$ along the
top and left boundaries. As $p, q$ are odd, the parity of the forwards copies of $\gamma$
alternate, starting with the even parity for the first copy. The
parity of the backwards copies of $\gamma$ also alternate, ending with
the even parity for the last copy. Thus, if we pair up the first
forwards copy of $\gamma$ with the last backwards copy of $\gamma$,
and then the second forwards copy of $\gamma$ with the penultimate
backwards copy of $\gamma$, etc., we will be able to apply Algorithm {\sc (c)} to all these pairs. Obviously, we have to make
sure that the connecting paths will not intersect each other in order
for the algorithm to work, but this can be guaranteed by the fact that
$p, q\gg m\ell>\ell+1$, which guarantees enough room in between these
connecting paths.

For the first eight tiles, this will result in all copies of $\gamma$ being accounted for, and therefore these tiles will be completely colored with just nodes on $\gamma$.
However,
for the long tiles $\Tcqadpa$, $\Tdpacqa$, $\Tcaqcbp$, and
$\Tcbpcaq$, there will be $m$ unmatched copies of $\gamma$ in the same orientation, with alternating parities. We will describe a different algorithm to handle these copies of $\gamma$.

The algorithm starts by collecting the remaining copies of $\gamma$ and applying Algorithm {\sc (c)} to make a single copy of $\gamma^m$. Next, we identify a rectangular region in the uncolored part of the long tile and connect this single copy of $\gamma^m$ to a part of its boundary. It will be clear once we finish describing the algorithm that the dimensions of the rectangular region needed to perform the algorithm are much smaller than $p, q$ and therefore such a region can be easily accommodated in the uncolored part of the long tile. We will call this rectangular region the \emph{hub}. To be clear, the boundary of the hub consists of a copy of $\gamma^m$ together with alternating
$v_0,v_1$. The hub will be completely filled to concretely manifest the fact that $\gamma^m$ is null-homotopic.

It remains to describe the coloring within the hub. For definiteness we assume that the copy of $\gamma^m$ occurs on the top boundary of the hub in the backwards orientation.

To describe the algorithm for filling the hub we need to define two auxiliary algorithms. The first one is called the {\em spur elimination} algorithm, or Algorithm {\sc (s)}, and the second one is the {\em $4$-cycle elimination} algorithm, or Algorithm {\sc (f)}. A {\em spur} is just a $2$-cycle. Suppose that $(u_0, u_1, u_0)$ is a $2$-cycle in $\Gamma$, and three adjacent vertices $x_1, x_2=(1,0)\cdot x_1, x_3=(2,0)\cdot x_1$ in a grid graph have been colored by $u_0, u_1, u_0$, in this order. Then to ``eliminate" the $2$-cycle $(u_0, u_1, u_0)$ we just color the vertex above $x_2$ by $u_0$. In general, if a cycle $\eta$ is equivalent to the identity cycle in $\pi_1(\Gamma)$, then we can obtain the identity cycle from $\eta$ by successively eliminating the spurs that occur in $\eta$. As an example, suppose a row of the grid graph has been colored by the cycle $\sigma=(u_0, u_1, u_2, u_3, u_2, u_4, u_2, u_1. u_0)$, which is equivalent to the identity cycle in $\pi_1(\Gamma)$. We can repeat the spur elimination scheme to accomplish the following coloring:
\begin{center}
\begin{tikzpicture}
\matrix [matrix of math nodes, row sep={2em,between origins}, column sep={2em,between origins}] {
          &      &        &            &    u_0 &       &     &       &      \\
          &      &        &     (u_0 & u_1   &u_0) &     &       & \\
          &      &u_0&(u_1&u_2  &u_1)&u_0  &\\
          &u_0&u_1&(u_2&u_4    &u_2)&u_1&u_0&  \\
    u_0&u_1&(u_2&u_3   &u_2)  &u_4&u_2&u_1&u_0  \\
};
\end{tikzpicture}
\end{center}
In each step we marked out the spur that was eliminated.  In general, suppose we have a rectangular region whose boundary consists of such an $\eta$ plus alternating
$v_0,v_1$, where $v_0, v_1$ are the first two nodes in $\eta$. Then working with one spur at a time, we can successively produce a coloring of the row above by nodes in $\eta$ until all nodes of $\eta$ other than $v_0, v_1$ are no longer needed. This is the spur elimination algorithm or Algorithm {\sc (s)}. An example of the application of Algorithm {\sc (s)} to a much longer cycle equivalent to identity is shown in
Figure~\ref{fig:spur-elimination}.  Each spur eliminated shortens the
cycle by one vertex, but because we only ever shorten by a single
vertex at a time, we can always situate the shortened cycle above its
predecessor.
\begin{figure}[hbtp]
\begin{tikzpicture}[scale=0.27]
    \colorlet{color8}{black!0}
    \colorlet{color9}{black!20}
    \colorlet{color3}{black!100}
    \colorlet{color5}{black!70}
    \colorlet{color1}{red!50}
    \colorlet{color2}{blue!30}
    \colorlet{color4}{blue!60}
\foreach \i/\j/\clrNum in {0/0/8, 0/1/9, 0/2/3, 0/3/5, 0/4/1, 0/5/2, 0/6/4, 0/7/2, 0/8/1, 0/9/5, 0/10/3, 0/11/9, 0/12/3, 0/13/5, 0/14/1, 0/15/5, 0/16/3, 0/17/9, 0/18/3, 0/19/5, 0/20/1, 0/21/2, 0/22/4, 0/23/8, 0/24/9, 0/25/3, 0/26/5, 0/27/1, 0/28/2, 0/29/4, 0/30/8, 0/31/9, 0/32/8, 0/33/4, 0/34/2, 0/35/1, 0/36/5, 0/37/3, 0/38/9, 0/39/8, 0/40/4, 0/41/2, 0/42/1, 0/43/5, 0/44/3, 0/45/9, 0/46/8, 1/0/9, 1/1/8, 1/2/9, 1/3/3, 1/4/5, 1/5/1, 1/6/2, 1/7/1, 1/8/5, 1/9/3, 1/10/9, 1/11/3, 1/12/5, 1/13/1, 1/14/5, 1/15/3, 1/16/9, 1/17/3, 1/18/5, 1/19/1, 1/20/2, 1/21/4, 1/22/8, 1/23/9, 1/24/3, 1/25/5, 1/26/1, 1/27/2, 1/28/4, 1/29/8, 1/30/9, 1/31/8, 1/32/4, 1/33/2, 1/34/1, 1/35/5, 1/36/3, 1/37/9, 1/38/8, 1/39/4, 1/40/2, 1/41/1, 1/42/5, 1/43/3, 1/44/9, 1/45/8, 1/46/9, 2/0/8, 2/1/9, 2/2/8, 2/3/9, 2/4/3, 2/5/5, 2/6/1, 2/7/5, 2/8/3, 2/9/9, 2/10/3, 2/11/5, 2/12/1, 2/13/5, 2/14/3, 2/15/9, 2/16/3, 2/17/5, 2/18/1, 2/19/2, 2/20/4, 2/21/8, 2/22/9, 2/23/3, 2/24/5, 2/25/1, 2/26/2, 2/27/4, 2/28/8, 2/29/9, 2/30/8, 2/31/4, 2/32/2, 2/33/1, 2/34/5, 2/35/3, 2/36/9, 2/37/8, 2/38/4, 2/39/2, 2/40/1, 2/41/5, 2/42/3, 2/43/9, 2/44/8, 2/45/9, 2/46/8, 3/0/9, 3/1/8, 3/2/9, 3/3/8, 3/4/9, 3/5/3, 3/6/5, 3/7/3, 3/8/9, 3/9/3, 3/10/5, 3/11/1, 3/12/5, 3/13/3, 3/14/9, 3/15/3, 3/16/5, 3/17/1, 3/18/2, 3/19/4, 3/20/8, 3/21/9, 3/22/3, 3/23/5, 3/24/1, 3/25/2, 3/26/4, 3/27/8, 3/28/9, 3/29/8, 3/30/4, 3/31/2, 3/32/1, 3/33/5, 3/34/3, 3/35/9, 3/36/8, 3/37/4, 3/38/2, 3/39/1, 3/40/5, 3/41/3, 3/42/9, 3/43/8, 3/44/9, 3/45/8, 3/46/9, 4/0/8, 4/1/9, 4/2/8, 4/3/9, 4/4/8, 4/5/9, 4/6/3, 4/7/9, 4/8/3, 4/9/5, 4/10/1, 4/11/5, 4/12/3, 4/13/9, 4/14/3, 4/15/5, 4/16/1, 4/17/2, 4/18/4, 4/19/8, 4/20/9, 4/21/3, 4/22/5, 4/23/1, 4/24/2, 4/25/4, 4/26/8, 4/27/9, 4/28/8, 4/29/4, 4/30/2, 4/31/1, 4/32/5, 4/33/3, 4/34/9, 4/35/8, 4/36/4, 4/37/2, 4/38/1, 4/39/5, 4/40/3, 4/41/9, 4/42/8, 4/43/9, 4/44/8, 4/45/9, 4/46/8, 5/0/9, 5/1/8, 5/2/9, 5/3/8, 5/4/9, 5/5/8, 5/6/9, 5/7/3, 5/8/5, 5/9/1, 5/10/5, 5/11/3, 5/12/9, 5/13/3, 5/14/5, 5/15/1, 5/16/2, 5/17/4, 5/18/8, 5/19/9, 5/20/3, 5/21/5, 5/22/1, 5/23/2, 5/24/4, 5/25/8, 5/26/9, 5/27/8, 5/28/4, 5/29/2, 5/30/1, 5/31/5, 5/32/3, 5/33/9, 5/34/8, 5/35/4, 5/36/2, 5/37/1, 5/38/5, 5/39/3, 5/40/9, 5/41/8, 5/42/9, 5/43/8, 5/44/9, 5/45/8, 5/46/9, 6/0/8, 6/1/9, 6/2/8, 6/3/9, 6/4/8, 6/5/9, 6/6/8, 6/7/9, 6/8/3, 6/9/5, 6/10/3, 6/11/9, 6/12/3, 6/13/5, 6/14/1, 6/15/2, 6/16/4, 6/17/8, 6/18/9, 6/19/3, 6/20/5, 6/21/1, 6/22/2, 6/23/4, 6/24/8, 6/25/9, 6/26/8, 6/27/4, 6/28/2, 6/29/1, 6/30/5, 6/31/3, 6/32/9, 6/33/8, 6/34/4, 6/35/2, 6/36/1, 6/37/5, 6/38/3, 6/39/9, 6/40/8, 6/41/9, 6/42/8, 6/43/9, 6/44/8, 6/45/9, 6/46/8, 7/0/9, 7/1/8, 7/2/9, 7/3/8, 7/4/9, 7/5/8, 7/6/9, 7/7/8, 7/8/9, 7/9/3, 7/10/9, 7/11/3, 7/12/5, 7/13/1, 7/14/2, 7/15/4, 7/16/8, 7/17/9, 7/18/3, 7/19/5, 7/20/1, 7/21/2, 7/22/4, 7/23/8, 7/24/9, 7/25/8, 7/26/4, 7/27/2, 7/28/1, 7/29/5, 7/30/3, 7/31/9, 7/32/8, 7/33/4, 7/34/2, 7/35/1, 7/36/5, 7/37/3, 7/38/9, 7/39/8, 7/40/9, 7/41/8, 7/42/9, 7/43/8, 7/44/9, 7/45/8, 7/46/9, 8/0/8, 8/1/9, 8/2/8, 8/3/9, 8/4/8, 8/5/9, 8/6/8, 8/7/9, 8/8/8, 8/9/9, 8/10/3, 8/11/5, 8/12/1, 8/13/2, 8/14/4, 8/15/8, 8/16/9, 8/17/3, 8/18/5, 8/19/1, 8/20/2, 8/21/4, 8/22/8, 8/23/9, 8/24/8, 8/25/4, 8/26/2, 8/27/1, 8/28/5, 8/29/3, 8/30/9, 8/31/8, 8/32/4, 8/33/2, 8/34/1, 8/35/5, 8/36/3, 8/37/9, 8/38/8, 8/39/9, 8/40/8, 8/41/9, 8/42/8, 8/43/9, 8/44/8, 8/45/9, 8/46/8, 9/0/9, 9/1/8, 9/2/9, 9/3/8, 9/4/9, 9/5/8, 9/6/9, 9/7/8, 9/8/9, 9/9/8, 9/10/9, 9/11/3, 9/12/5, 9/13/1, 9/14/2, 9/15/4, 9/16/8, 9/17/9, 9/18/3, 9/19/5, 9/20/1, 9/21/2, 9/22/4, 9/23/8, 9/24/4, 9/25/2, 9/26/1, 9/27/5, 9/28/3, 9/29/9, 9/30/8, 9/31/4, 9/32/2, 9/33/1, 9/34/5, 9/35/3, 9/36/9, 9/37/8, 9/38/9, 9/39/8, 9/40/9, 9/41/8, 9/42/9, 9/43/8, 9/44/9, 9/45/8, 9/46/9, 10/0/8, 10/1/9, 10/2/8, 10/3/9, 10/4/8, 10/5/9, 10/6/8, 10/7/9, 10/8/8, 10/9/9, 10/10/8, 10/11/9, 10/12/3, 10/13/5, 10/14/1, 10/15/2, 10/16/4, 10/17/8, 10/18/9, 10/19/3, 10/20/5, 10/21/1, 10/22/2, 10/23/4, 10/24/2, 10/25/1, 10/26/5, 10/27/3, 10/28/9, 10/29/8, 10/30/4, 10/31/2, 10/32/1, 10/33/5, 10/34/3, 10/35/9, 10/36/8, 10/37/9, 10/38/8, 10/39/9, 10/40/8, 10/41/9, 10/42/8, 10/43/9, 10/44/8, 10/45/9, 10/46/8, 11/0/9, 11/1/8, 11/2/9, 11/3/8, 11/4/9, 11/5/8, 11/6/9, 11/7/8, 11/8/9, 11/9/8, 11/10/9, 11/11/8, 11/12/9, 11/13/3, 11/14/5, 11/15/1, 11/16/2, 11/17/4, 11/18/8, 11/19/9, 11/20/3, 11/21/5, 11/22/1, 11/23/2, 11/24/1, 11/25/5, 11/26/3, 11/27/9, 11/28/8, 11/29/4, 11/30/2, 11/31/1, 11/32/5, 11/33/3, 11/34/9, 11/35/8, 11/36/9, 11/37/8, 11/38/9, 11/39/8, 11/40/9, 11/41/8, 11/42/9, 11/43/8, 11/44/9, 11/45/8, 11/46/9, 12/0/8, 12/1/9, 12/2/8, 12/3/9, 12/4/8, 12/5/9, 12/6/8, 12/7/9, 12/8/8, 12/9/9, 12/10/8, 12/11/9, 12/12/8, 12/13/9, 12/14/3, 12/15/5, 12/16/1, 12/17/2, 12/18/4, 12/19/8, 12/20/9, 12/21/3, 12/22/5, 12/23/1, 12/24/5, 12/25/3, 12/26/9, 12/27/8, 12/28/4, 12/29/2, 12/30/1, 12/31/5, 12/32/3, 12/33/9, 12/34/8, 12/35/9, 12/36/8, 12/37/9, 12/38/8, 12/39/9, 12/40/8, 12/41/9, 12/42/8, 12/43/9, 12/44/8, 12/45/9, 12/46/8, 13/0/9, 13/1/8, 13/2/9, 13/3/8, 13/4/9, 13/5/8, 13/6/9, 13/7/8, 13/8/9, 13/9/8, 13/10/9, 13/11/8, 13/12/9, 13/13/8, 13/14/9, 13/15/3, 13/16/5, 13/17/1, 13/18/2, 13/19/4, 13/20/8, 13/21/9, 13/22/3, 13/23/5, 13/24/3, 13/25/9, 13/26/8, 13/27/4, 13/28/2, 13/29/1, 13/30/5, 13/31/3, 13/32/9, 13/33/8, 13/34/9, 13/35/8, 13/36/9, 13/37/8, 13/38/9, 13/39/8, 13/40/9, 13/41/8, 13/42/9, 13/43/8, 13/44/9, 13/45/8, 13/46/9, 14/0/8, 14/1/9, 14/2/8, 14/3/9, 14/4/8, 14/5/9, 14/6/8, 14/7/9, 14/8/8, 14/9/9, 14/10/8, 14/11/9, 14/12/8, 14/13/9, 14/14/8, 14/15/9, 14/16/3, 14/17/5, 14/18/1, 14/19/2, 14/20/4, 14/21/8, 14/22/9, 14/23/3, 14/24/9, 14/25/8, 14/26/4, 14/27/2, 14/28/1, 14/29/5, 14/30/3, 14/31/9, 14/32/8, 14/33/9, 14/34/8, 14/35/9, 14/36/8, 14/37/9, 14/38/8, 14/39/9, 14/40/8, 14/41/9, 14/42/8, 14/43/9, 14/44/8, 14/45/9, 14/46/8, 15/0/9, 15/1/8, 15/2/9, 15/3/8, 15/4/9, 15/5/8, 15/6/9, 15/7/8, 15/8/9, 15/9/8, 15/10/9, 15/11/8, 15/12/9, 15/13/8, 15/14/9, 15/15/8, 15/16/9, 15/17/3, 15/18/5, 15/19/1, 15/20/2, 15/21/4, 15/22/8, 15/23/9, 15/24/8, 15/25/4, 15/26/2, 15/27/1, 15/28/5, 15/29/3, 15/30/9, 15/31/8, 15/32/9, 15/33/8, 15/34/9, 15/35/8, 15/36/9, 15/37/8, 15/38/9, 15/39/8, 15/40/9, 15/41/8, 15/42/9, 15/43/8, 15/44/9, 15/45/8, 15/46/9, 16/0/8, 16/1/9, 16/2/8, 16/3/9, 16/4/8, 16/5/9, 16/6/8, 16/7/9, 16/8/8, 16/9/9, 16/10/8, 16/11/9, 16/12/8, 16/13/9, 16/14/8, 16/15/9, 16/16/8, 16/17/9, 16/18/3, 16/19/5, 16/20/1, 16/21/2, 16/22/4, 16/23/8, 16/24/4, 16/25/2, 16/26/1, 16/27/5, 16/28/3, 16/29/9, 16/30/8, 16/31/9, 16/32/8, 16/33/9, 16/34/8, 16/35/9, 16/36/8, 16/37/9, 16/38/8, 16/39/9, 16/40/8, 16/41/9, 16/42/8, 16/43/9, 16/44/8, 16/45/9, 16/46/8, 17/0/9, 17/1/8, 17/2/9, 17/3/8, 17/4/9, 17/5/8, 17/6/9, 17/7/8, 17/8/9, 17/9/8, 17/10/9, 17/11/8, 17/12/9, 17/13/8, 17/14/9, 17/15/8, 17/16/9, 17/17/8, 17/18/9, 17/19/3, 17/20/5, 17/21/1, 17/22/2, 17/23/4, 17/24/2, 17/25/1, 17/26/5, 17/27/3, 17/28/9, 17/29/8, 17/30/9, 17/31/8, 17/32/9, 17/33/8, 17/34/9, 17/35/8, 17/36/9, 17/37/8, 17/38/9, 17/39/8, 17/40/9, 17/41/8, 17/42/9, 17/43/8, 17/44/9, 17/45/8, 17/46/9, 18/0/8, 18/1/9, 18/2/8, 18/3/9, 18/4/8, 18/5/9, 18/6/8, 18/7/9, 18/8/8, 18/9/9, 18/10/8, 18/11/9, 18/12/8, 18/13/9, 18/14/8, 18/15/9, 18/16/8, 18/17/9, 18/18/8, 18/19/9, 18/20/3, 18/21/5, 18/22/1, 18/23/2, 18/24/1, 18/25/5, 18/26/3, 18/27/9, 18/28/8, 18/29/9, 18/30/8, 18/31/9, 18/32/8, 18/33/9, 18/34/8, 18/35/9, 18/36/8, 18/37/9, 18/38/8, 18/39/9, 18/40/8, 18/41/9, 18/42/8, 18/43/9, 18/44/8, 18/45/9, 18/46/8, 19/0/9, 19/1/8, 19/2/9, 19/3/8, 19/4/9, 19/5/8, 19/6/9, 19/7/8, 19/8/9, 19/9/8, 19/10/9, 19/11/8, 19/12/9, 19/13/8, 19/14/9, 19/15/8, 19/16/9, 19/17/8, 19/18/9, 19/19/8, 19/20/9, 19/21/3, 19/22/5, 19/23/1, 19/24/5, 19/25/3, 19/26/9, 19/27/8, 19/28/9, 19/29/8, 19/30/9, 19/31/8, 19/32/9, 19/33/8, 19/34/9, 19/35/8, 19/36/9, 19/37/8, 19/38/9, 19/39/8, 19/40/9, 19/41/8, 19/42/9, 19/43/8, 19/44/9, 19/45/8, 19/46/9, 20/0/8, 20/1/9, 20/2/8, 20/3/9, 20/4/8, 20/5/9, 20/6/8, 20/7/9, 20/8/8, 20/9/9, 20/10/8, 20/11/9, 20/12/8, 20/13/9, 20/14/8, 20/15/9, 20/16/8, 20/17/9, 20/18/8, 20/19/9, 20/20/8, 20/21/9, 20/22/3, 20/23/5, 20/24/3, 20/25/9, 20/26/8, 20/27/9, 20/28/8, 20/29/9, 20/30/8, 20/31/9, 20/32/8, 20/33/9, 20/34/8, 20/35/9, 20/36/8, 20/37/9, 20/38/8, 20/39/9, 20/40/8, 20/41/9, 20/42/8, 20/43/9, 20/44/8, 20/45/9, 20/46/8, 21/0/9, 21/1/8, 21/2/9, 21/3/8, 21/4/9, 21/5/8, 21/6/9, 21/7/8, 21/8/9, 21/9/8, 21/10/9, 21/11/8, 21/12/9, 21/13/8, 21/14/9, 21/15/8, 21/16/9, 21/17/8, 21/18/9, 21/19/8, 21/20/9, 21/21/8, 21/22/9, 21/23/3, 21/24/9, 21/25/8, 21/26/9, 21/27/8, 21/28/9, 21/29/8, 21/30/9, 21/31/8, 21/32/9, 21/33/8, 21/34/9, 21/35/8, 21/36/9, 21/37/8, 21/38/9, 21/39/8, 21/40/9, 21/41/8, 21/42/9, 21/43/8, 21/44/9, 21/45/8, 21/46/9, 22/0/8, 22/1/9, 22/2/8, 22/3/9, 22/4/8, 22/5/9, 22/6/8, 22/7/9, 22/8/8, 22/9/9, 22/10/8, 22/11/9, 22/12/8, 22/13/9, 22/14/8, 22/15/9, 22/16/8, 22/17/9, 22/18/8, 22/19/9, 22/20/8, 22/21/9, 22/22/8, 22/23/9, 22/24/8, 22/25/9, 22/26/8, 22/27/9, 22/28/8, 22/29/9, 22/30/8, 22/31/9, 22/32/8, 22/33/9, 22/34/8, 22/35/9, 22/36/8, 22/37/9, 22/38/8, 22/39/9, 22/40/8, 22/41/9, 22/42/8, 22/43/9, 22/44/8, 22/45/9, 22/46/8, 23/0/9, 23/1/8, 23/2/9, 23/3/8, 23/4/9, 23/5/8, 23/6/9, 23/7/8, 23/8/9, 23/9/8, 23/10/9, 23/11/8, 23/12/9, 23/13/8, 23/14/9, 23/15/8, 23/16/9, 23/17/8, 23/18/9, 23/19/8, 23/20/9, 23/21/8, 23/22/9, 23/23/8, 23/24/9, 23/25/8, 23/26/9, 23/27/8, 23/28/9, 23/29/8, 23/30/9, 23/31/8, 23/32/9, 23/33/8, 23/34/9, 23/35/8, 23/36/9, 23/37/8, 23/38/9, 23/39/8, 23/40/9, 23/41/8, 23/42/9, 23/43/8, 23/44/9, 23/45/8, 23/46/9} {
    \filldraw[fill=color\clrNum, draw=black] (\j,\i) rectangle ++(1,1);
}

\end{tikzpicture}
\caption{\label{fig:spur-elimination}A process for eliminating spurs in a cycle one at a time.}
\end{figure} 

Next we turn to the $4$-cycle elimination algorithm, or Algorithm {\sc (f)}. Suppose $\eta$ is conjugate to a $4$-cycle in $\pi_1(\Gamma)$, we may assume $\eta$ is of the form
\[
\eta=(w_0,w_1,\ldots,w_j,u_0,u_1,u_2,u_3,u_0,w_j,\ldots,w_1, w_0)
\]
where $\theta=(u_0,u_1,u_2,u_3, u_0)$ is the $4$-cycle that $\eta$ is conjugate to.  Now, if $\eta$ appears horizontally as colors assigned to some vertices in a grid graph, we can eliminate $\theta$ by the following coloring scheme:
\begin{center}
\begin{tikzpicture}
\matrix [matrix of math nodes, row sep={2.8em,between origins}, column sep={2.8em,between origins}] {
          &      &\ldots&w_{j-2}&w_{j-1}&w_j  &u_0&w_j  &w_{j-1}&w_{j-2}&\ldots & &\\
          &w_0&\ldots&w_{j-1}&w_j    &u_0&u_1&u_0&w_j    &w_{j-1}&\ldots&w_0& &\\
    w_0&w_1&\ldots&w_j    &u_0  &u_1&u_2&u_3&u_0  &w_j    &\ldots&w_1&w_0 \\
};
\end{tikzpicture}
\end{center}
We then proceed to eliminate the entire cycle $\eta$ by
spur eliminations after eliminating the $4$-cycle.

We are now ready to describe the coloring of the hub. First we note a
simple fact about the homotopy of cycles. Suppose $\eta=\alpha\beta$
is a cycle and $\eta'=\alpha\gamma\theta\gamma^{-1}\beta$ is a cycle
obtained from $\eta$ by adding a conjugate of a $4$-cycle
$\theta$. Then
$$ \eta'=(\alpha\gamma)\theta(\alpha\gamma)^{-1}\alpha\beta $$ is in
fact the product of a conjugate of a $4$-cycle with $\eta$. Now
consider the cycle $\gamma^m$. Because $\gamma^m$ is null-homotopic in
$\pi_1^*(\Gamma)$, there are 
conjugates of $4$-cycles $\lambda_1,\dots, \lambda_k$, all begin and end at $v_0$, such that
$\lambda_k\dots\lambda_1\gamma^m$ is equivalent to the identity cycle in $\pi_1(\Gamma)$.

Inside the hub, we will color a row with nodes in
$\lambda_k\cdots\lambda_2\lambda_1\gamma^m$, along with fillers.
Note that since $m$ is even, the total parity change
along $\lambda_k\cdots\lambda_2\lambda_1\gamma^m$ is $0$.
Then above
the row we perform the $4$-cycle elimination algorithm for each of
$\lambda_k,\dots, \lambda_1$, and connect $\gamma^m$ to the
boundary. Below the row we perform the spur elimination algorithm. The
result of the coloring is depicted in
Figure~\ref{fig:hub-construction}.

 \begin{figure}[h]
\begin{tikzpicture}[scale=0.06]

\draw (-100,0) rectangle (100,-110);
\draw (50,2) -- (50,-2);
\draw (85,2) --(85,-2);
\node at (70, 5) {$\gamma^m$};

\draw (-80, -20) -- (85,-20);
\draw (-80, -22) -- (-80, -18);
\draw (-50, -22) -- (-50, -18);
\draw (-10, -22) -- (-10, -18);
\draw (20, -22) -- (20, -18);
\draw (50, -22) -- (50, -18);
\draw (85, -22) -- (85, -18);

\node at (-65, -25) {$\lambda_k$};
\node at (-30, -25) {$\ldots$};
\node at (5, -25) {$\lambda_2$};
\node at (35, -25) {$\lambda_1$};
\node at (70, -25) {$\gamma^m$};

\draw[dotted] (-80,-20) -- (-65, -5) -- (-50, -20);
\draw[dotted] (-10,-20) -- (5, -5) -- (20, -20) -- (35,-5) -- (50, -20);
\draw[dotted] (50,0) -- (50,-20);
\draw[dotted] (85,0) -- (85,-20) --(0,-100) -- (-80, -20);

\node at (-65, -14) {\sc (f)};
\node at (5,-14) {\sc (f)};
\node at (35,-14) {\sc (f)};
\node at (68, -10) {\sc (c)};
\node at (0, -50) {\sc (s)};
\node at (-70, -70) {\sc (t)};
\node at (70, -70) {\sc (t)};

\end{tikzpicture}
\caption[Construction of a hub.]{\label{fig:hub-construction}Coloring within a hub.
We use conjugates of the four-cycles $\lambda_1,\lambda_2,\ldots \lambda_k$
to eliminate $\gamma^m$ when $\lambda_k\dots\lambda_1\gamma^{m}$ is equivalent to the identity in $\pi_1(\Gamma)$. The algorithm performed in each region is shown.}
\end{figure}

After performing all of these algorithms we fill the rest of the
tiles by alternating $v_0, v_1$. It is clear that we have described a
complete coloring of the 12 tiles by nodes in $\Gamma$ in a way that
adjacent vertices are colored by adjacent nodes in $\Gamma$. This
finishes the definition of a homomorphism $\Gamma_{1,p,q}\to \Gamma$.
\end{proof}

The sufficient condition of Theorem~\ref{thm:odd-length-cycle}
for the existence of a continuous homomorphism from $\FofZZ$ to a graph $\Gamma$
is a search problem suitable for a computer.
For example, Theorem~\ref{thm:odd-length-cycle} applies to two
well-studied and well-known graphs, the Chv\'atal graph and the
Gr\"otsch graph.

\begin{exn}
Let $\Gamma_C$ be the Chv\'atal graph and $\Gamma_G$ the Gr\"otsch graph.
There exist continuous graph homomorphisms
$\varphi_C\colon \FofZZ\to \Gamma_C$ and $\varphi_G\colon \FofZZ\to\Gamma_G$.
\end{exn}

\begin{proof}
We can establish that each graph $\Gamma\in \{\Gamma_{C}, \Gamma_G\}$ has a cycle $\gamma$ with
$\pi_1^*(\gamma)$ of order $2$ by coloring the nodes of a grid graph
with nodes of $\Gamma$, placing
$\gamma$ at the top of the grid graph and $\gamma^{-1}$ at the bottom, and
leaving the sides alternating between two colors.  The alternation
between two colors will contribute nothing to homotopy class of the
boundary of the box, whence $\gamma^2$ will be null-homotopic in $\pi_1^*(\Gamma)$.
The Chv\'atal and Gr\"otsch graphs are given along with their solutions in
Figures~\ref{fig:chvatal-graph} and \ref{fig:grotsch-graph}
respectively.
\begin{figure}[htbp]
\begin{tikzpicture}
    \begin{scope}[every node/.style={nodeStyle, minimum size=3.5ex}]
        \node (n0) at (-2, 2) {0};
        \node (n1) at (2, 2) {1};
        \node (n2) at (2, -2) {2};
        \node (n3) at (-2, -2) {3};
        \node (n4) at (-0.500000000000000, 1) {4};
        \node (n5) at (0.500000000000000, 1) {5};
        \node (n6) at (1, 0.500000000000000) {6};
        \node (n7) at (1, -0.500000000000000) {7};
        \node (n8) at (0.500000000000000, -1) {8};
        \node (n9) at (-0.500000000000000, -1) {9};
        \node (n10) at (-1, -0.500000000000000) {10};
        \node (n11) at (-1, 0.500000000000000) {11};
    \end{scope}

    \draw (n0) -- (n1);
    \draw (n0) -- (n3);
    \draw (n0) -- (n4);
    \draw (n0) -- (n11);
    \draw (n1) -- (n2);
    \draw (n1) -- (n5);
    \draw (n1) -- (n6);
    \draw (n2) -- (n3);
    \draw (n2) -- (n7);
    \draw (n2) -- (n8);
    \draw (n3) -- (n9);
    \draw (n3) -- (n10);
    \draw (n4) -- (n5);
    \draw (n4) -- (n7);
    \draw (n4) -- (n8);
    \draw (n5) -- (n9);
    \draw (n5) -- (n10);
    \draw (n6) -- (n7);
    \draw (n6) -- (n9);
    \draw (n6) -- (n10);
    \draw (n7) -- (n11);
    \draw (n8) -- (n9);
    \draw (n8) -- (n11);
    \draw (n10) -- (n11);

\matrix (m) [matrix of math nodes, row sep={4ex,between origins},
column sep={4ex,between origins}] at (5,0) {
         0 & 11 &  7 &  2 &  1 &  0 \\
         3 & 10 &  6 &  1 &  0 &  3 \\
         0 &  3 &  9 &  5 &  4 &  0 \\
         3 &  2 &  8 &  4 &  0 &  3 \\
         0 &  1 &  2 &  7 & 11 &  0 \\
    };
\end{tikzpicture}
\caption[The Chv\'atal graph.]{The Chv\'atal graph, given at left.
The odd cycle $\gamma=(0,11,7,2,1,0)$ has order 2 in this graph,
because there is a (necessarily null-homotopic) graph homomorphism from a
grid graph to the Chv\'atal graph with boundary homotopic to $\gamma^2$, given at right.}
\label{fig:chvatal-graph}   
\end{figure}  
\begin{figure}[htbp]
\begin{tikzpicture}
    \begin{scope}[every node/.style={nodeStyle, minimum size=3.5ex}]
        \node (n0) at (-1.90211303259031, 0.618033988749894) {0};
        \node (n1) at (0.000000000000000, 2.00000000000000) {1};
        \node (n2) at (1.90211303259031, 0.618033988749895) {2};
        \node (n3) at (1.17557050458495, -1.61803398874989) {3};
        \node (n4) at (-1.17557050458495, -1.61803398874990) {4};
        \node (n5) at (-0.951056516295154, 0.309016994374947) {5};
        \node (n6) at (0.000000000000000, 1.00000000000000) {6};
        \node (n7) at (0.951056516295153, 0.309016994374947) {7};
        \node (n8) at (0.587785252292473, -0.809016994374947) {8};
        \node (n9) at (-0.587785252292473, -0.809016994374948) {9};
        \node (n10) at (0.000000000000000, 0.000000000000000) {10};
    \end{scope}

    \draw (n0) -- (n1);
    \draw (n0) -- (n4);
    \draw (n0) -- (n6);
    \draw (n0) -- (n9);
    \draw (n1) -- (n2);
    \draw (n1) -- (n5);
    \draw (n1) -- (n7);
    \draw (n2) -- (n3);
    \draw (n2) -- (n6);
    \draw (n2) -- (n8);
    \draw (n3) -- (n4);
    \draw (n3) -- (n7);
    \draw (n3) -- (n9);
    \draw (n4) -- (n5);
    \draw (n4) -- (n8);
    \draw (n5) -- (n10);
    \draw (n6) -- (n10);
    \draw (n7) -- (n10);
    \draw (n8) -- (n10);
    \draw (n9) -- (n10);

\matrix (m) [matrix of math nodes, row sep={4ex,between origins},
column sep={4ex,between origins}] at (5,0) {
         0 &  1 &  2 &  3 &  9 &  0 \\
         6 &  2 &  1 &  7 & 10 &  6 \\
         0 &  1 &  5 & 10 &  6 &  0 \\
         6 &  0 &  4 &  8 &  2 &  6 \\
         0 &  9 &  3 &  2 &  1 &  0 \\
    };
\end{tikzpicture}
\caption[The Gr\"otzsch graph.]{The Gr\"otzsch graph.
The odd cycle $\gamma=(0,1,2,3,9,0)$ has order 2 for similar
reasons as with the Chv\'atal graph.}
\label{fig:grotsch-graph}
\end{figure}   
\end{proof}

\todo{Mention that none of our graphs contain loop edges}
\todo{Degenerate cycles are actually zero-cycles.
One-cycles would be loops, which never exist.  How do I handle this?}

\todo{all lists start indexing at zero; last index is always presented as $k-1$
so that $k$ is the length.  Pure computer science style.}

It follows immediately from Theorem~\ref{thm:pcn2} that a
simple condition sufficient to guarantee existence of continuous homomorphisms from $\FofZZ$ to $\Gamma$
is that $\Gamma$ contains a copy of $K_4$, the complete graph with $4$-vertices.
It is also easy to check directly with the same method as above that any $3$-cycle in $K_4$ has order $2$.
Thus Theorem~\ref{thm:odd-length-cycle} applies also in this case.
Note, however,  that none of the above graphs contains a copy of $K_4$.
Therefore the positive condition in Theorem~\ref{thm:odd-length-cycle} is strictly weaker.

\begin{rem}
According to Theorems~\ref{thm:neghom} and \ref{thm:odd-length-cycle},
the negative and positive conditions for a graph $\Gamma$ must necessarily
be mutually exclusive. It may be reassuring to the reader that
this mutual exclusiveness can be established by a simple direct argument,
similar to what we noted in the proof of Theorem~\ref{thm:odd-length-cycle}.

To see this,
suppose $\Gamma$ satisfies the positive condition of Theorem~\ref{thm:odd-length-cycle}, and let $\gamma$ be an odd-length
cycle in $\Gamma$ such that $\gamma$ has finite order $m$ in $\pi_1^*(\Gamma)$.
Note that $m$ must be even, as an odd-length cycle cannot be the identity
element in $\pi_1^*(\Gamma)$. By raising $\gamma$ to a suitable odd power,
we may assume that $m=2^{m_0}$. Let $\gamma=(v_0,v_1,\dots, v_\ell)$.
We violate the negative condition of Theorem~\ref{thm:neghom} in a strong way by showing that for all sufficiently large
$p,q>m\ell$ with $\gcd(p,q)=1$ there is a $p$-cycle $\alpha$ and a $q$-cycle $\beta$
such that $\pi_1^*(\alpha^q)=\pi_1^*(\beta^p)$ in $\pi_1^*(\Gamma)$.
We can assume without loss of generality that $p$ is odd.

Suppose first that $q$ is also odd.
Let $0<a,b<m$ be such that $a\equiv p\mod m$ and $b\equiv q\mod m$. Then $a, b$ are odd and $aq-bp \equiv 0 \mod m$.
Let $\alpha$
be the cycle obtained by following $\gamma^a$ with alternations of
$v_0,v_1$ so that $\alpha$ has length $p$. This is possible
since $p>a\ell$ and both are odd. Let $\beta$ be obtained in a
similar manner, extending $\gamma^b$ by alternations of $v_0,v_1$
to have length $q$. Then $\pi_1^*(\alpha^q)=\pi_1^*(\gamma^{aq})=
\pi_1^*(\gamma^{bp})=\pi_1^*(\beta^p)$ since $m| (aq-bp)$.

Next assume $p$ is odd and $q$ is even. Let $2^k$ be the highest power of $2$
which divides $q$. If $k\geq m_0$ then we can take $\alpha$ to be the extension
of $\gamma$  by alternations of $v_0,v_1$ so that $\alpha$ has length $p$. Then $\pi_1^*(\alpha^q)$ is the identity,
and the result is immediate if we let $\beta$ be a $q$-cycle consisting only of alternations of $v_0, v_1$.
So assume $1 \leq k <m_0$. Let $a,b$ be such that $0<a<2^{m_0-k}$ is odd,
$0<b<2^{m_0}$ is an odd multiple of $2^k$, and
$a\frac{q}{2^k}-\frac{b}{2^k}p \equiv 0 \mod 2^{m_0-k}$. Then let $\alpha$ be the extension of
$\gamma^a$ by alternations of $v_0,v_1$ to have length $p$, and $\beta$ the extension
of $\gamma^b$ to have length $q$. Then $\pi_1^*(\alpha^q)=\pi_1^*(\gamma^{aq})=\pi_1^*(\gamma^{bp})
=\pi_1^*(\beta^p)$ since $2^{m_0}| (aq-bp)$.
\end{rem}

\chapter{Applications to Decidability Problems} \label{sec:otherapp}

\section{Three general problems} \label{sec:problems}

We formulate three general problems for $\fzn$. These problems can be considered at both the continuous
and Borel levels, and are interesting in both contexts. We will, however, confine our attention to
the continuous versions in this paper.

Recall from Definition~\ref{def:sft} that a $\Z^n$-subshift of finite type is a
closed invariant set $Y\subseteq \bsft^{\Z^n}$ consisting of all $x \in \bsft^{\Z^n}$
avoiding a prescribed set of patterns $\{p_1,\dots,p_k\}$. Thus, a subshift is determined
by the tuple $\vec p=(\bsft;p_1,\dots,p_k)$ which can be viewed as a tuple of integers via some coding. More explicitly, each pattern $p$ as an $n$-dimensional matrix of integers can be coded by a single integer, which we still denote as $p$, via a computable procedure. Then the entire tuple $\vec{p}=(\bsft; p_1,\dots, p_k)$ can be encoded as a single integer
$m=\langle \bsft, p_1,\dots,p_k\rangle$, where $\langle\cdot,\cdots,\cdot\rangle: \omega^{<\omega}\to\omega$ is a fixed computable bijection. We let
$Y_{\vec p}$ denote the subshift determined by $\vec p$. We also use $Y_m$ to denote this subshift,
where $m$ is the integer $m=\langle \bsft, p_1,\dots,p_k\rangle$ coding the subshift.
The general subshift problem is the following.

\begin{prob}[Subshift problem] \label{prob:subs}
For which $(\bsft;p_1,\dots,p_k)$ does there exist a continuous (or Borel) equivariant map
$\pi \colon \fzn \to Y_{\vec p}$? What is the complexity of the set of integers
$m=\langle \bsft,p_1,\dots,p_k\rangle$ coding subshifts $Y_m$ for which there is such a continuous
(or Borel) equivariant map? In particular, is this set computable?
\end{prob}

A finite graph $\Gamma$ can be coded by an integer, say by viewing the vertex set as
$\{0,\dots,k\}$ for some $k$. This then suggests the general graph homomorphism problem.

\begin{prob}[Graph homomorphism problem] \label{prob:gh}
For which finite graphs $\Gamma$ does there exist a continuous (or Borel)
graph homomorphism from $\fzn$ to $\Gamma$? What is the complexity of the set of integers
$m$ coding finite graphs $\Gamma_m$ for which there is such a continuous (or Borel)
graph homomorphism? In particular, is this set computable?
\end{prob}

Recall that in Definition~\ref{def:tiling} we gave the definition of a clopen (or Borel)
tiling of $\fzn$ by a set of tiles $\{T_1,\dots,T_k\}$, where each $T_i$ is a finite subset
of $\Z^n$. We assume here that $k$ is finite, and so the sequence of tiles can again
be coded by a integer $m$. Partly restating Question~\ref{ques:tiling} we have the following
tiling problem.

\begin{prob}[Tiling problem] For which sets of tiles $\{T_1,\dots,T_k\}$
does there exist a clopen (or Borel) tiling of $\fzn$? What is the complexity of the set
of integers $m$ which code finite sets of tiles for which there is a clopen (or Borel) tiling
of $\fzn$? In particular, is this set computable?
\end{prob}

In each of these three problems, when discussing the complexity of the set of integers $m$
we of course must use a ``lightface'' notion of complexity as we are dealing with a set of integers.
In each of the three cases, it is not immediately clear from the definitions
that the complexity of the relevant set of integers is even $\Delta^1_1$ (hyperarithmetical).
For example, the set of integers $m=\langle \bsft,p_1,\dots,p_k\rangle$ coding subshifts $Y_m$
for which there is a continuous equivariant map from $\fzn$ to $Y_m$ seems to involve
quantifying over reals coding continuous functions from $\fzn$ to $\bsft^{\Z^n}$. Superficially at least,
this seems to compute the set as a $\Sigma^1_2$ set of integers.

Our first observation is that Theorem~\ref{thm:tilethm} shows that in all
of the three problems, the relevant set of integers is actually $\Sigma^0_1$ (or computably enumerable, or semi-computable).

\begin{thmn} \label{thm:sr}
For each of the subshift, graph homomorphism, and
tiling problems for $F(2^{\Z^n})$, the set of integers
$m$ for which the continuous version of the problem has a positive answer is $\Sigma^0_1$.
\end{thmn}

\begin{proof}
In each case, Theorem~\ref{thm:tilethm} shows that the continuous object exists
(e.g., for the subshift problem, a continuous equivariant map into the subshift)
if and only if a corresponding object exists for the graph $\gnpq$ for some $n<p,q$ where $\gcd(p,q)=1$
(e.g., for the subshift problem, there is a map $g \colon \gnpq \to \bsft$ which respects the
subshift, see Definition~\ref{def:resp}). For given $n,p,q$, it is a finite
problem to determine if such a $g$ from $\gnpq$ exists. Thus, the set of $m$
coding subshifts for which there is a continuous equivariant map from $\fzs$ to the subshift
is a $\Sigma^0_1$ set.
\end{proof}

In the remaining sections of this chapter, we show that the one-dimensional subshift problem corresponds to a computable set, and that for the general subshift problem and the graph homomorphism problem, the corresponding sets are $\Sigma^0_1$-complete, and in particular not computable. The exact degree of computability of the set corresponding to the tiling problem is unknown.

\section{The subshift problem for $F(2^\Z)$} \label{sec:ssonedim}

We show first that the general subshift problem in
dimension one, that is for $F(2^{\Z})$, is decidable. It follows immediately that the one-dimensional graph homomorphism problem and tiling problem are also decidable.
In contrast, we will show in \S\ref{sec:sstwodim} that the subshift problem
in dimensions two or higher is not decidable. In fact, we will show in
\S\ref{sec:gh} that even the specific instance of the graph homomorphism
problem is not decidable for $F(2^{\Z^n})$ for $n \geq 2$. Recall a
subshift of finite type $Y\subseteq \bsft^\Z$ is determined by
a sequence $(\bsft;\ell;p_1,\dots,p_k)$ where $\ell \geq 1$ and
$p_i\in \bsft^\ell$ are the forbidden patterns of length $\ell$ (width $\ell-1$).

\begin{thmn} \label{thm:ssonedim}
The set of integers $m$ coding subshifts $Y_m\subseteq \bsft^{\Z}$ for which there is a continuous,
equivariant map from $F(2^\Z)$ to $Y_m$ is a $\Delta^0_1$ (i.e., computable) set.
\end{thmn}

The rest of this section is devoted to a proof of Theorem~\ref{thm:ssonedim}.
We recall the one-dimensional version of the tile analysis, which was
given in Theorem~\ref{thm:twotilesthm}. Recall the graph $\Gamma^{(1)}_{n,p,q}$
from \S\ref{sec:onedimtile} which was used in that theorem. We use also
the other notation of \S\ref{sec:onedimtile}, for example the tile graphs $T_1(n, p, q), T_2(n, p, q)$ and the block label $R^{(1)}_\times$ associated to certain parts of
$\Gamma^{(1)}_{n,p, q}$.

Given a subshift $Y\subseteq \bsft^\Z$ determined by $(\bsft; \ell; p_1,\dots,p_k)$,
we define a directed graph
$\Lambda=\Lambda(Y)$ as follows.
Let the vertex set $V(\Lambda)$ of $\Lambda$ be the $v\in \bsft^\ell$
which are not in $\{ p_1,\dots,p_k\}$. If $u,v\in V(\Lambda)$, we put a directed edge
from $u$ to $v$ if $(v(0),\dots, v(\ell-2))=(u(1),\dots,u(\ell-1))$.
Let $(V(\Lambda),E(\Lambda))$ be the resulting directed graph.
The vertices of $\Lambda$ correspond to $\ell$-size window states (that is
possible values of $y\res [k,k+\ell)$ for $y\in Y$) and the edge relation corresponds to a shift
of the window one unit to the right.

We say $f \colon F(2^\Z)\to V(\Lambda)$ is a {\em directed graph homomorphism}
if for all $x \in F(2^\Z)$ we have that $(f(x), f(1\cdot x))\in E(\Lambda)$, where
$1$ is the generator for the action of $\Z$ on $F(2^\Z)$.
We likewise define the notion of $g \colon \Gamma^1_{n,p,q} \to V(\Lambda)$
being a directed graph homomorphism. The following fact is now immediate from the
definitions.

\begin{lemn}\label{lem:dgh}
There is a continuous, equivariant map $f \colon F(2^\Z) \to Y$, where $Y$ is the subshift
determined by $(\bsft;\ell;p_1,\dots,p_k)$, if and only if there is a continuous directed graph homomorphism
$f' \colon F(2^\Z) \to \Lambda(Y)$.
\end{lemn}

A directed graph homomorphism $f' \colon F(2^\Z) \to \Lambda(Y)$ is, in turn,
a type of equivariant subshift map where the subshift has width $1$. So, Lemma~\ref{lem:dgh} and Theorem~\ref{thm:twotilesthm} immediately
give the following.

\begin{lemn}
There is a continuous, equivariant map $f \colon F(2^\Z)\to Y$,
where $Y\subseteq \bsft^\Z$ is the subshift given by $(\bsft;\ell;p_1,\dots,p_k)$,
if and only if there are $p,q>1$ with $\gcd(p,q)=1$ and a
directed graph homomorphism
$g \colon \Gamma^{(1)}_{1,p,q} \to \Lambda(Y)$.
\end{lemn}

It remains to show that it is computable in the integer $m$ coding
the subshift determined by $(\bsft;\ell;p_1,\dots,p_k)$ whether there  are $p,q>1$ with
$\gcd(p,q)=1$ and a directed graph
homomorphism $g:\Gamma^{(1)}_{1,p,q} \to \Lambda(Y)$.

Consider the directed graph $\Lambda=\Lambda(Y)$. Let $V=V(\Lambda)$
be the vertex set of $\Lambda$, and let $N=|V|$. For $u, v \in V$ say $u \sim v$ if there
is a directed path in $\Lambda$ from $u$ to $v$ and also a directed path from $v$ to $u$.
This is clearly an equivalence relation, and we call the equivalence classes
$C_1,\dots,C_t$ the {\em directed components} of $\Lambda$.
Note that all the vertices in a directed cycle lie in the same directed
component of $\Lambda$.

We say a directed cycle $\gamma=(u_0,u_1,\dots, u_k,u_0)$
in $\Lambda$ is {\em primitive} if no subsegment $(u_i, u_{i+1},\dots, u_j)$
is a cycle. Clearly a primitive cycle in $\Lambda$ has length at most
$N\leq \bsft^\ell$.

\begin{lemn} \label{prim}
Let $C$ be a directed component of $\Lambda$. Then the following are equivalent:
\begin{enumerate}[label={\rm (\arabic*)}, ref=\arabic*]
\item \label{cla}
The $\gcd$ of the primitive cycle lengths in $C$ is $1$.
\item \label{clb}
The $\gcd$ of the (directed) cycle lengths in $C$ is $1$.
\item \label{clc}
The $\gcd$ of the lengths of the cycles beginning and ending at $v_0$
(a fixed vertex in $C$) is $1$.
\end{enumerate}
\end{lemn}

\begin{proof}
Clearly (\ref{cla}) implies (\ref{clb}). To see (\ref{clb})  implies (\ref{cla}),
suppose $d>1$ divides all the lengths of primitive cycles in $C$.
We show $d$ divides the length of all cycles in $C$, a contradiction. This is done by an induction on the length of the cycle.
Let $\gamma=(u_0,\dots,u_k,u_0)$
be a cycle in $C$ of length $|\gamma|=k+1$. Either $\gamma$ is primitive, and we are done, or there
is a subsegment $\delta=(u_i,\dots,u_j)$, $i<j$, of $\gamma$ which is a primitive cycle
(so $u_j=u_i$). Then $\gamma'=(u_0,\dots,u_i, u_{j+1},\dots,u_k, u_0)$
is a cycle of length $|\gamma|-(j-i)$. Since $\delta$ is primitive,
by assumption $d$ divides  $|\delta|=j-i$.
By induction $d$ divides $|\gamma'|$, and so $d$ divides $|\gamma|$.

Clearly (\ref{clc}) implies (\ref{clb}) and so it remains to show that
(\ref{cla}) implies (\ref{clc}). Assume (\ref{cla}) and assume towards a contradiction
that $d>1$ divides all the lengths of the cycles which begin and end at $v_0$.
Let $a>1$ be such that $d \nmid a$ and $a$ is the length of a primitive cycle $\eta$
in $C$, say starting and ending at $v_1 \in C$. Let $\alpha$ be a directed path
in $C$ from $v_0$ to $v_1$, and $\beta$ a directed path in $C$ from $v_1$
to $v_0$. Then the cycle $\rho=\alpha \eta \beta$ is a cycle starting and ending at $v_0$
with length $|\alpha|+ |\eta|+|\beta| =|\eta|+|\delta|$ where $\delta$
is the cycle $\alpha \beta$ (which starts and ends at $v_0$). By assumption,
$d$ divides $|\delta|$ and $|\rho|$, and so $d$ divides $|\eta|=a$, a contradiction.
\end{proof}

We now have the following algorithm to determine if there is a directed graph
homomorphism from $\Gamma^{(1)}_{1,p,q}$ to $\Lambda$. First, compute the
directed components $C_1,\dots,C_t$ of $\Lambda$. Next, for each of these components
$C$, compute the $\gcd$ of the lengths of the primitive cycles in $C$. This is
computable as these cycles have length $\leq N$. The following lemma
finishes the algorithm.

\begin{lemn} The following are equivalent:
\begin{enumerate}[label={\rm (\alph*)}]
\item There are $p, q>1$ with $\gcd(p,q)=1$ and a directed graph homomorphism $g:\Gamma^{(1)}_{1,p,q}\to \Lambda$.
\item There is a directed component $C$ of $\Lambda$ such that the $\gcd$
of the lengths of the primitive cycles in $C$ is $1$.
\end{enumerate}
\end{lemn}

\begin{proof}
Suppose first that $g \colon \Gamma^{(1)}_{1,p,q}\to \Lambda$ is a directed graph
homomorphism, where $\gcd(p,q)=1$. Since $n=1$, there is a unique vertex $r$ in $\Gamma^{(1)}_{1, p, q}$ with label $R^{(1)}_\times$. Let $v=g(r)$.
Let $C$ be the directed component of
$v$ in $\Lambda$. Note that $T_1(1, p, q)$ has length $p+1$ and $T_2(1, p, q)$ has length $q+1$.
As we move from the left to the right in $T_1(n, p, q)$,
applying $g$ gives a directed path in $C$
of length $p$, starting and ending at $v$. Likewise, considering $T_2$
gives a directed path starting and ending at $v$ of length $q$. Since
$\gcd(p,q)=1$, the $\gcd$ of all the cycle lengths in $C$ is $1$. By
Lemma~\ref{prim}, the $\gcd$ of the lengths of the primitive cycles in $C$ is $1$.

Suppose next that the $\gcd$ of the lengths of the primitive cycles
in the component $C$ is $1$. From Lemma~\ref{prim} we may fix a vertex
$v_0\in C$ and directed cycles $\eta_1,\dots,\eta_m$ starting and ending at $v_0$
such that $\gcd ( |\eta_1|,\dots,|\eta_m|)=1$. Any sufficiently large integer
is a non-negative integral linear combination of the $|\eta_i|$.
So, for any sufficiently large $p,q$ (and we may take $\gcd(p,q)=1$)
there are directed cycles $\gamma$ and $\delta$ of lengths $p$ and $q$ respectively
which come from combinations of $\eta_1, \dots, \eta_m$ and therefore start and end at $v_0$. Using $\gamma$ and $\delta$ we define $g$
on $T_1$ and $T_2$ respectively in the obvious manner, namely, set $g(r)=v_0$, and
follow $\gamma$ to color the points of $T_1$, and likewise follow $\delta$ for $T_2$. This gives a directed graph homomorphism $g:\Gamma^{(1)}_{1,p,q}\to \Lambda$.
\end{proof}

This completes the proof of Theorem~\ref{thm:ssonedim}.

\section{The subshift problem for $F(2^{\Z^2})$} \label{sec:sstwodim}

In this section we show that the general subshift problem for dimension at least $2$ is $\Sigma^0_1$-complete, and in particular undecidable. In \S\ref{sec:gh} we will show a similar result for the graph homomorphism problem. Since the graph homomorphism problem can be viewed as a special case of the subshift problem, the main theorem of \S\ref{sec:gh} will imply the main result of this section. However, we present a detailed proof in this section for two reasons. The first is that our proof here will be self-contained in the sense that we will construct a computable reduction from the halting problem of Turing machines to the subshift problem. In contrast, the proof in \S\ref{sec:gh} uses the $\Sigma^0_1$-completeness of a special case of the word problem as a black box. The second reason is that our proof here establishes the undecidability of the subshift problem for subshifts $Y\subseteq \bsft^{\Z^2}$ whenever $\bsft\geq 4$. In contrast, the proof in \S\ref{sec:gh} will require a much larger $\bsft$.

\begin{thmn} \label{thm:subsnr}
Let $\bsft\geq 4$. The set of integers $m$ coding subshifts $Y_m \subseteq \bsft^{\Z^2}$ for which there is a continuous,
equivariant map from $\fzs$ to $Y_m$ is a $\Sigma^0_1$-complete set.
\end{thmn}

We note that we do not have a similar result for the subshift problem for $\bsft=2$ or $\bsft=3$. In particular, it is unknown if the subshift problem for $\bsft=2$ is decidable.

The rest of this section is devoted to a proof of Theorem~\ref{thm:subsnr}.
Let $\ssc \subseteq \omega$ be the set of integers $m=\langle \bsft,p_1,\dots,p_k\rangle$ coding subshifts
$Y_m$ such that there is a continuous, equivariant map from $\fzs$ to $Y_m$.
We show $\ssc$ is $\Sigma^0_1$-complete by defining a computable reduction from the halting problem to $\ssc$.

\subsection{Good Turing machines}
We adopt some conventions regarding Turing machines. Our Turing
machines operate on a bi-infinite memory tape divided into discrete
cells. In a typical step, the machine positions its head over a cell
and reads the symbol there.  Then, according to the symbol and its
present state, the machine looks up a table of instructions and takes
action following the instruction applicable to the symbol read and the
state. A typical action involves the machine writing a symbol in the
cell followed by positioning its head either at the present cell or
one cell to the left or one cell to the right. Upon completion of an
action the machine changes its state to a new state following the
instruction. After this the machine is ready to proceed to the next
instruction or it can halt.  The Turing machines we work with will
have some additional features. First, they will all have the alphabet
$$ \Sigma=\{B, S, E, 0, 1\}. $$ The symbol $B$ indicates that the cell
is blank.  The symbol $S$ is intended to mark the start of a region on
the tape (the ``data region'') where the machine will do its work.  The symbol $E$ marks the
end of the working region. The symbols $0$ and $1$ are used by the
Turing machine for its computations. Each of our Turing machines will
have a set of states $Q=\{q_0,\dots, q_t\}$, with two special states
$q_0$ and $q_t$. The state $q_0$ is the initial state of the machine
at the beginning of the computation. The state $q_t$ is the halting
state. We think of each instruction of a Turing machine to be a
quadruple
$$ I=(q, \sigma, \alpha, q') $$ where $q, q'\in Q$, $\sigma\in
\Sigma$, and $\alpha$ is a pair $(\tau, P)$ where $\tau\in\Sigma$ and
$P\in\{-1,0,1\}$. In the end, a Turing machine is written as a finite
set of instructions
$$ M=\{ (q_i, \sigma_i, \alpha_i, q'_i)\,:\,1\leq i\leq s\}. $$
Figure~\ref{fig:tm} illustrates a typical step in the computation of a Turing machine.

\begin{figure}[h]
\begin{tikzpicture}[scale=0.03]

\pgfmathsetmacro{\s}{0.03}

\pgfmathsetmacro{\n}{20}
\pgfmathsetmacro{\r}{0.05}

\draw (-150,0) to (150,0);
\draw (-150,\n) to (150, \n);

\foreach \i in {1,...,14}
{
\draw ({-150+\i*20},0) to ({-150+\i*20},20);
}

\node[scale=0.75] at (-150, 10) {$\cdots\cdots$};
\node[scale=0.75] at (150, 10) {$\cdots\cdots$};
\node[scale=0.75] at (-120, 10) {$B$};
\node[scale=0.75] at (-100, 10) {$B$};
\node[scale=0.75] at (-80, 10) {$S$};
\node[scale=0.75] at (-60, 10) {$0$};
\node[scale=0.75] at (-40, 10) {$1$};
\node[scale=0.75] at (-20, 10) {$1$};
\node[scale=0.75] at (0, 10) {$0$};
\node[scale=0.75] at (20, 10) {$E$};
\node[scale=0.75] at (40, 10) {$B$};
\node[scale=0.75] at (60, 10) {$B$};
\node[scale=0.75] at (80, 10) {$B$};

\draw (-40,20) to (-45,25);
\draw (-40,20) to (-35,25);
\draw (-45,25) to (-65,25);
\draw (-35,25) to (-15,25);
\draw (-65,25) to (-65, 45);
\draw (-15,25) to (-15, 45);
\draw (-65,45) to (-15,45);

\node[scale=0.75] at (-25, 35) {$q$};

\end{tikzpicture}
\caption{A typical step in the computation of a Turing machine} \label{fig:tm}
\end{figure}

A {\em tape configuration} is an element of $\Sigma^\Z$ representing the content of the tape. The following definition imposes more restrictions on the Turing machines we would like to work with.

\begin{defnn} \label{goodm}
Let $M$ be a Turing machine.
Let $p(i) \in \Z$ denote the position of the reading head of $M$ on the tape at step $i$. Let $z_i$ be the tape configuration of $M$ at step $i$. We say that $M$ is {\em good} if it has the following properties:
\begin{enumerate}
\item \label{g0} $p(0)=0$; $z_0(0)=S$, $z_0(1)=E$, and $z_0(k)=B$ for all $k\neq 0,1$.
\item \label{g1}
For any $i\geq 0$, $p(i)\geq 0$; $z_i(0)=S$, $z_i(k)=B$ for all $k<0$, and for any $k\neq 0$, $z_i(k)\neq S$.
\item \label{g2}
For any $i\geq 0$, there is a unique $K>0$ such that $z_i(K)=E$ and $z_i(k)=B$ for all $k>K$. We refer to this cell as the {\em ending cell} of the working region.
\item \label{g3}
For any $i\geq 0$, if $K$ is the ending cell of the working region, then $p(i)\leq K+1$.
\item \label{g4}
For any $i\geq 0$, if $K$ is the ending cell of the working region, then $z_i(k)\in \{0,1,E\}$ for all $0<k<K$. Moreover, if $z_i(k)=E$ for any $0<k<K$, then $k=K-1$.
\item \label{g5}
If $M$ halts at step $i$, then $p(i)=0$, $z_i(1)=E$, and $z_i(k)=B$ for all $k\neq 0,1$.
\end{enumerate}
\end{defnn}

In English, a good machine only uses the non-negative portion of the
tape, and always maintains a working region with a starting cell with
symbol $S$ and an ending cell with symbol $E$. The symbol in the cell
immediately preceding the ending cell could have a symbol $E$; this
is intended to address the situation where the machine needs to extend
its working region and therefore needs to temporarily write a symbol
$E$ next to the previous ending cell. The point is that at any time
during the computation, the working region should always end with an
ending cell with symbol $E$. Except for this convention, the working
region should always consist of $0$s and $1$s. During the computation,
the reading head stays in the work region, with the only exception
that it might be positioned one cell to the right of the ending cell
(again, to accommodate extension of the working region). If a good
machine ever halts, it halts in the state $q_t$ with its head over the
beginning cell, with the tape configuration identical to the initial
configuration of the computation.

Although good Turing machines operate in a very restricted way, it is
a standard fact that they are able to simulate any computation of
Turing machines. For programming techniques for Turing machines the
reader can consult standard references such as Chapter 8 of
\cite{HMD}. For example, the technique to keep all computations on a
semi-infinite tape while requiring that the machine never prints the
blank symbol $B$ is presented in \S 8.5.1 of \cite{HMD}.

We use the usual G\"{o}del numbering for Turing machines to assign a
natural number to each Turing machine. Conversely, for each natural
number $n\in\omega$ there is a corresponding Turing machine $M_n$. The
standard programming techniques for Turing machines give that there is
a total computable function $f \colon \omega \to \omega$ such that for
all $n \in\omega$, $M_{f(n)}$ is a good Turing machine, and $M_n$
halts if and only if $M_{f(n)}$ halts. We fix this computable function
$f$ for the rest of the proof. Our halting set is $$H=\{ n \colon
M_{n} \text{ halts}\}=\{ n \colon M_{f(n)} \text{ halts}\}.$$ This is
a $\Sigma^0_1$-complete set of integers.

We will define another computable function $h \colon \omega \to \omega$ such that
if $n$ codes a good Turing machine $M_n$, then $h(n)=\langle \bsft,p_1,\dots,p_k\rangle$
codes a subshift $Y_{h(n)}$ of finite type, and $M_n$ halts if and only if $h(n) \in \ssc$.
Then $h\circ f$ gives a computable reduction of $H$ to $\ssc$, and so
$\ssc$ is $\Sigma^0_1$-complete.

\subsection{The canonical allowable pattern} \label{cap}
For the rest of the proof we fix a good Turing machine $M=M_n$ and describe the subshift $Y=Y_{h(n)}$.
It will be clear from the construction that the map $h$ being implicitly defined is computable.
We describe the subshift by giving the local rules an element $x \in \bsft^{\Z^2}$ must satisfy
in order to be an element of $Y$ (we are thereby defining the set of patterns $p_i$ for the subshift).

The idea is that, for an element $x\in \bsft^{\Z^2}$ to be in the
subshift, $x$ must appear to be coding a valid computation sequence
following the instructions of the Turing machine $M$. To describe the
set of patterns for the subshift, we first define the  {\em canonical allowable
pattern} $\pi(M)\in \{ 0,1,3\}^{\Z^2}$ associated  to $M$.
Each horizontal line $\ell_j=\{ (i,j)\in \Z^2 \colon i\in\Z\}$,
for an even $j \geq 0$,  of $\pi(M)$ will code a valid tape
configuration of $M$ at a particular step of the computation, and the line above, $\ell_{j+1}$, will contain a
code for the state and the head position of $M$ at this step.
When two lines $\ell_j, \ell_{j+1}$, where $j\geq 0$ is even, in $\pi(M)$ code the information of a step during the
computation, the two lines above, $\ell_{j+2}, \ell_{j+3}$, will code the information of the
next step obtained by following the applicable instruction of $M$.

To implement this idea, we use the following coding scheme for the alphabet of $M$:
$$ \begin{array}{rcl}
B&\mapsto & 11111 \\
S& \mapsto & 01110 \\
0&\mapsto & 01000 \\
1& \mapsto & 01100 \\
E &\mapsto & 11110
\end{array}
$$
It is easy to see that a bi-infinite string that codes a tape configuration of a good Turing machine this way is uniquely readable. To code the head position and the state of the Turing machine, we assume without loss of generality that the set of states $Q=\{ q_0=1, q_1=2, \dots, q_t=|Q|\}$ consists of integers. We then use the string
$$ 000001^q00000 $$
to code the state $q\in Q$, where $1^q$ denotes $q$ many consecutive 1s. The obvious way to code the position is to line up the leading 5 consecutive 0s in the above code with the code for the symbol of the cell the head is reading. However, we will postpone this to address another important consideration.

The consideration is regarding the parities of positions of $1$s in $\pi(M)$. To be precise, an element $(i,j)\in \Z^2$ has {\em even parity} if $i+j$ is even, and {\em odd parity} if $i+j$ is odd. The consideration is that we would like to make sure
that in our definition of the canonical allowable pattern $\pi(M)$,
all 1s appear in positions of the
same parity.  For this we use the method of {\em padding}, which is to
append a $0$ to each digit of the code, no matter whether it is 0 or
1. Equivalently, we apply the replacement rule
$$ 0\mapsto 00, \ \ 1\mapsto 10 $$ to the results of above coding
schemes. Now all lines of $\pi(M)$ coding tape configurations have 1s
occurring only in even positions. It only remains to determine the
exact position for the (padded) string coding the state of the
machine. For this we first align the beginning of the string with the
beginning position of the code for the cell the head is reading and
then shift one unit to the right. Thus in all lines of $\pi(M)$ coding
the state and the head position, all 1s occur in odd
positions. Figure~\ref{fig:cap} illustrates the construction of
$\pi(M)$; some spaces are added for readability.

\begin{figure}[htbp]
\begin{tikzpicture}[scale=0.03]

\pgfmathsetmacro{\s}{0.03}

\pgfmathsetmacro{\n}{80}
\pgfmathsetmacro{\r}{0.05}

\draw (-150,60) to (150,60);
\draw (-150,\n) to (150, \n);

\foreach \i in {1,...,14}
{
\draw ({-150+\i*20},60) to ({-150+\i*20},80);
}

\node[scale=0.75] at (-150, 70) {$\cdots\cdots$};
\node[scale=0.75] at (150, 70) {$\cdots\cdots$};
\node[scale=0.75] at (-120, 70) {$B$};
\node[scale=0.75] at (-100, 70) {$B$};
\node[scale=0.75] at (-80, 70) {$S$};
\node[scale=0.75] at (-60, 70) {$0$};
\node[scale=0.75] at (-40, 70) {$1$};
\node[scale=0.75] at (-20, 70) {$1$};
\node[scale=0.75] at (0, 70) {$0$};
\node[scale=0.75] at (20, 70) {$E$};
\node[scale=0.75] at (40, 70) {$B$};
\node[scale=0.75] at (60, 70) {$B$};
\node[scale=0.75] at (80, 70) {$B$};

\draw (-40,80) to (-45,85);
\draw (-40,80) to (-35,85);
\draw (-45,85) to (-65,85);
\draw (-35,85) to (-15,85);
\draw (-65,85) to (-65, 105);
\draw (-15,85) to (-15, 105);
\draw (-65,105) to (-15,105);

\node[scale=0.75] at (-25, 95) {$q_2$};

\draw (-10,55) to (10,55);
\draw (-10,55) to (-10,45);
\draw (10,55) to (10, 45);
\draw (-10,45) to (-20,45);
\draw (10,45) to (20,45);
\draw (20,45) to (0,35);
\draw (-20,45) to (0,35);

\node[scale=0.75] at (0, 30) {$\cdots\cdots$};
\node[scale=0.75] at (0, 0) {$\cdots\cdots$};
\node[scale=0.75] at (19,20) {$000000000\ 0101010000\ 0000000$};
\node[scale=0.75] at (0, 10) {$0010101000\ 0010000000\ 0010100000\ 0010100000\ 0010000000\ 1010101000$};

\end{tikzpicture}
\caption{From the computation of $M$ to the canonical allowable pattern $\pi(M)$} \label{fig:cap}
\end{figure}

Thus the canonical allowable pattern $\pi(M)$ will start with two
lines $\ell_0, \ell_1$ coding the initial configuration (a single $S$ followed by a
single $E$), the initial state ($q_0=1$), and the initial head
position (over the starting cell). Then $\pi(M)$ will extend upward as
the computation of $M$ proceeds, with every two lines representing a
valid transition from the two lines below according to the applicable
instructions of the machine.
For lines below the starting configuration line, $\pi(M)$ will consist
entirely of codes for the symbol $B$, with every other line shifted to the right
by $1$ to maintain the same parity of $1$s.
Note that so far $\pi(M)$ is a pattern of
0s and 1s, with all 1s occurring in positions of the same parity. If
$M$ does not halt, then the region of $\pi(M)$ not coding $B$s will extend
upwards forever and will be
infinite. In case $M$ halts, this extension process will stop, at
which time we will introduce a new symbol 3 (we reserve the symbol 2
for another purpose below) and add a $10\times 10$ block of the symbol 3
directly above the 10 leading 0s in the code for the halting state
$q_t$. For all regions that have not been covered by this definition, $\pi(M)$ will consist of
alternating $0$s and $1$s so that the same parity of $1$s is maintained.
This finishes the definition of the canonical allowable pattern
$\pi(M)$. Thus in case $M$ halts, the part of $\pi(M)$ not coding $B$s
will be a finite pattern
with symbols 0, 1, and 3. Note that every line of the domain of
$\pi(M)$ has  a finite interval of non-blank codes, and that the non-blank portion of $\pi(M)$ is
connected. Figure~\ref{fig:cannhalt} illustrates the  canonical allowable pattern $\pi(M)$ in case $M$ halts.

\begin{figure}[h]

\begin{tikzpicture}[scale=0.05]

\draw[fill, color=darkgray] (0,0) rectangle (5,90);
\draw[fill, pattern=north east lines] (0,92) rectangle (5,97);
\draw[fill, color=darkgray] (7,85) rectangle (12,90);
\draw[fill,color=darkgray] (7,5) to (12,5) to (50,60) to (12,85) to (7,85) to (45,60) to (7,5);
\draw[fill, color=darkgray] (7,0) rectangle (12,5);
\draw[fill,color=lightgray] (5,5) to (7,5) to (45,60) to (7,85) to (5,85) to (5,5);

\draw (5,50) to (39,50);
\draw (5,45) to (36,45);

\draw[fill,color=blue] (30,45) rectangle (32, 50);

\draw[draw=none,fill=lightgray] (5,0) rectangle (7,5);
\draw[draw=none,fill=lightgray] (5,85) rectangle (7,90);


\node at (-50,95) (sb) {\small{stopping block}};
\draw[->] (-25,95) to (-1,95);

\node at (-60,30) (sb) {\small{state and position}};
\draw[->] (-30,30) to (29,47);

\end{tikzpicture}
\caption{A bird's eye view of the canonical allowable pattern $\pi(M)$ for a halting computation.
Black areas represent codes for starting cells and ending cells.
The gray area consists of $0$s and $1$s  and the lined area consists of $3$s.} \label{fig:cannhalt}
\end{figure}

Here is the intuition behind the introduction of the block of 3s for
halting computations. In a general element of the subshift (to be defined), there
might be different areas appearing to code valid computations of the
machine. For each of these areas the 1s occur in positions with the
same parity, but the parities of 1s for different areas do not
necessarily match.
These computation regions of (possibly) different parity will be separated
by ``threads'' of the symbol $2$.
The block of 3s is introduced to allow these parity
changing threads to emanate/terminate, thus
allowing any potential parity conflicts to be resolved.  In case $M$
halts, this will allow us to show, using Theorem~\ref{thm:tilethm},
that there is a continuous, equivariant map into the
subshift. Conversely, if there is a continuous equivariant map into
the subshift, Theorem~\ref{thm:tilethm} will imply that halting
configurations must exist.

\subsection{Prototypical elements and the subshift of finite type}
We now define the notion of a {\em prototypical element} for $M$. A prototypical
element will be an element of $\{ 0,1,2,3\}^{ \Z^2}=4^{\Z^2}$.
These elements will be the basis for our definition of the subshift
$Y$ associated to $M$. The idea is that within regions separated by $2$s,
we appear to see a valid (coded and padded) Turing machine computation
(with alternating rows giving the tape and state/position information
as described previously).

\begin{defnn} \label{def:accep}
Let $\ell=\ell(M)=100 |Q|$.
We say that $x \in 4^{\Z^2}$ is {\em prototypical} if
for any $\ell \times \ell$ rectangle $R\subseteq \Z^2$, $x \res R$
satisfies the following:

\begin{enumerate}
\item(no parity violations) \label{acc_a}
There do not exist points $a,b\in R$
of different parities in $\Z^2$
with $x(a)=x(b)=1$,
and a path $p$ from $a$ to $b$ in the undirected Cayley graph on $R$
such that $x \res p$ takes values in $\{ 0,1\}$.

\item(no adjacent $2$ colors) \label{acc_b}
If $a,b \in R$ and $x(a)=x(b)=2$,
then $a, b$ are not related by a generator of $\Z^2$.
Furthermore, if $x(a)=2$, $j \leq 15$, and $x(a+(i,0))\notin \{2,3\}$
for $i \leq j$, then the values $x(a+(1,0)),\dots,x(a+(j,0))$
are alternating $0$s and $1$s (starting with either a $0$ or a $1$).
Similarly, if we look to the left of $a$ up to $15$ positions
we see an alternating pattern of $0$s and $1s$ until we see a $3$.

\item(no Turing machine violations) \label{acc_c}
If $D\subseteq R$ is the set of $a\in R$ such that $x(a)\in \{0,1,3\}$,
then $x\res D$ has no violations of the rules for the Turing machine $M$.
Specifically, we mean the following.

\begin{enumerate}
\item
If $I$ is a horizontal segment of $x \res D$,
then $I$ is consistent with being a coded and padded sequence
of a portion of the data line of a Turing machine,
or the coded and padded
sequence of a portion of the state/position line of a Turing machine, or
consistent with being a portion of a $10 \times 10$ region of $3$s surrounded to the left
and right by blanks.
\item
If $I=(a,b)\times \{ j\}$ is a horizontal segment of $x \res D$ coding $0$s and $1$s,
then  $S=(a-3,b+3)\times [j-3,j+3] \cap R$ does not contain any $2$s and
is consistent with the rules of a coded and padded Turing Machine
(including the rule about a $10 \times 10$ region of $3$s after a halting line).
Also, if $I$ is an interval of $3$s, then the two lines below this is a properly aligned
halting configuration.

\end{enumerate}
\end{enumerate}
\end{defnn}

Thus,  a prototypical $x \in 4^{\Z^2}$ has the property that in each
$\ell \times \ell$ window, each of the regions bounded by the
$2$s appears to be a (coded and padded) instance of a computation according to $M$.
Note that in (\ref{acc_c}) of Definition~\ref{def:accep} we only require that
$x$ satisfies the rules of the Turing machine $M$, we
do {\em not} require that it is part of the actual canonical allowable pattern $\pi(M)$.


We are now ready to define the subshift $Y$ associated to a good Turing machine $M$.

\begin{defnn} \label{def:tursubs}
Let $\ell=\ell(M)=100 |Q|$ as before.
Let $W$ be the set of all $\ell \times \ell$ patterns which occur in some prototypical element of $4^{\Z^2}$.
Define $Y$ to be the subshift of $4^{\Z^2}$ determined by $(p_1,\dots,p_k)$,
where the $p_i$ are all $\ell\times\ell$ patterns not in $W$.
\end{defnn}

The prototypical elements $x \in \bsft^{\Z^2}$ are precisely the elements of the subshift $Y$.
It is obvious from the definition that $Y$ has width $\ell-1$.

In the following we prove some basic facts about prototypical elements.

\begin{lemn} \label{ttc}
Let $x \in 4^{\Z^2}$ be a prototypical element for the Turing machine $M$
and suppose that $x(a)\neq 3$ for all $a \in \Z^2$. Then there
is a proper $3$-coloring $x' \in 3^{\Z^2}=\{0,1,2\}^{\Z^2}$ such that if $x'(a)\neq x(a)$, then $x(a)=0$ and $x'(a)=1$.
Moreover, for every $a\in \Z^2$, whether $x'(a)\neq x(a)$ is determined by the values of $x$
in the region $\{b\in \Z^2 \colon \|a-b\|\leq 20\}$.
\end{lemn}

\begin{proof}
Let $D=\{ a \in \Z^2 \colon x(a) \in \{ 0,1\} \}$. We claim that
in any connected component $C$
of $D$ that the positions of $1$s all have the same parity.
To see this claim, let $a,b \in \Z^2$ be in the same component $C$ of $D$.
Let $p$ be a path in $C$ connecting $a$ and $b$.
So, $x(c)\in \{ 0,1\}$ for all $c \in p$.

We show a subclaim that for any $c \in p$
there is $d \in \Z^2$ with $\| c-d\| \leq 20$ such that $x(d)=1$
and furthermore there is a path from $c$ to $d$ staying within
$C\cap B(c,20)$ (the points of distance $\leq 20$ from $c$).
Let $c_1$ be the first point to the left of $c$ for which $x(c_1)=2$
(if this does not exist,  the point $d$ can be taken to the
left of $c$ and within distance $11$ of $c$ since the maximum number of consecutive
$0$s in a horizontal line of a prototypical element is $11$ and there are no $3$s in $x$).
Likewise define $c_2$ to the
right of $c$.
If $\| c_1-c_2\|\geq 13$, then the above fact that the maximum number of consecutive
$0$s in a horizontal line of a prototypical element is $11$ gives the result.
If $3 \leq\| c_1-c_2\| \leq 12$, then
condition (\ref{acc_b}) of Definition~\ref{def:accep} implies there is a $1$ in the horizontal interval
between $c_1$ and $c_2$ and we are done.
Finally, suppose $c_1$ and $c_2$ are the points immediately to the left and right of $c$.
The path $p$ must move from $c$ to the point immediately above or below $c$.
Say, for example, $p$ moves from $c$ to $c'$, the point immediately
above $c$. Let $c'_1$, $c'_2$ be the points immediately above $c_1$
and $c_2$. By (\ref{acc_b}) of Definition~\ref{def:accep}, $x(c'_1)$ and
$x(c'_2)$ are not equal to $2$. Repeating the argument given for $c_1$ and $c_2$,
shows that there is a $1$ on the horizontal segment
between $c'_1$ and $c'_2$, which proves the subclaim.

For each point $c\in p$, let $d(c)$ be a point of $C$ connected to $c$ by
a path in $C$ of length $\leq 20$ with $x(d(c))=1$. Since $\ell >40$, an easy induction
using (\ref{acc_a}) of Definition~\ref{def:accep} shows that
all of the points $d(c)$ have the same parity. In particular, $a$ and $b$
have the same parity. This proves the claim.

We define $x'$ as follows. Let $x'(a)=x(a)$ if $x(a)=1$ or $2$. If $x(a)=0$,
then consider the connected component $C$ of $D$ containing $a$.
If $a$ has the same parity as the positions of $1$s in $C$, then let $x'(a)=1$; otherwise let $x'(a)=x(a)=0$.
Using condition~(\ref{acc_b}) of Definition~\ref{def:accep} and the above claim, it is easy to see that $x'$ is a
proper $3$-coloring.

To see the ``moreover" part of the lemma, note that by the subclaim, for
every $a\in \Z^2$ such that there is $b$ in the connected component $C$
of $a$ in the subgraph induced by $D$ with $x(b)=1$, there is such a $b$ with $\|a-b\|\leq 20$ and
a path $p$ from $a$ to $b$ such that $x\res p$ takes
values in $\{0, 1\}$.
Thus $x'(a)\neq x(a)$ exactly when $a$ and $b$
are of the same parity and $x(a)=0$. By our claim this does not depend
on the choice of the point $b$ and the path $p$. Thus the values of
$x$ in the region $\{b\in\Z^2 \colon \|a-b\|\leq 20\}$ completely determine
whether $x'(a)\neq x(a)$.
\end{proof}

Fix for the moment $n,p,q$ with $p,q>n\geq \ell-1$, where $\ell$ is as in
Definition~\ref{def:accep} for the good Turing machine $M$.
Let $H_1=\Gdpacqa$, $H_2=\Gcqadpa$ be the labeled grid graphs for the two
horizontal long tiles used in defining $\gnpq$ as shown in Figure~\ref{fig:Gamma-npq-horiz}.
Let $H=H_{n,p,q}$ be the quotient graph of the disjoint union of $H_1$ and $H_2$
by identifying corresponding vertices of the blocks having the same label,
as in the definition of $\gnpq$ (see \S\ref{subsec:tilethm}). Thus, $H$ is an induced
subgraph of the graph $\gnpq$. Note that we may also obtain $H$ by first stacking
$H_1$ on top of $H_2$ so that the labeled blocks along the bottom side  of $H_1$ coincide with the
corresponding labeled blocks along the top side of $H_2$, and then identifying corresponding vertices
as in $\gnpq$.
If $z$ is a map from the vertex set $V(H)$ of $H$
to $4=\{ 0,1,2,3\}$, we say $z$ is prototypical if for any $\ell\times \ell$ rectangle $R$ contained in one of the grid graphs
$H_1$ or $H_2$, $z\res R$ satisfies (\ref{acc_a})--(\ref{acc_c}) of Definition~\ref{def:accep}.

From Lemma~\ref{ttc} and Theorem~\ref{thm:pcn} we have the following.

\begin{lemn} \label{ttd}
Let $p,q>n \geq \ell-1$, where $\ell$ is as in
Definition~\ref{def:accep} for the good Turing machine $M$, and assume $\gcd(p,q)=1$.
Let $H_{n,p,q}$ be as defined above. If $z\colon V(H_{n,p,q}) \to 4$
is prototypical, then $3 \in \ran(z)$.
\end{lemn}

\begin{proof}
Suppose $z\colon V(H_{n,p,q}) \to 4$ is prototypical and $z(c)\neq 3$
for all vertices $c$ of $H_{n,p,q}$. Let $\bar{x}$ be the element of
$3^{\Z^2}$ obtained by tiling $\Z^2$ with copies of $H_1$ and $H_2$ in
the manner illustrated in Figure~\ref{fig:incon1}. That is, we stack
$H_1$ and $H_2$ alternatively, with the labeled blocks along the
bottom side of a copy of $H_1$ coinciding with the labeled blocks
along the top side of the copy of $H_2$ below it, and similarly the
labeled blocks along the bottom side of a copy of $H_2$ coinciding
with the labeled blocks along the top side of the copy of $H_1$ below
it. Also, neighboring columns share their labeled blocks on the sides.

\begin{figure}[h]
\begin{tikzpicture}[scale=0.01]

\pgfmathsetmacro{\s}{0.03}

\pgfmathsetmacro{\n}{20}
\pgfmathsetmacro{\r}{0.05}

\tikzset{
pics/myshapec/.style={
code={
\draw (-150,0) rectangle (150,80);
\draw (-150,\n) to (150,\n);
\draw (-150,80-\n) to (150,80-\n);
\draw (-150+\n,0) to (-150+\n, 80);
\draw (150-\n,0) to (150-\n,80);

\foreach \i in {1,...,2}
{
\draw ({-150+\n/2+\i*(300-\n)/3-\n/2},0) to ({-150+\n/2+\i*(300-\n)/3-\n/2},\n);
\draw ({-150+\n/2+\i*(300-\n)/3+\n/2},0) to ({-150+\n/2+\i*(300-\n)/3+\n/2},\n);
}

\foreach \i in {1,...,4}
{
\draw ({-150+\n/2+\i*(300-\n)/5-\n/2},80-\n) to ({-150+\n/2+\i*(300-\n)/5-\n/2},80);
\draw ({-150+\n/2+\i*(300-\n)/5+\n/2},80-\n) to ({-150+\n/2+\i*(300-\n)/5+\n/2},80);
}

}}}

\tikzset{
pics/myshaped/.style={
code={
\draw (-150,0) rectangle (150,80);
\draw (-150,\n) to (150,\n);
\draw (-150,80-\n) to (150,80-\n);
\draw (-150+\n,0) to (-150+\n, 80);
\draw (150-\n,0) to (150-\n,80);

\foreach \i in {1,...,2}
{
\draw ({-150+\n/2+\i*(300-\n)/3-\n/2},80-\n) to ({-150+\n/2+\i*(300-\n)/3-\n/2},80);
\draw ({-150+\n/2+\i*(300-\n)/3+\n/2},80-\n) to ({-150+\n/2+\i*(300-\n)/3+\n/2},80);
}

\foreach \i in {1,...,4}
{
\draw ({-150+\n/2+\i*(300-\n)/5-\n/2},0) to ({-150+\n/2+\i*(300-\n)/5-\n/2},\n);
\draw ({-150+\n/2+\i*(300-\n)/5+\n/2},0) to ({-150+\n/2+\i*(300-\n)/5+\n/2},\n);
}

}}}

\foreach \j in {-1,0,1}
{
\pic[xscale=0.01,yscale=0.01] at (0-\j*280,0) {myshapec};
\pic[xscale=0.01,yscale=0.01] at (0-\j*280,-80+\n) {myshaped};
\pic[xscale=0.01,yscale=0.01] at (0-\j*280,-160+2*\n) {myshapec};
}

\end{tikzpicture}
\caption{Tiling of $\Z^2$ with $H_1$ and $H_2$.} \label{fig:incon1}
\end{figure}

Since $n\geq \ell-1$, any $\ell\times \ell$ region is properly contained in
a copy of $H_1$ or $H_2$, and it follows that $\bar{x}$ is a
prototypical element of $3^{\Z^2}$. Let $x'$ be obtained from
$\bar{x}$ as in Lemma~\ref{ttc}. So $x'$ is a proper $3$-coloring of
$\Z^2$ and $x'$ is obtained from $\bar{x}$ by only changing some $0$
values of $x$ to $1$. Moreover, for any $a\in \Z^2$ whether $x'(a)\neq
\bar{x}(a)$ is determined by the values of $x$ in the region
$\{b\in\Z^2 \colon \|a-b\|\leq 20\}$. Since $n\geq \ell-1>2\cdot 20$, from the periodic
nature of $\bar{x}$ it follows that $x'$ also respects the identified
vertices in the tiling (e.g., for two elements $a,a'$ of $\Z^2$ in the
same position of similarly labeled blocks we have
$x'(a)=x'(a')$). Thus, $x'$ induces a proper $3$-coloring of the graph
$H_{n,p,q}$. In particular, this gives a proper $3$-coloring for each
of the long horizontal tiles $H_1=\Gdpacqa$ and $H_2=\Gcqadpa$.  This
contradicts Theorem~\ref{thm:pcn}.
\end{proof}

\subsection{Proof of Theorem~\ref{thm:subsnr}}

Let $M$ be a good Turing machine and $Y$ the subshift associated to $M$
as in Definition~\ref{def:tursubs}. To finish the proof of Theorem~\ref{thm:subsnr}
it suffices to show that  $M$ halts if and only if there is a continuous, equivariant map from
$\fzs$ to $Y$.

First suppose that $M$ halts in $N$ steps. Fix $n> \ell$ and $p,q>\max\{2n, 100N\}$ with $\gcd(p,q)=1$.
We may also assume that $n,p,q$ are odd, and $q= p+2$.
From Theorem~\ref{thm:tilethm} it suffices to
find $g \colon \gnpq \to 4$ which respects the subshift $Y$.
We describe such a $g$ as a coloring of $\gnpq$ with colors $\{0,1,2,3\}$.

We use the following terminology. For a subset $D\subseteq \Z^2$,
a {\em checkerboard pattern} on $D$ is a proper 2-coloring with colors $\{0,1\}$.
If $D$ is a region connected by generators of $\Z^2$, there are only
two checkerboard patterns possible on $D$. For each of them the positions of $1$s are of the same parity, and
the parities of 1s for the two patterns are the opposite.
By a {\em zig-zag path} in $\Z^2$ we mean a sequence $u_0, u_1,\dots$ of points
such that $u_{i+1}= (\pm 1, \pm 1) +u_i$ for all $i$. Note that all points on such a path have the same
parity, in particular no two points of the path are connected by a generator of $\Z^2$.

We describe the map $g$ by first describing the values of $g$
on the regions with labels $R_\times$, $R_a$, $R_b$, $R_c$, and $R_d$.
For $R_\times$, we use a zig-zag path connecting the
corner points of the region to each other as shown in Figure~\ref{fig:grx}.
This is possible as $n$ is odd. The points on the path
are given color $2$. As demonstrated in Figure~\ref{fig:grx}, the path created an interior region.
This region is given a checkerboard pattern. Finally, the points on the
top and right edges of $R_\times$ are given color $0$, and the
points on the other two edges are given color $1$.

\begin{figure}[h]

\begin{tikzpicture}[scale=0.05]

\pgfmathsetmacro{\d}{10}

\definecolor{darkgreen}{rgb}{0.2,1,0.5}


\foreach \i in {0,...,2}
{
\foreach \j in {0,...,3}
{
\draw (\d*2*\i, \d*2*\j) circle (0.75);
}
\foreach \j in {0,...,2}
{
\draw[fill] (\d*2*\i, \d*2*\j+\d) circle (0.75);
}
}

\foreach \i in {0,...,2}
{
\foreach \j in {0,...,2}
{
\draw[fill] (\d*2*\i+\d, \d*2*\j) circle (0.75);
\draw (\d*2*\i+\d,\d*6) circle (0.75);
}
\foreach \j in {0,...,2}
{
\draw (\d*2*\i+\d, \d*2*\j+\d) circle (0.75);
}
}

\foreach \j in {0,...,3}
{
\draw[fill,color=darkgreen] (0,2*\j*\d) circle (0.75);
}
\foreach \j in {0,...,2}
{
\draw[fill,color=darkgreen] (\d,2*\j*\d+\d) circle (0.75);
}
\draw[fill,color=darkgreen] (2*\d,0) circle (0.75);
\draw[fill,color=darkgreen] (2*\d,6*\d) circle (0.75);
\draw[fill,color=darkgreen] (3*\d,\d) circle (0.75);
\draw[fill,color=darkgreen] (3*\d,5*\d) circle (0.75);
\draw[fill,color=darkgreen] (4*\d,0) circle (0.75);
\draw[fill,color=darkgreen] (4*\d,6*\d) circle (0.75);
\foreach \j in {0,...,2}
{
\draw[fill,color=darkgreen] (5*\d,2*\j*\d+\d) circle (0.75);
}
\foreach \j in {0,...,3}
{
\draw[fill,color=darkgreen] (6*\d,2*\j*\d) circle (0.75);
}
\foreach \j in {0,...,2}
{
\draw (6*\d,2*\j*\d+\d) circle (0.75);
}

\draw[thick, color=darkgreen] (0,0) to (\d,\d) to (2*\d,0) to (3*\d, \d) to (4*\d,0) to (5*\d,\d) to (6*\d,0);
\draw[thick, color=darkgreen] (0,0) to (\d,\d) to (0,2*\d) to (\d, 3*\d) to (0,4*\d) to (\d,5*\d) to (0,6*\d);
\draw[thick, color=darkgreen] (0,6*\d) to (\d,5*\d) to (2*\d,6*\d) to (3*\d, 5*\d) to (4*\d,6*\d) to
(5*\d,5*\d) to (6*\d,6*\d);
\draw[thick, color=darkgreen] (6*\d,0) to (5*\d,\d) to (6*\d,2*\d) to (5*\d, 3*\d) to (6*\d, 4*\d)
to (5*\d,5*\d) to (6*\d,6*\d);

\end{tikzpicture}
\caption{The map $g$ on the $R_\times$-labeled region.} \label{fig:grx}
\end{figure}

Each of the $R_a$-labeled and $R_b$-labeled regions is given a
checkerboard pattern with the upper-left corner point getting a color
1. It is easy to check that this is consistent. Because $n, p, q$ are
odd, points in the top row of the $R_a$-labeled region get different
colors than their adjacent points in the bottom row of the
$R_\times$-labeled region, and points in the bottom row of the
$R_a$-labeled region get different colors than their adjacent points
in the top row of the $R_\times$-labeled region. Similarly for the
$R_b$-labeled regions.

We color the $R_c$-labeled and $R_d$-labeled regions in two steps. In
the first step, each of these regions is given a background
checkerboard pattern with the upper-left corner point getting a color
0. As above, the coloring of the $R_\times$-labeled region and the
fact that $p$, $q$ are odd give that this is a consistent proper
two-coloring of these regions. Next, we override the background
coloring by inserting in each $R_c$-labeled and $R_d$-labeled region a
portion of the canonical allowable pattern $\pi(M)$ that corresponds
to the beginning steps of the computation of $M$, with the positions
of 1s in the inserted piece having the same parity as the positions of
1s in the background. For example, we may insert the first two lines
of $\pi(M)$ corresponding to the initial step of the computation into
the top two rows of $R_c$ and $R_d$. Since $p-n, q-n>n\geq \ell-1$, there is
enough room in $R_c$ and $R_d$ for the insertion. We will make sure
that the code for the starting cell appears near the horizontal center
of the region. Note that the width of the codes for the two initial
steps of the computation is no more than $10N$ (since each step can
access only one more cell and each symbol is coded by a 0,1-string
with ten digits). By placing the code for the starting cell around the
center, we ensure that the codes for the entire computation, when
extended beyond the $R_c$-labeled or $R_d$-labeled regions, will stay
in a region directly above the $R_c$-labeled or $R_d$-labeled
region. This is illustrated in Figure~\ref{fig:glt}.

\begin{figure}[htbp]
\begin{tikzpicture}[scale=0.03]

\pgfmathsetmacro{\s}{0.03}

\pgfmathsetmacro{\n}{20}
\pgfmathsetmacro{\r}{0.05}

\definecolor{darkgreen}{rgb}{0.2,1,0.5}

\tikzset{
pics/myshape/.style={
code={
\draw[fill=darkgray] (-\r,\r) rectangle (0,0.7);
\draw[fill=lightgray] (0,\r) to (0.1,0.5) to (0,0.7);
\draw[fill=white] (0,0.7+\r) circle (\r);
}}}

\tikzset{
pics/myshapeb/.style={
code={
\draw[fill=darkgray] (-\r,\r) rectangle (0,\n*\s/2);
\draw[fill=lightgray] (0,\r) to (0.2*\n*\s/2,\n*\s/2) to (0,\n*\s/2);
}}}

\draw (-150,0) rectangle (150,180);
\draw (-150,\n) to (150,\n);
\draw (-150,180-\n) to (150,180-\n);
\draw (-150+\n,0) to (-150+\n, 180);
\draw (150-\n,0) to (150-\n,180);

\foreach \i in {1,...,2}
{
\draw ({-150+\n/2+\i*(300-\n)/3-\n/2},0) to ({-150+\n/2+\i*(300-\n)/3-\n/2},\n);
\draw ({-150+\n/2+\i*(300-\n)/3+\n/2},0) to ({-150+\n/2+\i*(300-\n)/3+\n/2},\n);
}

\foreach \i in {1,...,4}
{
\draw ({-150+\n/2+\i*(300-\n)/5-\n/2},180-\n) to ({-150+\n/2+\i*(300-\n)/5-\n/2},180);
\draw ({-150+\n/2+\i*(300-\n)/5+\n/2},180-\n) to ({-150+\n/2+\i*(300-\n)/5+\n/2},180);
}

\foreach \i in {0,...,3}
{
\draw[very thick, color=darkgreen] ({-150+\i*(300-\n)/3},0) rectangle ({-150+\i*(300-\n)/3+\n},\n);
}

\foreach \i in {0,...,5}
{
\draw[very thick, color=darkgreen] ({-150+\i*(300-\n)/5},180-\n) rectangle ({-150+\i*(300-\n)/5+\n},180);
}

\pic  foreach \i in {0,...,2} at ({-150+\n/2+300/6+\i*(300-\n)/3},{\n/2}) {myshape};

\pic  foreach \i in {0,...,4} at ({-150+\n/2+300/10+\i*(300-\n)/5},{180-\n/2}) {myshapeb};

\foreach \i in {0,...,2}
{
\draw[thick, color=darkgreen] ({-150+\n+\i*(300-\n)/3)},\n) to
({-150+ \n/2+300/6-1+\i*(300-\n)/3)}, {\n/2+0.7/\s+\r/(\s)-0.6});
}

\foreach \i in {1,...,3}
{
\draw[thick, color=darkgreen] ({-150+\i*(300-\n)/3}, \n) to ({-150+ \n/2+300/6-1+(\i-1)*(300-\n)/3+2},
{\n/2+0.7/\s+\r/(\s)-0.6});
}

\draw[thick, color=darkgreen] ({-150+\n/2+ (300-\n)/6+2},{\n/2+(0.7+2*\r)/\s}) to
(-150+\n,180-\n);

\draw[thick, color=darkgreen] ({-150+\n/2+ (300-\n)/6+2+1},{\n/2+(0.7+2*\r)/\s+1}) to
({-150+\n+(300-\n)/5},180-\n);

\draw[thick, color=darkgreen] ({-150+\n/2+ (300-\n)/6+(300-\n)/3+2+1},{\n/2+(0.7+2*\r)/\s+1}) to
({-150+\n+(300-\n)/5+(300-\n)/5},180-\n);

\draw[thick, color=darkgreen] ({-150+\n/2+ (300-\n)/6+(300-\n)/3+2+1},{\n/2+(0.7+2*\r)/\s+1}) to
({-150+\n+(300-\n)/5+2*(300-\n)/5},180-\n);

\draw[thick, color=darkgreen] ({-150+\n/2+ (300-\n)/6+2*(300-\n)/3+2+1},{\n/2+(0.7+2*\r)/\s+1}) to
({-150+\n+(300-\n)/5+3*(300-\n)/5},180-\n);

\draw[thick, color=darkgreen] ({-150+\n/2+ (300-\n)/6+2*(300-\n)/3+2+1},{\n/2+(0.7+2*\r)/\s+1}) to
(150-\n,180-\n);

\end{tikzpicture}
\caption{Mapping $\Gcqadpa$ into the subshift when $M$ halts.} \label{fig:glt}
\end{figure}
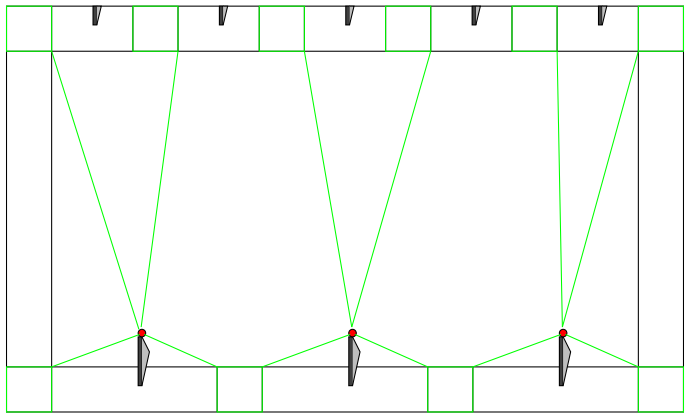

Note that the coloring we have defined so far respects the subshift $Y$, since each $\ell\times\ell$ pattern appearing in the coloring can be extended to a prototypical element of $4^{\Z^2}$.

Finally, it remains to color the ``interior'' parts of the tiles of $\gnpq$.
As these points are not identified with any other points in forming the quotient
$\gnpq$, we only need to make sure that the resulting coloring of each tile respects the subshift $Y$ separately.
We consider the long tile $\Gcqadpa$, the other cases being similar. Figure~\ref{fig:glt}
shows the coloring of the long tile which defines $g \res \Gcqadpa$.

In $\Gcqadpa$ we extend each copy of the codes for the entire computation from the
$R_d$-labeled regions along the bottom side into the interior. This is possible as we noted that $p,q>100N$,
whereas the width of the codes for each step of the computation is no more than $10N$.
Recall that at the top of each copy of the codes for the entire computation there is a $10\times 10$ rectangular region
consisting of color-$3$ points (shown by circles in Figure~\ref{fig:glt}). We run
zig-zag color-$2$ paths (shown in green in Figure~\ref{fig:glt}) as shown in the figure.
One end of each of these paths terminates in a corner point of an $R_\times$-labeled region,
which sets its parity. The other end terminates in one of the color-$3$ rectangular regions.
Thus, we maintain the subshift conditions required by $Y$. Note that since
$q=p+2$, there are at most $4$ color-$2$ paths needing to terminate in any given
color-$3$ region, and so the $10 \times 10$ size of the color-$3$ regions is
enough to accommodate them. Note that the color-$2$ paths isolate regions of different
parity in $\Gcqadpa$, and there is no problem filling in each of the isolated
regions by a checkerboard pattern of $0$s and $1$s.

The key point in the construction is that the color-$2$ paths, which are needed to
isolate regions of different parity in the tile, are free to terminate arbitrarily in the
special color-$3$ regions. From the construction we easily see that the map
$g \colon \gnpq \to 4$ defined above respects the subshift $Y$.

Next suppose that $M$ does not halt. From Theorem~\ref{thm:tilethm} it is enough to
fix arbitrary $n\geq \ell-1$ and $p, q>n$ and show that there does not exist $g \colon \gnpq \to 4$
which respects the subshift $Y$. Suppose that such a $g$ exists.
Restricting $G$ to the horizontal long tiles gives a $g'\colon H_{n,p,q}\to 4$
which respects the subshift, that is, $g'$ is a prototypical element in the
sense of Lemma~\ref{ttd}. From Lemma~\ref{ttd} we have that $3 \in \ran(g')$.
In particular, $g$ restricted to some horizontal long tile must attain the
color $3$.


Say $g \res \Gcqadpa$ takes the value $3$ (the other case is similar).
The rules of the subshift require that
a point of color $3$ occurs in a rectangular region of color $3$ points which is immediately
above a line representing a halting configuration of $M$. Consider the diagram shown in
Figure~\ref{fig:incon}.

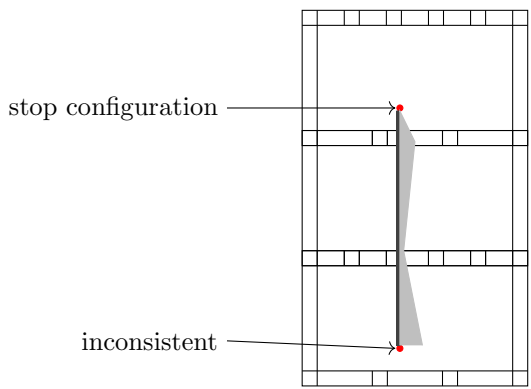
\begin{figure}[h]
\begin{tikzpicture}[scale=0.01]

\pgfmathsetmacro{\s}{0.03}

\pgfmathsetmacro{\n}{20}
\pgfmathsetmacro{\r}{0.05}

\tikzset{
pics/myshapec/.style={
code={
\draw (-150,0) rectangle (150,180);
\draw (-150,\n) to (150,\n);
\draw (-150,180-\n) to (150,180-\n);
\draw (-150+\n,0) to (-150+\n, 180);
\draw (150-\n,0) to (150-\n,180);

\foreach \i in {1,...,2}
{
\draw ({-150+\n/2+\i*(300-\n)/3-\n/2},0) to ({-150+\n/2+\i*(300-\n)/3-\n/2},\n);
\draw ({-150+\n/2+\i*(300-\n)/3+\n/2},0) to ({-150+\n/2+\i*(300-\n)/3+\n/2},\n);
}

\foreach \i in {1,...,4}
{
\draw ({-150+\n/2+\i*(300-\n)/5-\n/2},180-\n) to ({-150+\n/2+\i*(300-\n)/5-\n/2},180);
\draw ({-150+\n/2+\i*(300-\n)/5+\n/2},180-\n) to ({-150+\n/2+\i*(300-\n)/5+\n/2},180);
}

}}}

\tikzset{
pics/myshaped/.style={
code={
\draw (-150,0) rectangle (150,180);
\draw (-150,\n) to (150,\n);
\draw (-150,180-\n) to (150,180-\n);
\draw (-150+\n,0) to (-150+\n, 180);
\draw (150-\n,0) to (150-\n,180);

\foreach \i in {1,...,2}
{
\draw ({-150+\n/2+\i*(300-\n)/3-\n/2},180-\n) to ({-150+\n/2+\i*(300-\n)/3-\n/2},180);
\draw ({-150+\n/2+\i*(300-\n)/3+\n/2},180-\n) to ({-150+\n/2+\i*(300-\n)/3+\n/2},180);
}

\foreach \i in {1,...,4}
{
\draw ({-150+\n/2+\i*(300-\n)/5-\n/2},0) to ({-150+\n/2+\i*(300-\n)/5-\n/2},\n);
\draw ({-150+\n/2+\i*(300-\n)/5+\n/2},0) to ({-150+\n/2+\i*(300-\n)/5+\n/2},\n);
}

}}}

\pic[xscale=0.01,yscale=0.01] at (0,0) {myshapec};

\pic[xscale=0.01,yscale=0.01] at (0,-180+\n) {myshaped};

\pic[xscale=0.01,yscale=0.01] at (0,-360+2*\n) {myshapec};

\draw (-20, 50) circle (4);
\draw (-20, 50-2*180+2*\n) circle (4);

\draw[fill, color=darkgray]  (-20-4, 50-2*180+2*\n+4) rectangle (-20, 50-4);

\draw[fill, color=lightgray]  (-20, 50-2*180+2*\n+4) to  (-20, 50-4)
to (0, 5) to (-15, -140) to (10,50-2*180+2*\n+4);

\node[anchor=east] at (-250,50) {stop configuration};
\draw[->] (-250,50) to (-20-6,50);

\node[anchor=east] at (-250,50-2*180+2*\n+10) {inconsistent};
\draw[->]  (-250,50-2*180+2*\n+10) to  (-20-6,50-2*180+2*\n);

\end{tikzpicture}
\caption{Problem with having map into subshift when $M$ does not halt.} \label{fig:incon}
\end{figure}

This diagram consists of the tile $\Gcqadpa$ with a copy of
$\Gdpacqa$ below, and another copy of $\Gcqadpa$ below that.
Both copies of $\Gcqadpa$ have a rectangular grid of color $3$ points
at the same location, as shown in the figure. The rules of the subshift
require that every other line below the top region of color $3$
represents a valid tape configuration consisting of a starting cell and an ending cell. The codes for the starting cells
are vertically aligned as shown in the figure. Each of these lines is obtained from two lines below by an applicable instruction of the machine $M$. Since $M$ does not halt, it cannot be the case that as we move
down from the top region of color $3$ we reach a line corresponding to an initial configuration
(as this would give a valid computation of $M$ which halted). The rules of the subshift then imply
that two lines below each line correspondent to a tape configuration
is another line correspondent to a tape configuration, with starting and ending cells.
This, however, contradicts that the region of color $3$ points occurs in the
bottom copy of $\Gcqadpa$ in the same location, since it violates the subshift requirements.

This contradiction shows that when $M$ does not halt that there is no continuous,
equivariant map into the subshift $Y$, which completes the proof of
Theorem~\ref{thm:subsnr}.

\section{The graph homomorphism problem}  \label{sec:gh}

Note that the subshift problem (stated in \S\ref{sec:problems}) is equivalent to a problem regarding
the existence of $\Z^n$-homomorphisms from $F(2^{\Z^n})$ to a finite
$\Z^n$-graph. In fact, for each subshift of finite type $Y_{\vec{p}}$, where $\vec{p}=(\bsft; p_1,\dots, p_k)$, 
we can associate a finite $\Z^n$-graph $\Gamma$ as follows. For a sufficiently large positive integer $a$ (we can take 
any $a$ larger than the width of $Y$, see Definition~\ref{def:sft} and the preceding discussion), 
let $V(\Gamma)$ be the set of all patterns $p$ with $\dom(p)=[0,a)^n$ in which none of $p_1,\dots, p_k$ occurs. 
For a generator $s$ of $\Z^n$ and patterns $p, q\in V(\Gamma)$, let $(p, q)\in E(\Gamma)$ with label $s$ 
if $s\cdot p$ is compatible with $q$. This defines a $\Z^n$-graph $\Gamma$, and it is 
straightforward to check that there is a continuous equivariant map from  $F(2^{\Z^n})$ to $Y_{\vec{p}}$ 
if and only if there is a continuous $\Z^n$-graph homomorphism from $F(2^{\Z^n})$ to $\Gamma$. 
Thus Theorem~\ref{thm:subsnr} implies that the existence
of a continuous $\Z^2$-homomorphism from $F(2^{\Z^2})$ to a finite
$\Z^2$-graph is $\Sigma^0_1$-complete, and in particular
undecidable. 

In this subsection we prove a result which answers the
graph homomorphism problem, Problem~\ref{prob:gh}. As with the
subshift problem, Problem~\ref{prob:subs} and
Theorem~\ref{thm:subsnr}, we show that the graph homomorphism problem
is not decidable. This is done via an undecidable word problem for
finitely presented groups.

We view finite graphs as coded by integers in a standard manner.

\begin{thmn} \label{thm:ghnr}
The set of finite graphs $\Gamma$ such that there is a continuous graph homomorphism from
$\fzs$ to $\Gamma$ is a $\Sigma^0_1$-complete set.
\end{thmn}

The rest of this section is devoted to a proof of Theorem~\ref{thm:ghnr}.

We will define a computable reduction from a variation of the word problem for finitely presented groups to the graph homomorphism problem.
Specifically, we use the word problem for torsion-free groups, which is known to be undecidable
(see Lemma 1 of \cite{BaumslagMiller}). So, there is a $\Sigma^0_1$-complete set $A \subseteq \omega$
and a computable
function $f \colon \omega \to \omega$ such that for all $n$, $f(n)$ is an integer coding
\begin{itemize}
\item a finite group presentation $\CG_n=\langle a_1,\dots, a_k| r_1,\dots,r_l\rangle$
for a torsion-free group $G_n$, and
\item a distinguished word
$w=w(a_1,\dots,a_k)$
\end{itemize}
(in our notation we suppress the dependence of $k$, $l$, the $r_i$, and $w$ on $n$),
such that $n \in A$ if and only if $w$ is the identity $e$ in $G_n=\langle a_1,\dots, a_k| r_1,\dots,r_l \rangle$.
We will need one extra property of the groups $G_n$:
\begin{quote}
\noindent $(*)$\hskip 20pt There is (a lower bound) $\alpha>0$
such that, if the distinguished word $w$ is not the identity in $G_n$, then for all integers $m\geq 1$, $w^m$ is not equal in $G_n$ to any word of length at most $\alpha m$.
\end{quote}

The property $(*)$ for the groups $G_n$ is easy to arrange. If
$\CG=\langle a_1,\dots,a_k| r_1,\dots,r_l\rangle$ is a finite presentation
of a group $G$ with distinguished word $w=w(\vec a)$, let
\[
\tilde{\CG}
=\langle a_1,\dots,a_k, b_1,\dots,b_k| r_1(\vec a),\dots, r_l(\vec a),
r_1(\vec b),\dots, r_l(\vec b)\rangle,
\]
and let $\tilde{w}=w(\vec a) w(\vec b)$. $\tilde{G}$ is just the free product of $G$
with itself. If $w=e$ in $G$, then clearly $\tilde{w}=e$ in $\tilde{G}$.
If $w \neq e$ in $G$, we also have that $\tilde{w}^m=
w(\vec a) w(\vec b)\cdots w(\vec a)w(\vec b)$
is not equal to any word of length $< 2m$ by the normal form theorem for
free products (see Theorem~2.6 of Chapter IV of \cite{Lyndon}). So,
$(*)$ is satisfied with any $\alpha<2$. Note that if $G$ is torsion-free,
then so is $\tilde{G}$. So, without loss of generality we henceforth
assume that all of the groups $G_n$ are torsion-free and satisfy $(*)$.

The following technical theorem gives an outline for the rest of the proof.

\begin{thmn}\label{thm:tech}
Let $\CG=\langle a_1, \dots, a_k| r_1, \dots, r_l\rangle$ be a presentation of a torsion-free group $G$ and let $w=w(a_1, \dots, a_k)$ be a word. Suppose that the pair $(\CG, w)$ satisfies $(*)$. Then there is a  finite connected graph $\Gamma$ with the following properties:
\begin{enumerate}[label={\rm (\roman*)}]
\item \label{techi} $\pi^*(\Gamma)\cong \langle a_1, \dots, a_k, z| r_1, \dots, r_l, z^2w^{-1}\rangle$, where $z$ is a new letter distinct from $a_1, \dots, a_k$;
\item \label{techii} if $w=e$ in $G$, then $\Gamma$ satisfies the conditions of Theorem~\ref{thm:odd-length-cycle}, and hence there is a continuous homomorphism from $F(2^{\Z^2})$ to $\Gamma$;
\item \label{techiii} if $w\neq e$ in $G$, then $\Gamma$ satisfies the conditions of Theorem~\ref{thm:neghom}, and hence there is no continuous homomorphism from $F(2^{\Z^2})$ to $\Gamma$.
\end{enumerate}
Furthermore, one can construct $\Gamma$ from $\CG$ using a computable procedure.
\end{thmn}

\subsection{Realizing a given reduced homotopy group}\label{subsec:rhg}
In \ref{techi} of Theorem~\ref{thm:tech} we need to construct a finite connected graph with a prescribed reduced homotopy group. In this subsection we give the procedure to do this as well as make some observations about the procedure that we need to use later in the proof. Our main objective is to show the following general fact.

\begin{lemn} \label{prestograph}
Let $\CG= \langle b_1,\dots, b_k|s_1,\dots, s_l\rangle$ be a finite presentation for a group $G$.
Then there is a finite connected graph $\Gamma(\CG)$ such that $\pi^*_1(\Gamma(\CG)) \cong G$. Moreover, the map
$\CG \mapsto \Gamma(\CG)$ is computable.
\end{lemn}

\begin{proof}
This fact can be seen as an instance of a general result about CW complexes associated
to group presentations, as for example in Proposition~2.3 of \cite{Lyndon}.
For the sake of completeness, and since we need to record a few extra points, we
briefly sketch a proof.

The graph $\Gamma(\CG)$ is constructed as follows.
$\Gamma(\CG)$ will have a distinguished vertex  $v_0$.
For each of the generators $b_i$, $1 \leq i \leq k$,
we add a cycle $\beta_i$ of some length $\lambda_i>20$ that starts and ends at the vertex $v_0$. Other than the common vertex $v_0$, we make the vertex sets of these
cycles pairwise disjoint. This gives a natural notion of length $\lambda(b_i)=\lambda_i$
which extends in the obvious manner to reduced words in the free group generated by the $b_i$.
For each word $s_j, 1 \leq j \leq l$, from the presentation $\CG$, we add to $\Gamma(\CG)$
a rectangular grid graph $R_j$ whose length and width are both greater than $4$ and whose perimeter $\per(R_j)$ is equal to $\lambda(s_j)$. This is possible because each $\lambda(b_i)$ is a large enough even number.
The vertices in the interiors of these grid graphs $R_j$ are pairwise disjoint, and
are disjoint from the vertices of $\beta_i$. The vertices on the sides of each $R_j$ are assigned so that $v_0$ occurs as the upper-left corner of each $R_j$ and, going clockwise, the boundary of $R_j$ coincides with the concatenation of the paths corresponding to the generators in the word $s_j$. This completes the definition of the graph $\Gamma(\CG)$.

It is obvious from this definition that $\Gamma(\CG)$ is a finite connected graph and that the map $\CG\mapsto \Gamma(\CG)$ is computable.


Next we verify that $\pi_1^*(\Gamma(\CG))\cong G$. For this we first introduce a ``normal form" for each cycle in $\Gamma(\CG)$ that starts and ends at $v_0$. Consider a cycle $\chi$ in $\Gamma(\CG)$ starting and ending at $v_0$.
Say $\chi=(v_0, v_1, \dots,v_M)$, where $v_0=v_M$ and the $v_m$ are vertices of $\Gamma(\CG)$. Consider a segment $\sigma=(v_{m_0}, \dots,v_{m_1})$
of this path where $v_{m_0}$, $v_{m_1}$ are vertices in the union of the $\beta_i$, and
for $m_0<m<m_1$, $v_m$ is an interior vertex of some $R_j$. Note that the value of $j$ is uniquely determined
by $\sigma$. Let $\bar{\sigma}$ be the result of replacing in $\sigma$ the segment $(v_{m_0+1},\dots,v_{m_1-1})$
with a path along the perimeter of $R_j$ (say, choose the shortest path, and if there are two of equal length,
choose one randomly). Doing this on each such segment  $\sigma$ of $\chi$ gives a new path $\bar{\chi}$.
Clearly $\chi$ is reduced homotopy equivalent to $\bar{\chi}$ in $\pi_1^*(\Gamma(\CG))$. Note that $\bar{\chi}$ corresponds to a word in the $b_i$, which we still denote by $\bar{\chi}$.
It is not difficult to check that if $\chi$ differs from another cycle $\psi$ by the insertion/deletion of a $4$-cycle in
$\Gamma(\CG)$, then $\bar{\chi}$, as a word, is equivalent to $\bar{\psi}$ in $G$. We use here the fact that if all
$\lambda_i>20$ and each $R_j$ has width and height greater than $4$,
then the only $4$-cycles in $\Gamma(\CG)$ are the ones coming from individual $R_j$. It follows that, if $\chi$ is reduced homotopy equivalent to $e$ in $\pi_1^*(\Gamma(\CG))$, then $\bar{\chi}$, as a word in $G$, is either trivial or a combination of $s_j$.

We are now ready to argue that $\pi_1^*(\Gamma(\CG))\cong G$. If $F$
is the free group generated by the $b_i$, then there is a natural
homomorphism $\varphi$ from $F$ onto $\pi_1^*(\Gamma(\CG))$ obtained
by sending $b_i$ to the equivalence class $[\beta_i]$ of the cycle
$\beta_i$ in $\pi_1^*(\Gamma(\CG))$. Clearly each $s_j$ is in the
kernel of $\varphi$, and so $\varphi$ induces a homomorphism
$\varphi'$ from $G$ onto $\pi_1^*(\Gamma(\CG))$.  To see that the
kernel of $\varphi'$ is trivial, it suffices to show that the kernel
of $\varphi$ is generated by $s_j$ in $F$.  Suppose $u$ is a word in
the $b_i$ and $\varphi(u)=e$. Since $u$ gives a path $\chi$ starting
and ending at $v_0$ in $\Gamma(\CG)$, $\varphi(u)=e$ means that $\chi$
is reduced homotopy equivalent to the trivial path in
$\pi_1^*(\Gamma(\CG))$. By the argument in the preceding paragraph,
this implies that $\bar{\chi}$ is either trivial or a combination of
$s_j$. However, since $\chi$ does not contain any interior vertices of
$R_j$, we have $\chi=\bar{\chi}$. This means that $u$ is in the
subgroup generated by $s_j$. This completes the proof of
Lemma~\ref{prestograph}.
\end{proof}

Applying Lemma~\ref{prestograph} to the presentation demonstrated in
\ref{techi} of Theorem~\ref{thm:tech} would establish that
clause. However, in order to continue with our proof we need to use some
details of the construction of $\Gamma(\CG)$ as well as some
additional properties. We specify these additional properties below.

Given a presentation $\CG_n=\langle a_1, \dots, a_k| r_1, \dots,
r_l\rangle$ for a group $G_n$, let $G'_n$ be the finitely presented
group with presentation given by
\[
\CG'_n=\langle a_1,\dots,a_k, z| r_1,\dots,r_l, z^2w^{-1}\rangle,
\]
that is, we add the extra generator $z$ and the extra relation $z^2=w$
to the presentation for $G_n$. For notational convenience, we denote
$a_0=z$ and $r_0=z^2w^{-1}$.

Since every word in the normal subgroup of $G'_n$ generated by the
$r_0,r_1,\dots,r_\ell$ has an even number of occurrences of $z$,
the parity of the number of occurrences of $z$ in an element of $G'_n$
is well-defined. We accordingly refer to an element $g \in G'_n$ as being
either odd or even.

When associating a finite connected graph $\Gamma(\CG'_n)$ to the
presentation $\CG'_n$, we deviate slightly from the construction in
the above proof. The deviation is that, when assigning a cycle
$\beta_0$ to the generator $z$, we let $\lambda_0=\lambda(z)>10$ be an
{\em odd} number and ensure that the length of $\beta_0$ is
$\lambda_0$. One can easily check that this does not invalidate the
proof as the only occurrences of $z$ in the relations are in
$r_0=z^2w^{-1}$, and in order to carry out the proof we only needed
all $\lambda(r_j)$ to be sufficiently large and even.

In addition, we note that there is no loss of generality to require
that the $\lambda_i$ are large and that the aspect ratios of $R_j$ are
less than $2$. This latter property can be achieved by choosing the
width of $R_j$, denoted $\text{\tt w}(R_j)$, and the height of $R_j$,
denoted $\text{\tt h}(R_j)$ to be close to $\frac{1}{4} \per(R_j)$,
subject to the requirement that $\text{\tt w}(R_j)+\text{\tt
  h}(R_j)=\frac{1}{2}\per(R_j)$, which is easy to arrange. Both these
requirements have no effect on the validity of the above proof.

Let the length of the word $w$ in $G_n$ be $L$. Let $\alpha$ be the constant given by property $(*)$ of $G_n$.
We henceforth assume that $\Gamma(\CG'_n)$ has all the following
properties.
\begin{enumerate}
\item
$\lambda_i=\lambda(a_i)>\max(20, \frac{12}{\alpha})$ are even for all $1\leq i\leq k$.
\item
$\lambda_0=\lambda(z)>\max(10, \frac{12}{\alpha})$ is odd.
\item
For each $0\leq j\leq l$, we have $\text{\tt w}(R_j), \text{\tt h}(R_j)>4$.
\item
For each $0\leq j\leq l$, the perimeter length $\per(R_j)$ is equal to $\lambda(r_j)$.
\item
For each $0\leq j\leq l$, we have $\frac{1}{2}\leq \frac{\text{\tt w}(R_j)}{\text{\tt h}(R_j)} \leq 2$.
\end{enumerate}

Condition (5) gives the following lemma which we will need in the argument later. For any cycle $\chi$ in $\Gamma(\CG'_n)$ let $|\chi|$ denote the length of $\chi$.

\begin{lemn} \label{lm} For any cycle $\chi$ in $\Gamma(\CG'_n)$, there is a cycle $\bar{\chi}$ in
$\Gamma(\CG'_n)$ such that
\begin{enumerate}
\item[\rm (a)] $\bar{\chi}$ is in the subgroup of $\pi_1(\Gamma(\CG'_n))$ generated by the cycles corresponding to $a_1,\dots, a_k$ and $z$;
\item[\rm (b)] $\bar{\chi}$ is equivalent to $\chi$ in $\pi^*_1(\Gamma(\CG'_n))$; and
\item[\rm(c)] $|\bar{\chi}|\leq 3|\chi|$.
\end{enumerate}
\end{lemn}

\begin{proof}
As in the proof of Lemma~\ref{prestograph}, each segment of the cycle $\chi$ with two endpoints on the boundary of
a grid graph $R_j$ can be replaced
by a path which stays on the boundary of $R_j$. Replacing all such segments of $\chi$
in this manner gives $\bar{\chi}$. If $\rho$ is the maximum aspect ratio of all $R_j$, then it is easily seen that $|\bar{\chi}|\leq (\rho+1)|\chi|$. By condition (5) above, $\rho\leq 2$ and hence $|\bar{\chi}|\leq 3|\chi|$.
\end{proof}

\subsection{Standard forms of elements of $G'_n$}

In this subsection we give an algebraic analysis of various standard forms of elements of $G'_n$. The standard forms will be instrumental in the proof of Theorem~\ref{thm:tech} in the next subsection.

The group $G'_n$ can be described as the amalgamated free product $$G'_n
=G_n *_{H,K} \Z$$ where $H=\langle w \rangle \leq G_n$ and
$K=\langle z^2 \rangle \leq \langle z\rangle \cong \Z$. Note that $H\cong K\cong \Z$ as $G_n$ is torsion-free.

\begin{defnn} \label{reducedform}
A {\em reduced form} for an element $g\in G'_n$ is a word of the form
$$ g=g_1\dots g_m $$
where for any $1\leq i<m$, either $g_i=z$ and $g_{i+1}\in G_n$ or $g_i\in G_n$ and $g_{i+1}=z$, and for any $1<i\leq m$, if $g_i\in G_n$ then $g_i\not\in H$.
\end{defnn}

Note that a reduced form may start with a power of $w$. In general, a {\em word} in $G'_n$ is a sequence $g_1\dots g_m$, where each {\em term} $g_i$ is either an element of $G_n$ or an element of $\langle z\rangle$. For any word in $G_n'$ that is not in reduced form, we can apply the following procedure, and repeat it if necessary, to arrive at a reduced form. First, if the word contains any term that is an odd power of $z$, say $z^{2k+1}$ with $k\neq 0$, we replace the term by two terms $w^{k}z$. Then, if it contains any term that is an even power of $z$, say $z^{2k}$, we replace the term by the term $w^k$. (Note that after these two steps, in the resulting word every term  that contains $z$ is just $z$ itself.) Next, if the resulting word contains adjacent terms that are both in $G_n$, we obtain a shorter word by combining the two terms. If the resulting word contains adjacent terms that are both $z$, we replace the two terms by a single $w$. If the resulting word contains three consecutive terms of the form $z w^k z$, we replace them by a single $w^{k+1}$. At any moment a term is obtained which is the identity of $G_n$, we delete the term. These steps are repeated as much as necessary. (Note that these steps only reduce the length of the word, and therefore the process must terminate.) If the last two terms of the resulting word are of the form $zw^k$, we replace them by $w^kz$ and repeat the length-reducing steps again as much as necessary. It is easy to see that, by following the procedure, we will eventually arrive at a reduced form for the original element. Also note that the number of occurrences of $z$ in the original word (including those appearing in either the positive or negative powers of $z$) does not increase through this reduction procedure. No power of $z$ other than $z$ itself appears as a term in a reduced form.

We next define a normal form for elements of $G'_n$. Let $U$ be a set of non-identity right coset representatives for the subgroup $H=\langle w\rangle$ of $G_n$.
Clearly $z$ is the unique coset representative for $K$ in $\Z$.

\begin{defnn} \label{normalform}
The {\em normal form} of an element $g\in G'_n$ is one of the following
\begin{align}
g&= w^ku_1zu_2z  \cdots z u_m  \label{c1}\\
g&= w^k zu_1z u_2 \cdots z u_m \label{c2}\\
g&= w^k u_1 z u_2 z \cdots u_m z \label{c3}\\
g&= w^k z u_1 z u_2 \cdots u_m z \label{c4}\\
g&= w^kz \label{c5} \\
g&= w^k \label{c6}
\end{align}
where $m\geq 1$, $u_i \in U$ and $k\in\Z$.
\end{defnn}

The normal form theorem for amalgamated free products (see Theorem~2.6 of \cite{Lyndon})
implies that every element $g$ of $G'_n$ can be written uniquely in a
normal form specified in the above definition. There is also a natural
procedure to turn a reduced form into a normal form. For the
convenience of the reader we describe this procedure below. Suppose
$$ g=g_1\dots g_m$$ is in reduced form. The procedure is defined
inductively on the length $m$ of the reduced form for $g$. If $m=1$
then either $g=z$, which is already in normal form (5), or $g\in G_n$,
in which case either $g\in H$, which is in normal form (6), or $g=w^k
u_1$ for some $u_1\in U$, which is in normal form (1). Next assume
$m=2$. Then either $g=w^k z$, or $g=g_1 z$ with $g_1\in G_n\setminus
H$, or $g=z g_2$ with $g_2\in G_n\setminus H$. The first case is
already in normal form (5). In the second case, write $g_1=w^k u_1$
for $u_1\in U$, and $g=w^k u_1 z$ is in normal form (3). In the third
case, write $g_2=w^k u_1$, and $g= w^k z u_1$ is in normal form
(2). In general, we assume $m>2$. Then either $g_m=z$ or $g_m\in
G_n\setminus H$. In case $g_m=z$ we write $g=g'z$ where $g'=g_1\dots
g_{m-1}$. Note that $g'$ is in reduced form but of length $m-1$. We
then inductively obtain the normal form of $g'$, and the normal form
of $g$ is the normal form of $g'$ followed by $z$. Suppose $g_m\in
G_n\setminus H$. Let $g_m=w^k u_m$ for $u_m\in U$, and
$$ g'=g_1\dots g_{m-3}(g_{m-2}w^k). $$ Then $g'$ is in reduced form
and of length $m-2$. By induction we can obtain the normal form of
$g'$. Then the normal form of $g$ is the normal form of $g'$ followed
by $z u_m$. This finishes the formal definition of the
procedure. Informally, given a reduced form, we obtain the normal form
by successively ``passing" any power of $w$ to the left, leaving
behind the coset representatives in place of general elements of
$G_n\setminus H$.

We next define another standard form for elements of $G'_n$.

\begin{defnn} \label{mc}
Let $g\in G'_n$. Among all representations of $g$ as a product of the form
$g= x y x^{-1}$ where $y$ is in normal form, we let $i(g)$ be the minimum number of
occurrences of $z$ in the normal form $y$. A {\em minimal conjugate form} of $g$ is
$$ g=xyx^{-1}$$
where $x, y\in G'_n$ and $y$ is in normal form with exactly $i(g)$ many occurrences of $z$.
\end{defnn}

The following fact is immediate from Definition~\ref{mc}.

\begin{fctn} \label{mcvs}
Let $g \in G'_n$ and $g=xyx^{-1}$ be in minimal conjugate form.
Then the normal form of $g$ has
at least $i(g)$ many occurrences of $z$.
\end{fctn}

\begin{proof}
Since $g=eg'e^{-1}$, where $g'$ is the normal form for $g$, it follows
that the number of occurrences of $z$ in $g'$ is at least $i(g)$.
\end{proof}

Note that if $g=xyx^{-1}$ is in minimal conjugate form, then $g$ is odd if and only if $y$ is odd.
In particular if $g$ is odd then $i(g)>0$.

\begin{lemn} \label{lif}
Let $g \in G'_n$ and $g=xyx^{-1}$ be in minimal conjugate form.
Suppose $y \notin \langle z \rangle$. Then for any $N>0$ we have
that $g^N=x y^N x^{-1}$ where the normal form for $y^N$ has exactly $N\cdot i(g)$ many occurrences of $z$.
\end{lemn}

\begin{proof}
Since $y\notin\langle z\rangle$, $y$ is not of normal form~$(\ref{c5})$ or $(\ref{c6})$.
Note that $y$ cannot be of normal form~(\ref{c4}) as then
$y= z^{-1} (w^kz^2) (u_1z \cdots u_m)z$ and
$$g= (xz^{-1}) (w^{k+1}u_1 z \cdots u_m) (zx^{-1}).$$
This gives a conjugate form for $g$ with a smaller number of occurrences of $z$
in the normal form of the middle element, which violates the
fact that $y$ attained the minimum value $i(g)$.

Assume next that $y$ is in normal form either~(\ref{c2}) or (\ref{c3}). The cases are similar,
so assume $y$ is as in normal form~(\ref{c2}), namely
$y=w^k z u_1 z u_2 \cdots u_m$. Then
$g^N= x (w^k z u_1 \cdots u_m)^N x^{-1}$. Passing each of the $w^k$ terms to the left
to put $(w^k z u_1 \cdots u_m)^N$ in normal form gives
\[
g^N= x (w^{\ell} z u'_1 z u'_2 \cdots u'_m z u'_{m+1} z \cdots u'_{2m} \cdots z u'_{(N-1)m+1} z
\cdots u'_{Nm}) x^{-1}
\]
for some $\ell\in\Z$. Here the coset representatives in $U$ have changed (compared to the
concatenation $(w^ku_1z\dots u_m)^N$), but the number of them is still equal to $Nm$.
So, in this case the conclusion of the lemma is immediate.

Finally, assume $y$ is in normal form~(\ref{c1}). Consider $y^N=(w^k u_1z \cdots u_m)^N$.
To show the normal form for $y^N$ has $Nm$ many coset terms, it suffices to show that
$u_m w^k u_1 \notin H$. Suppose $u_m w^k u_1 = h\in H$. Then
\[
\begin{array}{rcl}y&=& (w^ku_1)(z u_2 \cdots u_{m-1}z)(h)(w^ku_1)^{-1} \\
&=& (w^ku_1z)(u_2 \cdots u_{m-1})(hw)(w^ku_1z)^{-1} \\
&=& (w^ku_1z)(w^\ell u'_2 \cdots u'_{m-1})(w^ku_1z)^{-1}
\end{array}
\]
for some $\ell\in \Z$. Thus
$$ g=(xw^ku_1z)(w^\ell u'_2\cdots u'_{m-1})(xw^ku_1z)^{-1} $$
and the normal form for the middle element contains fewer occurrences of $z$ than $y$.
This violates the choice of $y$.
\end{proof}

\subsection{Proof of Theorem~\ref{thm:tech}} In this subsection we prove \ref{techii} and \ref{techiii} of Theorem~\ref{thm:tech}, and thus complete the proof of Theorem~\ref{thm:ghnr}. Recall from \S\ref{subsec:rhg} that we have obtained a finite connected graph $\Gamma(\CG'_n)$ via a computable procedure so that $\pi_1^*(\Gamma(\CG'_n))\cong G'_n$.

To show \ref{techii} of Theorem~\ref{thm:tech}, suppose $n \in A$. Then $w=e$ in $G'_n$. Recall from condition (2) in \S\ref{subsec:rhg} that the generator $z$ in $G'_n$ corresponds to an odd-length cycle $\zeta$ in $\Gamma(\CG'_n)$.
However $z^2=w$ in $G'_n$, and so $z^2=e$ in $G'_n$. Thus $\zeta$ is an odd-length cycle with $\pi_1^*(\zeta)$ having finite order in $\pi_1^*(\Gamma(\CG'_n))$, which
shows $\Gamma(\CG'_n)$ satisfies the positive condition of Theorem~\ref{thm:odd-length-cycle}.

It remains to verify \ref{techiii} of Theorem~\ref{thm:tech}. For the rest of the argument, assume $n \notin A$ and thus $w\neq e$ in $G'_n$. We show that
$\pi^*_1(\Gamma(\CG'_n))\cong G'_n$ satisfies the negative condition of Theorem~\ref{thm:neghom}. For this, let $p$, $q$ be large odd primes. Toward a contradiction, suppose $\gamma$ is a $p$-cycle in $\Gamma(\CG'_n)$ and assume
$\pi_1^*(\gamma^q)$ is a $p^\text{th}$ power in $\pi^*_1(\Gamma(\CG'_n))$.

To ease notation we use the isomorphism between $\pi^*_1(\Gamma(\CG'_n))$ and $G'_n$ tacitly and, instead of writing $\pi_1^*(\gamma)$, we simply write $\gamma$ as an element of $\pi_1^*(\Gamma(\CG'_n))$. Thus, suppose $\gamma^q=\delta^p$ in $G'_n$. In the group $\G'_n$ express $\gamma$ and $\delta$ in their respective minimal conjugate forms, say $\gamma= x v x^{-1}$ and
$\delta= y u y^{-1}$, where $i(\gamma)$, $i(\delta)$ are the number of occurrences
of $z$ in $v$, $u$, respectively. For slight notational simplicity let $i_v=i(\gamma)$ and $i_u=i(\delta)$.

We claim that $i_v, i_u\neq 0$. Assume $i_v=0$. Then $v\in G_n$, the subgroup of $G'_n$ generated by $a_1, \dots, a_k$. By condition (1) of \S\ref{subsec:rhg}, in $\Gamma(\CG'_n)$ the cycle corresponding to $xvx^{-1}$ has even length. But the lengths of reduced homotopy equivalent cycles differ by an even number, thus $\gamma$ has even length, contradicting that $p$ is odd. Since $p$, $q$ are both odd, $\delta$ is also an odd-length cycle in $\Gamma(\CG'_n)$. By a similar argument, $i_u \neq 0$.

Continuing our proof, assume first that $v \notin \langle z\rangle$. We have
\[
\gamma^q= xv^q x^{-1}= y u^p y^{-1}= \delta^p
\]
in $\G'_n$. By Lemma~\ref{lif}, the normal form of $v^q$ has exactly
$qi_v$ many occurrences of $z$.  We first claim that $u\notin \langle
z\rangle$ as well. If $u=z^a$, then we would have that $xv^{qN}
x^{-1}=yz^{apN} y^{-1}= y w^{\frac{apN}{2}} y^{-1}$ for all even
$N$. Thus, $v^{qN}=x^{-1}y w^{\frac{apN}{2}} y^{-1} x$ for all even
$N$. From Lemma~\ref{lif}, the normal form of $v^{qN}$ has exactly
$qNi_v$ many occurrences of $z$. However, we may put $x^{-1}y
w^{\frac{apN}{2}} y^{-1} x$ into reduced form which involves no more
occurrences of $z$ than those in the reduced forms of $x^{-1}y$ plus
those in $y^{-1}x$, which is a constant that does not depend on
$N$. It follows that the normal form of $x^{-1}y w^{\frac{apN}{2}}
y^{-1} x$ has a constant number of occurrences of $z$ that does not
depend on $N$. Because the normal form for an element of $G'_n$ is
unique, we have a contradiction for large enough even $N$. So, $u
\notin \langle z\rangle$. By Lemma~\ref{lif} we also have that the
normal form of $u^p$ has $pi_u$ many occurrences of $z$.

Now, for any $N$ we have $x v^{qN}x^{-1}= y u^{pN}y^{-1}$, where the
normal forms of $v^{qN}$ and $u^{pN}$ have $ qNi_v$ and $ pNi_u$
occurrences of $z$ respectively.  Next we claim that $p | i_v$.
Otherwise, $p$ is not a divisor of $qi_v$ and so $qi_v \neq
pi_u$. Thus $| qNi_v-pNi_u| \geq N$. Without loss of generality assume
$qNi_v-pNi_u\geq N$. Then $v^{qN}=x^{-1}y u^{pN} y^{-1} x$.  Comparing
the number of occurrences of $z$ for the normal forms of both sides,
we get a contradiction when $N$ is large. So, we have $p| i_v$.

Next, we let $\bar{\gamma}$ be a cycle in $\Gamma(\CG'_n)$ given by
Lemma~\ref{lm} for the $p$-cycle $\gamma$. By Lemma~\ref{lm} (c) we have
$|\bar{\gamma}|\leq 3|\gamma|=3p$. By Lemma~\ref{lm} (b),
$\bar{\gamma}=\gamma=xvx^{-1}$ as elements of $G'_n$. By
Fact~\ref{mcvs}, the canonical form of $\bar{\gamma}$ has at least
$i_v$ many occurrences of $z$. Since $p|i_v$, we conclude that the
canonical form of $\bar{\gamma}$ has at least $p$ many occurrences of
$z$. Finally, by Lemma~\ref{lm} (a), $\bar{\gamma}$ is a cycle that is
a combination of cycles corresponding to the generators $a_1,\dots,
a_k, z$, which in particular gives a representation of $\bar{\gamma}$
in $G'_n$ as a product of the generators $a_1,\dots, a_k, z$. For
definiteness suppose
\begin{equation}\label{cyclerep} \bar{\gamma}= g_1\dots g_m
\end{equation}
is this representation. Since the normal form of $\bar{\gamma}$
contains at least $p$ many occurrences of $z$, this representation
also contains at least $p$ many occurrences of $z$, because the
procedures to produce the normal form from any representation does not
increase the number of occurrences of $z$. Thus, as a cycle in
$\Gamma(\CG'_n)$, $\bar{\gamma}$ contains at least $p$ many
occurrences of the cycle corresponding to $z$. It follows that
$|\bar{\gamma}|\geq p\lambda_0$ where $\lambda_0=\lambda(z)>10$. This contradicts
$|\bar{\gamma}|\leq 3p$.

This finishes the proof for the case $v\not\in\langle z\rangle$. Note that the above argument to show $u\not\in\langle z\rangle$ from assuming $v\not\in\langle z\rangle$ is symmetric for $u$ and $v$. Thus the remaining case is in fact $u, v\in\langle z\rangle$.

Suppose $u=z^a$ and $v=z^b$. Thus $\gamma=x z^a x^{-1}$ and $\delta=y z^b y^{-1}$. Note that $a$ is odd, since otherwise the cycle corresponding to $xz^ax^{-1}$ in $\Gamma(\CG'_n)$ would have even length and therefore not reduced homotopy equivalent to the odd-length cycle $\gamma$. Similarly, $b$ is also odd.
So,
\[
\gamma^q=x w^{\frac{aq-1}{2}} z x^{-1} = y w^{\frac{bp-1}{2}} z y^{-1}=\delta^p.
\]

\begin{lemn} \label{tfg}
For any integers $c, d$ and $x, y\in G'_n$, if $x w^c z x^{-1}= y w^d z y^{-1}$
then $c=d$ and $x^{-1}y\in \langle z \rangle$.
\end{lemn}

\begin{proof} By conjugating both sides of the identity by $x$, the lemma is reduced to the case $x=e$ in $G'_n$, which we prove below.
Note that $y\in \langle z \rangle$ implies $c=d$, and so the conclusion of the lemma holds.
Suppose the conclusion fails. Let $y\not\in\langle z\rangle$ be of minimal length in normal form
such that for some $c, d$ we have
\[
w^c z = y w^d z y^{-1}.
\]

We claim $y$ cannot end with the $z$ term.
To see this, let $y=y'z$ be in normal form
(so $y'$ ends with a $\G_n$ term), then
$$ w^c z = y' z w^d z z^{-1} {y'}^{-1}= y' w^d z {y'}^{-1},$$
and $y'=yz^{-1}\not\in\langle z\rangle$. This violates the minimality assumption on the length of $y$.
So, $y$ ends with a $\G_n\setminus H$ term.
It follows that the normal form
of $y w^d z y^{-1}$ obtained from the procedure starting with a reduced form must contain a $\G_n\setminus H$ term (that is, some $u\in U$). But $w^cz$ is already in its normal form. This violates the uniqueness of the normal forms guaranteed by the normal form theorem.
\end{proof}

From the lemma we have that $aq=bp$ and $x^{-1}y\in \langle z \rangle$. In
particular, $p | a$. Let $a=pc$ for some odd integer $c$. Then
$aq=pcq=bp$. In $G'_n$ we have
$$\gamma=xz^{pc}x^{-1}=x(x^{-1}y) z^{pc} (y^{-1}x)x^{-1}=yz^{pc}y^{-1} $$ and $\delta=y z^{qc}
y^{-1}$. Consider again a cycle $\bar{\gamma}$ in $\Gamma(\CG'_n)$ given by
Lemma~\ref{lm} for the $p$-cycle $\gamma$. By Lemma~\ref{lm} (c), $|\bar{\gamma}|\leq
3|\gamma|=3p$. By Lemma~\ref{lm} (a), $\bar{\gamma}$ is a combination
of the cycles corresponding to the generators $a_1, \dots, a_k,
z$. This gives a representation of $\bar{\gamma}$ in $G'_n$ as a
product $g_1\dots g_m$. Let $\ell$ be the minimum of $\lambda_0=\lambda(z)$ and
$\lambda(a_i)$ for all $1\leq i\leq k$. By conditions (1) and (2)
in \S\ref{subsec:rhg}, $\ell>\frac{12}{\alpha}$. We have that
$$3p\geq |\bar{\gamma}|\geq m\ell>m\frac{12}{\alpha}. $$ Thus $m\leq
\frac{\alpha p}{4}$.

For any integer $N$ we have $y z^{2pcN} y^{-1}=\gamma^{2N}=\bar{\gamma}^{2N}$, and so
$$w^{pcN}=z^{2pcN}=y^{-1}\bar{\gamma}^{2N}y. $$
Let $Y$ be the length of $y$ as a word in the $a_i$ and $z$.
Then as a word in the $a_i$ and $z$, the length of the word
$y^{-1}\bar{\gamma}^{2N}y$ is at most $2Y+N\frac{\alpha p}{2}$.
On the other hand, from the definition of $\alpha$ we have that any
representation of $w^{pcN}$ in the $a_i$ and $z$ must have length
at least $\alpha pcN \geq \alpha pN$. For $N$ large enough, this is a contradiction.

Thus, in all cases we have a contradiction from assuming that there is a
$p$-cycle in $\Gamma(\G'_n)$ such that $\gamma^q$ is a $p^\text{th}$ power in
$\pi^*_1(\Gamma(\G'_n))$, assuming $p$, $q$ are sufficiently large odd primes.
So, $\Gamma(\G'_n)$ satisfies the negative condition of Theorem~\ref{thm:neghom}.
This completes the proof of Theorems~\ref{thm:tech} and~\ref{thm:ghnr}.

\section{The tiling problem} \label{sec:tiling}

Recall from Definition~\ref{def:tiling} that a continuous tiling of $\fzs$
by tiles $T_i \subseteq \Z^2$
is a clopen subequivalence relation $E$ of $\fzs$ such that every $E$
class is isomorphic to one of the tiles $T_i$. We consider the continuous tiling
problem for $\fzs$ with a given finite set of finite tiles $T_1,\dots,T_n$.
In \S\ref{sec:problems} we stated the general tiling problem, which asks for which
sets of tiles $\{ T_i\}$ there is a continuous (or Borel) tiling of $\fzs$.
Of course, this problem could be asked for more general group actions, but we confine
our remarks here to continuous tilings of $\fzs$ by finite sets of finite tiles.

As we pointed out in \S\ref{sec:problems}, the continuous tiling problem for
finite sets of finite tiles (which can be coded as integers) is a $\Sigma^0_1$
set. That is, the set $A$ of integers $n$ coding a finite set of finite tiles for which there
is a continuous tiling of $\fzs$ is a $\Sigma^0_1$ set of integers
(this was Theorem~\ref{thm:sr}). However, unlike the subshift and graph-homomorphism
problems for which we know the corresponding sets are $\Sigma^0_1$-complete
(Theorems~\ref{thm:subsnr} and \ref{thm:ghnr}), we do not know if the
tiling problem is $\Sigma^0_1$-complete (i.e., the whether the set $A$ is
$\Sigma^0_1$-complete). This question seems to be difficult even for specific
rather simple sets of tiles. We restate the tiling problem:

\begin{probn} \label{prob:tiling}
Is the set of integers $n$ coding a finite set of finite tiles
\[
T_1(n),\dots,T_i(n)\subseteq \Z^2
\]
($i$ also depends on $n$)
for which there is a continuous tiling of $\fzs$ a $\Sigma^0_1$-complete set?
\end{probn}

In Theorem~\ref{thm:negthreetiles} and the discussion before it  we showed that
for the four rectangular tiles of dimensions $d\times d$, $d \times (d+1)$,
$(d+1)\times d$, and $(d+1)\times (d+1)$, a continuous tiling of $\fzs$ was possible, but
if either the ``small" tile $T_s$ of dimensions $d \times d$, or the ``large''
tile $T_\ell$ of dimensions $(d+1)\times (d+1)$ is omitted from the set, then there is
no continuous tiling of $\fzs$ by the set of three remaining tiles. This argument
(say in the case of omitting the large tile) proceeded by showing that, for
$p$ relatively prime to $d$, the $p \times p$ torus tile $\Tcaca$
could not be tiled by these three tiles. This argument does not work
if both the small tile and the large tile are present but one of the other tiles are omitted.
In particular, we can ask:

\begin{probn} \label{prob:sl}
Is there a continuous tiling of $\fzs$ by the tiles $\{ T_s,T_\ell\}$,
where $T_s$ is the $d \times d$ rectangle, and $T_\ell$ is the $(d+1)\times (d+1)$
rectangle (for $d \geq 2$)?
\end{probn}

It seems possible at the moment that the answer to Problem~\ref{prob:sl}
could depend on the value of $d$. We currently do not know the answer for any
$d\geq 2$.

To illustrate the nature of the problem, consider the $d=2$ instance of
Problem~\ref{prob:sl}. Here we have the $2\times 2$ tile $T_s$ and the
$3 \times 3$ tile $T_\ell$. Theorem~\ref{thm:tilethm} gives a way, in principle,
to answer the question. Namely, we must see if for all sufficiently large
$p,q$ with $\gcd(p,q)=1$, whether we can tile (in the natural sense) the graph $\gnpq$.
A necessary condition for doing this is that we be able to tile each of the $12$
graphs $T^1_{n,p,q}, \dots, T^{12}_{n,p,q}$ individually.

First, unlike the case where we omit the small or large tile, we can now show
that $T_s$ and $T_\ell$ suffice to tile the $4$ torus tiles (e.g., $\Tcaca$)
for $n=1$ and $p,q$ large enough with $\gcd(p,q)=1$:

\begin{fact} \label{sltorus}
Let $p,q\geq 60$ with $\gcd(p,q)=1$. Then the $p \times q$ torus $\Tcaca$
can be tiled by $2 \times 2$ and $3 \times 3$ tiles.
\end{fact}

\begin{proof}
First of all, it is easy to see that for any positive integers $a$ and $b>1$, the $6a\times b$ and $b\times 6a$ grid graphs can be tiled by $T_s$ and $T_\ell$, the $2\times 2$ and $3\times 3$ tiles.

Let $\Lambda \subseteq \Z^2$ be the lattice generated by the vectors
$v_1=(3,2)$ and $v_2=(2,-3)$. $\Lambda$ has horizontal and vertical periods
of
$$\displaystyle\left| \det \begin{pmatrix} 3 & 2 \\ 2 & -3\end{pmatrix}\right|=13.$$
There is a tiling of $\Z^2$ by $T_s, T_\ell$ using the points of $L$
for the centers of the $T_\ell$ tiles and the points of $(2,0)+\Lambda$
for the upper-left corners of the $T_s$ tiles.
This is illustrated in Figure~\ref{fig:pt}. From the periodicity,
we can use restrictions of this tiling to tile any torus
of dimensions $13a\times 13b$, for positive integers $a, b$.

\begin{figure}[ht]
\centering
\begin{tikzpicture}[scale=0.35]

\tikzset{
lr/.pic={
\draw (0,0) rectangle (2,2);
}
}

\tikzset{
sr/.pic={
\draw (0,0) rectangle (1,1);
}
}

\foreach \i in {0,...,13}
{
\foreach \j in {0,...,13}
{
\draw[fill,radius=0.1,black] (1*\i,1*\j) circle;
}
}
\pic at  (0,0) [transform shape]  {lr};
\pic at (3,2) [transform shape] {lr};
\pic at (6,4) [transform shape] {lr};
\pic at (9,6) [transform shape] {lr};

\pic at (1,5) [transform shape] {lr};
\pic at (4,7) [transform shape] {lr};
\pic at (7,9) [transform shape] {lr};
\pic at (10,11) [transform shape] {lr};

\pic at (2,10) [transform shape] {lr};
\draw (5,13) to (5,12) to (7,12) to (7,13);

\pic at (8,1) [transform shape] {lr};
\pic at (11,3) [transform shape] {lr};

\pic at (3,0) [transform shape] {sr};
\pic at (6,2) [transform shape] {sr};
\pic at (9,4) [transform shape] {sr};
\pic at (12,6) [transform shape] {sr};
\pic at (1,3) [transform shape] {sr};
\pic at (4,5) [transform shape] {sr};
\pic at (7,7) [transform shape] {sr};
\pic at (10,9) [transform shape] {sr};
\pic at (2,8) [transform shape] {sr};
\pic at (5,10) [transform shape] {sr};
\pic at (8,12) [transform shape] {sr};
\pic at (11,1) [transform shape] {sr};
\pic at (0,11) [transform shape] {sr};

\draw (5,0) to (5,1) to (7,1) to (7,0);
\draw (8,0) to (9,0);
\draw (10,0) to (12,0);
\draw (13,8) to (12,8) to (12,10) to (13,10);
\draw (13,0) to (13,2);
\draw (13,11) to (13,12);
\draw (0,3) to (0,5);
\draw (0,6) to (0,7);
\draw (0,8) to (1,8) to (1,10) to (0,10);
\draw (0,13) to (2,13);
\draw (3,13) to (4,13);

\end{tikzpicture}
\caption{A periodic tiling by $T_s$, $T_\ell$.} \label{fig:pt}
\end{figure}
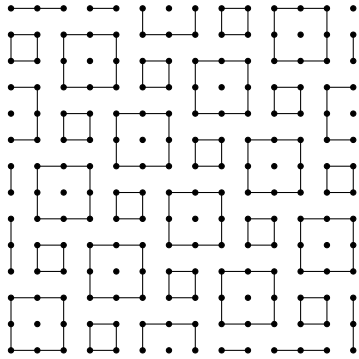

Next, we define a procedure by which we can stretch any torus tiling by a multiple
of $6$ in either direction. Consider the horizontal direction, and suppose we have a tiling
of a torus of dimensions $a \times b$. We describe a tiling of the torus of dimensions $(a+6c)\times b$. Draw a vertical line $L_0$ through
the torus. The line $L_0$ will be partitioned into segments where it intersects
various copies of the $T_s,T_\ell$ tiles (it may intersect the middle of a $T_\ell$
tile, or along one of its edges). Place another vertical line, $L_1$, $6c$ units
to the right of the line $L_0$. We place copies of $T_s$ and $T_\ell$ so that they intersect
$L_1$ in the same way that $L_0$ intersected copies of these tiles. To the right of
$L_1$ we will copy the original tiling of the $a\times b$ torus. Since $6$ is a multiple
of both $2$ and $3$, we may fill in the region between $L_0$ and $L_1$ as follows.
For each segment $w$ of $L_0$ which intersects a $T_\ell$ tile,
we add $2c-1$ copies of  $T_\ell$ tiles horizontally to its right until we reach
the corresponding tile intersecting $L_1$. We likewise proceed for each segment
of $L_0$ which intersects a $T_s$ tile, adding $3c-1$ copies of $T_s$ horizontally
to its right. Figure~\ref{fig:st} illustrates the stretching procedure.

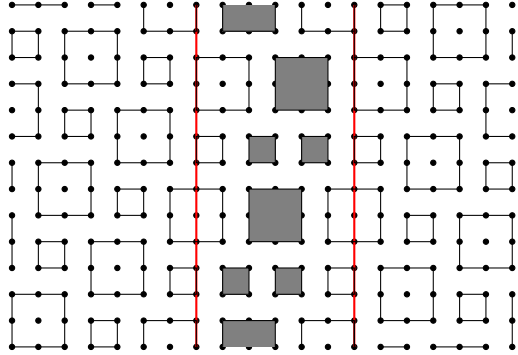
\begin{figure}[ht]
\centering
\begin{tikzpicture}[scale=0.35]

\tikzset{
lr/.pic={
\draw (0,0) rectangle (2,2);
}
}

\tikzset{
sr/.pic={
\draw (0,0) rectangle (1,1);
}
}

\tikzset{
lrf/.pic={
\draw[fill=gray] (0,0) rectangle (2,2);
}
}

\tikzset{
srf/.pic={
\draw[fill=gray] (0,0) rectangle (1,1);
}
}

\foreach \i in {0,...,19}
{
\foreach \j in {0,...,13}
{
\draw[fill,radius=0.1,black] (1*\i,1*\j) circle;
}
}

\draw[color=red,thick] (7,0) to (7,13);
\draw[color=red,thick] (13,0) to (13,13);

\pic at  (0,0) [transform shape]  {lr};
\pic at (3,2) [transform shape] {lr};
\pic at (6,4) [transform shape] {lr};
\pic at (15,6) [transform shape] {lr}; 

\pic at (1,5) [transform shape] {lr};
\pic at (4,7) [transform shape] {lr};
\pic at (7,9) [transform shape] {lr};
\pic at (16,11) [transform shape] {lr}; 

\pic at (2,10) [transform shape] {lr};
\draw (5,13) to (5,12) to (7,12) to (7,13);

\pic at (14,1) [transform shape] {lr}; 
\pic at (17,3) [transform shape] {lr}; 

\pic at (3,0) [transform shape] {sr};
\pic at (6,2) [transform shape] {sr};
\pic at (15,4) [transform shape] {sr}; 
\pic at (18,6) [transform shape] {sr}; 
\pic at (1,3) [transform shape] {sr};
\pic at (4,5) [transform shape] {sr};
\pic at (7,7) [transform shape] {sr};
\pic at (16,9) [transform shape] {sr}; 
\pic at (2,8) [transform shape] {sr};
\pic at (5,10) [transform shape] {sr};
\pic at (14,12) [transform shape] {sr}; 
\pic at (17,1) [transform shape] {sr}; 
\pic at (0,11) [transform shape] {sr};

\draw (5,0) to (5,1) to (7,1) to (7,0);
\draw (14,0) to (15,0); 
\draw (16,0) to (18,0); 
\draw (19,8) to (18,8) to (18,10) to (19,10); 
\draw (13,0) to (13,1);
\draw (13,11) to (13,12);
\draw (0,3) to (0,5);
\draw (0,6) to (0,7);
\draw (0,8) to (1,8) to (1,10) to (0,10);
\draw (0,13) to (2,13);
\draw (3,13) to (4,13);

\pic at (13,9) [transform shape] {lr};
\pic at (13,7) [transform shape] {sr};
\pic  at (12,4) [transform shape] {lr};
\pic at (12,2) [transform shape] {sr};
\draw (11,0) to (11,1) to (13,1) to (13,0);
\draw (11,13) to (11,12) to (13,12) to (13,13);
\draw (19,11) to (19,12);
\draw (19,0) to (19,2);

\draw[fill=gray] (8,13) to (8,12) to (10,12) to (10,13);
\pic at (10,9) [transform shape] {lrf};
\pic at (9,7) [transform shape] {srf};
\pic at (11,7) [transform shape] {srf};
\pic at (9,4) [transform shape] {lrf};
\pic at (8,2) [transform shape] {srf};
\pic at (10,2) [transform shape] {srf};
\draw[fill=gray] (8,0) to (8,1) to (10,1) to (10,0);

\end{tikzpicture}
\caption{Stretching the torus tiling.} \label{fig:st}
\end{figure}

The above algorithm, together with the observation at the beginning of this proof, allow us to tile any torus of dimensions
$(13a+6b)\times (13c+6d)$ for nonnegative integers $a,b,c,d$ where $a+b, c+d>0$. Since $\gcd(13,6)=1$,
this shows any $p \times q$ torus tile for $p,q$ large enough ($p,q \geq 60$) can be tiled by
$T_s$ and $T_\ell$, which proves the fact.
\end{proof}

We note that some smaller values of
$p,q$ can also work for different reasons. For example, the $17 \times 19$ torus can be tiled by $2\times 2$ and
$3\times 3$ tiles as shown in Figure~\ref{fig:et} (this tiling was discovered by a computer program),
even though $17$ is not of the form $13a+6b$.

It seems likely that the tiling problem for continuous tilings of $\fzs$ by
$T_s$ and $T_\ell$ hinges on whether the ``long tiles'' such as
$\Tcqadpa$ of Figure~\ref{fig:Gamma-npq-horiz} can be tiled by $T_s, T_\ell$
(for large enough $p,q$). We do not know the answer to this question, but some
computer experiments done by the authors for values of $p,q \leq 50$ were unable
to find a tiling of the long tiles corresponding to these values of $p,q$.

\begin{figure}[htbp]
\centering
\begin{tikzpicture}[scale=0.35]

\tikzset{
lr/.pic={
\draw (0,0) rectangle (2,2);
}
}

\tikzset{
sr/.pic={
\draw (0,0) rectangle (1,1);
}
}

\foreach \i in {0,...,17}
{
\foreach \j in {0,...,19}
{
\draw[fill,radius=0.1,black] (1*\i,1*\j) circle;
}
}

\pic at  (5,0) [transform shape]  {lr};
\pic at  (13,1) [transform shape]  {lr};
\pic at  (8,2) [transform shape]  {lr};
\pic at  (3,3) [transform shape]  {lr};
\pic at  (11,4) [transform shape]  {lr};
\pic at  (6,5) [transform shape]  {lr};
\pic at  (1,6) [transform shape]  {lr};
\pic at  (9,7) [transform shape]  {lr};
\pic at  (4,8) [transform shape]  {lr};
\pic at  (7,10) [transform shape]  {lr};
\pic at  (2,11) [transform shape]  {lr};
\pic at  (14,12) [transform shape]  {lr};
\pic at  (7,13) [transform shape]  {lr};
\pic at  (0,14) [transform shape]  {lr};
\pic at  (12,15) [transform shape]  {lr};
\pic at  (7,16) [transform shape]  {lr};
\pic at  (15,17) [transform shape]  {lr};

\pic at (1,0) [transform shape] {sr};
\pic at (8,0) [transform shape] {sr};
\pic at (3,1) [transform shape] {sr};
\pic at (16,1) [transform shape] {sr};
\pic at (1,2) [transform shape] {sr};
\pic at (11,2) [transform shape] {sr};
\pic at (16,3) [transform shape] {sr};
\pic at (6,3) [transform shape] {sr};
\pic at (1,4) [transform shape] {sr};
\pic at (14,4) [transform shape] {sr};
\pic at (9,5) [transform shape] {sr};
\pic at (16,5) [transform shape] {sr};
\pic at (4,6) [transform shape] {sr};
\pic at (14,6) [transform shape] {sr};
\pic at (12,7) [transform shape] {sr};
\pic at (16,7) [transform shape] {sr};
\pic at (7,8) [transform shape] {sr};
\pic at (14,8) [transform shape] {sr};
\pic at (2,9) [transform shape] {sr};
\pic at (12,9) [transform shape] {sr};
\pic at (10,10) [transform shape] {sr};
\pic at (14,10) [transform shape] {sr};
\pic at (5,11) [transform shape] {sr};
\pic at (12,11) [transform shape] {sr};
\pic at (0,12) [transform shape] {sr};
\pic at (10,12) [transform shape] {sr};
\pic at (5,13) [transform shape] {sr};
\pic at (12,13) [transform shape] {sr};
\pic at (3,14) [transform shape] {sr};
\pic at (10,14) [transform shape] {sr};
\pic at (5,15) [transform shape] {sr};
\pic at (15,15) [transform shape] {sr};
\pic at (3,16) [transform shape] {sr};
\pic at (10,16) [transform shape] {sr};
\pic at (1,17) [transform shape] {sr};
\pic at (5,17) [transform shape] {sr};
\pic at (3,18) [transform shape] {sr};
\pic at (13,18) [transform shape] {sr};

\draw (0,1) to (0,2);
\draw (13,0) to (14,0);
\draw (15,0) to (17,0);
\draw (0,3) to (0,4);
\draw (0,5) to (0,6);
\draw (0,7) to (0,8);
\draw (0,9) to (1,9) to (1,11) to (0,11);
\draw (17,9) to (16,9) to (16,11) to (17,11);
\draw (10,0) to (10,1) to (12,1) to (12,0);
\draw (10,19) to (10,18) to (12,18) to (12,19);
\draw (0,17) to (0,19);
\draw (1,19) to (2,19);
\draw (3,0) to (4,0);
\draw (5,19) to (7,19);
\draw (8,19) to (9,19);
\draw (17,12) to (17,13);
\draw (17,14) to (17,16);

\end{tikzpicture}
\caption{A tiling of the $17 \times 19$ torus tile.} \label{fig:et}
\end{figure}
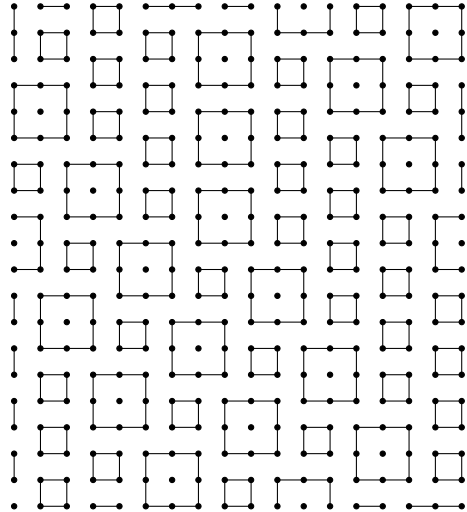

\bibliographystyle{plain}
\bibliography{continuous}

\end{document}

\begin{exn}
Consider the case of $\varphi\signatureSep \Gamma\to S$ and
$\psi\signatureSep \Delta\to S$ with $\psi\le\varphi$, where $\Gamma$
is the Cayley graph of $\scriptstyle{\left(\faktor{\Z}{5\Z}\right)}^2$
(i.e., a $5\times 5$ torus graph) and $\Delta$ is the Cayley graph of
$\Z^2$, both as $\Z^2$-graphs. Here, unlike the case in Example~\ref{ex:pullbacks-restriction},
$\Delta$ is much larger than $\Gamma$.  Because of the $\Z^2$-graph
structure, if $\varphi\signatureSep \Gamma\to S$ is a map and
$\psi\signatureSep \Delta\to S$ is a pullback of $\Gamma$, then $\psi$
is a doubly periodic tiling of $\Z^2$ by $\varphi$.  Indeed, if
$\psi=\varphi\circ\sigma$, where
$\sigma\signatureSep\Z^2\to\scriptstyle{\left(\faktor{\Z}{5\Z}\right)}^2$,
then $\sigma$ is determined entirely by $\sigma(g)$ for any one
$g\in\Z^2$, thus there are exactly 25 possible such $\sigma$.  This
phenomenon, where $\Z^2$-pullbacks of certain graphs are simply
tilings, will be exploited later when building tilings.
\end{exn}

\end{comment}